\documentclass[a4paper, openany]{book}
\usepackage[utf8]{inputenc}
\usepackage[english]{babel}
\usepackage{paralist,url,verbatim,anysize}
\usepackage{amscd}
\usepackage{color}
\usepackage[usenames,dvipsnames,svgnames,table]{xcolor}
\definecolor{purple}{HTML}{961C8C}
\usepackage{setspace}

\usepackage{mathrsfs}
\usepackage{xparse}
\usepackage{suffix}
\usepackage{microtype}
\usepackage{manfnt}
\usepackage{futitle}
\usepackage{nicefrac}
\usepackage[e]{esvect}
\usepackage{amsmath,amsthm,amssymb,amsfonts}
\usepackage{graphicx,graphics}
\usepackage[bookmarksopen=true, backref=page, linktoc=page]{hyperref}
\hypersetup{colorlinks=true, citecolor=blue, urlcolor=TealBlue}
\usepackage{bbm}
\usepackage[font=small,labelfont=bf]{caption}
\usepackage[labelfont=rm, textfont=rm, labelformat=simple]{subcaption}

\makeatletter
\renewcommand\part{%
  \if@openright
    \cleardoublepage
  \else
    \clearpage
  \fi
  \thispagestyle{empty}%
  \if@twocolumn
    \onecolumn
    \@tempswatrue
  \else	
    \@tempswafalse
  \fi
  \null\vfil
  \secdef\@part\@spart}
\makeatother

\makeatletter

\let\old@part\part

\WithSuffix\def\part*{
  \ifx\next[%
    \let\next\thesis@part@star%
  \else
    \def\next{\thesis@part@star[]}%
  \fi
  \next}

\def\thesis@part@star[#1]#2{
  \ifthenelse{\equal{#1}{}}
   {
    \def\thesis@part@short{#2}
    \old@part*{#2}}
   {
    \def\thesis@part@short{#1}
    \old@part*[#1]{#2}}

  \addcontentsline{toc}{part}{\thesis@part@short}
}

\makeatother

\theoremstyle{plain}
\newtheorem{theorem}{\bf Theorem}[section]
\newtheorem{conjecture}[theorem]{Conjecture}
\newtheorem*{theorem*}{Theorem}
\newtheorem*{conjecture*}{Conjecture}
\newtheorem*{problem*}{Problem}
\newtheorem{cor}[theorem]{Corollary}
\newtheorem{lemma}[theorem]{Lemma}
\newtheorem{prp}[theorem]{Proposition}
\newtheorem{mthm}{\bf Main Theorem}

\newtheorem{problemm}{\bf Problem}

\newtheorem{problem}[theorem]{Problem}
\theoremstyle{definition}
\newtheorem{rem}[theorem]{Remark}
\newtheorem{definition}[theorem]{Definition}
\newtheorem{obs}[theorem]{Observation}

\newtheorem{example}[theorem]{Example}
\newtheorem{fml}[theorem]{Formula}
\newtheorem{examples}[theorem]{Examples}
\DeclareMathAlphabet{\mathpzc}{OT1}{pzc}{m}{it}
\newenvironment{frcseries}{\fontfamily{frc}\selectfont}{}

\newcommand{\size}{\operatorname{size} }  
\newcommand{\conv}{\operatorname{conv} }
\newcommand{\Sp}{\operatorname{sp} }
\newcommand{\rot}{{\operatorname{r}} }
\newcommand{\spi}{\operatorname{s} }
\newcommand\diam{\operatorname{diam}}
\newcommand{\sd}{\operatorname{sd} }
\newcommand{\SSp}{\operatorname{sp}^s }
\newcommand{\RR}{\operatorname{R} }
\newcommand{\F}{\operatorname{F} }
\newcommand\Defn[1]{\emph{\color{RubineRed}#1}}

\newcommand{\cl}{\operatorname{cl} }
\newcommand{\fat}{\operatorname{fat} }
\newcommand{\R}{\mathbb{R}}
\newcommand{\N}{\mathbb{N}}
\newcommand{\intx}{\operatorname{int} }
\newcommand{\rint}{\operatorname{relint} }
\newcommand{\Z}{\mathbb{Z}}
\newcommand{\pp}{\operatorname{p} }
\newcommand{\n}{\vv{n}}
\newcommand\CF{\mathrm{C}}
\newcommand{\CT}{\operatorname{T}}
\newcommand{\PS}{\CT^{s}}
\newcommand{\eq}{{\mathrm{eq}}}
\newcommand{\e}{\varepsilon}
\newcommand{\T}{\mathfrak{T}}
\newcommand{\RC}{\mathfrak{C}}

\newcommand{\aff}{\mathrm{aff} }
\newcommand{\cR}{\mathcal{RS}}
\newcommand{\cm}[1]{}
\newcommand{\bigslant}[2]{{\raisebox{.3em}{$#1$} \Big/ \raisebox{-.3em}{$#2$}}}
\newcommand{\Lk}{\mathrm{Lk}}
\newcommand{\St}{\mathrm{St}}

\newcommand{\NE}{\operatorname{NE} }

\newcommand{\CAT}{\mathrm{CAT}}
\DeclareFontFamily{OT1}{pzc}{}
\DeclareFontShape{OT1}{pzc}{m}{it}{<-> s * [1.2] pzcmi7t}{}
\newcommand{\ii}[1]{\iota(#1)}
\newcommand{\comp}{\operatorname{c}}
\newcommand{\CC}{\mathbb{C}}

\newcommand{\HA}{\mathscr{A}}

\newcommand{\EH}{\HA_e}
\newcommand{\SH}{\HA_\sigma}
\newcommand{\io}{(\operatorname{A})}
\newcommand{\iit}{(\operatorname{B})}
\newcommand{\iii}{(\operatorname{C})}
\newcommand{\K}{\mathscr{K}}
\newcommand{\FD}{F}
\newcommand{\OD}{O}
\newcommand{\s}{\mathbf{s}}

\newcommand{\cc}{2}
\begin{document}
\pagenumbering{roman}
\pagestyle{empty}
\title{{\Huge {\bf   Methods from Differential Geometry in Polytope Theory}}}
\author{Karim Alexander Adiprasito \\ Institut f\"ur Mathematik,\\ Freie Universit\"at Berlin \\
Betreuer und erster Gutacher: Prof.\ G\"unter M.\ Ziegler, PhD\\
Zweiter Betreuer und Gutachter: Prof.\ Gil Kalai, PhD}
\department{Fachbereich Mathematik und Informatik}
\doctortype{Doktor der Mathematik}
\date{12.\ Juni 2013\\
Datum der Disputation: \\
27. Mai 2013}
\maketitle 
\addtocontents{toc}{\protect\thispagestyle{empty}} 
\newpage
\mbox{} 
\newpage
\vspace*{\fill}
\begingroup
\centering
{\Huge {\bf Methods from Differential Geometry in Polytope Theory}}\\
\bigskip
\bigskip
	
\bigskip
{\huge by Karim Alexander Adiprasito, \\ Dipl.\ Math.}\\
\bigskip
\bigskip
\bigskip
{\huge Advisor and first referee:\ Prof.\ G\"unter M.\ Ziegler,\ PhD}\\
\bigskip
{\huge Second advisor and referee: Prof.\ Gil Kalai, PhD}\\
\bigskip
\bigskip
\begin{figure*}[h!tbf]  
\centering 
\includegraphics[width=0.74\linewidth]{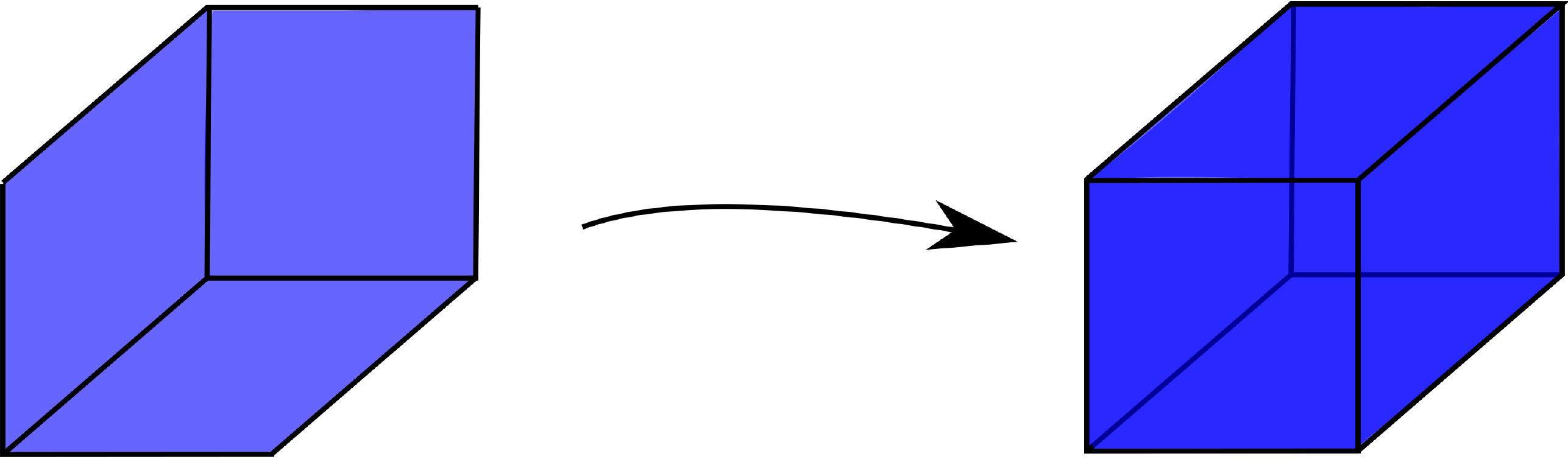}
\end{figure*}
\bigskip
\bigskip
{\huge Department of Mathematics}\\
\bigskip
{\huge Freie Universit\"at Berlin}\\
\bigskip
{\huge 12.\ Juni 2013}\\
\bigskip
{\huge Date of Defense:} \\
\bigskip
{\huge 27. Mai 2013}\\
\bigskip
\endgroup
\vspace*{\fill}
\newpage
\mbox{}
\newpage
\tableofcontents 
\thispagestyle{empty}
\setstretch{1.15}

\part*{Preface}

\chapter*{Introduction}
\addcontentsline{toc}{chapter}{Introduction}
\pagestyle{plain}

Combinatorial topology is the attempt to understand topological spaces by combinatorial means. Indeed, early results in algebraic topology rely heavily on combinatorial intuition and methods. Compare for instance the classical formulation and proof of Poincar\'e duality, Whitehead's simple homotopy theory or the beginnings of PL topology. However, the continued development of algebraic topology also led to the discovery that combinatorial methods were often not flexible enough to capture all aspects of a topological space, leading to the development of other methods.

\smallskip

The last decades, with the rising influence of computer science, geometric group theory, low-dimensional topology and discrete geometry, have witnessed a comeback in the study of the connection between combinatorics and topology. Furthermore, new theories have contributed to the field, predominantly metric and differential geometry. This lead to astonishing progress in various fields, and new methods have been found to analyze the fruitful connections between combinatorics, topology, and geometry.

\smallskip

The purpose of this thesis is to study classical objects, such as polytopes, polytopal complexes, and subspace arrangements. We will tackle problems, old and new, concerning them. We do so by using some of the new tools that have been developed in combinatorial topology, especially those tools developed in connection with (discrete) differential geometry, geometric group theory and low-dimensional topology. 

\smallskip

This thesis is divided into three parts: In {\bf Part}~\ref{pt:metgeo}, we use the theory of Alexandrov spaces of bounded curvature to prove upper bounds on the diameter of simplicial complexes (Main Theorem~\ref{mthm:diam}). In particular, we prove the Hirsch conjecture for a wide class of simplicial complexes. This is a simple toy example for the application of metric intuition to combinatorics, and should serve as an introduction to the underlying idea of this thesis. 

\smallskip

The main focus of {\bf Part}~\ref{pt:dmt} is to use Forman's discretization of Morse theory to study complements of subspace arrangements in $\R^d$. With this elementary tool, we are able to generalize results whose proofs previously required deep algebraic geometry (Main Theorem~\ref{mthm:mini}). In particular, guided by the theory of complex arrangements and their Milnor fibers, we establish an analogue of the Lefschetz hyperplane theorem for complements of real arrangements (Main Theorem~\ref{mthm:lef}). We start this part of the thesis by presenting progress on conjectures and problems of Lickorish, Kirby, Chillingworth and others (Main Theorems~\ref{mthm:g},~\ref{mthm:l} and~\ref{mthm:h}).

\smallskip

Finally, {\bf Part}~\ref{pt:pup} contains arguably the most beautiful and demanding instances of applications of methods from differential geometry to combinatorics presented in this thesis: We answer a problem going back to the research of Steinitz and Legendre, and resolve a classical problem of Perles, Shephard and others. For this, we use several (discretized) tools from classical differential geometry: Among others, the theory of conjugate nets and the theory of locally convex hypersurfaces with boundary are used to establish Main Theorems~\ref{mthm:mld} and~\ref{mthm:mpu}. Some of the techniques developed are furthermore used to prove a universality property of projectively unique polytopes (Main Theorem~\ref{mthm:upu}), which in turn provides a counterexample to an old conjecture by Shephard (Main Theorem~\ref{mthm:stp}).

\newpage

\noindent{\bf{\large Part~\ref{pt:metgeo}: A toy example: Metric geometry and the Hirsch conjecture}
 
\smallskip 
 
\noindent{\bf Chapter~\ref{ch:Hirsch}: The Hirsch conjecture for normal flag complexes}}

\medskip

The simplex algorithm is an important method in linear programming. To estimate its performance, Warren Hirsch (cf.\ \cite{Dantzig}) proposed what is known today as the Hirsch conjecture for the diameter of polyhedra. While the case of unbounded polyhedra was quickly proven false by Klee and Walkup \cite{KleeWalkup}, the case of bounded polyhedra (i.e.\ polytopes) remained an important open problem.

\begin{conjecture*}[(Bounded) Hirsch conjecture \cite{KleeWalkup}]
The diameter of the facet-ridge graph of any $d$-polytope on $n$ vertices is $\le n-d$.
\end{conjecture*}

Recently, this last remaining classical case of the Hirsch conjecture was proven wrong by Santos~\cite{Santos}. Nevertheless, establishing upper bounds for the diameter of polytopes is still an important and fruitful subject. A nice example for an upper bound on the diameter is the following.

\begin{theorem*}[Provan {\&} Billera {\cite{BP, PB}}]
If $C$ is the boundary complex of any polytope, then the derived subdivision of $C$ satisfies the Hirsch diameter bound.
\end{theorem*}

This result can be generalized using curvature bounds on polytopal complexes. Recall that a simplicial complex is \emph{normal} if the link of every face of codimension at least $2$ is connected. Furthermore recall that the derived subdivision of any simplicial complex is \emph{flag}, that is, it coincides with the clique complex of its $1$-skeleton.

\begin{mthm}
\label{mthm:diam}
Let $C$ be a normal flag simplicial complex. Then $C$ satisfies the non-revisiting path property, and in particular the Hirsch diameter bound.
\end{mthm}\enlargethispage{3mm}

The idea for the proof of this result is to endow $C$ with a metric length structure (cf.\ \cite{BuragoBuragoIvanov}), choose a suitable geodesic segment between any pair of facets, and then construct the desired path of facets along~it. A simple argument then proves that the constructed path satisfies the Hirsch diameter bound. We give two proofs of Main Theorem~\ref{mthm:diam}: A geometric one (Section~\ref{sec:geomproof}) and a combinatorial one (Section~\ref{sec:combproof}).

\smallskip
\noindent\emph{The results of this chapter are found in \cite{Hirsch}.}

\medskip

\noindent{\bf{\large Part~\ref{pt:dmt}: Discrete Morse theory for stratified spaces}
 
\smallskip 
 
\noindent{\bf Chapter~\ref{ch:convcollapse}: Collapsibility of convex and star-shaped simplicial complexes}}

\medskip
\emph{Simplicial collapsibility} is a notion due to Whitehead~\cite{Whitehead}, and lies at the very core of PL topology. It is often used to simplify a simplicial complex while preserving its homotopy type. It has important applications, for example to the generalized Poincar\'e conjecture \cite{Stallings, ZeemanP}. However, few criteria are known that guarantee the existence of a collapsing sequence. A classical conjecture concerning this problem was first formulated by Lickorish:

\begin{conjecture*}[Lickorish {\cite[Pr.\ 5.5 (A)]{Kirby}}]
Any simplicial subdivision of a $d$-simplex is collapsible.
\end{conjecture*}

Goodrick \cite{Goodrick} proposed a bolder conjecture:

\begin{conjecture*}[Goodrick {\cite[Pr.\ 5.5 (B)]{Kirby}}]
Any simplicial complex in $\R^d$ whose underlying set is star-shaped is collapsible.
\end{conjecture*}

The conjecture of Lickorish was verified up to dimension $3$ by Chillingworth \cite{CHIL}; Goodrick's conjecture is only proven in dimension $2$ and lower. On the other hand, triangulations of topological balls are not collapsible in general~\cite{Bing}. A motivation for the conjecture of Lickorish is the behavior of collapsibility under subdivisions. 

\begin{problem*}[Hudson~{\cite[Sec.\ 2, p.\ 44]{Hudson}}]
\label{pro:2} Is collapsibility preserved under subdivisions? If $C$ is a simplicial complex that collapses to a subcomplex $C'$, is it also true that any simplicial subdivision $D$ of $C$ collapses to the restriction of $D$ to the underlying space of $C'$?\end{problem*}

The elementary approach of collapsing along a linear functional, which proves collapsibility of polytopes, does not seem to work here. However, Whitehead~\cite{Whitehead} proved that the problems have a positive solution if one allows for additional stellar subdivisions. We provide further progress towards a positive solution. Let $\sd^r$ denote the $r$-th derived subdivision. Concerning Goodrick's conjecture, we prove the following:

\begin{mthm}
\label{mthm:g}
If $C$ is any polytopal complex in $\R^d$, $d\geq 3$, whose underlying space is star-shaped, then $\sd^{d-2} C$ is collapsible.
\end{mthm}

Furthermore, we provide progress on Lickorish's conjecture and Hudson's problem.

\begin{mthm}
\label{mthm:l}
If $C$ is any subdivision of a $d$-polytope, then $\sd C$ is collapsible.
\end{mthm}

\begin{mthm}
\label{mthm:h}
Let $C$ be a polytopal complex that collapses to a subcomplex $C'$, and let $D$ be any subdivision of $C$. If $D'$ denotes the restriction of $D$ to the underlying space of $C'$, then $\sd D$ collapses to~$\sd D'$.
\end{mthm}

\noindent\emph{Main Theorem \ref{mthm:g} is joint work with B.\ Benedetti. All results of this chapter are found in~\cite{AB-MGC}.}

\medskip
 
\noindent{\bf Chapter~\ref{ch:minimal}: Combinatorial Stratifications and minimality of $2$-arrangements}

\medskip

A \emph{$c$-arrangement} $\HA$ is a collection of $(d-c)$-dimensional affine subspaces in $\R^d$ such that the codimension of every nonempty intersection of elements of $\HA$ is divisible by $c$. The class of $c$-arrangements, $c=2$, includes in particular the class of complex hyperplane arrangements. A classical line of problems is concerned with describing the topology of the complement $\HA^{\comp}=\R^d{\setminus}\HA$ of a $c$-arrangement. 

The homological structure of $\HA^{\comp}$ is understood by the formula of Goresky and MacPherson~\cite{GM-SMT}; however it is far more challenging to understand the homotopy type of $\HA^{\comp}$. In this chapter we discuss whether the complements of $c$-arrangements are \emph{minimal}. A topological space is minimal if there exists a CW complex homotopy equivalent to it with number of $i$-cells equal to the $i$-th (rational) Betti number of the space. Previous results include the following:

\begin{compactenum}[(1)]
\item The complement of any complex hyperplane arrangement is a minimal space. (This was shown in a long line of works by Hattori \cite{Hattori}, Falk \cite{Falk} and finally Dimca--Papadima \cite{DimcaPapadima}, among others.)
\item The complement of any $c$-arrangement, $c\neq 2$, is a minimal space. (This follows easily from the work of Goresky and MacPherson, together with classical cellular approximation results.)
\item The complement of a general subspace arrangement can have arbitrary torsion in cohomology, and is thus not a minimal space in general. (This follows from the formula of Goresky and MacPherson.)
\item Complements of $c$-pseudoarrangements (cf.\ \cite[Sec.\ 8 \& 9]{BjZie}) are minimal if $c\neq 2$, but not necessarily minimal if $c=2$. 
(This follows with an argumentation similar to point (2).)
\end{compactenum}

\noindent This leaves us with the task to examine affine $2$-arrangements in $\R^d$. We prove the following:

\begin{mthm}
\label{mthm:mini}
The complement $\HA^{\comp}$ of any affine $2$-arrangement $\HA$ in $\R^d$ is minimal. 
\end{mthm}

For the proof, we use discrete Morse theory on the combinatorial stratifications of $\R^d$, as defined by Bj\"orner--Ziegler \cite{BjZie}, together with suitable notions of Alexander duality and Poincar\'e duality for discrete Morse functions. We also establish the following Lefschetz-type theorem for complements of $2$-arrangements:

\begin{mthm}
\label{mthm:lef}
Let $\HA^{\comp}$ denote the complement of any affine $\cc$-arrangement $\HA$ in $\R^d$, and let $H$ be any hyperplane in $\R^d$ in general position with respect to $\HA$. Then $\HA^{\comp}$ is homotopy equivalent to $H\cap\HA^{\comp}$ with finitely many $\lceil\nicefrac{d}{\cc}\rceil$-dimensional cells attached.
\end{mthm}

Using this theorem, a simple inductive argument finishes the proof of Main Theorem~\ref{mthm:mini}.

\smallskip
\noindent\emph{The results of this chapter are found in \cite{A}.}

\medskip

\newpage

\noindent{\bf{\large Part~\ref{pt:pup}: Constructions for projectively unique polytopes}
 
\smallskip 
 
\noindent{\bf Chapter~\ref{ch:substacked}: Subpolytopes of stacked polytopes}}

\medskip

In this chapter, we provide a construction method for projectively unique polytopes. As an application, we provide the following universality theorem.
\begin{mthm}
\label{mthm:upu}
If $P$ is any polytope with rational vertex coordinates, then there exists a polytope $P'$ that is projectively unique and has a face projectively equivalent to $P$. 
\end{mthm}

Encouraged by this result, we consider an a question of Shephard, who asked whether every polytope is a \emph{subpolytope} of some stacked polytope, i.e.\ whether it can be obtained as the convex hull of a subset of the vertices of some stacked polytope. Shephard proved this wrong in~\cite{Shephard74}, but he conjectured it to be true in a combinatorial sense:

\begin{conjecture*}[Shephard~\cite{Shephard74}, Kalai {\cite[p.\ 468]{Kalai}}, \cite{KalaiKyoto}]
Every combinatorial type of polytope can be realized as a subpolytope of some stacked polytope.
\end{conjecture*}
 
The idea to disprove the conjecture is to establish the existence of a polytope that is not the subpolytope of any stacked polytope (as Shephard does), but that has the additional property that its realization is unique up to projective transformation. A refined method gives the following result.

\begin{mthm}
\label{mthm:stp}
There is a combinatorial type of $5$-dimensional polytope that is not realizable as the subpolytope of any stacked polytope. 
\end{mthm}

\noindent\emph{The results of this chapter are joint work with A.\ Padrol, and are found in \cite{AP}.}

\medskip

\noindent{\bf Chapter~\ref{ch:lowdim}: Many polytopes with low-dimensional realization space}
\medskip

Realization spaces of polytopes are a long-studied subject, with many challenging problems that remain open to the present day. Intuitively speaking, the \emph{realization space} $\cR(P)\subset \R^{d\times n}$ of a $d$-polytope $P$ on $n$ vertices is the set of all possible vertex coordinates of polytopes combinatorially equivalent to $P$. One of the first questions posed on realization spaces goes back to Legendre's influential 1794 monograph \cite{Legendre}:

\begin{problem*}[Legendre \cite{Legendre}]
How can one determine the dimension of the realization space of a polytope in terms of its $f$-vector?
\end{problem*}

Legendre and Steinitz \cite{steinenc} solved this problem for dimension $3$ (for $d\leq 2$, it is clear). Much later, in the 1960's, Perles and Shephard considered a problem of a similar flavor: A polytope $P$ in $\R^d$ is \emph{projectively unique} if the projective linear group $\mathrm{PGL}(\R^{d+1})$ acts transitively on $\cR(P)$. Na\"ive intuition would suggest that projectively unique polytopes should have few vertices, or more accurately, that for every fixed dimension, there are only a finite number of projectively unique polytopes (up to combinatorial equivalence). Indeed, for $d\leq3$, this follows from the formulas available for $\dim \cR(P)$. Consequently, Perles and Shephard \cite{PerlesShephard} asked the following.

\begin{problem*}[Perles \& Shephard \cite{PerlesShephard}, Kalai \cite{Kalai}]
Is it true that, for each $d\geq 2$, the number of distinct combinatorial types of projectively unique $d$-polytopes is finite?
\end{problem*}

By using several ideas from discrete differential geometry and the theory of convex hypersurfaces, we obtain:

\begin{mthm}
\label{mthm:mld}
For each $d\geq 4$, there exists an infinite family  of combinatorially distinct $d$-dimensional polytopes $\operatorname{CCTP}_d[n]$ such that $\dim \cR (\operatorname{CCTP}_d[n])\leq76+d(d+1)$ for all $n\geq 1$.
\end{mthm}

In particular, for each $d\geq 4$, only trivial bounds can be given for $\dim \cR (P)$ in terms of a nonnegative combination of entries of the $f$-vector of a $d$-polytope $P$. Building on this result, we proceed to answer the problem of Perles and Shephard negatively.

\begin{mthm}
\label{mthm:mpu}
For each $d\geq 69$, there exists an infinite family of combinatorially distinct $d$-dimensional polytopes $\operatorname{PCCTP}_{d} [n],\, n\geq 1$, all of which are projectively unique.
\end{mthm}

\noindent\emph{The results of this chapter are joint work with G.\ M.\ Ziegler, and are found in \cite{AZ12}.}

\section*{Acknowledgments} 
\addcontentsline{toc}{chapter}{Acknowledgments}

This thesis would not have been possible without the support, advice and guidance of {\bf G\"unter Ziegler}, whom I cannot possibly thank enough. Special thanks also to {\bf Gil Kalai}, who invited me to Jerusalem on several occasions, agreed to be my second advisor, and provided many problems for me to think about. Many thanks to {\bf Igor Pak} for fruitful collaboration and inviting me to visit him at UCLA, to {\bf Raman Sanyal} for countless hours of talking math and spending part of his Miller Fellowship to get me to UC Berkeley, {\bf Bruno Benedetti} for successful collaboration on several papers and inviting me to Stockholm on several occasions and {\bf Tudor Zamfirescu} for guidance, proofreading my first papers and introducing me to research mathematics. My thanks to {\bf Pavle Blagojevi\'c} for many inspiring conversations on topology, geometry and combinatorics. Many thanks also to {\bf Alexander Samorodnitsky}, {\bf Nati Linial}, {\bf Itai Benjamini} and {\bf Eran Nevo}, for many hours of 
discussing math while I was in Israel. Thanks also to {\bf Isabella Novik}, {\bf Uli Wagner} and {\bf Alex Engstr\"om} for discussing math with me, and to {\bf Elke Pose} for helping me with bureaucracy so many times. Thanks to my coauthor {\bf Arnau Padrol} for his persistence in making me think about Shephard's conjecture and providing some of the figures for Chapter~\ref{ch:substacked} (Figures~\ref{fig:Q31},~\ref{fig:Qd1} and \ref{fig:stacked}), and to {\bf Miriam Schl\"oter} for providing some of the figures of Chapters~\ref{ch:substacked} and~\ref{ch:lowdim} (Figures~\ref{fig:lawrence},~\ref{fig:subdirect} and~\ref{fig:cube}). During my studies, the mathematical institutes of {\bf Technische Universit\"at Berlin}, {\bf Freie Universit\"at Berlin}, {\bf UC Berkeley}, {\bf Hebrew University of Jerusalem}, {\bf UCLA}, and {\bf KTH Stockholm}, the {\bf Mathematisches Forschungsinstitut Oberwolfach} and the {\bf Institut Mittag Leffler} provided great places for me to work at. I was kindly supported by the {\bf 
Deutsche Forschungsgemeinschaft} within the research training group ``Methods for Discrete Structures'' (GRK1408), and I thank the members of the {\bf RTG Methods for Discrete Structures} for support and many new impulses, with special thanks to my BMS Mentor {\bf Martin Skutella} and administrative assistant {\bf Dorothea Kiefer}. Last but not least, I thank {\bf Kaie Kubjas}, {\bf Marie-Sophie Litz} and my love {\bf Jasmin Matz} who, additionally to my coauthors, proofread sections of this thesis.

\medskip

I emphatically want to thank again my coauthors: Part of Chapter~\ref{ch:convcollapse} is joint work with {\bf Bruno Benedetti}, Chapter~\ref{ch:substacked} is joint work with {\bf Arnau Padrol}, and Chapter~\ref{ch:lowdim} is joint work with {\bf G\"unter M. Ziegler}.

{\chapter*{Basic set-up}
\addcontentsline{toc}{chapter}{Basic set-up}}

\vskip -3mm
In this section, we outline the basic notions we shall need throughout this thesis. We base our notation for polytopes and geometric polytopal complexes on the books \cite{Grunbaum, RourkeSanders, Z}. Concerning metric geometry and the intrinsic geometry of polytopal complexes, we base our notation on the second and third section of \cite{BuragoBuragoIvanov} and the second section of \cite{DM-NP}. We will also assume a basic knowledge in topology (which we will not review here); for this, we refer the reader to the books \cite{Hatcher, Munkres}.

\smallskip

The sphere $S^d$ shall always be considered as the unit sphere in $\R^{d+1}$ (with midpoint at the origin), and we consider euclidean space $\R^d$ and the sphere $S^d$ to be endowed with their canonical metrics of uniform sectional curvature.
A set $M\subset S^d$ is \Defn{convex} if any two points in it can be connected by a subarc of a great circle of length $\le \pi$ that lies entirely in $M$; by convention, any subset of $S^0$ is convex. A \Defn{(euclidean) polytope} in $\R^d$ is the convex hull of finitely many points in~$\R^d$. A \Defn{(spherical) polytope} in $S^d$ is the convex hull of a finite number of points in some open hemisphere of $S^d$. A \Defn{(euclidean) polyhedron} in $\R^d$ is the intersection of a finite number of closed halfspaces in $\R^d$.  Analogously, a \Defn{(spherical) polyhedron} in $S^d$ is the intersection of a finite number of closed hemispheres in $S^d$. Polytopes in a fixed open hemisphere $\OD$ of $S^d$ are in one-to-one correspondence with polytopes in $\R^d$ via \Defn{radial projections} (from $\R^n$ to $\OD\subset S^d$) and \Defn{central} or \Defn{gnomonic projections} (from $\OD$ to $\R^d$). Thus, the theories of spherical and euclidean polytopes coincide. We refer to \cite{Grunbaum, Z} for the basic notions 
in the theory of polytopes and polyhedra.

\smallskip

An \Defn{$i$-dimensional subspace in $S^d$} is the intersection of the unit sphere $S^d$ with some $(i+1)$-dimensional linear subspace of $\R^{d+1}$.  A \Defn{hyperplane in $S^d$} is a $(d-1)$-dimensional subspace in $S^d$. We use $\Sp(X)$ to denote the \Defn{linear span} of a set $X$ in~$\R^d$, and $\aff(X)$ shall denote the \Defn{affine span} in $\R^d$. Finally, if $X$ is a subset of $S^d\subset \R^{d+1}$, we define $\SSp(X)=\Sp(X)\cap S^d$, the \Defn{spherical span} of $X$ in $S^d$. We use $\conv X$ to denote the \Defn{convex hull} of a set $X$, and $\cl X$, $\intx X$, $\rint X$ and $\partial X$ shall denote the \Defn{closure}, \Defn{interior}, \Defn{relative interior} and \Defn{boundary} of $X$ respectively.

\smallskip

A \Defn{(geometric) polytopal complex} (cf.\ \cite{RourkeSanders}) in $\R^d$ (resp.\ $S^d$) is a collection of polytopes in $\R^d$ (resp.~$S^d$) such that the intersection of any two polytopes is a face of both, and that is closed under passing to faces of the polytopes in the collection. An \Defn{(abstract) polytopal complex} is a set of polytopes that are attached along isometries of their faces, cf.\ \cite[Sec.\ 2.1]{DM-NP},~{\cite[Sec.~3.2]{BuragoBuragoIvanov}}. Our polytopal complexes are usually finite, i.e.\ the number of polytopes in the collection is finite. Two polytopal complexes $C,\, C'$ are \Defn{combinatorially equivalent}, denoted by $C\cong C'$, if their \Defn{face posets} are isomorphic. A \Defn{realization} of a polytopal complex is a geometric polytopal complex combinatorially equivalent to it.

\newcommand{\RS}{\mathrm{R}}
\smallskip

The elements of a polytopal complex are called the \Defn{faces}, and the inclusion maximal faces are the \Defn{facets} of the polytopal complex. The \Defn{dimension} of a polytopal complex is the maximum of the dimensions of its faces; the dimension of the empty set is $-1$. A polytopal complex is \Defn{pure} if all its facets are of the same dimension. We abbreviate \Defn{$d$-polytope}, \Defn{$d$-face} and \Defn{polytopal $d$-complex} to denote a polytope, face and polytopal complexes of dimension $d$, respectively. A polytopal complex is \Defn{simplicial} if all its faces are simplices. A polytope combinatorially equivalent to the regular unit cube $[0,1]^k\subset \R^d$, $k
\ge 0$, shall simply be called \Defn{cube}, and a polytopal complex is \Defn{cubical} if all its faces are cubes. The set of dimension $k$ faces of a polytopal complex $C$ is denoted by~$\F_k(C)$, and the cardinality of this set is denoted by $f_k(C)$. 

\smallskip

The \Defn{underlying space} $|C|$ of a polytopal complex $C$ is the union of its faces. With abuse of notation, we often speak of the polytopal complex when we actually mean its underlying space. For example, we often do not distinguish in notation between a polytope and the complex formed by its faces. In another instance of abuse of notation, if $A$ is a set (for example $A\subset \R^d$), and $C$ is a polytopal complex, then we write $C\subset A$ to denote the fact that $|C|$ lies in $A$. We define the \Defn{restriction} $\RS(C,A)$ of a polytopal complex $C$ to a set $A$ as the inclusion-maximal subcomplex $D$ of $C$ such that $D\subset A$. Finally, the \Defn{deletion} $C-D$ of a subcomplex $D$ from $C$ is the subcomplex of $C$ given by $\RS(C, C{\setminus} \rint{D})$.

\smallskip
\newcommand{\TT}{\mathrm{T}}
\newcommand{\RN}{\mathrm{N}}

Polytopal complexes come with two kinds of metric structures on them; combinatorial metrics (which we consider in Chapter~\ref{ch:Hirsch}) that measure the distance between faces, and ``continuous'' metrics that measure the distance between points in the complex and that give more relevance to the geometry. A~central example of such a metric is the \Defn{intrinsic length metric} (cf.\ \cite[Sec.~2 {\&} Sec.~3.2]{BuragoBuragoIvanov}): If $x$ and $y$ are two points in a polytopal complex $C$, the \Defn{intrinsic distance} between $x$ and $y$ is the infimum over the length of all rectifiable curves in $C$ connecting $x$ and $y$, and is denoted by $\mathrm{d}_{\mathrm{in}}(x,y)$. Two points in different connected components of $C$ are simply at distance $\infty$ from each other. A \Defn{polyhedral space} is a metric space that is locally isometric (cf.\ \cite[Sec.~3.3]{BuragoBuragoIvanov}) to a finite polytopal complex with its intrinsic length~metric. For example, $\R^d$, $S^d$ and polytopal complexes (with 
their intrinsic metric) are polyhedral spaces.

\smallskip

Let $\TT_p X$ denote the tangent space of a polyhedral space $X$ in a point $p$ (cf.\ \cite[Sec.\ 3.6.6]{BuragoBuragoIvanov}). We use $\TT^1_p X$ to denote the restriction of $\TT_p X$ to unit vectors. The space $\TT^1_p X$ comes with a natural metric: The distance between two points in it is the \Defn{angle} $\measuredangle(\cdot,\cdot)$ (cf.\ \cite[Sec.\ 3.6.5]{BuragoBuragoIvanov}); with this metric, $\TT^1_p X$ is itself a polyhedral space. We say that two elements $a$, $b$ in $\TT^1_p X$  are \Defn{orthogonal} if their enclosed angle is~$\nicefrac{\pi}{2}$. If $Y$ is any polyhedral subspace of $X$, then $\RN^1_{(p,Y)} X$ denotes the subset of $\TT^1_p X$ of points that are orthogonal to all elements of $\TT^1_p Y \subset \TT^1_p X$. If $Y=p$, then $\RN^1_{(p,Y)} X= \TT_p^1 X$.

\smallskip

Let $P$ be any polytope in the polyhedral space $X^d=\R^d$ (or $X^d=S^d$), let $\sigma$ be any nonempty face of~$P$, and let $p$ be any interior point of $\sigma$. Clearly, $\RN^1_{(p,\sigma)} X^d$ is isometric to the unit sphere of dimension $d-\dim \sigma -1$, and is considered as such. With this notion, $\RN^1_{(p,\sigma)} P$ is a polytope in the sphere~$\RN^1_{(p,\sigma)} X^d$. The polytope $\RN^1_{(p,\sigma)} P$ and its embedding into $\RN^1_{(p,\sigma)} X^d$ are uniquely determined up to \Defn{ambient isometry}, that is, for any second point $q$ in $\rint \sigma$ there is an isometry
\[\phi:\RN^1_{(p,\sigma)} X^d\longmapsto\RN^1_{(q,\sigma)} X^d\]
that restricts to an isometry 
\[\phi:\RN^1_{(p,\sigma)} P\longmapsto\RN^1_{(q,\sigma)} P.\]
Thus, we shall abbreviate $\RN^1_{\sigma} P:=\RN^1_{(p,\sigma)} P$ and $\RN^1_{\sigma} X^d:=\RN^1_{(p,\sigma)} X^d$ whenever the point $p$ is not relevant.

\smallskip

Now, let $C$ be any polytopal complex, and let $\sigma$ be any face of $C$. The \Defn{star} of $\sigma$ in $C$, denoted by $\St(\sigma, C)$, is the minimal subcomplex of $C$ that contains all faces of $C$ containing $\sigma$. If $C$ is simplicial and $v$ is a vertex of $C$ such that $\St(v,C)=C$, then $C$ is called a \Defn{cone} with apex $v$ over the base $C-v$.

\smallskip

Let $\tau$ be any face of $C$ containing $\sigma$, and assume that  $\sigma$ is nonempty and $p$ is any interior point of $\sigma$. Then the set $\RN^1_{(p,\sigma)} \tau$ of unit tangent vectors in $\RN^1_{(p,\sigma)} \cm{|}C\cm{|}$ pointing towards $\tau$ forms a spherical polytope isometrically embedded in~$\RN^1_{(p,\sigma)} \cm{|}C\cm{|}$. Again, $\RN^1_{(p,\sigma)} \tau$ and its embedding into $\RN^1_{(p,\sigma)} \cm{|}C\cm{|}$ are uniquely determined up to ambient isometry, so we abbreviate $\RN^1_{\sigma} \tau:=\RN^1_{(p,\sigma)} \tau$ and $\RN^1_{\sigma} \cm{|}C\cm{|}:=\RN^1_{(p,\sigma)} \cm{|}C\cm{|}$ unless $p$ is relevant in another context. The collection of all polytopes in $\RN^1_{\sigma} \cm{|}C\cm{|}$ obtained this way forms a polytopal complex, denoted by $\Lk_p(\sigma, C)$, the \Defn{link} of $\sigma$ in $C$ (cf.\ \cite[Sec.\ 2.2]{DM-NP}). Unless $p$ is relevant, we omit it in the notation for the link. This is still well-defined: Up to isometry, $\Lk_p(\sigma, C)$ does not depend on~$p$. 
If $C$ is a geometric polytopal complex in $X^d=\R^d$ (or $X^d=S^d$), then $\Lk_p(\sigma, C)$ is naturally a polytopal complex in $\RN^1_{(p,\sigma)} X^d$: $\Lk_p(\sigma, C)$ is the collection of spherical polytopes $\RN^1_{(p,\sigma)} \tau$ in the $(d-\dim \sigma -1)$-sphere $\RN^1_{(p,\sigma)} X^d$, where $\tau$ ranges over the faces of $C$ containing $\sigma$. Up to ambient isometry, this does not depend on the choice of $p$; we shall consequently omit it whenever possible.

\smallskip

We set $\Lk(\emptyset,C):=C$. For every face $a$ of $\Lk(\sigma,C)$, there exists a unique face $A$ of $C$, $A\supset a$, with $\Lk(\sigma,A)=a$. We write $A=\sigma \ast a$, and say that $A$ is the \Defn{join} of $a$ with $\sigma$. If $C$ is simplicial, and $v$ is a vertex of $C$, then $\Lk(v,C)$ is combinatorially equivalent to $\St(v,C)-v$. 

\newcommand\st{\mathrm{st}}
\smallskip

If $C$ is a geometric polytopal complex, then a \Defn{subdivision} of $C$ is any polytopal complex $C'$ with $|C|=|C'|$ such that each face of $C'$ is contained in some face of $C$. Now, let $C$ denote any polytopal complex, and let $\tau$ denote any face of $C$. Let $v_\tau$ denote a point anywhere in the relative interior of $\tau$. Define
\[
\st(\tau,C):=(C-\tau) \cup \{\conv \{v_\tau\}\cup \sigma : \sigma \in \St(\tau,C)-\tau \}.
\]
The complex $\st(\tau,C)$ is called a \Defn{stellar subdivision} of $C$ at $\tau$.

\smallskip

A \Defn{derived subdivision} $\sd C$ of a polytopal $d$-complex $C$ is any subdivision obtained from $C$ by first performing a stellar subdivision at all the $d$-faces of $C$, then stellarly subdividing the result at the $(d-1)$-faces of $C$ and so on. Any two derived subdivisions of the same complex are combinatorially equivalent; the combinatorial type of a derived subdivision is that of the \Defn{order complex} of the face poset of nonempty faces of~$C$ (cf.\ \cite[Sec.\ 4.7(a)]{BLSWZ}); in particular, the $k$-faces of $\sd C$ correspond to chains of length $k+1$ in that poset. An example of a derived subdivision is the \Defn{barycentric subdivision} which uses as vertices the barycenters of all faces of $C$. The process of derivation can be iterated, by setting $\sd^0 C := C$ and recursively $\sd^{r+1}(C) : = \sd\,  (\sd^r C)$. If~$r$~is a positive integer, then $\sd^r C$ is a simplicial complex, even if $C$ is not simplicial.

\newpage
\mbox{}
\thispagestyle{empty}
\newpage
\pagestyle{headings}
\pagenumbering{arabic}	

\part{A toy example: Metric geometry and the Hirsch conjecture}\label{pt:metgeo}
\newcommand{\did}{\mathrm{d}}
\vspace*{\fill}
\thispagestyle{empty}
\begingroup
{\em
{\bf Metric Geometry}, and in particular the theory of Alexandrov spaces of bounded curvature, was developed to study limit phenomena in Riemannian geometry. From this viewpoint, many results from differential geometry extend to more general classes of metric spaces, often leading to more insightful proofs of classical theorems. Among the crowning achievements of the theory of spaces of bounded curvature are Gromov's characterization of groups of polynomial growth, Perelman's proof of Poincar\'e's hypothesis and the recent proof of the virtually Haken conjecture by Agol.

The main purpose of this part of the thesis is to provide an example of the application of metric geometry to combinatorics: We prove an upper bound on the diameter of normal flag simplicial complexes using the theory of spaces of curvature bounded above. This simple example is supposed to serve as an introduction to this thesis. To help the reader unfamiliar with metric geometry, we provide two proofs of this result: a geometric one and a combinatorial one.
}
\endgroup
\vspace*{\fill}

\chapter{The Hirsch conjecture for normal flag complexes}\label{ch:Hirsch}

\theoremstyle{plain}
\newtheorem{thmmain}{\bf Theorem}[chapter]
\renewcommand{\thethmmain}{\arabic{chapter}.\arabic{section}.\Alph{thmmain}}
\newtheorem{mcor}[thmmain]{\bf Corollary}

\section{Introduction}
The Hirsch conjecture was, for a long time, one of the most important open problems in the theory of polytopes. It was proposed in the sixties by Warren Hirsch in a letter to George Dantzig to give a theoretical upper bound on the running time of the simplex algorithm in linear programming.

\begin{conjecture}[Hirsch conjecture {\cite[Sec.\ 7.3,\ 7.4]{Dantzig}}]
Let $P$ denote a $d$-dimensional polyhedron in $\R^d$ that is defined as the intersection of $n$ closed halfspaces. Then the diameter of the $1$-skeleton of $P$ is $\le n-d$.
\end{conjecture}

The case of unbounded polyhedra was quickly resolved when a counterexample was given by Klee and Walkup~\cite{KleeWalkup}. It remained to treat the case of bounded polyhedra (i.e.\ polytopes), the \emph{bounded Hirsch conjecture}. We state the conjecture in a form dual to the classical formulation.

\begin{conjecture}[(Bounded) Hirsch conjecture {\cite{KleeWalkup}}]\label{cj:BHC}
The diameter of the facet-ridge graph of any $d$-polytope on $n$ vertices is $\le n-d$.
\end{conjecture}
 
A related conjecture, the \emph{$W_v$-conjecture}, or \emph{non-revisiting path conjecture}, was formulated by Klee and Wolfe, cf.~\cite{Klee}.

\begin{conjecture}[(Bounded) Non-revisiting path conjecture, or (bounded) $W_v$-conjecture]\label{cj:BWv}
For any two facets of a simplicial polytope $P$ there exists a non-revisiting path connecting them.
\end{conjecture}

Conjectures~\ref{cj:BHC} and~\ref{cj:BWv} are equivalent for the class of all simplicial polytopes; for a given polytope, the $W_v$-conjecture easily implies the Hirsch conjecture, cf.~\cite{KleeKleinschmidt}. Recently, the two conjectures were disproved by Santos~\cite{Santos}. It was, however, already known that it may be difficult for the simplex algorithm to find a short way to the optimum, even if such a short way might exist~\cite{KleeMinty}, thus suggesting that the Hirsch conjecture's practical relevance might not be as high as its theoretical importance; nevertheless, Conjecture~\ref{cj:BHC} attracted much attention, and its falsification is a great achievement. The demise of the Hirsch conjecture opens several possible directions of research, among others the problem of determining for which polytopes the Hirsch conjecture is true. 

In this part of the thesis, we prove that the $W_v$-conjecture is true for all flag polytopes. Our methods apply more generally to all simplicial complexes that are normal and flag.

\begin{thmmain}\label{THM:HIRSCHA}
 Let $C$ be a normal flag simplicial $d$-complex with $n$ vertices. Then $C$ satisfies the non-revisiting path property, and in particular the Hirsch diameter bound \[\diam(C)\le n-d-1.\]
\end{thmmain}

We provide two proofs of this theorem, a geometric proof (Section~\ref{sec:geomproof}), based on the theory of spaces of curvature bounded above, and a combinatorial one (Section~\ref{sec:combproof}) that discretizes the geometric intuition. The combinatorial proof is clearly more elementary than the geometrical one, but it also provides less insight as to why the theorem is true. Our idea for both proofs is reminiscent of an idea of Larman~\cite{Larman}. In order to aid the understanding of both proofs, we attempt to use a parallel notation whenever possible. In Section~\ref{sec:corH}, we state some straightforward corollaries of Theorem~\ref{THM:HIRSCHA}. In particular, we will see that Theorem~\ref{THM:HIRSCHA} generalizes a classical result of Provan and Billera, cf.\ Remark~\ref{rem:prb}. 

\enlargethispage{5mm}
\subsection{Set-up}

\begin{definition}[Flag complexes]
If $C$ is a simplicial complex, then a subset of its vertices is a \Defn{non-face} if no face of $C$ contains all vertices of the set. A simplicial complex $C$ is \Defn{flag} if every inclusion minimal non-face is an edge, that is, every inclusion minimal non-face of $C$ consists of $2$ vertices.
\end{definition}

\begin{definition}[Dual graph, normal complex and diameter]
If $C$ is a pure simplicial $d$-complex, then we define $G^\ast(C)$, the \Defn{dual graph}, or \Defn{facet-ridge graph} of $C$, by considering the set $\F_d(C)$ of facets of $C$ as the set of vertices of $G^\ast(C)$, two of which we connect by an edge if they intersect in a face of codimension~$1$. A pure simplicial complex $C$ is \Defn{normal} if for every face $\sigma$ of $C$ (including the empty face), $G^\ast(\St(\sigma,C))$ is connected. We define $\diam(C)$ as the diameter of the graph $G^\ast(C)$.  We say that $C$ satisfies the \Defn{Hirsch diameter bound} if $\diam(C)\le n-d-1.$
\end{definition}

An \Defn{interval} in $\mathbb{Z}$ (resp.\ $\mathbb{R}$) is a set $[a,b]:=\{x\in \mathbb{Z}: a\le x\le b\}$ (resp.\ $\{x\in \mathbb{R}: a\le x\le b\}.$)

\begin{definition}[Curves, Facet paths and vertex paths]
If $X$ is a metric space and $I$ is an interval in~$\R$, an immersion $\gamma:I\mapsto X$ is called a \Defn{curve}. If $C$ is a pure simplicial $d$-complex, and $I$ is an interval in~$\Z$, then a \Defn{facet path} is a map $\varGamma$ from $I$ to $\F_d(C)$ such that for every two consecutive elements $i$, $i+1$ of~$I$, we have that $\dim (\varGamma(i)\cap \varGamma(i+1))=d-1$. A \Defn{vertex path} in $C$ is a map $\gamma$ from $I$ to $\F_0(C)$ such that for every two consecutive elements $i$, $i+1$ of~$I$, we have that $\gamma(i)$ and $\gamma(i+1)$ are joined by an edge. 

\end{definition}

All curves and paths are considered with their natural order from the startpoint (the image of $\min I$) to the endpoint (the image of $\max I$). For example, the \Defn{last facet} of a facet path $\varGamma$ in a subcomplex $S$ of $C$ is the image of the maximal $z\in I$ such that $\gamma(z)\in S$. As common in the literature, we will not strictly differentiate between a curve (or path) and its image; for instance, we will write $\gamma\subset S$ to denote the fact that the image of a curve $\gamma$ lies in a set $S$.

If $\gamma$ and $\delta$ are two curves in any metric space such that the endpoint of $\gamma$ coincides with the starting point of $\delta$, we use the notation $\gamma \cdot \delta$ to denote their \Defn{concatenation} or \Defn{product} (cf.~\cite[Sec.~2.1.1.]{BuragoBuragoIvanov}). Analogously, if the last facet of a facet path $\varGamma$ and the first facet of a facet path $\varDelta$ coincide, we can \Defn{concatenate} $\varGamma$ and $\varDelta$ to form a facet path $\varGamma\cdot \varDelta$. Concatenations of more than two paths are represented using the symbol~$\prod$. If $i$ and $j$ are elements in the domain of a facet path $\varGamma$, then $\varGamma_{[i,j]}$ is the restriction of $\varGamma$ to the interval $[i,j]$ in $\mathbb{Z}$. If a facet path $\varGamma$ is obtained from a facet path $E$ by restriction to some interval, then $\varGamma$ is a \Defn{subpath} of $E$, and we write $\varGamma \subset E$. Two facet paths coincide up to \Defn{reparametrization} if they coincide up to an order-
preserving bijection of their respective domains.

\begin{definition}[$W_v$-property, cf.~\cite{Klee}]
Let $C$ be a pure simplicial complex. The facet path $\varGamma$ is \Defn{non-revisiting} for every pair $i, j$ in the domain of $\varGamma$ such that $\varGamma(i)$ and $\varGamma(j)$ lie in $\St(v,C)$ for some vertex $v\in C$, the subpath $\varGamma_{[i,j]}$ of $\varGamma$ lies in $\St(v,C)$. We say that $C$ satisfies the \Defn{non-revisiting path property}, or \Defn{$W_v$-property}, if for every pair of facets of $C$, there exists a non-revisiting facet path connecting the two.
\end{definition}

\begin{lemma}[obvious, cf.~\cite{KleeKleinschmidt}]
Any pure simplicial complex that satisfies the $W_v$-property satisfies the Hirsch diameter bound.
\end{lemma}

Finally, if $M$ is a metric space with metric $\did(\cdot,\cdot)$, then the \Defn{distance between two subsets} $A, B$ of $X$ is defined as $\did(A,B):=\inf \{\did(a,b): a\in A, b\in B\}.$ Furthermore, if $\sigma$, $\sigma'$ are two faces of a simplicial complex $C$ that lie in a common face, then $\sigma\circledast \sigma'$ denotes the minimal face of $C$ containing them both.
 
\section{The geometric proof}\label{sec:geomproof}

In this section, we give a geometric proof of Theorem~\ref{THM:HIRSCHA}. The idea is fairly simple: By a criterion of Gromov, any flag complex has a canonical metric which satisfies a certain upper curvature bound. If $X, Y$ are two facets in the complex, we now just have to choose a segment (with respect to the aforementioned metric) connecting them; a simple argument then shows that the facets this segment encounters along the way give the desired non-revisiting facet path connecting $X$ and $Y$. We need some modest background from the theory of spaces of curvature bounded above, which we review here. For a more detailed introduction, we refer the reader to the textbook by Burago--Burago--Ivanov~\cite{BuragoBuragoIvanov}.

\subsubsection*{$\CAT(1)$ spaces and convex subsets} 

Consider a metric space $M$ with metric $\did(\cdot,\cdot)$. We say that $M$ is a \Defn{length space} if for every pair of points $a$ and $b$ in the same connected component of $M$, the value of $\did(a,b)$ is also the infimum of the lengths of all rectifiable curves from $a$ to~$b$. In this case, we write $\did_{\mathrm{in}}(\cdot,\cdot)$ instead of $\did(\cdot,\cdot)$, and say that $\did_{\mathrm{in}}(\cdot,\cdot)$ is an \Defn{intrinsic length metric}.
A curve that attains the distance $\did_{\mathrm{in}}(a,b)$ is denoted by $[a,b]$, and is called a \Defn{segment}, or \Defn{shortest path}, connecting $a$ and $b$. A \Defn{geodesic} $\gamma:I\mapsto M$ is a curve that is locally a segment, that is, every point in $I$ has an open neighborhood $J$ such that $\gamma$, restricted to $\cl(J)$, is a segment. A \Defn{geodesic triangle} $[a,b,c]$  in $M$ is given by three vertices $a,b,c$ connected by some three segments  $[a,b],\, [b,c]$ and $[a,c]$. All of the metric spaces we consider in this section are compact length spaces; in these spaces, any two points that lie in a common connected component are connected by at least one segment. 

A \Defn{comparison triangle} for a geodesic triangle $[a,b,c]$ in $M$ is a geodesic triangle $[\bar{a},\bar{b},\bar{c}]$ in $S^2$ such that $\did_{\mathrm{in}}(\bar{a},\bar{b})=\did_{\mathrm{in}}(a,b)$, $\did_{\mathrm{in}}(\bar{a},\bar{c})=\did_{\mathrm{in}}(a,c)$ and $\did_{\mathrm{in}}(\bar{b},\bar{c})=\did_{\mathrm{in}}(b,c)$. The space $M$ is a \Defn{$\CAT(1)$ space} if it is a length space in which the following condition is satisfied:

\smallskip
\noindent {\textsc{Triangle condition}:} \emph{For each geodesic triangle $[a,b,c]$ inside $M$ and for any point $d$ in the relative interior of $[a, b]$, one has $\did_{\mathrm{in}}(c,d)\leq \did_{\mathrm{in}}(\bar{c},\bar{d})$, where $[\bar{a},\bar{b},\bar{c}]$ in $S^2$ is any comparison triangle for $[a,b,c]$ and $\bar{d}$ is the unique point on $[\bar{a},\bar{b}]$ with $\did_{\mathrm{in}}(a,d) = \did_{\mathrm{in}}(\bar{a},\bar{d})$.}
\smallskip

Let $A$ be any subset of a length space $M$. The set $A$ is \Defn{convex} if any two points of $A$ are connected by a segment that lies in $A$. The set $A$ is \Defn{locally convex} if every point in $A$ has an open neighborhood $U$ such that $U\cap A$ is convex. We are interested in the following property of $\CAT(1)$ spaces.

\begin{prp}[cf.~\cite{Tietze},~\cite{Nakajima}, {\cite[Thm.\ 8.3.3]{Papa}},~\cite{BuxWitzel}]\label{prp:hdct} Let $M$ denote a compact $\CAT(1)$ length space. Let $A$ be any locally convex subset of $M$ such that any two points in $A$ are connected by a rectifiable curve in $A$ of length $\le \pi$. Then $A$ is convex.
\end{prp}

\subsubsection*{Non-obtuse simplices and convex vertex-stars} 

\begin{definition}[Non-obtuse simplices]
Let $\varDelta$ denote a simplex of dimension $d\geq 1$, and let $\sigma$ denote a $(d-2)$-face of $\varDelta$. The dihedral angle of $\varDelta$ at $\sigma$ is the length of the arc $\RN^1_\sigma \varDelta$. The simplex $\varDelta$ is \Defn{non-obtuse} if all dihedral angles of $\varDelta$ are smaller or equal to $\nicefrac{\pi}{2}$. By convention, every $0$-simplex is \Defn{non-obtuse} as well.
\end{definition}

For the rest of this section we consider every simplicial complex $C$ to be endowed with its intrinsic length metric $\did_{\mathrm{in}}(\cdot,\cdot)$. With this notion, Proposition~\ref{prp:hdct} gives the following:

\begin{cor}\label{cor:hdct}
Let $C$ be a pure simplicial $d$-complex such that each face of $C$ is non-obtuse and $C$ is a $\CAT(1)$ metric space. Then $\St(v,C)$ is convex in $C$ for every vertex $v$ of $C$.
\end{cor}

\begin{proof}
The proof is by induction on $d$; the case $d=0$ is trivial. Assume now $d\ge 1$. For every vertex $w\in C$, the simplicial complex $\Lk (w,C)$ is a $\CAT(1)$ complex (cf.~\cite[Thm.\ 4.2.A]{GromovHG}) all whose faces are non-obtuse. Thus, $\St(v,C)$ is locally convex since for every $w\in \St(v,C), w\neq v$, we have that $\Lk(w,\St(v,C))=\St(\Lk(w,w\circledast v),\Lk(w,C))$ is convex in $\Lk (w,C)$ by inductive assumption. Furthermore, since every face of $C$ is non-obtuse, every point in $\St(v,C)$ can be connected to $v$ by a segment in $\St(v,C)$ of length $\le \nicefrac{\pi}{2}$. Applying Proposition~\ref{prp:hdct} finishes the proof.
\end{proof}

\subsection*{Geometric proof of Theorem~\ref{THM:HIRSCHA}.}

\begin{lemma} \label{lem:Hirsch}
Let $C$ be a normal simplicial $d$-complex such that each simplex of $C$ is non-obtuse and $C$ is a $\CAT(1)$ metric space. Let furthermore $X$ be any facet of $C$, and let $\mathcal{Y}$ be any finite set of points in $C$. Then, there exists a non-revisiting facet path $\varGamma$ from the facet $X$ of $C$ to some facet of $C$ containing a point of $\mathcal{Y}$.
\end{lemma}

\begin{proof}
The proof, as well as the construction of the desired facet path, is by induction on the dimension $d$ of $C$. The case $d=0$ is easy: If $X$ and $\mathcal{Y}$ intersect, the path is trivial of length $0$. If not, the desired facet path is given by $\varGamma:\{0,1\}\mapsto C$, with $\varGamma(0):=X$ and $\varGamma(1):=Y$, where $Y$ is any facet of $C$ intersecting $\mathcal{Y}$. We proceed by induction on $d$, assuming that $d\ge 1$. 

\smallskip

\emph{Some preliminaries}: If $\alpha$ is any point in $C$, let us denote by $\sigma_\alpha$ the minimal face of $C$ containing $\alpha$. If $\omega$ is any second point in $C$, let $\operatorname{S}_{\alpha}^{\omega}$ denote the set of segments from $\alpha$ to $\omega$. For an element $\gamma\in\operatorname{S}_{\alpha}^{\omega}$ with $\TT_\alpha^1 \gamma\notin \TT_\alpha^1 \sigma_\alpha$, \Defn{the tangent direction of $\gamma$ in $\Lk_\alpha(\sigma_{\alpha},C)=\Lk(\sigma_{\alpha},C)$ at $\alpha$} is defined as the barycenter of $\Lk(\sigma_{\alpha},\tau)$, where $\tau$ is the minimal face of $C$ that contains $\sigma_\alpha$ and such that $\TT_\alpha^1 \gamma\in \TT_\alpha^1 \tau$.
Define $\operatorname{T}_{\alpha}^{\omega}$ to be the union of tangent directions in $\Lk (\sigma_{\alpha},C)$ at $\alpha$ over all segments $\gamma\in\operatorname{S}_{\alpha}^{\omega}$. Finally, set
\[\operatorname{S}_{\alpha}^{\mathcal{\varOmega}}:=\bigcup_{\omega\in \mathcal{\varOmega}} \operatorname{S}_{\alpha}^{\omega}\ \ \text{and}\ \ \operatorname{T}_{\alpha}^{\mathcal{\varOmega}}:=\bigcup_{\omega\in \mathcal{\varOmega}} \operatorname{T}_{\alpha}^{\omega}.\]
for any collection $\mathcal{\varOmega}$ of points in $C$. Clearly, $\operatorname{T}_{\alpha}^{\mathcal{\varOmega}}$ is finite.

\smallskip

Returning to the proof, let $x_0$ denote any point of $X_0:=X$ minimizing the distance to the set $\mathcal{Y}_0:=\mathcal{Y}$. Set $i:=0$. The construction process for the desired facet path goes as follows: 

\smallskip

\noindent {\bf Construction procedure} If ${X_i}$ intersects $\mathcal{Y}$, set $\ell:=i$ and stop the procedure. If not, consider the face $\sigma_i:=\sigma_{x_i}$ of $C$ containing $x_i$ in its relative interior. The simplicial complex $\Lk (\sigma_{i},C)$ is a normal $\CAT(1)$ complex (cf.~\cite[Thm.\ 4.2.A]{GromovHG}) all whose faces are non-obtuse. Now, we use the construction technique for dimension $d-\dim \sigma_i - 1\le d-1$ to find a (non-revisiting) facet path $\varGamma'_{X'_iX'_{i+1}}$ in $\Lk(\sigma_{i},C)$ from $X'_i:=\Lk(\sigma_{i},X_i)$ to some facet $X'_{i+1}$ of $\Lk(\sigma_{i},C)$ that intersects $\operatorname{T}_{z}^{\mathcal{Y}_i}$. We may assume that $\varGamma'_{X'_iX'_{i+1}}$ intersects  $\operatorname{T}_{z}^{\mathcal{Y}_i}$ only in the last facet $X'_{i+1}$. Lift the facet path $\varGamma'_{X'_iX'_{i+1}}$ in $\Lk(\sigma_{i},C)$ to a facet path $\varGamma_{X_iX_{i+1}}$ in $C$ from $X_i$ to $X_{i+1}:=\sigma_{i}\ast X'_{i+1}$ by join with $\sigma_i$, i.e.\ define \[\varGamma_{X_iX_{
i+1}}:=\sigma_i\ast \varGamma'_{X_i'X_{i+1}'}.\] 
Let $\gamma_i$ be any element of $\operatorname{S}_{x_i}^{\mathcal{Y}_i}$ whose tangent direction in $\Lk(\sigma_{i},C)$ at $x_i$ lies in $X'_{i+1}$, let $\overline{\gamma}_i$ denote the restriction of $\gamma_i$ to $X_{i+1}$, and let $x_{i+1}$ be the last point of $\gamma_i$ in $X_{i+1}$. Finally, let $\mathcal{Y}_{i+1}$ denote the subset of points $y$ of $\mathcal{Y}_i$ with
\begin{equation}\tag{$\ast$} \label{eq:xx}
\did_{\mathrm{in}}(y,x_i)=\did_{\mathrm{in}}(y,x_{i+1})+\did_{\mathrm{in}}(x_i,x_{i+1}).
\end{equation}
Now, increase $i$ by one, and repeat the construction procedure from the start.

\smallskip

Define the facet path \[\varGamma:=\prod_{i\in (0,\, \cdots,\ell-1)} \varGamma_{X_{i}X_{i+1}}.\] Associated to $\varGamma$, define the curve \[\gamma=\prod_{i\in(0,\, \cdots,\ell-1)} \overline{\gamma}_i \] from $x$ to some element $y$ of $\mathcal{Y}$, the \Defn{necklace} of $\varGamma$, and define the \Defn{pearls} of $\varGamma$ to be the faces $\sigma_i$. Finally, we denote by $\chi_i$, $0\le i\le \ell$, the element in the domain of $\varGamma$ corresponding to $X_i$; with this, the facet paths $\varGamma_{X_{i}X_{i+1}}$ coincide up to reparametrization with the subpaths $\varGamma_{[\chi_{i},\chi_{i+1}]}$ of $\varGamma$ for each $i$. 
For any element $a\neq \chi_\ell$ in the domain of $\varGamma$, let $i$ be chosen so that $a\in [\chi_{i},\chi_{i+1}-1]$. We say that $a$ is \Defn{associated to the pearl $\sigma_i$} of $\varGamma$. By convention, \Defn{$\chi_\ell$ is associated to the pearl~$\sigma_{\ell-1}$}.

By Equation \eqref{eq:xx}, $\gamma$ is a segment. Thus, by Corollary~\ref{cor:hdct}, if $v$ is any vertex of $C$, then $\gamma$ intersects $\intx \St(v,C)$ in a connected component. We will see that this fact extends to the combinatorial setting. First, we make the following claim.

\begin{figure}[ht]
	\centering
  \includegraphics[width=.7\linewidth]{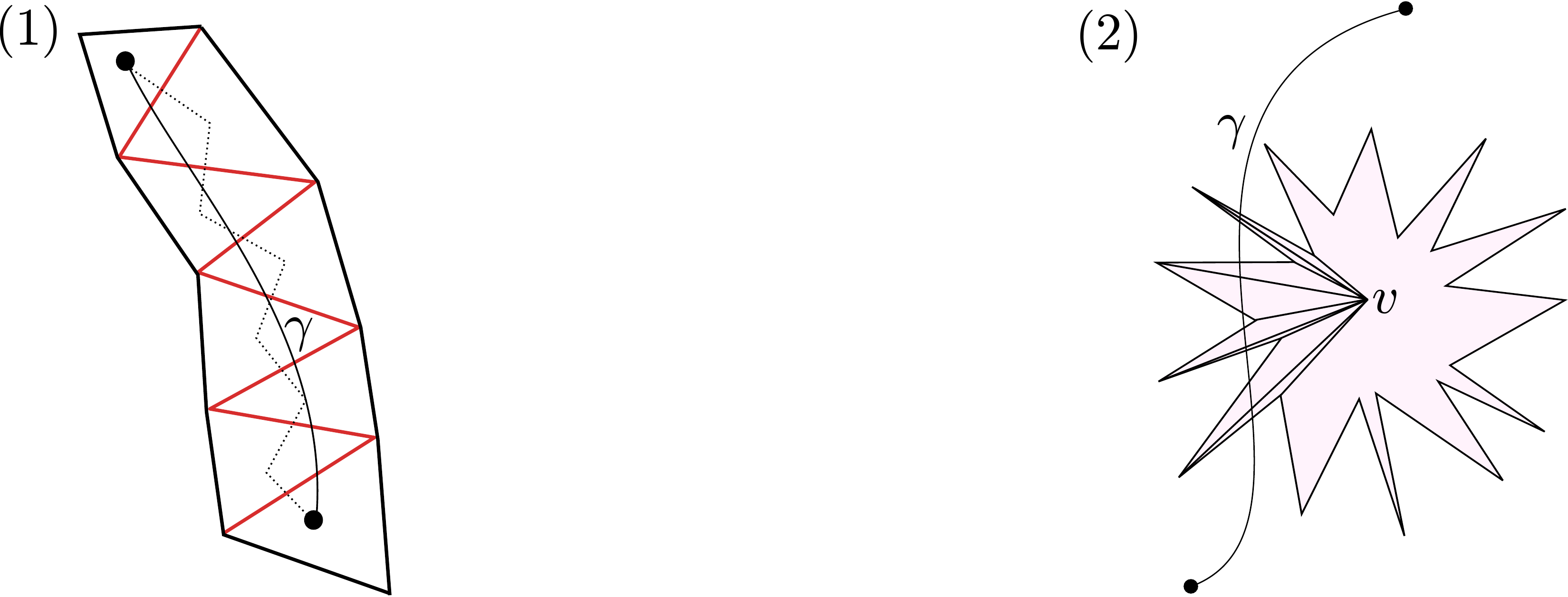}
 	\caption{(1) The segment $\gamma$ in a complex with convex vertex stars. Since all vertex stars are convex, the path $\gamma$ enters the star of any vertex in $C$ only once. 
\newline (2) If the convexity assumption on vertex stars is removed, the segment $\gamma$ may enter a single vertex star multiple times.}
	\label{fig:hirsch}
\end{figure}
\vskip -2mm

\smallskip 
\noindent \emph{Let $a$ denote any element in the domain of $\varGamma$, and let $v$ be any vertex of $\varGamma(a)$. Let $\hat{x}$ denote the last point of $\gamma$ in $\St(v,C)$, and assume that $\hat{x}$ is not in $\varGamma(a)$. Let $\sigma_i$ be the pearl associated to $a$. Then $\varGamma_{[a,\chi_{i+1}]}$ lies in $\St(v,C)$. In particular, $X_{i+1}$ lies in $\St(v,C)$.}
\smallskip

\noindent To prove the claim, we need only apply an easy induction on the dimension:
\begin{compactitem}[$\circ$]
\item If $v$ is a vertex of the pearl $\sigma_i$, this follows directly from construction of $\varGamma$. Since $v=\sigma_i$ if $d=1$, this in particular proves the case $d=1$. 
\item If $v$ is not in $\sigma_i$, we consider the facet path $\varGamma':=\Lk(\sigma_i,\varGamma_{[\chi_i,\chi_{i+1}]})$ in $\Lk(\sigma_i,C)$. The point $\hat{x}$ lies in $\St(v,C)$, which is convex in $C$ by Corollary \ref{cor:hdct}. Thus, the construction of $\gamma$ and $\varGamma$ implies that the restriction of $\gamma$ to the interval $[\gamma^{-1}(x_i),\gamma^{-1}(\hat{x})]$ lies in $\St(v,C)$: indeed, if $\did_{\mathrm{in}}(x_i,\hat{x})<\pi$, then $x_i$ and $\hat{x}$ are connected by a unique segment in $C$, thus, this segment must lie in $\St(v,C)$. If $\did_{\mathrm{in}}(x_i,\hat{x})\geq \pi$, then connecting $x_i$ to $v$ and $v$ to $\hat{x}$ by segments gives a segment from $x_i$ to $\hat{x}$; thus, $\varGamma(a)$ must contain $\sigma_{i+1}$ by construction of $\varGamma$, which contradicts the assumption that $a$ was associated to $\sigma_i$.

In particular, since $[\gamma^{-1}(x_i),\gamma^{-1}(\hat{x})]$ lies in $\St(v,C)$, the tangent direction of $\overline{\gamma}_i$ in $\Lk(\sigma_i,C)$ at $x_i$ is a point in $\St(v',\Lk(\sigma_i,C))$, where $v':=\Lk(\sigma_i,v\circledast \sigma_i)$. However, since $\varGamma(a)$ does not contain $\hat{x}$, this tangent direction does not lie in~$\varGamma'(a)$. Hence, the path $\varGamma'_{[a,\chi_{i+1}]}=\Lk(\sigma_i,\varGamma_{[a,\chi_{i+1}]})$ is contained in $\St(v',\Lk(\sigma_i,C))$ by induction assumption. We obtain \[\varGamma_{[a,\chi_{i+1}]}=\sigma_i\ast \varGamma'_{[a,\chi_{i+1}]}\subset\sigma\ast \St(v',\Lk(\sigma_i,C))\subset\St(v,C).\]
\end{compactitem}
We can now use induction on $d$ to conclude that $\varGamma$ is non-revisiting. The case $d=0$ is trivial, assume therefore $d\ge 1$. Consider any second $b$ with $\varGamma(b) \in \St(v,C)$ such that $b\ge a$. Let $j$, $j\ge i$, be chosen such that $\sigma_j$ is the pearl associated with~$b$. There are two cases to consider
\begin{compactitem}[$\circ$]
\item {\bf If $i=j$:} By induction assumption, the facet path $\varGamma'_{X'_iX'_{i+1}}$ (as defined above) is non-revisiting. Thus, the facet path $\varGamma_{[\chi_i,\chi_{i+1}]}$, which coincides with $\varGamma_{X_iX_{i+1}}=\sigma_i\ast \varGamma'_{X'_iX'_{i+1}}$ up to reparametrization, is non-revisiting. Since $\varGamma_{[a,b]}$ is a subpath of $\varGamma_{[\chi_i,\chi_{i+1}]}$, this finishes the proof of this case.
\item {\bf If $i<j$:} The claim proves that $\varGamma_{[a,\chi_{i+1}]}$ lies in $\St(v,C)$ and that for every $k$, $i< k < j$, $\varGamma_{[\chi_k,\chi_{k+1}]}$ lies in $\St(v,C)$. Thus, \[\varGamma_{[a,\chi_{j}]}= \varGamma_{[a,\chi_{i+1}]} \cdot \Big(\prod_{k,\ i< k< j} \varGamma_{[\chi_k,\chi_{k+1}]}\Big)\subset \St(v,C).\]
Since $\varGamma_{[a,b]}=\varGamma_{[a,\chi_{j}]}\cdot\varGamma_{[\chi_{j},b]}$, it only remains to prove that  $\varGamma_{[\chi_{j},b]}\subset \St(v,C)$; this was proven in the previous case. \qedhere
\end{compactitem}
\end{proof}

\begin{cor} \label{cor:hirsch}
Let $C$ be a normal simplicial $d$-complex such that each simplex of $C$ is non-obtuse and $C$ is a $\CAT(1)$ metric space. Then $C$ satisfies the non-revisiting path property.  
\end{cor}

\begin{proof}
If $X$ and $Y$ are any two facets of $C$, apply Lemma 
\ref{lem:Hirsch} to the facet $X$ and the set $\mathcal{Y}=\{y\}$, where $y$ is any interior point of $Y$.
\end{proof}

We can give the first proof of Theorem~\ref{THM:HIRSCHA}.
\begin{proof}[\textbf{Proof of Theorem~\ref{THM:HIRSCHA}}]
We change $C$ to a combinatorially equivalent abstract simplicial complex $C'$ by endowing  every face with the metric of a regular spherical simplex with orthogonal dihedral angles. By Gromov's Criterion~\cite[Sec.\ 4.2.E]{GromovHG}, the resulting metric space is $\CAT(1)$, because $C$ is flag. By construction, every simplex of $C'$ is non-obtuse, so we can apply Corollary~\ref{cor:hirsch} and conclude that $C$ satisfies the non-revisiting path property.
\end{proof}

\section{The combinatorial proof}\label{sec:combproof}

In this section, we give a purely combinatorial proof of Theorem~\ref{THM:HIRSCHA}. The proof is articulated into two parts: We first construct a facet path between any pair of facets of $C$, the so called combinatorial segment, and then we prove that the constructed path satisfies the non-revisiting path property.

Let $\did_{\mathrm{co}}(x,y)$ denote the (combinatorial) distance between two vertices $x,\, y$ in the $1$-skeleton of the simplicial complex $C$. If $\mathcal{Y}$ is a subset of $\F_0(C)$, let $\pp(x,\mathcal{Y})$ denote the elements of $\mathcal{Y}$ that realize the distance $\did_{\mathrm{co}}(x,\mathcal{Y})$, and let $\widetilde{\operatorname{T}}_x^{\mathcal{Y}}$ denote the set of vertices $y$ in $\St(x,C)-x$ with the property that \[\did_{\mathrm{co}}(y,\mathcal{Y})+1=\did_{\mathrm{co}}(x,\mathcal{Y}).\] 
$\St(x,C)-x$ is naturally combinatorially equivalent to $\Lk(x,C)$; let $\operatorname{T}_x^{\mathcal{Y}}$ denote the subset of $\F_0(\Lk(x,C))$ corresponding to $\widetilde{\operatorname{T}}_x^{\mathcal{Y}}$ in $\F_0(\St(x,C)-x)$.

\subsection*{Construction of a combinatorial segment}\label{ssc:cseg}
\enlargethispage{2mm}
\noindent{\bf Part 1: From any facet $X$ to any vertex set $\mathcal{Y}$.}

\smallskip 
We construct a facet path from a facet $X$ of $C$ to a subset $\mathcal{Y}$ of $\F_0(C)$, i.e.\ a facet path from $X$ to a facet intersecting $\mathcal{Y}$, with the property that $\mathcal{Y}$ is intersected by the path $\varGamma$ only in the last facet of the path. 

If $C$ is $0$-dimensional, and $X$ and $\mathcal{Y}$ intersect, the path is trivial of length $0$. Else, the desired facet path is given by $\varGamma:\{0,1\}\mapsto C$, where $\varGamma(0):=X$ and $\varGamma(1):=Y$, which is any facet of $C$ intersecting $\mathcal{Y}$ . 

If $C$ is of a dimension $d$ larger than~$0$, set $X_0:=X$, let $x_0$ be any vertex of $X_0$ that minimizes the distance $\did_{\mathrm{co}}$ to $\mathcal{Y}$, set $\mathcal{Y}_0:=\pp({x_0},\mathcal{Y})$ and set $i:=0$. Then proceed as follows:

\smallskip

\noindent {\bf Construction procedure} If $X_i$ intersects $\mathcal{Y}$, set $\ell:=i$ and stop the construction procedure. Otherwise, use the construction for dimension $d-1$ to construct a facet path $\varGamma'_{X'_iX'_{i+1}}$ in $\Lk(x_i,C)$ from the facet $X'_i:=\Lk(x_i,X_i)$ to the vertex set  $\operatorname{T}_{x_i}^{\mathcal{Y}_i}$. Denote the last facet of the path by $X'_{i+1}$. By forming the join of that path with $x_i$, we obtain a facet path $\varGamma_{X_iX_{i+1}}$ from $X_i$ to the facet $X_{i+1}:=x_i\ast X'_{i+1}$ of $C$. Denote the vertex of $\widetilde{\operatorname{T}}_{x_i}^{\mathcal{Y}_i}$ it intersects by $x_{i+1}$. Define $\mathcal{Y}_{i+1}:=\pp(x_{i+1},\mathcal{Y}_{i})$. Finally, increase $i$ by one, and repeat from the start.
\smallskip

The concatenation of these facet paths is a \Defn{combinatorial segment} from $X$ to $\mathcal{Y}$.

\medskip

\noindent{\bf Part 2: From any facet $X$ to any other facet $Y$.}

\smallskip 

Using Part 1, construct a facet path from $X$ to the vertex set $\F_0(Y)$ of $X_{\ell+1}:=Y$. If $C$ is of dimension $0$, then this finishes the construction. If $C$ is of dimension $d$ greater than $0$, let $x_\ell$ denote any vertex shared by $X_{\ell}$ and $X_{\ell+1}$, and apply the $(d-1)$-dimensional construction to construct a facet path in $\Lk(x_{\ell},C)$ from $\Lk(x_{\ell},X_{\ell})$ to $\Lk(x_{\ell},X_{\ell+1})$. 

Lift this facet path to a facet path $\varGamma_{X_{\ell}X_{\ell+1}}$ in $C$ by forming the join of the path with $x_{\ell}$. This finishes the construction: The \Defn{combinatorial segment} from $X$ to $Y$ is defined as the concatenation \[\varGamma:=\prod_{i\in (0,\, \cdots,\ell)} \varGamma_{X_{i}X_{i+1}}.\]

\subsection*{The combinatorial segment is non-revisiting}\label{ssc:cnrv}

We start off with some simple observations and notions for combinatorial segments in complexes of dimension $d\ge 1$.

\begin{compactitem}[$\circ$]
\item A combinatorial segment $\varGamma$ comes with a vertex path $(x_0, x_1, \, \cdots , x_{\ell})$. This is a \Defn{shortest vertex path} in $C$, i.e.\ it realizes the distance $\ell = \did_{\mathrm{co}}(\F_0(X),\mathcal{Y})$ resp.\ $\ell = \did_{\mathrm{co}}(\F_0(X),\F_0(Y))$. The path $\gamma$ is the \Defn{necklace} of $\varGamma$, the vertices $x_i$, $0\le i\le\ell$, are the \Defn{pearls} of~$\varGamma$.	
\item As in the geometric proof, we denote by $\chi_i$, $0\le i\le \ell+1$, the element in the domain of $\varGamma$ corresponding to $X_i$. Let $a\neq \chi_{\ell+1}$ be any element in the domain of $\varGamma$. If $i$ is chosen so that $a\in [\chi_i,\chi_{i+1}-1]$, then $x_i$ is the \Defn{pearl associated to $a$ in $\gamma$}. By convention, we say that \Defn{$\chi_{\ell+1}$ is associated to the pearl $x_\ell$}.
 \end{compactitem}

\begin{lemma}\label{lem:S}
Assume that $\dim C\ge 1$. If $a$ is an element in the domain of $\varGamma$ associated with pearl $x_i$ such that $i<\ell$, and $v$ is any vertex of $\varGamma(a)$ such that $x_{j}, j>i,$ lies in $\St(v,C)$, then the subpath $\varGamma_{[a,\chi_{j}]}$ lies in $\St(v,C)$. In particular, in this case $X_{j}$ is a facet of $\St(v,C)$ as well.
 \end{lemma}

\begin{proof} 
The lemma is clear if $v$ is in $\gamma$ (i.e.\ if $v$ coincides with $x_i$), and in particular it is clear if $\dim C=1$. To see the case $v\neq x_i$, we use induction on the dimension of $C$: Assume now $\dim C>1$. 

Since $x_j$ is connected to $\varGamma(a)$ by an edge, we have $j=i+1$. Indeed, if $j-i>2$, then the vertex path $(x_0,\cdots x_i, v, x_j, \cdots, x_\ell)$ is a vertex path that is strictly shorter than the necklace of $\varGamma$, which is not possible. If on the other hand $j-i=2$, then $\varGamma(a)$ contains $x_{i+1}$, which is consequently the pearl associated to $a$, in contradiction with the assumption.

Now, consider the combinatorial segment $\varGamma':=\Lk(x_i,\varGamma_{[\chi_i,\chi_{i+1}]})$ in $\Lk(x_i,C)$. Let $v':=\Lk(x_i,v \circledast x_i)$. Since the complex $\St(v,C)$ contains $x_{i+1}$ and since $C$ is flag, we obtain that $\St(v',\Lk(x_i,C))$ contains the pearl $\Lk(x_i,x_{i+1} \circledast x_i)$ of $\varGamma'$. Furthermore, $\varGamma'(a)$ is contained in $\St(v',\Lk(x_i,C))$ since $v\in \varGamma(a)$. Hence, the subpath $\varGamma'_{[a,\chi_{i+1}]}$ of $\varGamma'$ lies in $\St(v',\Lk(x_i,C))$ by induction assumption, so \[\varGamma_{[a,\chi_{i+1}]}=x_i\ast \varGamma'_{[a,\chi_{i+1}]}\subset x_i\ast \St(v',\Lk(x_i,C))\subset\St(v,C).\qedhere\]
\end{proof}

\begin{proof}[\textbf{Second proof of Theorem~\ref{THM:HIRSCHA}}]
Again, we use induction on the dimension; a combinatorial segment is clearly non-revisiting if $\dim C=0$. Assume now $\dim C\geq 1$, and consider a combinatorial segment $\varGamma$ that connects a facet $X$ with a facet $Y$ of $C$, as constructed above. Let $a,\ b$ be in the domain of $\varGamma$, associated with pearls $x_i$ and $x_j$, respectively. Assume that both $\varGamma(a)$ and $\varGamma(b)$ lie in the star of some vertex $v$ of $C$. Then the subpath $\varGamma_{[a,b]}$ of $\varGamma$ lies in the star $\St(v,C)$ of $v$ entirely. To see this, there are two cases to consider:
\begin{compactitem}[$\circ$]
\item {\bf If $i{=}j$:} By the inductive assumption, the combinatorial segment $\varGamma_{[\chi_i,\chi_{i+1}]}$ is non-revisiting, since it was obtained from the combinatorial segment $\Lk(x_i,\varGamma_{[\chi_i,\chi_{i+1}]})$ in the complex $\Lk(x_{i},C)$ by join with~$x_{i}$. Hence, the subpath $\varGamma_{[a,b]}$ of $\varGamma_{[\chi_{i},\chi_{i+1}]}$ is non-revisiting, and consequently lies in $\St(v,C)$.
\item {\bf If $i{<}j$:} Clearly, $v$ is connected to $x_i$ and $x_j$ via edges of $C$, since $\St(v,C)$ contains both $\varGamma(a)$ and~$\varGamma(b)$. Thus, we have that $\varGamma_{[a,b]}\subset \varGamma_{[\chi_{i},\chi_{j+1}]}$, that $\varGamma(a)$ lies in $\varGamma_{[\chi_{i},\chi_{j}]}$ and that $\varGamma(b)$ lies in $\varGamma_{[\chi_{j},\chi_{j+1}]}$. In particular, since $\St(v,C)$ contains $x_j$, we have that $\varGamma_{[a,\chi_{j}]}$ lies in $\St(v,C)$ by Lemma~\ref{lem:S}. Furthermore, we argued in the previous case that $\varGamma_{[\chi_{j},b]}$ lies in $\St(v,C)$, so that we obtain \[\varGamma_{[a,b]}=\varGamma_{[a,\chi_{j}]}\cdot \varGamma_{[\chi_{j},b]}\subset \St(v,C).\qedhere\]
\end{compactitem}
\end{proof}

\section{Some corollaries of Theorem~\ref{THM:HIRSCHA}}\label{sec:corH}

Theorem~\ref{THM:HIRSCHA} proves that flag triangulations of several topological objects naturally satisfy the Hirsch diameter bound.

First, a simplicial complex is said to \Defn{triangulate a manifold} if its underlying space is a manifold.

\begin{cor}\label{cor:HirschB}
All flag triangulations of connected manifolds satisfy the non-revisiting path property, and in particular the Hirsch diameter bound.
\end{cor}

A simplicial polytope is a \Defn{flag polytope} if its boundary complex is a flag complex.

\begin{cor}\label{cor:HirschC}
Every flag polytope satisfies the non-revisiting path conjecture, and in particular the Hirsch conjecture.
\end{cor}

\begin{rem}\label{rem:prb}
By a result of Provan and Billera~\cite{PB}, every \Defn{vertex-decomposable} simplicial complex satisfies the Hirsch diameter bound. As a corollary, they obtain the following famous result:

\begin{theorem}[Provan {\&} Billera {\cite[Cor.\ 3.3.4.]{PB}}]
Let $C$ be any shellable simplicial $d$-complex. Then the derived subdivision $\operatorname{sd} C$ of $C$ satisfies $\diam(\operatorname{sd} C)\leq f_0(\operatorname{sd} C)-d-1$. In particular, if $C$ is the boundary complex of any polytope, then $\sd C$ satisfies the Hirsch diameter bound.
\end{theorem}
The derived subdivision of an arbitrary triangulated manifold, however, is not vertex-decomposable in general. There are two reasons: There are topological obstructions (all vertex-decomposable manifolds are spheres or balls) as well as combinatorial obstructions (some spheres have non-vertex-decomposable derived subdivisions, cf.~\cite{HZ, BZ}). That said, the derived subdivision of any simplicial complex is flag. So, by Corollary~\ref{cor:HirschB}, we have the following:

\begin{cor}
The derived subdivision of any triangulation of any connected manifold satisfies the Hirsch diameter bound.
\end{cor}
 
\end{rem}
\part{Discrete Morse theory for stratified spaces}\label{pt:dmt}

\vspace*{\fill}
\thispagestyle{empty}
\begingroup
{\em
{\bf Morse theory} is a powerful tool to study smooth manifolds. Since it was introduced by Morse, several attempts have been made to generalize Morse theory to a variety of settings. This lead, among other things, to the development of ``non-smooth'' Morse theories, cf.\ \cite{Bestvina, GM-SMT, KirbySieb, Kuhnel, MorseTop}, which have proven to be exceedingly powerful tools for many purposes. A rather young development in this direction is discrete Morse theory by Forman \cite{FormanADV, FormanUSER}, which is based on the classical theory of simplicial collapses due to Whitehead \cite{Whitehead}. We will study the latter in Chapter \ref{ch:convcollapse} and conjectures in PL topology concerned with it.

The main purpose of this part of the thesis, however, is to demonstrate that discrete Morse theory is an excellent tool to study stratified spaces (Chapter~\ref{ch:minimal}): We apply it to generalize previous results on the topology of complements of subspace arrangements that previously relied on deep complex algebraic geometry. 
}
\endgroup
\vspace*{\fill}

\setcounter{thmmain}{0}
\setcounter{figure}{0}

\chapter{Collapsibility of convex and star-shaped complexes}\label{ch:convcollapse}

\section{Introduction}

\enlargethispage{1mm}

Collapsibility is a notion due to Whitehead~\cite{Whitehead}, created as part of his \emph{simple homotopy theory}. It is designed to reduce complicated CW-complexes to simpler ones while preserving the homotopy type. It is a powerful tool in PL topology, used for instance in some proofs of the generalized Poincar\'e conjecture~\cite{Stallings, ZeemanP}. Unfortunately, for a given simplicial complex, it is very hard to decide whether it is collapsible or not, cf.~\cite{Tancer10}. Three classical problems in this area are the conjectures/problems of Lickorish, Goodrick and Hudson:

\smallskip

\begin{compactitem}[$\circ$]
\item[{\bf (Goodrick's conjecture)} {\cite[Pb.\ 5.5 (B)]{Kirby}}] \emph{Every simplicial complex in $\R^d$ with a star-shaped underlying space is collapsible.}
\item[{\bf (Hudson's problem)} {\cite[Sec.\ 2, p.\ 44]{Hudson}}] \emph{Let $C,\ C'$ denote simplicial complexes, such that $C$ collapses to  the subcomplex $C'$. Is it true that, for every simplicial subdivision $D$ of $C$, $D$ collapses to the subcomplex $\RS(D,C')$?}
\item[{\bf (Lickorish's conjecture)} {\cite[Pb.\ 5.5 (A)]{Kirby}}] \emph{Every simplicial subdivision of a $d$-simplex is collapsible.}
\end{compactitem}

\smallskip

The three problems are closely related: Goodrick's conjecture implies Lickorish's conjecture, and so would a positive answer to Hudson's Problem. From the viewpoint of combinatorial topology, the last conjecture is probably the most natural, and it would simplify several results and proofs in combinatorial topology if it was proven true (cf.~\cite{Glaser, Hudson, ZeemanBK}).

To this day, the most striking progress on these questions was made by Chillingworth: He proved the collapsibility of convex $d$-balls for dimension $d \le 3$~\cite{CHIL}. His methods furthermore imply that Hudson's problem has a positive answer for simplicial complexes of dimension less than $4$. Finally, Goodrick's conjecture is only proven for the trivial cases of star-shaped complexes in $\R^d$, $d\le 2$. In this chapter, we provide some further progress: 

\smallskip

\begin{compactitem}[$\circ$]
\item[{\bf Theorem}~\ref{thm:ConvexEndo}] \emph{For every polytopal complex $C$ whose underlying space is a star-shaped set in~$\R^d$, $d\ge 3$, the simplicial complex $\sd^{d-2}  C$ is non-evasive, and in particular collapsible.}
\item[{\bf Theorem}~\ref{thm:hudson}] \emph{Let $C,\ C'$ denote polytopal complexes with $C$ collapsing to $C'$. Then for every subdivision $D$ of $C$, $\sd D$ collapses to $\sd \RS(D,C')$.}
\item[{\bf Theorem}~\ref{thm:liccon}] \emph{For every subdivision $C$ of a convex $d$-polytope, $\sd C$ is collapsible.}
\end{compactitem}

\subsection{Setup: Collapsibility and Non-evasiveness}
An \Defn{elementary collapse} is the deletion of a \Defn{free} face $\sigma$ from a polytopal complex~$C$, i.e.\ the deletion of a nonempty face $\sigma$ of $C$ that is strictly contained in only one other face of $C$. In this situation, we say that $C$ \Defn{(elementarily) collapses} onto $C-\sigma$, and write $C\searrow_e C-\sigma.$ We also say that the complex $C$ \Defn{collapses} to a subcomplex $C'$, and write~$C\searrow C'$, if $C$ can be reduced to $C'$ by a sequence of elementary collapses. 
A \Defn{collapsible} complex is a complex that collapses onto a single vertex. The notion of collapses has a topological motivation: If $C$ collapses to a subcomplex $C'$, then $C'$ is a deformation retract of $C$. In particular, collapsible complexes are contractible, i.e.\ homotopy equivalent to a point.

For our research, we will need the following elementary lemmas.

\begin{lemma}\label{lem:ccoll}
Let $C$ be a simplicial complex, and let $C'$ be a subcomplex of $C$. Then any cone over base $C$ collapses to the corresponding subcone over base $C'$.
\end{lemma}

\begin{lemma}\label{lem:cecoll}
Let $C$ be a simplicial complex, and let $v$ be any vertex of $C$. Assume that $\Lk(v,C)$ collapses to a subcomplex $S$. Then $C$ collapses to 
$(C-v)\cup (v\ast S).$ In particular, if $\Lk(v,C)$ is collapsible, then  $C\searrow C-v.$  
\end{lemma}

\begin{lemma}\label{lem:uc}
Let $C$ denote a simplicial complex that collapses to a subcomplex $C'$, and let $D$ be a simplicial complex such that $D\cup C$ is a simplicial complex. If $D\cap C=C'$, then $D\cup C\searrow D$.
\end{lemma}

\begin{proof}
It is enough to consider the case $C\searrow_e C'=C-\sigma$, where $\sigma$ is a free face of $C$. For this, notice that the natural embedding $C\mapsto D\cup C$ takes the free face $\sigma\in C$ to a free face of $D\cup C$.
\end{proof}

\Defn{Non-evasiveness} is a further strengthening of collapsibility for simplicial complexes that emerged in theoretical computer science~\cite{KahnSaksSturtevant}. A $0$-dimensional complex is non-evasive if and only if it is a point. Recursively, a $d$-dimensional simplicial complex ($d>0$) is non-evasive if and only if there is some vertex $v$ of the complex whose link and deletion are both non-evasive. The derived subdivision of every collapsible complex is non-evasive, and every non-evasive complex is collapsible~\cite{Welker}. A \Defn{non-evasiveness step} is the deletion from a simplicial complex $C$ of a single vertex whose link is non-evasive. Given two simplicial complexes $C$ and $C'$, we write $C\searrow_{\NE} C'$ if there is a sequence of non-evasiveness steps which lead from $C$ to $C'$. We will need the following lemmas, which are well known and easy to prove, cf.~\cite{Welker}.

\begin{lemma}\label{lem:conev}
Every simplicial complex that is a cone is non-evasive.
\end{lemma}

\begin{lemma}\label{lem:nonev}
If $C\searrow_{\NE} C'$, then $\sd^m  C   \searrow_{\NE}   \sd^m  C'$ for all non-negative $m$.
\end{lemma}

\begin{lemma}\label{lem:cone}
Let $C$ be a simplicial complex and let $v$ be any vertex of $C$. Let $m$ be a non-negative integer. If $\sd^m  \Lk(v,C)$ is non-evasive, then $\sd^m   C   \searrow_{\NE}   \sd^m (C-v)$.
\end{lemma}

\begin{proof}
We claim that the following, more general statement is true: 

\smallskip
\noindent {\em If $C$ is any simplicial complex, $v$ is any of its vertices, and $m$ is any non-negative integer, then $(\sd^m   C)-v   \searrow_{\NE}   \sd^m (C-v)$.}
\smallskip

The case $m=0$ is trivial. We treat the case $m =1$ as follows: The vertices of $\sd C$ correspond to faces of~$C$, the vertices that have to be removed in order to deform $(\sd C) -v$ to $\sd (C-v)$ correspond to the faces of $C$ strictly containing $v$. The order in which we remove the vertices of $(\sd C)-v$ is by increasing dimension of the associated face.

Let $\tau$ be a face of $C$ strictly containing $v$, and let $w$ denote the vertex of $\sd C$ it corresponds to. Assume all vertices corresponding to faces of $\tau$ have been removed from $(\sd C)-v$ already, and call the remaining complex~$D$. Then $\Lk(w, D)$ is combinatorially equivalent to the order complex of $\mathrm{L}(\tau,C)\cup \F(\tau-v)$, whose elements are ordered by inclusion. Here, $\mathrm{L}(\tau,C)$ is the set of faces of $C$ strictly containing $\tau$, and $\F(\tau-v)$ denotes the set of nonempty faces of $\tau-v$. Every maximal chain contains the face $\tau-v$, so $\Lk(w, D)$ is a cone, which is non-evasive by Lemma~\ref{lem:conev}. Thus, $D \searrow_{\NE} D-w$. The iteration of this procedure shows $(\sd C) -v \searrow_{\NE}  \sd (C-v)$, as desired.

The general case follows by induction on $m$: Assume that $m\geq 2$. Then 
\[(\sd^{m} C)-v= (\sd (\sd^{m-1} C))-v   \searrow_{\NE}   \sd ((\sd^{m-1} C) - v) \searrow_{\NE}  \sd (\sd^{m-1} (C - v))=  \sd^m (C - v),\]
by applying the induction assumption twice, together with Lemma~\ref{lem:nonev} for the second deformation.
\end{proof}

\section{Non-evasiveness of star-shaped complexes}

Here we show that \emph{any} subdivision of a star-shaped set becomes collapsible after $d-2$ derived subdivisions ({Theorem~\ref{thm:ConvexEndo}}); this provides progress on Goodrick's conjecture. In the remaining part of this chapter, we work with geometric polytopal complexes. Let us recall some definitions.

\begin{definition}[Star-shaped sets]
A subset $X \subset \R^d$ is \Defn{star-shaped} if there exists a point $x$ in $X$, a \Defn{star-center} of $X$, such that for each $y$ in $X$, the segment $[x,y]$ lies in $X$. 
Similarly, a subset $X \subset S^d$ is \Defn{star-shaped} if $X$ lies in a closed hemisphere of $S^d$ and there exists a \Defn{star-center} $x$ of $X$, in the interior of the hemisphere containing $X$, such that for each $y$ in $X$, the segment from $x$ to $y$ lies in $X$. With abuse of notation, a polytopal complex $C$ (in $\R^d$ or in $S^d$) is star-shaped if its underlying space is star-shaped.  
\end{definition}

\begin{definition}[Derived neighborhoods]
Let $C$ be a polytopal complex. Let $D$ be a subcomplex of $C$. The \Defn{(first) derived neighborhood} $N(D,C)$ of $D$ in $C$ is the polytopal complex
\[N(D,C):=\bigcup_{\sigma\in \sd D} \St(\sigma,\sd C) \] 
\end{definition}

\begin{lemma}\label{lem:star-shaped}
Let $C$ be a star-shaped polytopal complex in a closed hemisphere $\overline{H}_+$ of $S^d$. Let $D$ be the subcomplex of faces of $C$ that lie in the interior of $\overline{H}_+$ and assume that every nonempty face $\sigma$ of $C$ in $H$ is the facet of a unique face $\tau$ of $C$ that intersects both $D$ and $H$. 
Then the complex $N(D,C)$ has a star-shaped geometric realization in $\R^d$.
\end{lemma}

\begin{proof}
Let $m$ be the midpoint of $\overline{H}_+$. Let $B_r (m)$ be the closed metric ball in $\overline{H}_+$ with midpoint $m$ and radius $r$ (with respect to the canonical metric $\mathrm{d}$ on $S^d$). If  $C\subset \intx \overline{H}_+$, then $C$ has a realization as a star-shaped set in $\R^d$ by central projection, and we are done. Thus, we can assume that $C$ intersects $H$. Let $x$ be a star-center of~$C$ in the interior of $\overline{H}_+$. Since $D$ and $\{x\}$ are compact and in the interior of $\overline{H}_+$, there is some open interval $J = (R, \nicefrac{\pi}{2})$, $0< R< \nicefrac{\pi}{2}$, such that if $r$ is in $J$, then the ball $B_r(m)$ contains both $x$ and $D$.

If $\sigma$ is any face of $C$ in $H$, let $v_\sigma$ be any point in the relative interior of $\sigma$. If $\tau$ is any face of $C$ intersecting~$D$ and $H$, define $\sigma(\tau):= \tau\cap H$. For each~$r\in J$, choose a point $w(\tau, r)$ in the relative interior of $B_r(m) \cap \tau$, so that for each $\tau$, $w(\tau, r)$ depends continuously on $r$ and tends to $v_{\sigma(\tau)}$ as $r$ approaches $\nicefrac{\pi}{2}$. Extend each family $w(\tau, r)$ continuously to $r=\nicefrac{\pi}{2}$ by defining $w(\tau, \nicefrac{\pi}{2}):=v_{\sigma(\tau)}$.

Next, we use these one-parameter families of points to produce a one-parameter family $N_r(D,C)$ of geometric realizations of $N(D,C)$, where $r\in (R, \nicefrac{\pi}{2}]$. For this, let $\varrho$ be any face of $C$ intersecting $D$. If $\varrho$ is in~$D$, let $x_\varrho$ be any point in the relative interior of $\varrho$ (independent of $r$). If $\varrho$ is not in $D$, let $x_\varrho:=w(\varrho, r)$. We realize $N_r(D,C)\cong N(D,C)$ such that the vertex of $N_r(D,C)$ corresponding to the face $\varrho$ is~$x_\varrho$. 

This realizes $N_r(D,C)$ as a simplicial complex in $B_r (m) \subset \overline{H}_+$. Furthermore, $N_r(D,C)$ is combinatorially equivalent to $N(D,C)$ for every $r$ in $(R, \nicefrac{\pi}{2}]$: For $r\in J$, this is obvious; for $r=\nicefrac{\pi}{2}$, we have to check that $v_{\sigma(\tau)}\neq v_{\sigma(\tau')}$ for any pair of distinct faces $\tau$, $\tau'$ of $C$ intersecting $D$ and $H$. This follows since by assumption on $C$ we have that $\sigma(\tau)$ and $\sigma(\tau')$ are distinct faces of $\RS(C,H)$; in particular, their relative interiors are disjoint and $v_{\sigma(\tau)}\neq v_{\sigma(\tau')}$. To finish the proof, we claim that if $r$ is close enough to $\nicefrac{\pi}{2}$, then $N_r(D,C)$ is star-shaped with star-center $x$.

To this end, let us define the \Defn{extremal faces} of $N_r(D,C)$ as the faces all of whose vertices are in~$\partial B_r(m)$. We say that a pair $\sigma,\ \sigma'$ of extremal faces is \Defn{folded} if there are two points $a$ and $b$, in $\sigma$ and $\sigma'$ respectively and satisfying $\did(a,b)\le \mathrm{d}(x,H)+\nicefrac{\pi}{2}$, such that the subspace $\SSp \{a,b\}\subset S^d$ contains $x$, but the segment $[a, b]$ is not contained in~$N_r(D,C)$. 

When $r=\nicefrac{\pi}{2}$, folded faces do not exist since for every pair of points $a$ and $b$ in extremal faces, $\did(a,b)\le \mathrm{d}(x,H)+\nicefrac{\pi}{2}<\pi$, the subspace $\SSp \{a,b\}$ lies in $\partial B_{\nicefrac{\pi}{2}}=H$; hence, all such subspaces have distance at least $\mathrm{d}(x,H)>0$ to $x$. Thus, since we chose the vertices of $N_r(D,C)$ to depend continuously on $r\in (R, \tfrac{\pi}{2}]$, we can find a real number $R'$, $R< R'< \nicefrac{\pi}{2}$, such that for any $r$ in the open interval $J' := \left(R',\nicefrac{\pi}{2} \right),$
the simplicial complex $N_r(D,C)$ contains no folded pair of faces. For the rest of the proof, let us assume that $r\in J'$, so that folded pairs of faces are avoided. Let $y$ be any point in $N_r(D,C)$. We claim that the segment $[x,y]$ lies in $N_r(D,C)$.

If $z\in [x,y] \cap N_r(D,C)$ is not in an extremal face, then there exists an open neighborhood $U$ of $z$ such that $U\cap N_r (D,C)=U\cap C.$ In particular, \[[x,y]\cap U\cap N_r(D,C)=[x,y]\cap U\cap C,\] so $[x,y]$ can not leave $N_r (D,C)$ in $z$. So if $[x,y]$ leaves and reenters $N_r(D,C)$, it must do so through extremal faces. Since the points of leaving and reentry are at distance at most $\mathrm{d}(x,y)\le \mathrm{d}(x,H)+\nicefrac{\pi}{2}$ from each other, the faces are consequently folded. But since $r$ is in $J'$, there are no folded pairs of faces. Therefore, we see that $[x,y]$ lies in $N_r(D,C)$, which is consequently star-shaped in~$B_r(m)$. Using a central projection to $\R^d$, we obtain the desired star-shaped realization of $N(D,C)$.
\end{proof}

\begin{definition}[Derived order]\label{def:extord}
Recall that an \Defn{extension} of a partial order $\prec$ on a set $S$ is any partial order $\widetilde{\prec}$ on any superset $T$ of $S$ such that $a\, \widetilde{\prec}\, b$ whenever $a \prec b$ for $a,b\in S$.

Let now $C$ be a polytopal complex, let $S$ denote a subset of its faces, and let $\prec$ denote any strict total order on $S$ with the property that $a\prec b$ whenever $b\subset a$. We extend this order to an irreflexive partial order $\widetilde{\prec}$ on $C$ as follows: Let $\sigma$ be any face of $C$, and let $\tau\subsetneq \sigma$ be any strict face of $\sigma$. 
\begin{compactitem}[$\circ$]
\item If $\tau$ is the minimal face of $\sigma$ under $\prec$, then $\tau\,\widetilde{\prec}\, \sigma$.
\item If $\tau$ is any other face of $\sigma$, then $\sigma\, \widetilde{\prec}\, \tau$.
\end{compactitem}
The transitive closure of the relation $\widetilde{\prec}$ gives an irreflexive partial order on the faces of $C$, and by the correspondence of faces of $C$ to the vertices of $\sd  C$, it gives an irreflexive partial order on $\F_0(\sd  C)$. Any strict total order that extends the latter order is a \Defn{derived order} of $\F_0(\sd  C)$ induced by $\prec$.
\end{definition}

\newcommand{\LLk}{\mathrm{LLk}^{ \nu}}
\newcommand{\SLk}{\mathrm{SLk}^{ \nu}}

\begin{definition}[$H$-splitting derived subdivisions]
Let $C$ be a polytopal complex in $\mathbb{R}^d$, and let $H$ be a hyperplane of $\mathbb{R}^d$. An \Defn{$H$-splitting derived subdivision} of $C$ is a derived subdivision, with vertices chosen so that for any face $\tau$ of $C$ that intersects the hyperplane $H$ in the relative interior, the vertex of $\sd C$ corresponding to $\tau$ in $C$ lies in $H$.
\end{definition}

\begin{figure}[htbf]
\centering
\includegraphics[width=0.6\linewidth]{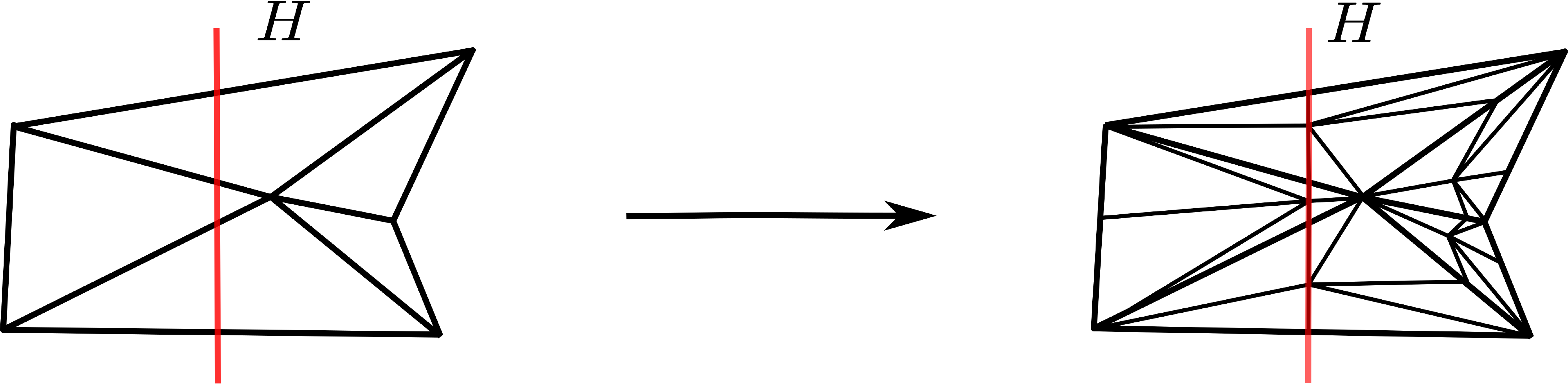}
\caption{An instance of the $H$-splitting derived subdivisions.}
\label{fig:split}
\end{figure}

\begin{definition}[Split link and Lower link]\label{def:slk}
Let $C$ be a simplicial complex in $\R^d$. Let $v$ be a vertex of $C$, and let $\overline{H}_+$ be a closed halfspace in $\R^d$ with outer normal $\nu$ at $v\in H:=\partial \overline{H}_+$. The \Defn{split link} (of $C$ at $v$ with respect to $\nu$), denoted by $\SLk (v, C)$, is the intersection of $\Lk (v,  C)$ with the hemisphere $\RN_v^1 \overline{H}_+$, that is, \[\SLk (v, C):= \{\sigma \cap \RN_v^1 \overline{H}_+: \sigma \in \Lk(v,C)\}.\]
The \Defn{lower link} $\LLk (v,C)$ (of~$C$ at $v$ with respect to $\nu$) is the subcomplex $\RS(\Lk(v,C), \intx \RN_v^1 \overline{H}_+)$ of $\Lk(v,C)$. The complex $\LLk (v, C)$ is naturally a subcomplex of $\SLk (v,T)$: we have $\LLk (v,C)=\RS(\SLk(v,C), \intx \RN_v^1 \overline{H}_+)$ 
\end{definition}

Finally, for a polytopal complex $C$ and a face $\tau$, we denote by $\mathrm{L}(\tau,C)$ the set of faces of $C$ strictly containing $\tau$.

\begin{thmmain}\label{thm:ConvexEndo}
Let $C$ be a polytopal complex in $\R^d$, $d\geq 3$, or a simplicial complex in $\R^2$. If $C$ is star-shaped in $\R^d$, then $\sd^{d-2}  (C) $ is non-evasive, and in particular collapsible. 
\end{thmmain}

\begin{proof}
The proof is by induction on the dimension. The case $d=2$ is easy: Every contractible planar simplicial complex is non-evasive (cf.\ \cite[Lemma 2.3]{AB-tight}). 

Assume now $d\ge 3$. Let $\nu$ be generic in $S^{d-1}\subset \R^{d}$, so that no edge of $C$ is orthogonal to~$\nu$. Let $H$ be a hyperplane through a star-center $x$ of $C$ such that $H$ is orthogonal to~$\nu$. From now on to the end of this proof, let $\sd C$ denote any $H$-splitting derived subdivision of $C$. Let $\overline{H}_+$ (resp.~$\overline{H}_-$) be the closed halfspace bounded by $H$ in direction $\nu$ (resp.\ $-\nu$), and let $H_+$ (resp.\ $H_-$) denote their respective interiors. We claim: 
\begin{compactenum}[(1)]
\item For $v\in \F_0(\RS( C, {H}_+))$, the complex $\sd^{d-3}  N (\LLk (v, C),\Lk(v,C))$ is non-evasive. 
\item For $v\in \F_0(\RS( C, {H}_-))$, the complex $\sd^{d-3}  N (\mathrm{LLk}^{-\nu} (v, C), \Lk(v,C))$ is non-evasive. 
\item $\sd^{d-3}  \RS(\sd  C, \overline{H}_+) \; \searrow_{\NE} \; \sd^{d-3}  \RS( \sd C, H).$
\item $\sd^{d-3}  \RS(\sd  C, \overline{H}_-) \; \searrow_{\NE} \; \sd^{d-3}  \RS( \sd  C, H).$
\item $\sd^{d-3}  \RS( \sd  C, H)$ is non-evasive.
\end{compactenum}
The combination of these claims will finish the proof directly.
\begin{compactenum}[(1)]
\item Let $v$ be a vertex of $C$ that lies in ${H}_+$. The complex $\SLk(v,C)$ is star-shaped in the $(d-1)$-sphere~$\RN^1_v \R^d$; its star-center is the tangent direction of the segment $[v,x]$ at $v$. Furthermore, $\nu$ is generic, so $\SLk(v,C)$ satisfies the assumptions of Lemma~\ref{lem:star-shaped}.

We obtain that the complex $N(\LLk (v, C),\SLk(v,C))\cong N(\LLk (v, C),\Lk(v,C))$ has a star-shaped geometric realization in $\R^{d-1}$. The inductive assumption hence gives that the simplicial complex $\sd^{d-3}  N (\LLk (v, C),\Lk(v,C))$ is non-evasive.

\item This is analogous to (1).

\item The vertices of $C$ are naturally ordered by the function $\langle \cdot, \nu  \rangle$, starting with the vertex of $C$ maximizing the functional $\langle \cdot, \nu  \rangle$. We give a strict total order on the vertices of $ \sd  C$ by using any derived order induced by this order (Definition~\ref{def:extord}). Let $v_0,v_1,\, \cdots, v_n$ denote the vertices of $\RS(\sd  C, {H}_+)\subset \sd C$, labeled according to the derived order (and starting with the maximal vertex $v_0$).

Define $C_i:=\RS(\sd  C, \overline{H}_+)-\{v_0,\, \cdots, v_{i-1}\}$ and $\varSigma_i:=\sd^{d-3} C_i.$ It remains to show that, for all $i$, $0\le i\le n$, we have $\varSigma_i\searrow_{\NE}\varSigma_{i+1}$.

\begin{compactitem}[$\circ$]
\item If $v_i$ corresponds to a face $\tau$ of $C$ of positive dimension, then let $w$ denote the vertex of $\tau$ minimizing~$\langle \cdot,\nu \rangle$. The complex $\Lk(v_i, C_i)$ is combinatorially equivalent to the order complex associated to the union $\mathrm{L}(\tau,C)\cup w$,  whose elements are ordered by inclusion. Since $w$ is the unique minimum in that order, $\Lk(v_i, C_i)$ is combinatorially equivalent to a cone over base $\sd\Lk (\tau,C)$. Every cone is non-evasive (Lemma~\ref{lem:conev}). Thus, $C_i \searrow_{\NE} C_i - v_i=C_{i+1}$. By Lemma~\ref{lem:nonev}, $\varSigma_i \searrow_{\NE} \varSigma_{i+1}$.
\item If $v_i$ corresponds to any vertex of $C$, we have by claim (1) that the $(d-3)$-rd derived subdivision of $\Lk (v_i,C_i)\cong N (\LLk (v_i, C),\Lk(v_i,C))$ is non-evasive. With Lemma~\ref{lem:cone}, we conclude that \[\varSigma_i = \sd^{d-3} C_i  \searrow_{\NE} \sd^{d-3} (C_i-v_i) = \sd^{d-3} C_{i+1} = \varSigma_{i+1}.\]
\end{compactitem}
We can proceed deleting vertices, until the remaining complex has no vertex in ${H}_+$. Thus, the complex $\sd^{d-3} \RS(\sd C, \overline{H}_+)$ can be deformed to $\sd^{d-3} \RS( \sd C, H)$ by non-evasiveness steps.
\item This is analogous to (3).
\item This is a straightforward consequence of the inductive assumption. In fact, $\RS(\sd C,H)$ is star-shaped in the $(d-1)$-dimensional hyperplane $H$: Hence, the inductive assumption gives that $\sd^{d-3}  \RS(\sd  C, H)$ is non-evasive.
\end{compactenum}
This finishes the proof: Observe that if $A$, $B$ and $A\cup B$ are simplicial complexes with the property that $A , B \searrow_{\NE} A \cap B$, then $A\cup B \searrow_{\NE} A\cap B$. Applied to the complexes $\sd^{d-3}  \RS(\sd  C, \overline{H}_+)$, $\sd^{d-3}  \RS(\sd  C, \overline{H}_-)$ and $\sd^{d-2}  C=\sd^{d-3}  \RS(\sd  C, \overline{H}_+)\cup \sd^{d-3}  \RS(\sd  C, \overline{H}_-)$, the combination of (3) and (4) shows that $\sd^{d-2}  C   \searrow_{\NE} \sd^{d-3} \RS (\sd C,H) $, which in turn is non-evasive by (5).
\end{proof}

\section{Collapsibility of convex complexes}

In the preceding section, we proved Goodrick's conjecture up to $(d-2)$ derived subdivisions. In this section, we establish that Lickorish's conjecture and Hudson's problem can be answered positively up to one derived subdivision (Theorems~\ref{thm:hudson} and~\ref{thm:liccon}). For this, we rely on the following Theorem~\ref{thm:ConvexEndo2}, the proof of which is very similar to the proof of Theorem~\ref{thm:ConvexEndo}. As usual, we say that a polytopal complex $C$ (in $\R^d$ or in $S^d$) is convex if its underlying space is convex. A hemisphere in $ S^d$ is in \Defn{general position} with respect to a polytopal complex $C\in S^d$ if it contains no vertices of $C$ in the boundary.

\begin{theorem}\label{thm:ConvexEndo2}
Let $C$ be a convex polytopal $d$-complex in $S^d$ and let $\overline{H}_+$ be a closed hemisphere of $S^d$ in general position with respect to $C$. Then we have the following:
\begin{compactenum}[\rm (A)]
\item If $\partial C\cap \overline{H}_+=\emptyset$, then $N(\RS(C,\overline{H}_+),C)$ is collapsible.
\item If $\partial C\cap \overline{H}_+$ is nonempty, and $C$ does not lie in $\overline{H}_+$, then $N(\RS(C,\overline{H}_+),C)$ collapses to the subcomplex $N(\RS(\partial C,\overline{H}_+),\partial C)$.
\item If $C$ lies in $\overline{H}_+$, then there exists some facet $F$ of $\sd \partial C$ such that $\sd C$ collapses to $C_F:=\sd \partial C-F$.
\end{compactenum}
\end{theorem}

In the statement of Theorem~\ref{thm:ConvexEndo2}(C), it is not hard to see that once the statement is proven for some facet $F$, then $F$ can be chosen arbitrarily among the facets of $C$. Indeed, for any simplicial ball $B$, the following statement holds (cf.~\cite[Prp.\ 2.4]{BZ}): If $\sigma$ is some facet of $\partial B$, and $B\searrow \partial B- \sigma$, then for any facet $\tau$ of $B$, we have that $B-\tau\searrow \partial B$. Thus, we have the following corollary:
\begin{cor}\label{cor:ConvexEndo2}
Let $C$ be a convex polytopal complex in $\R^d$.  Then for any facet $\sigma$ of $C$ we have that $\sd C-\sigma$ collapses to $\sd \partial C$.
\end{cor}
\begin{proof}[\textbf{Proof of Theorem~\ref{thm:ConvexEndo2}}]
Claim (A), (B) and (C) can be proved analogously to the proof of Theorem~\ref{thm:ConvexEndo}, by induction on the dimension. Let us denote the claim that (A) is true for dimension $\le d$ by $\io_d$, the claim that (B) is true for dimension $\le d$ by $\iit_d$, and finally the claim that (C) is true for dimension $\le d$ by $\iii_d$. Clearly, the
claims are true for dimension $d=0$. For $d>0$ the inductive proof goes as follows:
\begin{compactitem}[$\circ$]
\item $\io_{d-1}$ implies $\io_{d}$,
\item $\io_{d-1}$ and $\iit_{d-1}$ imply $\iit_{d}$, and
\item $\io_{d-1}$, $\iit_{d-1}$ and $\iii_{d-1}$ imply $\iii_{d}$.
\end{compactitem}
The proofs are very similar in nature; we can even treat cases (A) and (B) simultaneously. We assume, from now on, that $\io_{d-1}$, $\iit_{d-1}$ and $\iii_{d-1}$ are proven already, and proceed to prove $\io_{d}$, $\iit_{d}$ and~$\iii_{d}$. Recall that $\mathrm{L}(\tau,C)$ is the set of faces of $C$ strictly containing a face $\tau$ of $C$. Furthermore, we will make use of the notions of derived order (Definition~\ref{def:extord}) and lower link (Definition~\ref{def:slk}) from the preceding section.

\smallskip

\noindent \textbf{Cases} (A) \textbf{and} (B): Let $x$ denote the midpoint  of $\overline{H}_+$, and let $\mathrm{d}(y)$ denote the distance of a point $y \in S^d$ to $x$ with respect to the canonical metric on $S^d$. 

If $\partial C\cap \overline{H}_+\neq \emptyset$, then $|C|$ is a polyhedron that intersects $S^d{\setminus}\overline{H}_+$ in its interior since $\overline{H}_+$ is in general position with respect to $C$. Thus, $\overline{H}_+{\setminus}C$ is star-shaped, and for every $p$ in $\intx C{\setminus}\overline{H}_+$, the point $-p\in \intx \overline{H}_+$ is a star-center for it. In particular, the set of star-centers of $\overline{H}_+{\setminus}C$ is open. Up to a generic projective transformation $\varphi$ of $S^d$ that takes $\overline{H}_+$ to itself, we may consequently assume that $x$ is a generic star-center of~$\overline{H}_+{\setminus} C$.

Let $\mathrm{M}(C, \overline{H}_+)$ denote the set of faces $\sigma$ of $\RS(C,\overline{H}_+)$ for which the function ${\arg\min}_{y\in \sigma} \mathrm{d}(y)$ attains its minimum in the relative interior of $\sigma$. With this, we order the elements of $\mathrm{M}(C, \overline{H}_+)$ strictly by defining $\sigma\prec \sigma'$ whenever $\min_{y\in\sigma}\mathrm{d}(y)<\min_{y\in\sigma'}\mathrm{d}(y)$.

This allows us to induce an associated derived order on the vertices of $\sd C$, which we restrict to the vertices of $N(\RS(C,\overline{H}_+),C)$. Let $v_0, v_1, v_2, \, \cdots, v_n$ denote the vertices of $N(\RS(C,\overline{H}_+),C)$ labeled according to the latter order, starting with the maximal element $v_0$. Let $C_i$ denote the complex $N(\RS(C,\overline{H}_+),C)-\{v_0, v_1, \, \cdots, v_{i-1}\},$ and define $\Sigma_i:=C_i\cup N(\RS(\partial C,\overline{H}_+),\partial C)$. We will prove that $\Sigma_i\searrow \Sigma_{i+1}$ for all $i$, $0\le i\le n-1$; this proves $\io_{d}$ and $\iit_{d}$. There are four cases to consider here. 

\newpage

\begin{compactenum}[(1)]
\item $v_i$ is in the interior of $\sd C$ and corresponds to an element of $\mathrm{M}(C, \overline{H}_+)$.
\item $v_i$ is in the interior of $\sd C$ and corresponds to a face of $C$ not in $\mathrm{M}(C, \overline{H}_+)$.
\item $v_i$ is in the boundary of $\sd C$ and corresponds to an element of $\mathrm{M}(C, \overline{H}_+)$.
\item $v_i$ is in the boundary of $\sd C$ and corresponds to a face of $C$ not in $\mathrm{M}(C, \overline{H}_+)$.
\end{compactenum}
We need some notation to prove these four cases. Recall that we can define $\RN$, $\RN^1$ and $\Lk$ with respect to a basepoint; we shall need this notation in cases (1) and (3). We shall abbreviate $v:=v_i$ for the duration of the proof of case (A) and (B). Furthermore, let us denote by $\tau$ the face of $C$ corresponding to $v$ in $\sd C$, and let $m$ denote the point ${\arg\min}_{y\in \tau} \mathrm{d}(y)$. Finally, define the ball $B_m$ as the set of points $y$ in $S^d$ with~$\mathrm{d}(y)\le\mathrm{d}(m)$. 

\smallskip
\noindent \emph{Case $(1)$}:  In this case, the complex $\Lk(v,\varSigma_i)$ is combinatorially equivalent to $N(\mathrm{LLk}_m(\tau,C),\Lk_m(\tau,C))$, where \[\mathrm{LLk}_m(\tau,C):= \RS(\Lk_m(\tau,C),  \RN^1_{(m,\tau)} B_m)\] is the restriction of $\Lk_m(\tau,C)$ to the hemisphere $ \RN^1_{(m,\tau)} B_m$ of $\RN^1_{(m,\tau)} S^d$. Since the projective transformation $\varphi$ was chosen to be generic, $\RN^1_{(m,\tau)} B_m$ is in general position with respect to $\Lk_m(\tau,C)$. Hence, by assumption $\io_{d-1}$, the complex \[N(\mathrm{LLk}_m(\tau,C),\Lk_m(\tau,C))\cong \Lk (v,\varSigma_i)\] is collapsible. Consequently,  Lemma~\ref{lem:cecoll} proves $\varSigma_i\searrow \varSigma_{i+1}=\varSigma_i-v.$

\smallskip
\noindent \emph{Case $(2)$}: If $\tau$ is not an element of $\mathrm{M}(C, \overline{H}_+)$, let $\sigma$ denote the face of $\tau$ containing $m$ in its relative interior. Then, $\Lk(v, \varSigma_i)=\Lk(v, C_i)$ is combinatorially equivalent to the order complex of the union $\mathrm{L}(\tau,C)\cup \sigma$, whose elements are ordered by inclusion. Since $\sigma$ is a unique global minimum of the poset, the complex $\Lk(v, \varSigma_i)$ is a cone, and in fact naturally combinatorially equivalent to a cone over base $\sd \Lk(\tau, C)$. Thus, $\Lk(v, \varSigma_i)$ is collapsible (Lemma~\ref{lem:ccoll}). Consequently, Lemma~\ref{lem:cecoll} gives $\varSigma_i\searrow \varSigma_{i+1}=\varSigma_i-v$.

\smallskip 

This takes care of (A). To prove case (B), we have to consider the two additional cases in which $v$ is a boundary vertex of $\sd C$. 

\smallskip
\noindent \emph{Case $(3)$}: As in case {(1)}, $\Lk (v,C_i)$ is combinatorially equivalent to the complex \[N(\mathrm{LLk}_m(\tau,C),\Lk_m(\tau,C)),\  \  \mathrm{LLk}_m(\tau,C):= \RS(\Lk_m(\tau,C),  \RN^1_{(m,\tau)} B_m)\] in the sphere $\RN^1_{(m,\tau)} S^d$. Recall that $\overline{H}_+{\setminus}C$ is star-shaped with star-center $x$ and that $\tau$ is not the face of $C$ that minimizes $\mathrm{d}(y)$ since $v\neq v_n$, so that $\RN^1_{(m,\tau)} B_m\cap \RN^1_{(m,\tau)} \partial C$ is nonempty. Since furthermore $\RN^1_{(m,\tau)} B_m$ is a hemisphere in general position with respect to the complex $\Lk_m(\tau,C)$ in the sphere $\RN^1_{(m,\tau)} S^d$, the induction assumption $\iit_{d-1}$ applies: The complex $N(\mathrm{LLk}_m(\tau,C),\Lk_m(\tau,C))$ collapses to 
\[N(\mathrm{LLk}_m(\tau,\partial C),\Lk_m(\tau,\partial C))\cong \Lk (v,C'_i),\ \  C'_i:=C_{i+1} \cup (C_i\cap N(\RS(\partial C,\overline{H}_+),\partial C)).\]
Consequently, Lemma~\ref{lem:cecoll} proves that $C_i$ collapses to $C'_i$. Since \[ \varSigma_{i+1}\cap C_i =(C_{i+1}\cup N(\RS(\partial C,\overline{H}_+),\partial C))\cap C_i=  C_{i+1} \cup (C_i\cap N(\RS(\partial C,\overline{H}_+),\partial C))= C'_i\]
Lemma~\ref{lem:uc}, applied to the union $\varSigma_i=C_i\cup \varSigma_{i+1}$ of complexes $C_i$ and $\varSigma_{i+1}$ gives that $\varSigma_i$ collapses onto~$\varSigma_{i+1}.$

\smallskip
\noindent \emph{Case $(4)$}: As observed in case {(2)}, the complex $\Lk(v,C_i)$ is naturally combinatorially equivalent to a cone over base $\sd \Lk(\tau, C)$, which collapses to the cone over the subcomplex $ \sd  \Lk(\tau, \partial C)$ by Lemma~\ref{lem:ccoll}. Thus, the complex $C_i$ collapses to $C'_i:=C_{i+1}\cup(C_i \cap N(\RS(\partial C,\overline{H}_+),\partial C))$ by Lemma~\ref{lem:cecoll}. Now, we have $\varSigma_{i+1}\cap C_i=C'_i$ as in case {(3)}, so that $\varSigma_i$ collapses onto $\varSigma_{i+1}$ by Lemma~\ref{lem:uc}.

\smallskip

This finishes the proof of cases (A) and (B) of Theorem~\ref{thm:ConvexEndo2}. It remains to prove the notationally simpler case (C).

\smallskip

\noindent \textbf{Case} (C): Since $C$ is contained in the open hemisphere $\intx \overline{H}_+$, we may assume, by central projection, that $C$ is actually a convex simplicial complex in $\R^d$. Let $\nu$ be generic in $S^{d-1}\subset\R^d$.

Consider the vertices of $C$ as ordered according to decreasing value of $\langle \cdot, \nu \rangle$, and order the vertices of $\sd C$ by any derived order extending that order.  Let $v_i$ denote the $i$-th vertex of $\F_0(\sd C)$ in the derived order, starting with the maximal vertex $v_0$ and ending up with the minimal vertex $v_n$. 

The complex $\Lk(v_0, C)= \LLk(v_0,C)$ is a subdivision of the convex polytope $\RN^1_{v_0} |C|$ in the sphere $\RN^1_{v_0} \R^d$ of dimension $d-1$. By assumption $\iii_{d-1}$, the complex 
$
\Lk(v_0,\sd C)\cong \sd \Lk(v_0,C)
$
collapses onto $\partial \sd  \Lk(v_0, C)-F'$, where $F'$ is some facet of $\partial \Lk(v_0,\sd C)$. By Lemma~\ref{lem:cecoll}, the complex $\sd C$ collapses to \[\varSigma_1:=(\sd C-v_0)\cup (\partial \sd C -F) = (\sd C-v_0)\cup C_F,\] where $F:=v_0\ast F'$ and $C_F=\partial \sd C - F$. 

We proceed removing the vertices one by one, according to their position in the order we defined. More precisely, set $C_i:=\sd  C-\{v_0,\, \cdots, v_{i-1}\}$, and set $\varSigma_i:= C_i\cup C_F$ . We shall now show that $\varSigma_i\searrow \varSigma_{i+1}$ for all $i$, $1\le i\le n-1$; this in particular implies $\iii_{d}$. There are four cases to consider:

\begin{compactenum}[(1)]
\item $v_i$ corresponds to an interior vertex of $C$.
\item $v_i$ is in the interior of $\sd C$ and corresponds to a face of $C$  of positive dimension.
\item $v_i$ corresponds to a boundary vertex of $C$.
\item $v_i$ is in the boundary of $\sd C$ and corresponds to a face of $C$ of positive dimension.
\end{compactenum}
We shall abbreviate $v:=v_i$.

\smallskip
\noindent \emph{Case $(1)$}: In this case, the complex $\Lk (v,\varSigma_i)$ is combinatorially equivalent to the simplicial complex $N(\LLk(v,C),\Lk(v,C))$ in the $(d-1)$-sphere $\RN^1_{v} \R^d$. By assumption $\io_{d-1}$, the complex \[N(\LLk(v,C),\Lk(v,C))\cong \Lk (v,\varSigma_i)\] is collapsible. Consequently, by Lemma~\ref{lem:cecoll}, the complex $\varSigma_i$ collapses onto $\varSigma_{i+1}=\varSigma_i-v$.

\smallskip
\noindent \emph{Case $(2)$}: If $v$ corresponds to a face $\tau$ of $C$ of positive dimension, let $w$ denote the vertex of $\tau$ minimizing~$\langle \cdot, \nu \rangle$. The complex $\Lk(v, \varSigma_i)$ is combinatorially equivalent to the order complex of the union $\mathrm{L}(\tau,C)\cup w$, whose elements are ordered by inclusion. Since $w$ is a unique global minimum of this poset, the complex $\Lk(v, \varSigma_i)$ is a cone (with a base naturally combinatorially equivalent to $\sd   \Lk(\tau, C)$). Thus, $\Lk(v, \varSigma_i)$ is collapsible since every cone is collapsible (Lemma~\ref{lem:ccoll}). Consequently, Lemma~\ref{lem:cecoll} gives $\varSigma_i\searrow \varSigma_{i+1}=\varSigma_i-v$.

\smallskip
\noindent \emph{Case $(3)$}: Similar to case {(1)}, $\Lk (v,C_i)$ is combinatorially equivalent to $N(\LLk(v,C),\Lk(v,C))$ in the $(d-1)$-sphere $\RN^1_{v} \R^d$. By assumption $\iit_{d-1}$, the complex \[N(\LLk(v,C),\Lk(v,C))\cong \Lk (v,C_i) \] collapses to 
\[N(\LLk(v, \partial C), \Lk(v,\partial C))\cong \Lk (v,C'_i),\ \  C'_i:=C_{i+1} \cup (C_i\cap C_F).\]
Consequently, Lemma~\ref{lem:cecoll} proves that $C_i$ collapses to $C'_i$. Since
\[ \varSigma_{i+1}\cap C_i =(C_{i+1}\cup C_F)\cap C_i=  C_{i+1} \cup (C_i\cap C_F)= C'_i\]
Lemma~\ref{lem:uc}, applied to the union $\varSigma_i=C_i\cup \varSigma_{i+1}$ of complexes $C_i$ and $\varSigma_{i+1}$ gives that $\varSigma_i$ collapses to the subcomplex~$\varSigma_{i+1}.$

\smallskip
\noindent \emph{Case $(4)$}: As seen in case {(2)}, $\Lk(v,C_i)$ is naturally combinatorially equivalent to a cone over base $\sd   \Lk(\tau, C)$, which collapses to the cone over the subcomplex $ \sd  \Lk(\tau, \partial C)$ by Lemma~\ref{lem:ccoll}. Thus, the complex $C_i$ collapses to $C'_i:=C_{i+1}\cup(C_i \cap C_F)$ by Lemma~\ref{lem:cecoll}. Now, we have $\varSigma_{i+1}\cap C_i=C'_i$ as in case {(3)}, so that $\varSigma_i$ collapses onto $\varSigma_{i+1}$ by Lemma~\ref{lem:uc}.
\end{proof}

\subsection{Lickorish's conjecture and Hudson's problem}

In this section, we provide the announced partial answers to Lickorishs's conjecture (Theorem~\ref{thm:liccon}) and Hudson's Problem (Theorem~\ref{thm:hudson}). 

\begin{thmmain}\label{thm:hudson}
Let $C, C'$ be polytopal complexes such that $C' \subset C$ and $C\searrow C'$. Let $D$ denote any subdivision of $C$, and define $D':=\RS(D,C')$. Then, $\sd D   \searrow   \sd D'$.
\end{thmmain}

\begin{proof}
It clearly suffices to prove the claim for the case where $C \searrow_e C'$, i.e.\ we may assume that $C'$ is obtained from $C$ by an elementary collapse. Let $\sigma$ denote the free face deleted in the collapsing, and let $\varSigma$ denote the unique face of $C$ that strictly contains it.

Let $(\delta, \varDelta)$ be any pair of faces of $\sd D$, such that $\delta$ is a facet of $\RS(\sd D,\sigma)$, $\varDelta$ is a facet of $\RS(\sd D,\varSigma)$, and $\delta$ is a codimension-one face of $\varDelta$. With this, the face $\delta$ is a free face of $\sd D$.
Now, by Corollary~\ref{cor:ConvexEndo2}, $\RS(\sd D,\varSigma)-\varDelta$ collapses onto $\RS(\sd D,\partial \varSigma)$. Thus, $\sd D-\varDelta$ collapses onto $\RS(\sd D,\sd D{\setminus}\rint\varSigma)$, or equivalently, \[\sd D-\delta\searrow \RS(\sd D,\sd D{\setminus}\rint\varSigma)-\delta.\] 
Now, $\RS(\sd D,\sigma)-\delta$ collapses onto $\RS(\sd D,\partial \sigma)$ by Corollary~\ref{cor:ConvexEndo2}, and thus \[\RS(\sd D,\sd D{\setminus}\rint\varSigma)-\delta \searrow \RS(\sd D,\sd D{\setminus}(\rint\varSigma\cup\rint\sigma)).\] To summarize, if $C$ can be collapsed onto $C-\sigma$, then \[\sd D   \searrow_e   \sd D-\delta \searrow  \RS(\sd D,\sd D{\setminus}(\rint\varSigma\cup\rint\sigma)) =  \RS(\sd D, C-\sigma)= \RS(\sd D, C').\qedhere\]
\end{proof}

\begin{thmmain}\label{thm:liccon}
Let $C$ denote any subdivision of a convex $d$-polytope. Then $\sd C$ is collapsible.
\end{thmmain}

\begin{proof}
It is a classical theorem of Bruggesser--Mani~\cite{BruggesserMani} that any $d$-polytope $P$ is shellable; we hold it for sufficiently clear that their proof implies that every polytope, if interpreted as the complex formed by its faces (including the polytope itself as a facet), is collapsible. Indeed, the shelling order provided by Bruggesser--Mani provides also an order by which the $d-1$-faces of $P$ can be removed using a sequence of collapses. Thus, for any subdivision $C$ of a convex polytope, $\sd C$ is collapsible by Theorem~\ref{thm:hudson}.
\end{proof}
\setcounter{thmmain}{0}
\setcounter{figure}{0}

\chapter{Minimality of 2-arrangements}\label{ch:minimal}
\section{Introduction}

A $c$-arrange\-ment is a finite collection of distinct affine subspaces of $\R^d$, all of codimension $c$, with the property that the codimension of the non-empty intersection of any subset of $\HA$ is a multiple of $c$. For example, after identifying $\CC$ with $\R^2$, any collection of hyperplanes in $\CC^d$ can be viewed as a $2$-arrangement in $\R^{2d}$. However, not all $2$-arrangements arise this way, cf.~\cite[Sec.\ III,\ 5.2]{GM-SMT},~\cite{Z-DRC}. 
In this chapter, we study the complement $\HA^{\comp}:=\R^{d}{\setminus} \HA$ of any $2$-arrangement $\HA$ in $\R^d$.

Subspace arrangements $\HA$ and their complements $\HA^{\comp}$ have been extensively studied in several areas of mathematics. Thanks to the work by Goresky and MacPherson~\cite{GM-SMT}, the homology of $\HA^{\comp}$ is well understood; it is determined by the \Defn{intersection poset} of the arrangement, which is the set of all nonempty intersections of its elements, ordered by reverse inclusion. In fact, the intersection poset determines even the homotopy type of the compactification of $\HA$~\cite{ZieZiv}. On the other hand, it does not determine the homotopy type of the complement of $\HA^{\comp}$, even if we restrict ourselves to complex hyperplane arrangements~\cite{Bartolo, Rybnikovold, Rybnikov}, and understanding the homotopy type of $\HA^{\comp}$~remains~challenging.

A standard approach to study the homotopy type of a topological space $X$ is to find a \Defn{model} for it, that is, a CW complex homotopy equivalent to it. By cellular homology any model of a space $X$ must use at least $\beta_i(X)$ $i$-cells for each $i$, where $\beta_i$ is the $i$-th (rational) Betti number. A natural question arises: Is the complement of an arrangement \Defn{minimal}, i.e., does it have a model with \textit{exactly} $\beta_i(X)$ $i$-cells for all $i$?

Building on previous work by Hattori~\cite{Hattori}, Falk~\cite{Falk} and Cohen--Suciu~\cite{CohenSuciu}, around 2000 Dimca--Papadima~\cite{DimcaPapadima} and Randell~\cite{Randell} independently showed that the complement of any complex hyperplane arrangement is a minimal space. Roughly speaking, the idea is to consider the distance to a complex hyperplane in general position as a Morse function on the Milnor fiber to establish a Lefschetz-type hyperplane theorem for the complement of the arrangement. An elegant inductive argument completes their proof. 

On the other hand, the complement of an arbitrary subspace arrangement is, in general, \emph{not} minimal. In fact, complements of subspace arrangements might have arbitrary torsion in cohomology (cf.~\cite[Sec.\ III, Thm.\ A]{GM-SMT}). This naturally leads to the following question:
\enlargethispage{3mm}
\begin{problem}[Minimality]\label{prb:mini}
Is the complement $\HA^{\comp}$ of an arbitrary $c$-arrangement $\HA$ minimal?
\end{problem}

The interesting case is $c=2$. In fact, if $c$ is not $2$, the complements of $c$-arrangements, and even $c$-arrangements of pseudospheres (cf.~\cite[Sec.\ 8 \& 9]{BjZie}), are easily shown to be minimal; see Section~\ref{sec:car}. 
In 2007, Salvetti--Settepanella~\cite{SalSet} proposed a combinatorial approach to Problem~\ref{prb:mini}, based on Forman's discretization of Morse theory~\cite{FormanADV}. Discrete Morse functions are defined on regular CW complexes rather than on manifolds; instead of critical points, they have combinatorially-defined \Defn{critical faces}. Any discrete Morse function with $c_i$ critical $i$-faces on a complex $C$ yields a model for $C$ with exactly $c_i$ $i$-cells (cf.\ Theorem~\ref{thm:MorseThmw}). 

Salvetti--Settepanella studied discrete Morse functions on the \Defn{Salvetti complexes}~\cite{Salvetti}, which are models for complements of complexified real arrangements. Remarkably, they found that all Salvetti complexes admit \Defn{perfect} discrete Morse functions, that is, functions with exactly $\beta_i(\HA^{\comp})$ critical $i$-faces. Formans's Theorem~\ref{thm:MorseThmw} now yields the desired minimal models for $\HA^{\comp}$.
 
This tactic does not extend to the generality of complex hyperplane arrangements. However, models for complex arrangements, and even for $c$-arrangements, have been introduced and studied by Bj\"orner and Ziegler~\cite{BjZie}. In the case of complexified-real arrangements, their models contain the Salvetti complex as a special case. While our notion of the combinatorial stratification is slightly more restrictive than Bj\"orner--Ziegler's, cf.\ Section~\ref{ssc:2-arrangements}, it still includes most of the combinatorial stratifications studied in~\cite{BjZie}. For example, we still recover the $\s^{(1)}$-stratification which gives rise to the Salvetti complex. With these tools at hand, we can tackle Problem~\ref{prb:mini} combinatorially:

\begin{problem}[Optimality of classical models]\label{prb:cmpt}
Are there perfect discrete Morse functions on the Bj\"orner--Ziegler models for the complements of arbitrary $\cc$-arrangements?
\end{problem}

We are motivated by the fact that discrete Morse theory provides a simple yet powerful tool to study stratified spaces. On the other hand, there are several difficulties to overcome. In fact, Problem~\ref{prb:cmpt} is more ambitious than Problem~\ref{prb:mini} in many respects:

\begin{compactitem}[$\circ$]
\item Few regular CW complexes, even among the minimal ones, admit perfect discrete Morse functions. For example, many $3$-balls~\cite{Bing} and many contractible $2$-complexes~\cite{Zeeman} are not collapsible.
\item There are few results in the literature predicting the existence of perfect Morse functions. For example, it is not known whether any subdivision of the $4$-simplex is collapsible, cf.\ Chapter~\ref{ch:convcollapse}.
\item Solving Problem~\ref{prb:cmpt} could help in obtaining a more explicit picture of the attaching maps for the minimal model; compare Salvetti--Settepanella~\cite{SalSet} and Yoshinaga ~\cite{Yoshi}.
\end{compactitem}

\noindent In this chapter, we answer both problems in the affirmative.

\begin{thmmain}[Theorem~\ref{thm:BZperfect}] \label{MTHM:BZP}
Any complement complex of any $\cc$-arrange\-ment $\HA$ in $S^d$ or~$\R^d$ admits a perfect discrete Morse function.
\end{thmmain}

\begin{mcor}\label{mcor:m}
The complement of any affine $\cc$-arrangement in $\R^d$, and the complement of any $\cc$-arrangement in $S^d$, is a minimal space. 
\end{mcor}

A crucial step on the way to the proof of Theorem~\ref{MTHM:BZP} is the proof of a Lefschetz-type hyperplane theorem for the complements of $\cc$-arrangements. The lemma we actually need is a bit technical (Corollary~\ref{cor:lef}), but roughly speaking, the result can be phrased in the following way:

\begin{thmmain}[Theorem~\ref{thm:lef}] \label{MTHM:LEFT}
Let $\HA^{\comp}$ denote the complement of any affine $\cc$-arrangement $\HA$ in $\R^d$, and let $H$ be any hyperplane in $\R^d$ in general position with respect to $\HA$. Then $\HA^{\comp}$ is homotopy equivalent to $H\cap\HA^{\comp}$ with finitely many $e$-cells attached, where $e = \lceil\nicefrac{d}{\cc}\rceil = d-\lfloor\nicefrac{d}{\cc}\rfloor$.
\end{thmmain}

An analogous theorem holds for complements of $c$-arrangements ($c\neq 2$, with $e = d-\lfloor\nicefrac{d}{c}\rfloor$); it is an immediate consequence of the analogue of Corollary~\ref{mcor:m} for $c$-arrangements, $c\neq 2$, cf.\ Section~\ref{sec:car}.
Theorem~\ref{MTHM:LEFT} extends a result on complex hyperplane arrangements, which follows from Morse theory applied to the Milnor fiber~\cite{DimcaPapadima,
HammLe, Randell}. The main ingredients to our study are:

\begin{compactitem}[$\circ$]
\item the formula to compute the homology of subspace arrangements in terms of the intersection lattice, due to Goresky and MacPherson~\cite{GM-SMT}; cf.\   Lemma~\ref{LEM:LHTCA};

\item the study of combinatorial stratifications as initiated by Bj\"orner and Ziegler~\cite{BjZie}; cf.\ Section~\ref{ssc:2-arrangements};

\item the study of the collapsibility of complexes whose geometric realizations satisfy certain geometric constraints, as discussed in the previous chapter; this is for example used in the proof of Theorem~\ref{thm:hemisphere};

\item the idea of Alexander duality for Morse functions, in particular the elementary notion of ``out-$j$ collapse'', introduced in Section~\ref{ssc:relcollapse};

\item the notion of (Poincar\'e) duality of discrete Morse functions, which goes back to Forman~\cite{FormanADV}. This is used to establish discrete Morse functions on complement complexes, cf.\ Theorem~\ref{ssc:cmpm}.
\end{compactitem}

\section{Preliminaries}\label{sec:prelim}
We use this section to recall the basic facts on discrete Morse theory, $\cc$-arrangements and combinatorial stratifications, and to introduce some concepts that we shall use for the proofs of the main results of this chapter. A notion central to (almost) every formulation of Lefschetz-type hyperplane theorems is that of \Defn{general position}. In our setting, we can make this very precise. Recall that a \Defn{polyhedron} in $S^d$ (resp.\ $\R^d$) is an intersection of closed hemispheres (resp.\ halfspaces).

\begin{definition} If $H$ is a hyperplane in $S^d$ (resp.\ $\R^d$), then $H$ is in \Defn{general position} with respect to a polyhedron if $H$ intersects the span (resp.\ affine span) of any face of the polyhedron transversally. This notion extends naturally to collections of polyhedra (such as for instance polytopal complexes or subspace arrangements): A hyperplane is in \Defn{general position} to such a collection if it is in general position with respect to all of its elements, and every intersection of its elements. We say that a hemisphere is in \Defn{general position} with respect to a collection of polyhedra if the boundary of the hemisphere is in general position with respect to the collection.
\end{definition}

To prove the main results, we work with arrangements in $S^d$, and treat the euclidean case as a special case of the spherical one. We will frequently abuse notation and treat a subspace arrangement both a collection of subspaces and the union of its elements; for instance, we will write $\R^d{\setminus} \HA$ to denote the complement of an arrangement $\HA$ in $\R^d$. 
\enlargethispage{3mm}

\subsection{Discrete Morse theory and duality, I}\label{sec:dmt}

We shall use this section to recall the main terminology for discrete Morse theory used here. For more information on discrete Morse theory, we refer the reader to Forman~\cite{FormanADV}, Chari~\cite{Chari} and Benedetti~\cite{B-DMT4MWB}. We base this section on the introduction to the subject given in~\cite{AB-MGC}. For \Defn{CW complexes} and \Defn{regular CW complexes} we refer the reader to Munkres~\cite[Ch.\ 4,\ \S 38]{Munkres}. Roughly speaking, a regular CW complex is a collection of open balls, the \Defn{cells}, identified along homeomorphisms of their boundaries. The closures of the cells of a regular CW complex $C$ are the \Defn{faces} of the complex, the union of which we denote by~$\F(C)$. Terminology used for polytopal complexes (\Defn{facet}, \Defn{dimension}, \Defn{collapse},$\, \cdots$) naturally extends to regular CW complexes. For the notion of \Defn{dual block complex}, we also refer to Munkres~\cite[Ch.\ 8,\ \S 64]{Munkres}, and Bj\"orner et al.~\cite[Prp.\ 4.7.26(iv)]{BLSWZ}.

A  \Defn{discrete vector field} $\varPhi$ on a regular CW complex $C$ is a collection of pairs~$(\sigma,\varSigma)$ of nonempty faces of $C$, the \Defn{matching pairs} of $\varPhi$, such that $\sigma$ is a codimension-one face of $\varSigma$, and no face of $C$ belongs to two different pairs of $\varPhi$. If $(\sigma,\varSigma)\in \varPhi$, we also say that $\sigma$ \Defn{is matched with} $\varSigma$ in $\varPhi$. 
A \Defn{gradient path} in $\varPhi$ is a union of pairs in $\varPhi$
\[ (\sigma_0, \varSigma_0), (\sigma_1, \varSigma_1),  \ldots, (\sigma_k, \varSigma_k),\ k\ge 1,\]
such that $\sigma_{i+1}\neq \sigma_i$ and $\sigma_{i+1}$ is a codimension-one face of $\varSigma_i$ for all $i\in\{0,1,\, \cdots, k-1\}$; the gradient path is \Defn{closed} if $\sigma_0 = \sigma_k$. A discrete vector field $\varPhi$ is a \Defn{Morse matching} if $\varPhi$ contains no closed gradient paths. Morse matchings, a notion due to Chari~\cite{Chari}, are only one of several ways to represent \Defn{discrete Morse functions}, which were originally introduced by Forman~\cite{FormanADV,FormanUSER}. 

Discrete Morse theory is a generalization of Whitehead's collapsibility: It is an easy exercise to show that a regular CW complex $C$ is collapsible if and only if it admits a Morse matching with the property that all but one face are in one of the matching pairs. 
Forman generalized this observation with the following notion: If $\varPhi$ is a Morse matching on a regular CW complex $C$, then $\CF(\varPhi)$ is the set of \Defn{critical faces}, that is, $\CF(\varPhi)$ is the set of faces of $C$ that are not in any of the matching pairs of $\varPhi$. We set $\CF_i(\varPhi)$ to denote the subset of elements of $\CF(\varPhi)$ of dimension $i$, and we use $c_i(\varPhi)$ to denote the cardinality of this set. Let $A \simeq B$ denote the homotopy equivalence of two spaces $A$ and $B$.

\begin{theorem}[Forman {\cite[Cor.\ 3.5]{FormanADV}}, {\cite[Thm.\ 3.1]{Chari}}]\label{thm:MorseThmw}
Let $C$ be a regular CW complex. Given any Morse matching $\varPhi$ on $C$, we have $C\simeq \Sigma(\varPhi)$, where $\Sigma(\varPhi)$ is a CW complex whose $i$-cells are in natural bijection with the critical $i$-faces of the Morse matching $\varPhi$.
\end{theorem}

Theorem~\ref{thm:MorseThmw} is a special case of a more powerful result of Forman, Theorem~\ref{thm:MorseThm}. To state his theorem in a form convenient to our research, however, we need some more notation; we return to this in Section~\ref{sec:morsethm}. The notion of Morse matchings with critical faces subsumes Whitehead's notion of collapses. 
\begin{compactitem}[$\circ$]
\item A regular CW complex $C$ is collapsible if and only if it admits a Morse matching with only one critical~face. 
\item More generally, a regular CW complex $C$ collapses to a subcomplex $C'$ if and only if $C$ admits a Morse matching $\varPhi$ with $\CF(\varPhi)= C'$. In this case, $\varPhi$ is called a \Defn{collapsing sequence} from $C$ to $C'$. 
\end{compactitem}
Recall that a \Defn{perfect Morse matching} $\varPhi$ on a regular CW complex $C$ is a Morse matching with $c_i(\varPhi)=\beta_i(C)$ for all $i$, where $\beta_i$ denotes the $i$-th Betti number (with respect to field of coefficients $\mathbb{Q}$). Recall also that a \Defn{model} for a topological space is a CW complex homotopy equivalent to it, and that a topological space is \Defn{minimal} if it has a model $\Sigma$ consisting of precisely $\beta_i(\Sigma)$ cells of dimension $i$ for all $i$.

\begin{cor}\label{cor:MorseThm}
Let $C$ denote a regular CW complex that admits a perfect Morse matching. Then $C$ is minimal.
\end{cor}

Let $C$ denote a regular CW complex. Let $\sd C$ denote any (simplicial) complex combinatorially equivalent to the order complex of the face poset $\mathcal{P}(C)$ of nonempty faces of $C$. The complex $C$ is a \Defn{(closed) PL $d$-manifold} if the link of every vertex of $\sd C$ has a subdivision that is combinatorially equivalent to some subdivision of the boundary of the $(d+1)$-simplex.

Now, recall that the faces of $\sd C$ correspond to chains in $\mathcal{P}(C)$. To any face $\sigma$ of $C$ we associate the union $\sigma^\ast$ of faces of $\sd C$ which correspond to chains in $\mathcal{P}(C)$ with minimal element $\sigma$. Assume now $C$ is a closed PL manifold of dimension $d$, then $\rint \sigma^\ast$, the cell \Defn{dual} to $\sigma$ in $C$, is an open ball of dimension $d-\dim \sigma$. The collection $C^\ast$ of cells $\rint \sigma^\ast,\ \sigma \in \F(C)$, together with the canonical attaching homeomorphisms, forms a regular CW complex, the \Defn{dual block complex} to $C$ (for details, we refer the reader to Proposition 4.7.26(iv) in~\cite{BLSWZ}). Naturally, $\sigma^\ast$ is the \Defn{dual face} to $\sigma$.

If $\varPhi$ is a Morse matching on $C$, then a matching $\varPhi^\ast$, the \Defn{dual} to $\varPhi$, is induced by the map $\sigma\mapsto\sigma^\ast$ as follows: \[\varPhi^\ast:=\{(\sigma,\varSigma)^\ast : (\sigma,\varSigma)\in \varPhi \},\ \text{where}\ (\sigma,\varSigma)^\ast{:=}(\varSigma^\ast,\sigma^\ast).\]

\begin{theorem}[Benedetti {\cite[Thm.\ 3.10]{B-DMT4MWB}},\ Forman {\cite[Thm.\ 4.7]{FormanADV}}]\label{thm:dual}
Let $C$ be a regular CW complex that is also a closed PL $d$-manifold, and let $\varPhi$ denote a Morse matching on $C$. Then $\varPhi^\ast$ is a Morse matching on the regular CW complex $C^\ast$, and the map assigning each face of $C$ to its dual in $C^\ast$ restricts to a natural bijection from $\CF(\varPhi)$ to $\CF(\varPhi^\ast)$.
\end{theorem}

\subsection{2-arrangements, combinatorial stratifications, complement complexes} \label{ssc:2-arrangements}

In this section, we introduce $\cc$-arrangements, their combinatorial stratifications and complement complexes, guided by~\cite{BjZie}. All subspace arrangements considered in this chapter are finite.

A \Defn{$\cc$-arrangement} $\HA$ in $S^d$ (resp.\ in $\R^d$) is a finite collection of distinct ${(d-\cc)}$-dim\-ensional subspaces of $S^d$ (resp.\ of $\R^d$), such that the codimension of any non-empty intersection of its elements is a multiple of~$\cc$.  Any $\cc$-arrangement with $d<\cc$ is the empty arrangement. 
A subspace arrangement $\HA$ is \Defn{essential} if the intersection of all elements of $\HA$ is empty, and \Defn{non-essential} otherwise. By convention, the non-essential arrangements include the empty arrangement. The following definitions apply more generally to the bigger class of \Defn{codim-$\cc$-arrangements}, which are finite collections of subspaces of codimension $\cc$, to allow us to pass between $\R^d$ and $S^d$ without problems, cf.\ Remark~\ref{rem:cod}. 

Recall that an interval in $\mathbb{Z}$ is a set of the form $[a,b]:=\{x\in \mathbb{Z}: a\le x\le b\}.$

\begin{definition}[Extensions and stratifications]
A \Defn{sign extension} $\SH$ of a codim-$\cc$-arrangement $\HA=\{h_i : i\in [1,n]\}$ in $S^d$ is any collection of hyperplanes $\{H_{i}\subset S^d : {i\in [1,n]}\}$ such that for each $i$, we have $ h_i\subset H_{i}$. We say that the subspaces $S^d$, $H_i$ and $h_i$ itself \Defn{extend} $h_i$. A \Defn{hyperplane extension} $\EH$ of $\HA$ in $S^d$ is a sign extension of $\HA$ together with an arbitrary finite collection of hyperplanes in $S^d$.

 Consider now any subset $S$ of elements of $\EH$ and any point $x$ in $S^d$. The \Defn{stratum} associated to $x$ and $S$ is the set of all points that can be connected to $x$ by a curve that lies in all elements of $S$, and intersects no element of $\EH\setminus S$, cf.~\cite[Sec. 2]{BjZie}. The nonempty strata obtained this way form a partition of $S^d$ into convex sets, the \Defn{stratification} of $S^d$ given by~$\EH$. To shorten the notation, we will sometimes say that a stratification $\s$ is induced by a codim-$\cc$-arrangement $\HA$ if it is given by some extension $\EH$ of $\HA$.
\end{definition}

The inclusion map of closures of strata of a stratification gives rise to canonical attaching homeomorphisms of the strata to each other.

\begin{definition}[Fine extensions and combinatorial stratifications]
Let $\HA$ be a {codim-$\cc$-arrangement}. We say that an extension $\EH$ of $\HA$ is \Defn{fine} if it gives rise to a stratification $\s(\EH)$ that, together with the canonical attaching maps, is a regular CW complex; in this case, the stratification $\s(\EH)$ is \Defn{combinatorial}. For instance, if $\HA$ is essential, then any stratification $\s$ of $S^d$ induced by it is combinatorial.
\end{definition}

Let $C$ be a subcomplex of a combinatorial stratification of $S^d$. Let $M$ be an arbitrary subset of $S^d$. Recall that the \Defn{restriction} $\RS(C,M)$ of $C$ to $M$ is the maximal subcomplex of $C$ all whose faces are contained in~$M$. If $D$ is a subcomplex of $C$ then $C-D:=\RS(C, C{\setminus} \rint D)$ is the \Defn{deletion} of $D$ from $C$. 

Dually, let $\s^\ast$ be the dual block complex of a combinatorial stratification $\s$ of $S^d$, and let $C$ be any subcomplex of $\s^\ast$. Then, $\RS^\ast(C,M)$ is the minimal subcomplex of $C\subset \s^\ast$ containing all those faces of $C$ that are dual to faces of $\s$ intersecting~$M$.

\begin{definition}[Complement complexes for spherical codim-$\cc$-arrangements]
Let $\HA$ be a codim-$\cc$-arrangement in $S^d$, and let $\s$ be a combinatorial stratification of $S^d$ induced by $\HA$. With this, define the \Defn{complement complex} of $\HA$ with respect to $\s$ as the regular CW complex $\K(\HA,\s):=\RS^\ast(\s^\ast,S^d{\setminus} \HA)$. Equivalently, $\K(\HA,\s)$ is the subcomplex of $\s^\ast$ consisting of the duals of the faces that are \emph{not} contained in $\RS(\s,\HA)$.
\end{definition}

\begin{definition}[Complement complexes for affine codim-$\cc$-arrangements]\label{def:affine}
Let $\HA$ denote a codim-$\cc$-arrangement of affine subspaces in $\R^{d}$, and let $\rho$ denote a radial projection of $\R^{d}$ to an open hemisphere $\OD$ in $S^d$. Extend the image $\rho(\HA):=\{\rho(h): h\in \HA\}$ to the codim-$\cc$-arrangement $\HA':=\{\SSp(h)\subset S^d: h\in \rho(\HA)\}$ in $S^d$, and
consider any combinatorial stratification $\s$ of $S^d$ induced by $\HA'$. We say that $\RS^\ast(\s^\ast,\OD{\setminus} \HA')$ is a \Defn{complement complex} of $\HA$.
\end{definition}

\begin{rem}[$\cc$-arrangement with respect to an open hemisphere] \label{rem:cod}
Definition~\ref{def:affine} is the reason for defining stratifications in the generality of codim-$\cc$-arrangements; if $\HA$ is an affine $\cc$-arrangement in $\R^d$, then $\HA'$ (compare the preceding definition) \emph{is not} in general a $\cc$-arrangement in $S^d$. However, $\HA'$ \emph{is} a \Defn{$\cc$-arrangement w.r.t.\ $\OD$} in $S^d$, that is, every non-empty intersection of elements of $\HA'$ and $\OD$ has a codimension divisible by $\cc$.
\end{rem}

\begin{lemma}[{\cite[Prp.\ 3.1]{BjZie}}]
Let $\HA$ be a codim-$\cc$-arrangement in $\R^d$ (resp.~$S^d$). Then every complement complex of $\HA$ is a model for the complement $\HA^{\comp}=\R^d{\setminus}\HA$ (resp.\ $S^d{\setminus}\HA$).
\end{lemma}

\subsection{Outwardly matched faces of subcomplexes; out-{\em j} collapses} \label{ssc:relcollapse}

Here, we introduce the notion of outward matchings. We will see in the next section that this notion allows us to define a rudimentary Alexander duality for Morse matchings, which is crucial to our proofs.

\begin{definition}[Outwardly matched faces]
Let $C$ be a regular CW complex. Let $D$ be a non-empty subcomplex. Consider a Morse matching $\varPhi$ on $C$. A face $\tau$ of $D$ is \Defn{outwardly matched} with respect to the pair $(C,D)$ if it is matched with a face that does not belong to $D$.
\end{definition}

\begin{definition}[Out-$j$ collapse of a pair, Out-$j$ collapsibility]
Let $C$ be a regular CW complex. Let $D$ be a subcomplex. Suppose that $C$ collapses onto a subcomplex $C'$. The pair $(C,  D)$ \Defn{out-$j$ collapses} to the pair $(C', D \cap C')$, and we write in symbols $(C, D) \searrow_{\textrm{out-}j}  (C', D \cap C' ),$
if the collapsing sequence that reduces $C$ to $C'$ can be chosen so that every outwardly matched face with respect to the pair $(C,D)$ has dimension $j$.

We say that the pair $(C, D)$ is \Defn{out-$j$ collapsible} if there is a vertex $v$ of $D$ such that $(C, D) \searrow_{\textrm{ out-}j} (v, v).$ For any integer $j$, the pair $(C, \emptyset)$ is \Defn{out-$j$ collapsible} if $C$ is collapsible. A collapsing sequence demonstrating an out-$j$ collapse is called an \Defn{out-$j$ collapsing sequence.}
\end{definition}

\begin{example} \label{ex:out}
Let $C$ be a collapsible complex. 
\begin{compactenum}[(a)]
\item The pairs $(C, C)$ and $(C, \emptyset)$ are out-$j$ collapsible for any $j$.
\item If $D$ is the $k$-skeleton of $C$, with $0 \le k < \dim C$, then $(C, D)$ is out-$k$ collapsible. 
\item If $D$ is any triangulation of the Dunce hat (cf.~\cite{Zeeman}) that is a subcomplex of $C$, then $(C, D)$ is not out-$j$ collapsible for any $j$, cf.\ Proposition~\ref{prp:relind}.
\end{compactenum}
\end{example}

\begin{prp}\label{prp:relind}
Let $(C,D)$ be an out-$j$ collapsible pair, with $D$ non-empty.  
The number of outwardly matched $j$-faces of $D$ is independent of the out-$j$ collapsing sequence chosen, and equal to 
$(-1)^j \cdot ( \chi(D) - 1 )$, where $\chi$ denotes the Euler characteristic. Moreover, the following are equivalent:
\begin{compactenum}[\rm (1)]
\item $D$ is contractible;
\item There exists a collapsing sequence that has no outwardly matched faces with respect to the pair $(C,D)$;
\item $D$ is collapsible.
\end{compactenum}
\end{prp}

\begin{definition}
With the above notation, if one of the conditions (1), (2), or (3) is satisfied, we say that $(C,D)$ is a \Defn{collapsible pair}. The pair $(C,\emptyset)$ is a \Defn{collapsible pair} if $C$ is collapsible. 
\end{definition}

\begin{proof}[\textbf{Proof of Proposition~\ref{prp:relind}}] Fix an out-$j$ collapsing sequence for $(C,D)$. Let $O$ be the set of outwardly matched faces of $D$. Let $v$ be the vertex onto which $C$ collapses. Let $Q$ be the set of the faces of $D$ that are matched together with another face of $D$. Clearly, the sets $O, \{v\}, Q$ form a partition of the set of all faces of $D$. We set
\[f_i := \# \{ \textrm{$i$-faces of $D$} \}, 
\quad o_i := \# \{ \textrm{$i$-faces in $O$} \}, 
\quad q_i := \# \{ \textrm{$i$-faces in $Q$} \}.\]
Clearly, $f_0 = o_0 + 1 + q_0$, and $f_i = o_i + q_i$ for $i \ge 1$. By definition of out-$j$ collapsibility, $o_i = 0$ unless~$i=j$. In particular, $\sum (-1)^i o_i = (-1)^j o_j$.
Now, faces matched together in a collapsing sequence must have consecutive dimension. It follows (by pairwise canceling) that $\sum (-1)^i q_i = 0$. Hence, 
\[\chi(D) \: = \; \sum (-1)^i f_i \; = \; 1 + \sum (-1)^i o_i + \sum (-1)^i q_i \; = \: 1 + (-1)^j o_j.\]
Hence $o_j = (-1)^j \cdot ( \chi(D) - 1 )$ and the first claim is proven. For the second part: 
\begin{compactitem}[$\circ$]
\item \makebox[4.4em][l]{(1) $\Rightarrow$ (2)}: If $D$ is contractible, $\chi(D) =1$. By the formula above, $o_j = 0$.
\item \makebox[4.4em][l]{(2) $\Rightarrow$ (3)}: By assumption, there is a collapsing sequence for $C$ which removes all faces of $D$ in pairs. Consider the restriction to $D$ of the collapsing sequence for $C$; this yields a collapsing sequence for $D$.
\item \makebox[4.4em][l]{(3) $\Rightarrow$ (1)}: This is implied by the fact if $D\searrow D'$, then $D'$ is a deformation retract of $D$. \qedhere
\end{compactitem}
\end{proof}

We also have the following elementary version of Lemma \ref{lem:cecoll} for outward matchings.

\begin{lemma}\label{lem:outcoll}
Let $C$ and $D$, $D \subset C$, be regular CW complexes, and let $v$ be any vertex of $C$. Assume that $v\notin D$ or $\Lk(v,D)$ is nonempty, and that $(\Lk(v,C),\Lk(v,D))$ is out-$j$ collapsible. Then $(C,D)$ out-$(j+1)$ collapses to $(C-v,D-v)$. In particular, if in this situation $(\Lk(v,C),\Lk(v,D))$ is a collapsible pair, then $(C,D)$ out-$j$ collapses to $(C-v,D-v)$ for every $j$.

If on the other hand $v\in D$ but $\Lk(v,D)$ is empty, then $(C,D)$ out-$0$ collapses to $(C-v,D-v)$ if and only if $\Lk(v,C)$ is collapsible.
\end{lemma}

\subsection{Complement matchings}\label{ssc:cmpm}

Let $\s$ denote a combinatorial stratification of the sphere $S^d$. As such, $\s$ is necessarily a closed PL manifold. To obtain Morse matchings on the complement complex $\K:=\K(\HA,\s)$, we first study Morse matchings on the stratification $\s$. Via duality, Morse matchings on $\s$ will then give rise to Morse matchings on $\K$. To explain the details of this idea is the purpose of this section.

\begin{definition}[Restrictions of matchings]
Let $\varPhi$ be a Morse matching on a regular CW complex $C$, and let $D$ be a subcomplex of $C$. Let us denote by $\varPhi_D$ the \Defn{restriction} of $\varPhi$ to $D$, that is, the collection of all matching pairs in $\varPhi$ involving two faces of $D$. 
\end{definition}

\begin{example}[Complement matching]\label{ex:cm}
Let $\HA$ be a codim-$\cc$-arrangement in $S^d$, and let $\s$ denote a combinatorial stratification of $S^d$ induced by it. Let $\K:=\K(\HA,\s)$ denote the associated complement complex. Consider the dual Morse matching $\varPhi^\ast$ on $\s^\ast$ of a Morse matching $\varPhi$ defined on $\s$.  The Morse matching $\varPhi^\ast$ has an outwardly matched face $\varSigma^\ast$ (matched with a face $\sigma^\ast$) with respect to the pair $(\s^\ast,\K)$ for every outwardly matched face $\sigma$ (matched with $\varSigma$) of $\varPhi$ with respect to the pair $(\s,\RS(\s,\HA))$. After we remove all such matching pairs from $\varPhi^\ast$, we are left with a Morse matching on $\s^\ast$ that has no outwardly matched faces with respect to the pair~$(\s^\ast,\K)$. If we furthermore remove all matching pairs that only involve faces of $\s^\ast$ not in $\K$, we obtain a Morse matching on the complex $\K$, the \Defn{complement matching}~$\varPhi^{\ast}_{\K}$ induced by $\varPhi$.
\end{example}

\begin{figure}[htbf]
\centering 
 \includegraphics[width=0.60\linewidth]{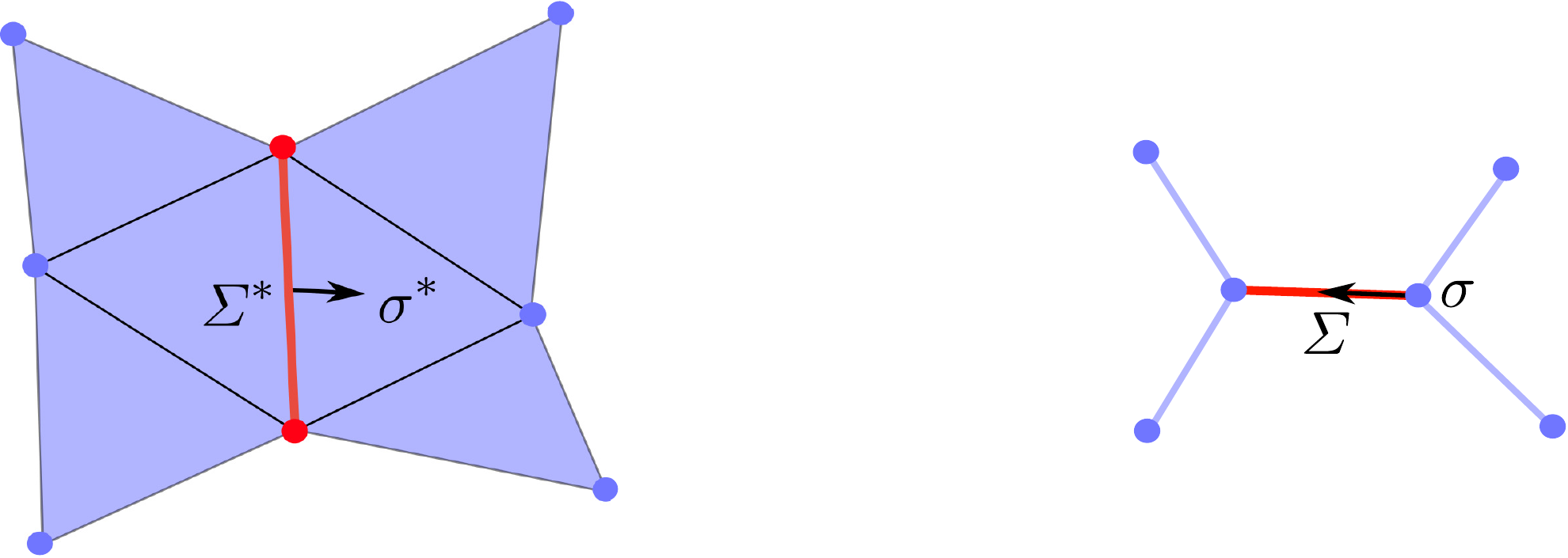} 
 \caption{\small An outward matching $(\varSigma^\ast,\sigma^\ast)$ of the pair $(\s^\ast,\K)$ corresponds to an outward matching $(\sigma,\varSigma)$ of the pair $(\s,\RS(\s,\HA))$.}
\label{fig:outward}
\end{figure}

The complement matching of a Morse matching is again a Morse matching. This allows us to study Morse matchings on the complement complex by studying Morse matchings on the stratification itself. The following theorem can be seen as a very basic Alexander duality for Morse matchings.

\begin{theorem}\label{thm:relcoll}
Let $\s$, $\HA$ and $\K$ be given as in Example~\ref{ex:cm}. Consider a Morse matching $\varPhi$ on $\s$. Then the critical $i$-faces of $\varPhi^{\ast}_{\K}$ are in one-to-one correspondence with the union of
\begin{compactitem}[$\circ$]
\item the critical $(d-i)$-faces of $\varPhi$ that are not faces of $\RS(\s,\HA)$, and
\item the outwardly matched $(d-i-1)$-faces of $\varPhi$ with respect to the pair $(\s,\RS(\s,\HA))$.
\end{compactitem}
\noindent If $M$ is furthermore an open subset of $S^d$ such that all noncritical faces of $\varPhi$ intersect $M$, then the critical $i$-faces of $\varPhi^{\ast}_{\RS^\ast(\K,M)}$ are in bijection with the union of
\begin{compactitem}[$\circ$]
\item the critical $(d-i)$-faces of $\varPhi$ that are not faces of $\RS(\s,\HA)$, and that intersect $M$, and
\item the outwardly matched $(d-i-1)$-faces of $\varPhi$ with respect to the pair $(\s,\RS(\s,\HA))$.
\end{compactitem}
\end{theorem}

\begin{proof} Both bijections are natural: Each critical $i$-face of $\varPhi^{\ast}_{\K}$ corresponds to either a critical $i$-face of $\varPhi^\ast$ in $\K$ or to an outwardly matched $i$-face of $\varPhi^\ast$ with respect to the pair $(\s^\ast, \K)$. Now, 
\begin{compactitem}[$\circ$]
\item critical $i$-faces of $\varPhi^\ast$ in $\K$ are in one-to-one correspondence with the critical $(d-i)$-faces of $\varPhi$ that are not faces of $\RS(\s,\HA)$ (cf.\ Theorem~\ref{thm:dual}), and
\item the outwardly matched $i$-faces of $\varPhi^\ast$ with respect to the pair $(\s^\ast, \K)$, are in one-to-one correspondence with the outwardly matched $(d-i-1)$-faces of $\varPhi$ with respect to the pair  $(\s,\RS(\s,\HA))$ (cf.\ Example~\ref{ex:cm}).
\end{compactitem}
This gives the desired bijection for the critical faces of $\varPhi^{\ast}_{\K}$. The bijection for the critical faces of $\varPhi^{\ast}_{\RS^\ast(\K,M)}$ is obtained analogously.\end{proof} 

\subsection{Discrete Morse theory and duality, II: A strong version of the Morse Theorem}\label{sec:morsethm}

The notions introduced in the past sections allow us to state the following stronger version of Theorem~\ref{thm:MorseThmw}. Recall that a \Defn{$k$-cell}, or \Defn{cell of dimension $k$} is just a $k$-dimensional open ball. We say that \Defn{a cell $B$ of dimension $k$ is attached to a topological space $X$} if $B$ and $X$ are identified along an embedding of $\partial B$ into $X$, cf.~\cite[Ch.\ 0]{Hatcher}.

\begin{theorem}[Forman {\cite[Thm.\ 3.4]{FormanADV}}]\label{thm:MorseThm}
Let $C$ be a regular CW complex, and let $D$ denote any subcomplex. Let $\varPhi$ denote a Morse matching on $C$ that does not have any outwardly matched faces with respect to the pair $(C,D)$. Then $C$ is up to homotopy equivalence obtained from $D$ by attaching one cell of dimension $k$ for every critical $k$-face of $\varPhi$ not in $D$.
\end{theorem}

\begin{proof}
Theorem 3.4 of Forman in~\cite{FormanADV} treats the case where $D$ contains all but one of the critical faces of $\varPhi$; this is clearly sufficient to prove the statement. For the sake of completeness, we sketch a reasoning using the language of Morse matchings here, inspired by Chari's proof of Theorem~\ref{thm:MorseThmw}~\cite[Thm.~3.1]{Chari}. Assume that $D$ is a strict subcomplex of $C:=C_0$. Set $\varPhi_0:=\varPhi$ and $i:=0$.

\smallskip

\noindent {\bf Deformation process} Let $G(C_i)$ denote the Hasse diagram of the face poset of $C_i$, i.e.\ let $G(C_i)$ be the graph
\begin{compactitem}[$\circ$]
\item whose vertices are the nonempty faces of $C_i$, and for which
\item two faces $\tau$, $\sigma$ are connected by an edge, directed from $\sigma$ to $\tau$, if and only if $\tau$ is a facet of $\sigma$. 
\end{compactitem}
We manipulate $G(C_i)$ to a directed graph $G_{\varPhi_i} (C_i)$ as follows: 

\begin{quote}
\noindent \emph{For every matching pair $(\sigma, \varSigma)$ of $\varPhi_i$, replace the edge directed from $\varSigma$ to $\sigma$ by an edge directed from $\sigma$ to $\varSigma$.}
\end{quote}
Finally, contract the vertices of $G_{\varPhi_i} (C_i)$ corresponding to $D$ to a single vertex, obtaining the directed graph $G_{\varPhi_i} (C_i)\cdot D$. Let $v_D$ denote the vertex corresponding to $D$ in that graph. Since $\varPhi_i$ contains no closed gradient path, and $(C_i,D)$ has no outwardly matched face with respect to $\varPhi_i$, the directed graph $G_{\varPhi_i} (C_i)\cdot D$ is \Defn{acyclic} and $v_D$ is a \Defn{sink}, i.e.\ every edge of the graph that contains $v_D$ points towards it.

 Consequently, $G_{\varPhi_i} (C_i)\cdot D$ has a \Defn{source} that is not $v_D$, i.e.\ a vertex such that every edge containing it points away from it. This vertex corresponds to a face $\sigma$ of $C_i$ not in $D$, which, since it is a source, must satisfy one of the following properties:

\begin{compactenum}[(1)]
\item there exists a face $\varSigma$ of $C_i$ such that $(\sigma,\varSigma)$ is a matching pair of $\varPhi_i$, or 
\item $\sigma$ is a critical face of $\varPhi_i$.
\end{compactenum}
In case $\sigma$ satisfies (1), $\sigma$ is a free face of $C_i$, and $C_i$ elementarily collapses to $C_i-\sigma$; in particular, $C_i$ is homotopy equivalent to $C_i-\sigma$. In case $\sigma$ satisfies (2), $\sigma$ is a facet of $C_i$: in particular, $C_i$ is obtained from $C_i-\sigma$ by attaching a cell of dimension $\dim \sigma$. 

Now, set $C_{i+1}:=C_i-\sigma$ and
\[
\varPhi_{i+1}:= \left\{ \begin{array}{ll}\varPhi_i \setminus \{(\sigma,\varSigma)\}&\text{ in case $\sigma$ satisfies (1) and}\\
\varPhi_i                             & \text{ in case $\sigma$ satisfies (2).}
\end{array}
\right.
\]
With this definition, we have that  
\begin{equation}\tag{$\ast\ast$} \label{eq:cr}
c_k(\varPhi_{i+1})= \left\{ \begin{array}{ll}c_k(\varPhi_i) &\text{ if $\sigma$ satisfies (1) or $k\neq \dim \sigma$ and }\\
c_k(\varPhi_i)-1                            & \text{ if $\sigma$ satisfies (2) and $k=\dim \sigma$.}
\end{array}
\right.
\end{equation}
If $C_{i+1}=D$, stop the deformation process. If $C_{i+1}\neq D$, increase $i$ by one and repeat from the start.

\smallskip

The homotopical characterization of how to obtain $C_{i+1}=C_i-\sigma$ from $C_{i}$, together with Equation~\eqref{eq:cr}, shows that $C$ is up to homotopy equivalence obtained from $D$ by attaching one cell of dimension $\dim \sigma$ for every critical face $\sigma$ of $\varPhi$ not in $D$.
\end{proof}

\section{Restricting stratifications to general position hemispheres}

In this section, we study Morse matchings on combinatorial stratifications of $S^d$. More precisely, we study Morse matchings on the restrictions of stratifications to a hemisphere. The main result of this section is Theorem~\ref{thm:hemisphere}, which will turn out to be crucial in order to establish Main Theorem~\ref{MTHM:BZP}. 

If $\HA$ is a subspace arrangement in $S^d$, and $H$ is a subspace of $S^d$, we define \[\HA^H:=\{h\cap H: h \in \HA\}.\] 
Recall that $\RN^1_p X$ denotes the subset of unit vectors in the tangent space of $X$ at $p$. If $\HA$ is a subspace arrangement in $S^d$, we define the \Defn{link} $\Lk(p,\HA)\subset \RN^1_p S^d$ of $\HA$ at $p$ by \[\Lk(p,\HA):=\{\RN^1_p h: h \in \HA,\ h\cap p\ne \emptyset\}.\] 
Similarly, if $C$ is a subcomplex of a combinatorial stratification of $S^d$, and $v$ is a vertex of $C$, the \Defn{link} $\Lk(v,C)\subset \RN^1_v S^d$ of $C$ at $v$ is the regular CW complex represented by the collection of faces \[\F(\Lk(v,C)):=\{\RN^1_v \sigma: \sigma \in \F(C),\ v\subset \sigma\}.\]
Our goal is to investigate whether, for any given hemisphere $\FD$, the pair $(\RS(\s,\FD), \RS(\s,\FD\cap \HA))$ is out-$j$ collapsible for some suitable integer $j$. With an intuition guided by the Lefschetz hyperplane theorems for complex varieties, one could guess that the right $j$ to consider is the integer 
\[\ii{d}  :=  \left\lfloor \nicefrac{d}{\cc} \right\rfloor  - 1.\]
This will turn out to be correct. Before we start with the main theorem of this section, we anticipate a special case; we consider the case of the empty arrangement.

\begin{lemma}\label{lem:hemisphere}
Let $\EH$ be a fine extension of the empty arrangement in $S^d$, and let $\s:=\s(\EH)$ be the associated combinatorial stratification of $S^d$. Let $\FD$ be a closed hemisphere that is in general position with respect to $\s(\EH)$. Then $\RS(\s,\FD)$ is collapsible.
\end{lemma}

\begin{proof}
The proof is by induction on the dimension, the case $d=0$ clearly being true. Assume now $d\ge 1$. Let $H$ denote any element of $\EH$, and let $\overline{H}_+$, $\overline{H}_-$ denote the closed hemispheres in $S^d$ bounded by $H$. The proof of the induction step is articulated into three simple parts:
\begin{compactenum}[(1)]
\item We prove $\RS(\s, \FD\cap \overline{H}_+)  \searrow  \RS(\s, \FD\cap H).$
\item We prove $\RS(\s, \FD\cap \overline{H}_-)  \searrow  \RS(\s, \FD\cap H).$
\item We show that $\RS(\s, \FD\cap {H})$ is collapsible.
\end{compactenum}

These three steps show that $\RS(\s, \FD)$ is collapsible: The combination of (1) and (2) gives that $\RS(\s, \FD)$ collapses to $\RS(\s, \FD\cap {H}),$ which is collapsible by step (3). We now show point (1); the proof of (2) is analogous and left out, and (3) is true by induction assumption.

Let $\zeta$ denote a central projection of $\intx \FD$ to $\R^{d}$, and let $\nu_+$ denote the interior normal to the halfspace $\zeta(\overline{H}_+ \cap \intx\FD)\subset \R^{d}$. Perturb $\nu_+$ to a vector $\nu$ such that the function $x \ \mapsto \ \langle \zeta(x),  \nu \rangle$ 
\begin{compactenum}[(a)]
\item \Defn{preserves the order given by $\langle  \zeta(\cdot), \nu_+  \rangle$} and
\item \Defn{induces a strict total order} on $\F_0(\RS(\s, \FD\cap \overline{H}_+))$, 
\end{compactenum}
that is, for any two vertices $v$, $w$ of $\RS(\s, \FD\cap \overline{H}_+)$, we have the following:
\begin{compactenum}[(a)]
\item  
If $ \langle  \zeta(v), \nu_+  \rangle  >  \langle \zeta(w),\nu_+ \rangle$, then $\langle  \zeta(v),  \nu  \rangle  > \langle  \zeta(w), \nu  \rangle$;
\item 
If $ \langle  \zeta(v), \nu  \rangle =  \langle \zeta(w),\nu \rangle$, then $v = w$.
\end{compactenum}
The function $x \mapsto \langle  \zeta(x), \nu  \rangle$ orders the $n$ vertices $v_0,v_1,\, \cdots, v_n$ of $\s$ in the interior of $\FD\cap \overline{H}_+$, with the labeling reflecting the order ($v_0$ is the vertex with the highest value under this function).
Let $\Sigma_i$ denote the complex $\RS(\s, \FD \cap \overline{H}_+)-\{v_0,v_1, \, \cdots, v_{i-1}\}$. 
We demonstrate $\RS(\s, \FD\cap \overline{H}_+)  \searrow  \RS(\s, \FD\cap H)$ by showing that, for all $i\in [0,n]$, we have $\Sigma_i  \searrow  \Sigma_i-v_i =\Sigma_{i+1}.$

To see this, notice that $\Lk(v_i,\s(\EH))$ is a combinatorial stratification of the $(d-1)$-sphere $\RN_{v_i}^1 S^d$, given by the fine hyperplane extension $\Lk({v_i}, \EH)$ of the empty arrangement. The complex $\Lk({v_i},\Sigma_i)$ is the restriction of $\Lk({v_i}, \s(\EH))$ to the general position hemisphere~$\RN_{v_i}^1 \FD_{v_i}$, where
\[\FD_{v_i} \; := \; 
\zeta^{-1}( \{  x\in \R^{d-1} : 
\langle  \zeta({v_i}),  \nu  \rangle 
\; \ge \; 
\langle  x,  \nu  \rangle \}), \]
since the vertices $v_0,\, \cdots, v_{i-1}$ were removed already. Thus, by induction assumption, we have that $\Lk({v_i}, \Sigma_i)$ is collapsible; consequently, $\Sigma_i$ collapses to $\Sigma_{i+1}$, as desired.
\end{proof}

\begin{theorem}\label{thm:hemisphere}
Let $\HA$ be a nonempty $\cc$-arrangement in $S^d$, let $\EH$ be a fine hyperplane extension of $\HA$, and let $\s:=\s(\EH)$ denote the combinatorial stratification of $S^d$ given by it. Let $\FD$ be a closed hemisphere that is in general position with respect to $\s$. Then, for any $k$-dimensional subspace $H$ of $\SH\subset\EH$ extending an element of~$\HA$, we have the following:
\begin{compactenum}[\rm (A)]
\item The pair $(\RS(\s,\FD\cap H), \RS(\s,\FD\cap \HA\cap H))$ is out-$\ii{d}$ collapsible. 
\item If $\HA$ is additionally non-essential, then  $(\RS(\s,\FD\cap H), \RS(\s,\FD\cap \HA\cap H))$ is a collapsible pair. 
\end{compactenum}
\end{theorem}

\begin{proof}
To simplify the notation, we set $\RS{[M]}  :=  \RS \left(  \s,  M\right)$ and $\RS'{[M]}  :=  \RS \left( \s,  M \cap \HA  \right)$ for any subset $M$ of~$S^d$.
We proceed by induction on $d$ and $k$. Let $\io_{d,k}$ denote the statement that part (A) is proven for all spheres of dimension $d$ and subspaces $H$ of dimension $k$. Let $\iit_{d,k}$ denote the statement that part (B) holds for arrangements in spheres of dimension $d$ and subspaces $H$ of dimension $k$. Since $H$ extends an element of $\HA$, we always have $d\ge k \ge d-2$.

For the base cases of the induction, it suffices to treat the cases $\io_{2,0}$ and $\io_{3,1}$. In both cases, $H$ is an element of $\HA$, so $\RS{[\FD\cap H]}= \RS'{[\FD\cap H]}$. Since a pair $(C,C)$ is a collapsible pair if and only if $C$ is collapsible, it suffices to prove that $\RS{[\FD\cap H]}$ is collapsible; in case $\io_{2,0}$, the complex $\RS{[\FD\cap H]}$ is a vertex; in case $\io_{3,1}$, the complex $\RS{[\FD\cap H]}$ is a tree; in both cases, the complex is trivially collapsible. Assume now $d\ge 2$ and $k\ge 0$. Our inductive proof proceeds like this:

\smallskip

\begin{compactenum}[\bf I.]
\item We prove that $\io_{d,k}$ implies $\iit_{d,k}$ for all $d$, $k$.
\item We prove that $\io_{k,k}$ implies $\io_{d,k}$ for $k= d-\cc$.
\item We prove that $\iit_{k-1,k-1}$, $\io_{k-1,k-1}$ and $\io_{d,k-1}$ together imply $\io_{d,k}$ for $k > d-\cc$.
\end{compactenum}

\medskip

\noindent \textbf{Part I. $\io_{d,k}$ implies $\iit_{d,k}$  for all $d$, $k$}

\medskip

Let $\sigma$ denote the intersection of all elements of the non-essential arrangement $\HA$. Since $\HA$ is nonempty, so is $\sigma$ and $\RS'[\FD\cap H]$ deformation retracts onto the contractible complex $\RS[\FD\cap \sigma ]$. Thus, by the second part of Proposition~\ref{prp:relind} and inductive assumption $\io_{d,k}$, the pair $(\RS[\FD\cap H], \RS[\FD\cap \HA\cap H])$ is a collapsible pair. 

\medskip

\noindent \textbf{Part II. $\io_{k,k}$ implies $\io_{d,k}$ for $k= d-\cc$}

\medskip

We have to show that $(\RS{[\FD\cap H]},  \RS'{[\FD\cap H]} )$ is out-$\ii{d}$ collapsible. We will see that it is even a collapsible pair. The $(d-2)$-dimensional subspace $H$ is an element of $\HA$, so that we have \[ (\RS{[\FD\cap H]},  \RS'{[\FD\cap H]} ) = (\RS( \s,  \FD\cap H)  ,  \RS( \s , \FD \cap H \cap \HA) )= ( \RS( \s,  \FD\cap H)  ,  \RS( \s,  \FD\cap H )  ).\]
But a pair $(C,C)$ is a collapsible pair if $C$ is collapsible. By definition, if a pair is out-$\ii{d}$ collapsible, the first complex in the pair is collapsible; so, since the pair $(\RS[\FD\cap H], \RS[\FD\cap H\cap \HA])$ is out-$\ii{d}$ collapsible by inductive assumption $\io_{k,k}$, it follows trivially that $\RS[\FD\cap H]$ is collapsible.

\medskip

\noindent \textbf{Part III. $\iit_{k-1,k-1}$, $\io_{k-1,k-1}$ and $\io_{d,k-1}$ together imply $\io_{d,k}$ for $k > d-\cc$}

\medskip

Let $h$ denote the element of $\HA$ extended by $H$. Let $\eta\in\SH\subset\EH$ be a codimension-one subspace of $H$ that extends $h$ as well. Let $\overline{\eta}_+$ and $\overline{\eta}_-$ be the closed hemispheres in $H$ bounded by $\eta$. We prove that the pair $(\RS{[\FD\cap H]},  \RS'{[\FD\cap H]} )$ is out-$\ii{d}$ collapsible. \newpage
The proof consists of three steps:
\begin{compactenum}[(1)]
\item We prove \[(\RS{[\FD\cap \overline{\eta}_+]},  \RS'{[\FD\cap \overline{\eta}_+]} )  \searrow_{\textrm{ out-}\ii{d}}  (\RS{[\FD\cap \eta ]},  \RS'{[\FD\cap {\eta}]} ).\]
\item Symmetrically, we have
\[(\RS{[\FD\cap \overline{\eta}_-]},  \RS'{[\FD\cap \overline{\eta}_-]} )  \searrow_{\textrm{ out-}\ii{d}}  (\RS{[\FD\cap \eta ]},  \RS'{[\FD\cap {\eta}]} ).\]
\end{compactenum}
The combination of these two steps proves 
\[(\RS{[\FD\cap H]},  \RS'{[\FD\cap H]} )  \searrow_{\textrm{ out-}\ii{d}}  (\RS{[\FD\cap \eta ]},  \RS'{[\FD\cap {\eta}]} ).\]
\begin{compactenum}[(1)]
\setcounter{enumi}{+2}
\item It then remains to show that $(\RS{[\FD\cap \eta]},  \RS'{[\FD\cap \eta]} )$ is out-$\ii{d}$ collapsible. This, however, is true by inductive assumption $\io_{d,k-1}$.
\end{compactenum}
\noindent Step (2) is completely analogous step (1), so its proof is left out. It remains to prove (1).
To achieve this, we establish a geometry-based strict total order on the vertices of $\RS{[\FD\cap \overline{\eta}_+]}$, and we collapse them away one at the time. 

In details: Let $\zeta$ be a central projection of $\intx\FD\cap H$ to $\R^{k}$, and let $\nu_+$ denote the interior normal to the halfspace $\zeta(\intx \overline{\eta}_+)\subset \R^{d-1}$. Perturb $\nu_+$ to a vector $\nu$ such that the function $x \mapsto \langle  \zeta(x),  \nu \rangle$ preserves the order given by $\langle  \zeta(\cdot),  \nu_+  \rangle$ and induces a strict total order on $\F_0(\RS(\s, \FD\cap \overline{\eta}_+))$ (see also the proof of Lemma~\ref{lem:hemisphere}). 

The function $\langle \zeta(x), \nu  \rangle$ induces a strict total order on the vertices $v_0,v_1,\, \cdots, v_n$ of $\s$ in the relative interior of $\FD\cap \overline{\eta}_+$, starting with the vertex $v_0$ maximizing it and such that the labeling reflects the order. Set $\Sigma_i:=\RS{[\FD\cap \overline{\eta}_+]}-\{v_0,v_1, \, \cdots, v_{i-1}\}.$ We show (1) by demonstrating that, for all $i\in [0,n]$, we have that \[(\Sigma_i, \RS(\Sigma_i, \HA))  \searrow_{\textrm{ out-}\ii{d}}  (\Sigma_{i+1}, \RS(\Sigma_{i+1}, \HA) )=(\Sigma_i-{v_i},\RS(\Sigma_i,\HA)-{v_i}).\]
To see this, notice that $\Lk({v_i},\s(\EH^H))$ is a combinatorial stratification of the $(k-1)$-sphere $\RN_{v_i}^1 H$, given by the hyperplane extension $\Lk({v_i}, \EH^H)$ of the $\cc$-arrangement $\Lk({v_i},\HA^H)$. The complex $\Lk({v_i}, \Sigma_i )$ is the restriction of $\Lk({v_i}, \s(\EH^H))$ to the general position hemisphere $\RN_{v_i}^1 \FD_{v_i}$, where
\[\FD_{v_i} \; := \; 
\zeta^{-1} (\{  x\in \R^{k-1} : 
\langle  \zeta({v_i}), \nu  \rangle 
\; \ge \; 
\langle  x,  \nu \rangle \}), \]
since ${v_i}$ maximizes $\langle  \zeta(x),  \nu  \rangle$ among the vertices $v_i, v_{i+1}, v_{i+2},\, \cdots$ and the vertices of $\RS[\FD\cap \eta]$. 

At this point, we want to apply the induction assumptions $\io_{k-1,k-1}$ and $\iit_{k-1,k-1}$ to $\Lk({v_i}, \Sigma_i )$. There are two cases to consider.
\enlargethispage{3mm}
\begin{compactitem}[$\circ$]
\item If $k=d=1$ mod $\cc$, then $\Lk({v_i},\HA^H)$ is a non-essential $\cc$-arrangement in $\RN_{v_i}^1 H$, and it is nonempty if and only if $v_i$ is in $\RS(\s,\HA)$. Thus, by inductive assumption~$\iit_{k-1,k-1}$ and Lemma~\ref{lem:hemisphere}, we have that the pair 
$\big(\Lk({v_i}, \Sigma_i),\Lk({v_i}, \RS(\Sigma_i,\HA))\big )$
is a collapsible pair. Lemma \ref{lem:outcoll} now proves
\[(\Sigma_i,\RS(\Sigma_i,\HA)) \searrow_{\textrm{ out-}\ii{d}} (\Sigma_i-{v_i},\RS(\Sigma_i,\HA)-{v_i}).\]

\item If $k = 0$ mod $\cc$ or $k=d-1=1$ mod $\cc$, then $\Lk({v_i},\HA^H)$ is a $\cc$-arrangement in $\RN_{v_i}^1 H$. By inductive assumption~$\io_{k-1,k-1}$ and Lemma~\ref{lem:hemisphere}, the pair $\big(\Lk({v_i}, \Sigma_i),\Lk({v_i}, \RS(\Sigma_i,\HA)) \big)$
is an out-$\ii{k-1}$ collapsible~pair. Moreover, if $k
>2$, then $v_i\in\RS(\s,\HA)$ implies that $\Lk({v_i},\HA^H)$ is nonempty. 
Finally, we have $\ii{k}=\ii{d}-1$ by assumption on $d$ and $k$. Using Lemma \ref{lem:outcoll}, we consequently obtain that
\[(\Sigma_i,\RS(\Sigma_i,\HA)) \searrow_{\textrm{ out-}\ii{d}} (\Sigma_i-{v_i},\RS(\Sigma_i,\HA)-{v_i}). \qedhere\]
\end{compactitem}
\end{proof}
\newpage

\section{Proof of Theorem~\ref{MTHM:LEFT}}
We have now almost all the tools to prove our version of the Lefschetz Hyperplane Theorem (Theorem~\ref{thm:lef}); it only remains for us to establish the following lemma:

\begin{lemma}\label{lem:lef}
Let $\FD$, $\FD'$ denote a pair of closed hemispheres in $S^d$. Let $\HA$ denote a $\cc$-arrangement w.r.t.\ the complement $\OD$ of~$\FD$, let $\EH$ be a fine extension of $\HA$, and let $\s:=\s(\EH)$ denote the combinatorial stratification of $S^d$ given by it. If $\FD'$ is in general position with respect to $\s(\EH\cup \partial \FD)$, then 
\[(\RS(\s,\FD\cup \FD'),\RS(\s,(\FD\cup \FD')\cap \HA)) \searrow_{\textrm{ out-}\ii{d}} (\RS(\s,\FD),\RS(\s,\FD\cap \HA)).\]
\end{lemma}

\begin{proof}
Let $\zeta$ denote a central projection of $\OD$ to $\R^d$, and let $\nu_+$ denote the outer normal to the halfspace $\zeta( \OD \cap \FD')\subset \R^{d}$. Perturb $\nu_+$ to a vector $\nu$ such that the function $x \mapsto \langle  \zeta(x),  \nu  \rangle$ preserves the order given by $\langle  \zeta(\cdot),  \nu_+  \rangle$ and induces a strict total order on $\F_0(\RS(\s, \OD\cap \FD'))$ (see also the proof of Lemma~\ref{lem:hemisphere}).

The function $x \mapsto  \langle \zeta(x), \nu \rangle$ gives an order on the vertices $v_0,v_1,\, \cdots, v_n$ of $\RS(\s, \OD\cap \FD')$, starting with the vertex $v_0$ with the {highest} value under this function and such that the vertices are labeled to reflect their order. Set \[\Sigma_i:=\RS(\s, \OD \cap \FD')-\{v_0,v_1, \, \cdots, v_{i-1}\}.\]
In order to prove 
\[(\RS(\s,\FD\cup \FD'),\RS(\s,(\FD\cup \FD')\cap \HA)) \searrow_{\textrm{ out-}\ii{d}} (\RS(\s,\FD),\RS(\s,\FD\cap \HA)),\]
it suffices to prove that, for all $i\in [0,n]$, we have
\[(\Sigma_i, \RS(\Sigma_i,\HA))  \searrow_{\textrm{ out-}\ii{d}}  (\Sigma_i-v_i, \RS(\Sigma_i,\HA)-{v_i} ) =(\Sigma_{i+1}, \RS(\Sigma_{i+1},\HA)).\]
Thus, let $v_i$ denote any vertex of $\s$ in $ \OD \cap \FD'$.

The complex $\Lk(v_i,\s)$ is a combinatorial stratification of the $(d-1)$-sphere $\RN_{v_i}^1 S^d$ given by the fine extension $\Lk(v_i,\EH)$ of the $\cc$-arrangement $\Lk(v_i,\HA)$, and the complex $\Lk(v_i, \Sigma_i)$ is the restriction of $\Lk(v_i, \s)$ to the general position hemisphere $\RN_v^1 \FD_{v_i}$, where
\[\FD_{v_i} \; := \; \zeta^{-1} 
(\{  x\in \R^{k-1} : \langle  \zeta({v_i}),  \nu \rangle \; \ge \; 
\langle  x,  \nu  \rangle \}). \] 
Thus, as in Part III of Theorem~\ref{thm:hemisphere}, there are two cases:
\begin{compactitem}[$\circ$]
\item If $d=1$ mod $\cc$, then $\Lk(v_i,\HA)$ is a non-essential $\cc$-arrangement in $\RN_{v_i}^1 S^d$, and it is nonempty if and only if $v_i$ is in $\RS(\s,\HA)$. Thus, by Theorem~\ref{thm:hemisphere}(B) and Lemma~\ref{lem:hemisphere}, the pair 
$\big(\Lk({v_i},\Sigma_i),\Lk({v_i},\RS(\Sigma_i,\HA))\big)$
is a collapsible pair. Consequently, Lemma \ref{lem:outcoll} proves that the pair $(\Sigma_i, \RS(\Sigma_i,\HA))$ out-$\ii{d}$ collapses to the pair 
\[(\Sigma_i-v_i, \RS(\Sigma_i,\HA) -{v_i}) =(\Sigma_{i+1}, \RS(\Sigma_{i+1},\HA)).\]
\item If $d = 0$ mod $\cc$, then $\Lk({v_i},\HA)$ is a $\cc$-arrangement in $\RN_{v_i}^1 S^d$. By Theorem~\ref{thm:hemisphere}(A) and Lemma~\ref{lem:hemisphere}, the pair 
$ \big(\Lk({v_i},\Sigma_i),\Lk({v_i},\RS(\Sigma_i,\HA))\big)$
is an out-$\ii{d-1}$ collapsible pair. Moreover, if $d> 2$, then $v_i\in\RS(\s,\HA)$ implies that $\Lk({v_i},\HA^H)$ is nonempty. Since $\ii{d-1}=\ii{d}-1$ and by Lemma \ref{lem:outcoll}, we obtain \[(\Sigma_i,\RS(\Sigma_i,\HA)) \searrow_{\textrm{ out-}\ii{d}} (\Sigma_i-{v_i},\RS(\Sigma_i,\HA)-{v_i}). \qedhere\]
\end{compactitem}
\end{proof}

\begin{cor}\label{cor:lef}
Let $\FD$ denote a closed hemisphere in $S^d$, let $\OD$ denote its open complement. Let $\HA$ denote a $\cc$-arrangement w.r.t.\ $\OD$, and let $\s$ denote a combinatorial stratification of $S^d$ induced by~$\HA$. If $H$ is a hyperplane in $S^d$ that is in general position with respect to $\s(\EH\cup \partial \FD)$, then
\[\big(\RS(\s, S^d{\setminus} (\OD \cap H)),\RS(\s,  \HA \cap S^d{\setminus} (\OD \cap H))\big)  \searrow_{\textrm{out-}\ii{d}} (\RS(\s,\FD), \RS(\s,\FD\cap \HA)).\]
\end{cor}

\begin{theorem}\label{thm:lef}
Consider any affine $\cc$-arrangement $\HA$ in $\R^d$, and any hyperplane $H$ in $\R^d$ in general position with respect to $\HA$. Then the complement $\HA^{\comp}$ of $\HA$ is homotopy equivalent to $H\cap\HA^{\comp}$ with $e$-cells attached to it, where $e = \lceil\nicefrac{d}{\cc}\rceil =d-\lfloor\nicefrac{d}{\cc}\rfloor$.
\end{theorem}
\begin{proof}
Define $\rho$, $\OD$ and $\HA'$ in $S^d$ as in Definition~\ref{def:affine}, and define the hyperplane $H_\rho:=\SSp(\rho(H))\subset S^d$ induced by $H$ in $\R^d$. Let $\FD:=\OD^{\comp}$ denote the closed hemisphere complementary to $\OD$. Let $\s$ be a generic combinatorial stratification of $S^d$ induced by $\HA'$. By Corollary~\ref{cor:lef}, \[\big(\RS(\s, S^d{\setminus} (H_\rho \cap \OD)),\RS(\s, \HA'\cap S^d{\setminus} (H_\rho \cap \OD)\big)  \searrow_{\textrm{out-}\ii{d}} (\RS(\s,\FD), \RS(\s,\FD\cap \HA')).\]
The associated out-$\ii(d)$ collapsing sequence gives a Morse matching $\varPhi$ on $\s$ with the following properties:
\begin{compactitem}[$\circ$]
\item the critical faces of $\varPhi$ are the faces of $\RS(\s,\FD)$ and the faces of $\s$ that intersect $H_\rho \cap \OD$;
\item the outwardly matched faces of the pair $(\s, \RS(\s, \HA'))$ are all of dimension $\ii{d}$.
\end{compactitem}
By Theorem~\ref{thm:relcoll} the restriction $\varPhi^{\ast}_{\RS^\ast(\K(\HA',\s), \OD)}$ of the complement matching induced by $\varPhi$ to the complex $\RS^\ast(\K(\HA',\s), \OD)$ has the following critical faces:
\begin{compactitem}[$\circ$]
\item the faces of ${\RS^\ast(\K(\HA',\s), \OD\cap H_\rho)}$, whose duals are critical faces of $\varPhi$, and
\item the critical faces of dimension $e=\lceil\nicefrac{d}{\cc}\rceil=d-\ii{d}-1$, which correspond to the outwardly matched faces of $\varPhi$.
\end{compactitem}
Furthermore, since the faces of $\RS^\ast(\K(\HA',\s), \OD\cap H_\rho)$ are critical in $\varPhi^{\ast}_{\RS^\ast(\K(\HA',\s), \OD)}$, this Morse matching has no outwardly matched faces with respect to the pair \[\big(\RS^\ast(\K(\HA',\s), \OD),\RS^\ast(\K(\HA',\s), \OD\cap H_\rho)\big)\]
Thus, by Theorem~\ref{thm:MorseThm}, we have that \[\RS^\ast(\K(\HA',\s), \OD)\simeq \R^d{\setminus} \HA\] is, up to homotopy equivalence, obtained from \[\RS^\ast(\K(\HA',\s), \OD\cap H_\rho) \simeq H{\setminus} \HA\] by attaching $e$-dimensional cells, as desired.
\end{proof}

\section{Proof of Theorem~\ref{MTHM:BZP}}

We start with an easy consequence of the formula of Goresky--MacPherson. For completeness, we provide the (straightforward) proof in the~appendix.

\begin{lemma}\label{LEM:LHTCA}
Let $\HA$ denote a subspace arrangement in $S^d$, let $\OD$ denote an open hemisphere of $S^d$, and let $H$ be a hyperplane in $S^d$. Then we have the following:
\begin{compactenum}[\rm (I)]
\item If $\OD$ is in general position with respect to $\HA$, then for all $i$, \[\beta_i(S^d{\setminus} \HA)\ge\beta_i(\OD{\setminus} \HA).\]
\item If $H$ is in general position with respect to $\HA$ and $\OD$, then for all $i$, \[\beta_i(\OD{\setminus} \HA)\ge\beta_i((\OD\cap H){\setminus} \HA).\]
\end{compactenum}
\end{lemma}
Besides this lemma, we need the following elementary concept. For any convex set $\sigma$ and any hyperplane $H$ in $S^d$, let us denote by $\sigma^H$ the intersection of $\sigma$ with $H$. If $C$ is any collection of polyhedra in~$S^d$, we define $C^H:=\bigcup_{\sigma\in C} \sigma^H$.

\begin{example}[Lifting a Morse matching]
Let $\varSigma$ be a polyhedron in $S^d$. Let $H$ denote a general position hyperplane in $S^d$. Then $\varSigma^H$ is a polyhedron in $H$, and if $\sigma^H$ is a facet of $\varSigma^H$, then there exists a unique facet $\sigma$ of $\varSigma$ with the property that $\sigma^H:=\sigma\cap H$.

Consider now a subcomplex $C$ of a combinatorial stratification of $S^d$, and a Morse matching $\varphi$ on the complex~$C^H$. Then we can match $\sigma$ with $\varSigma$ for every matching pair $(\sigma^H, \varSigma^H)$ in the matching $\varphi$ of $C^H$. This gives rise to a Morse matching $\varPhi$ on $C$ from a Morse matching $\varphi$ on $C^H$, the \Defn{lift} of $\varphi$ to $C$.

\begin{figure}[htbf]
\centering 
 \includegraphics[width=0.64\linewidth]{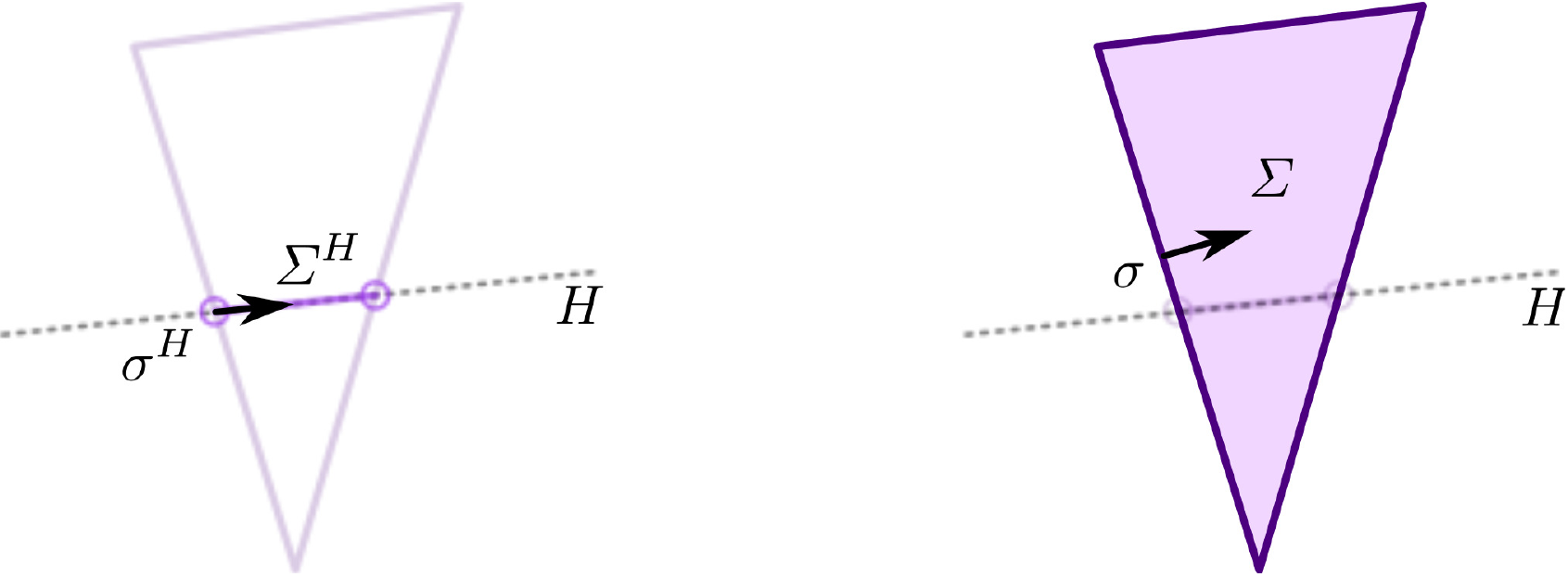} 
 \caption{\small The lift of a matching.}
\label{fig:lift}
\end{figure}
\end{example}

\begin{lemma}\label{lem:1} Let $\FD$ denote a closed hemisphere of $S^d$, and let $\OD:=\FD^{\comp}$ be its complement. Let $\HA$ be a $\cc$-arrangement w.r.t.\ $\OD$, and let $\s$ be a combinatorial stratification of $S^d$ induced by $\HA$. Finally, let $\K:=\K(\HA,\s)$ denote the associated complement complex. 

Then, there exists a Morse matching $\varPsi$ on $\s$ whose critical faces are the subcomplex $\RS(\s, \FD)$ and some additional facet of $\s$ such that the restriction $\varPsi^{\ast}_{\RS^\ast(\K,\OD)}$ of the complement matching to ${\RS^\ast(\K,\OD)}$ is perfect.
\end{lemma}

\begin{proof}
We abbreviate $\RS^\ast[M]:=\RS^\ast(\K,M)$ for any set $M$ in $S^d$. We prove the lemma by induction on the dimension; The case of $d=0$ is clearly true, since in this case $\RS(\s, \OD)$ is just a vertex. Assume now $d\ge 1$. Let $H$ denote a generic hyperplane in $S^d$. By induction on the dimension, there exists a Morse matching $\varphi$ on $\s^H$ such that the restriction $\varphi^{\ast}_{\RS^\ast(\K(\HA^H,\s^H),\OD\cap H)}$ of the complement matching induced by $\varphi$ to $\RS^\ast(\K(\HA^H,\s^H),\OD\cap H)$ is perfect on the latter. We lift $\varphi$ to a Morse matching $\varPhi$ of the faces of $\s$ intersecting~$H$.
Now, by Corollary~\ref{cor:lef}, 
\begin{equation}\tag{$\dagger$} \label{eq:co}
\big(\RS(\s, S^d{\setminus} (\OD \cap H)),\RS(\s,  \HA \cap S^d{\setminus} (\OD \cap H))\big)  \searrow_{\textrm{out-}\ii{d}} (\RS(\s,\FD), \RS(\s,\FD\cap \HA)).
\end{equation}
Denote the associated out-$\ii{d}$ collapsing sequence by $X$. Define the Morse matching $\varPsi$ as the union of the Morse matchings $\varPhi$ and $X$. We claim that $\varPsi$ gives the desired Morse matching on $\s$.

The Morse matching $\varPsi$ has the desired critical faces by construction, so it remains to show that $\varPsi^{\ast}_{\RS^\ast[\OD]}$ is perfect. By construction, $\varPsi^{\ast}_{\RS^\ast[\OD]}$ contains no outwardly matched faces with respect to the pair $(\RS^\ast[\OD],\RS^\ast[\OD\cap H])$. Furthermore, Observation~\eqref{eq:co} and Theorem~\ref{thm:relcoll} show that every critical face of $\varPsi^{\ast}_{\RS^\ast[\OD]}$ that is not in $\RS^\ast[\OD\cap H]$ is of dimension $e:=\lceil\nicefrac{d}{\cc}\rceil=d-\ii{d}-1$. Theorem~\ref{thm:MorseThm} now shows that the complex $\RS^\ast[\OD]$ is obtained, up to homotopy equivalence, from the complex $\RS^\ast[\OD\cap H]$ by attaching cells of dimension $e$. The attachment of each of these cells contributes to the homology of $\RS^\ast[\OD]$ either by deleting a generator of the homology in dimension $e-1$, or by adding a generator for the homology in dimension $e$. But if any of the $e$-cells deletes a generator in homology, then 
\[\beta_{e-1}(\OD{\setminus} \HA)=\beta_{e-1}(\RS^\ast[\OD])<\beta_{e-1}(\RS^\ast[\OD\cap H])=\beta_{e-1}((\OD\cap H){\setminus} \HA),\]
in contradiction with Lemma~\ref{LEM:LHTCA}(II). Thus, every $e$-cell attached must add a generator in homology. Since \[c_i(\varPhi^{\ast}_{\RS^\ast[\OD\cap H]})=c_i (\varphi^{\ast}_{\RS^\ast(\K(\HA^H,\s^H),\OD\cap H)}) =\beta_i((\OD\cap H){\setminus} \HA)\] for all $i$, we consequently have
$c_i(\varPsi^{\ast}_{\RS^\ast[\OD]})=\beta_i\big(\OD{\setminus} \HA\big)$
for all $i$, as desired.
\end{proof}

\begin{theorem} \label{thm:BZperfect}
Any complement complex of any $\cc$-arrangement $\HA$ in $S^d$ or~$\R^d$ admits a perfect Morse matching.
\end{theorem}

\begin{proof}
We distinguish two cases: (1) the case of affine $\cc$-arrangements in $\R^d$, and (2) the case of $\cc$-arrangements in $S^d$.

\begin{compactenum}[(1)]
\item If $\HA$ is any affine $\cc$-arrangement in $\R^{d}$, consider a radial projection $\rho$ of $\R^{d}$ to an open hemisphere $\OD$ in $S^d$, and the arrangement $\HA'$ induced by $\HA$, as defined in Definition~\ref{def:affine}. Any complement complex of the affine $2$-arrangement $\HA$ is of the form $\RS^\ast(\K(\HA',\s),\OD)$, where $\s$ is some combinatorial stratification induced by $\HA'$. The Morse matching $\varPsi^{\ast}_{\RS^\ast(\K(\HA',\s),\OD)}$ constructed in Lemma~\ref{lem:1} provides a perfect Morse function on it.

\item Suppose instead $\HA$ is any $\cc$-arrangement in $S^d$. Any complement complex of $\HA$ is of the form $\K(\HA,\s)$, where $\s$ is some combinatorial stratification induced by $\HA$. Let $\FD$ denote a generic closed hemisphere in $S^d$. By Lemma~\ref{lem:1}, there is a Morse matching $\varPsi$ on $\s$ whose critical faces are the subcomplex $\RS(\s, \FD)$ and some additional facet of $\s$ such that the restriction $\varPsi^{\ast}_{\RS^\ast(\K(\HA,\s),\OD)},\ \OD:=\FD^{\comp},$ of the complement matching to $\RS^\ast(\K(\HA,\s),\OD)$ is perfect. 

By Theorem~\ref{thm:hemisphere}(A), $\RS^\ast(\s,\FD)$ is out-$\ii{d}$ collapsible; by considering the union of the associated out-$\ii{d}$ collapsing sequence with the matching $\varPsi$, we obtain a Morse matching $\varOmega$ on $\K(\HA,\s)$ such that only outwardly matched faces of dimension $\ii{d}$ with respect to the pair $(\s,\RS(\s,\HA))$ are added when passing from $\varPsi$ to $\varOmega$. By Theorems~\ref{thm:relcoll} and~\ref{thm:MorseThm}, $\K(\HA,\s)$ is obtained from $\RS^\ast(\K(\HA,\s),\OD)$, up to homotopy equivalence, by attaching $e$-dimensional cells, where $e:=\lceil\nicefrac{d}{\cc}\rceil=d-\ii{d}-1$. Each of these cells can either add a generator for homology in dimension $e$, or delete a generator for the homology in dimension $e-1$. But if one of the cells deletes a generator, then \[\beta_{e-1}(\RS^\ast(\K(\HA,\s),\OD))<\beta_{e-1}(\K(\HA,\s)),\] which contradicts Lemma~\ref{LEM:LHTCA}(I). Thus, every cell attached adds a generator in homology, and in particular, since 
\[c_i(\varPsi^{\ast}_{\RS^\ast(\K(\HA,\s),\OD\cap H)})=\beta_i\big(\OD{\setminus} \HA\big)\]
for all $i$, we have
$c_i(\varOmega^{\ast})=\beta_i(S^d{\setminus} \HA)$
for all $i$. \qedhere
\end{compactenum}
\end{proof}

\newcommand\Ho{\mathrm{H}}

\section{Appendix}
\subsection{Proof of Lemma~\ref{LEM:LHTCA}}
Let us recall the Goresky--MacPherson formula. For an element $p$ of a poset $\mathcal{P}$, we denote by $\mathcal{P}_{<p}$ the poset of elements of $\mathcal{P}$ that precede $p$ with respect to the order of the poset. Let $\Delta(\mathcal{P})$ denote the order complex of a poset $\mathcal{P}$.  Let $\widetilde{\Ho}_\ast$ resp.\ $\widetilde{\Ho}^\ast$ denote reduced homology resp.\ cohomology, and let $\widetilde{\beta}_\ast$ denote the reduced Betti number.

\begin{theorem}[Goresky--MacPherson {\cite[Sec.\ III, Thm.\ A]{GM-SMT}}, {\cite[Cor.\ 2.3]{ZieZiv}}]\label{thm:gmf}
Let $\HA$ denote an arrangement of affine subspaces in $\R^d$. Let $\mathcal{P}(\HA)$ denote the poset of nonempty intersections of elements of $\HA$, ordered by reverse inclusion. Then, for all $i$,
\[\widetilde{\Ho}^i(\R^d{\setminus}{\HA}; \mathbb{Z} )\cong\bigoplus_{p\in \mathcal{P}}\widetilde{\Ho}_{d-2-i-\dim p}\big(\Delta(\mathcal{P}_{<p}(\HA)); \mathbb{Z} \big).\]
\end{theorem}

Apart from the Goresky--MacPherson formula, we will make use of the following observation: We say that a subposet $\mathcal{Q}$ of a poset $\mathcal{P}$ is a \Defn{truncation} of $\mathcal{P}$ if every element of $\mathcal{P}$ not in $\mathcal{Q}$ is a maximal element of~$\mathcal{P}$. In this case, for every element $q$ of $\mathcal{Q}$, we have that $\mathcal{Q}_{<q}=\mathcal{P}_{<q}$, and in particular, $\Delta(\mathcal{Q}_{<q})=\Delta(\mathcal{P}_{<q})$.

\smallskip

We turn to the proof of of Lemma~\ref{LEM:LHTCA}.

\begin{proof}[\textbf{Proof of Lemma~\ref{LEM:LHTCA}}]

The arrangement $\HA$ gives rise to a subspace arrangement $\HA_{\Sp}$ in $\R^{d+1}$ by considering, for every element $H$ of $\HA$, the subspace $\Sp(H)$ in $\R^{d+1}$. If $x$ is the midpoint of the open hemisphere $\OD$ in $S^d\subset \R^{d+1}$, and $H_{\TT}$ denotes the hyperplane in $\R^{d+1}$ tangent to $S^d$ in $x$, then $\HA_{\Sp}^{H_{\TT}}$, defined as the collection of intersections of elements of $\HA_{\Sp}$ with $H_{\TT}$, is a subspace arrangement in $H_{\TT}$. 

Furthermore, if $\zeta$ is a central projection of $\OD$ to $\R^d$, then for every element $h$ of $\HA$, $\zeta(h\cap \OD)$ is an affine subspace in $\R^d$. The collection of subspaces $\zeta(h\cap \OD),\ h\in \HA,$ gives an affine subspace arrangement $\HA_\zeta$ in~$\R^d$. Now, we have that
\begin{compactenum}[(I)]
\item $\R^{d+1}{\setminus} \HA_{\Sp}$ is homotopy equivalent to  $S^{d}{\setminus} \HA$, and $H_{\TT}{\setminus}\HA_{\Sp}^{H_{\TT}}$ is homeomorphic to $\OD{\setminus} {\HA}$; and
\item $\R^{d}{\setminus} {\HA}_{\zeta}$ is homeomorphic to $\OD{\setminus} {\HA}$, and $\zeta(\OD\cap H) {\setminus} {\HA}_{\zeta}$ is homeomorphic to $(\OD\cap H){\setminus} {\HA}$.
\end{compactenum}
Thus, both statements (I), (II) of the lemma are special cases of the following claim for affine subspace arrangements in $\R^d$.

\smallskip 

\noindent \emph{Let $\HA$ denote any affine subspace arrangement in $\R^d$, and let $H$ denote any hyperplane in $\R^d$ in general position with respect to $\HA$. Then, for all $i$ \[{\beta}_i(\R^d{\setminus} \HA)\ge{\beta}_i(H{\setminus} \HA).\]}
\vskip -4mm
\noindent Recall that we use, for an element $p$ of $\mathcal{P}(\HA)$ intersecting the hyperplane $H$, the notation $p^{H}:=p\cap H$ to denote the corresponding subspace of $H$. Recall also that $\HA^{H}$ denotes the affine arrangement obtained as the union of $h^{H}$, $h\in \HA$. Now, $\mathcal{P}(\HA^{H})$ is isomorphic to a truncation of $\mathcal{P}(\HA)$, so for every element $p$ of $\mathcal{P}(\HA)$ intersecting $H$, we have
\[\Delta(\mathcal{P}_{<{p^{H}}}(\HA^{H}))\cong\Delta(\mathcal{P}_{<p}(\HA)).\]
By this observation, and using the Goresky--MacPherson formula (Theorem~\ref{thm:gmf}), we obtain
\[\widetilde{\beta}_i(H{\setminus}{\HA})=\sum_{p\in \mathcal{P}(\HA^{H})}\widetilde{\beta}_{d-2-i-\dim p}\big(\Delta(\mathcal{P}_{<p}(\HA^{H}))\big)=\sum_{\substack {p\in \mathcal{P}(\HA) \\ p\cap H\ne \emptyset}}\widetilde{\beta}_{d-2-i-\dim p}\big(\Delta(\mathcal{P}_{<p}\HA))\big).\] 
Using Theorem~\ref{thm:gmf} once more, we furthermore see that
\[\widetilde{\beta}_i(\R^d{\setminus}{\HA})=\sum_{p\in \mathcal{P}(\HA)}\widetilde{\beta}_{d-2-i-\dim p}\big(\Delta(\mathcal{P}_{<p}(\HA))\big)\ge \sum_{\substack {p\in \mathcal{P}(\HA) \\ p\cap H\ne \emptyset}}\widetilde{\beta}_{d-2-i-\dim p}\big(\Delta(\mathcal{P}_{<p}(\HA))\big).\]
Consequently, we get \[\widetilde{\beta}_i(\R^d{\setminus}{\HA})\ge\widetilde{\beta}_i(H{\setminus}{\HA}). \qedhere\]
\end{proof}
\enlargethispage{4mm}

\subsection{Minimality of {\em c}-arrangements}\label{sec:car}
Recall that a \Defn{$c$-arrangement} $\HA$ in $S^d$ (resp.\ in $\R^d$) is a finite collection $(h_i)_{i\in [1,n]}$ of distinct $(d-c)$-dimensional subspaces of $S^d$ (resp.\ of $\R^d$), such that the codimension of any non-empty intersection of its elements is divisible by $c$. 
\begin{quote}
\emph{Is the complement of any $c$-arrangement $\HA$ a minimal space?}
\end{quote}
For $c=2$, the answer is positive, as we saw in Corollary~\ref{mcor:m}. For $c  \neq 2$, the answer is also positive. In fact, the complement $\HA^{\comp}$ of any $c$-arrangement $\HA$ in $S^d$ or $\R^d$, $c\neq 2$, has the following properties:
\begin{compactitem}[$\circ$]
\item $\HA^{\comp}$ has the homotopy type of a CW complex. (This holds for arbitrary subspace arrangements.)
\item $\HA^{\comp}$ has no torsion in homology. (This holds for arbitrary $c$-arrangements, cf.~\cite[Sec.\ III Thm.\ B]{GM-SMT}.)
\item Every connected component of $\HA^{\comp}$ is simply connected. (This is easy to prove, but holds only if $c \neq 2$.)
\end{compactitem} 
Now, \emph{any topological space satisfying these three properties is a minimal space}. This is proven in Hatcher~\cite[Prp.\ 4C.1]{Hatcher}; see also Anick~\cite[Lem.\ 2]{Anick} and Papadima--Suciu~\cite[Rem.\ 2.14]{PapadimaSuciu}. Consequently, Corollary~\ref{mcor:m} can be re-stated as follows:

\begin{cor}
The complement of any $c$-arrangement is a minimal space.
\end{cor}

In the rest of this section, we provide an analogue of Theorem~\ref{MTHM:BZP} for $c$-arrangements. First, some definitions. A \Defn{sign extension} $\SH$ of a $c$-arrangement $\HA=(h_i)_{i\in [1,n]}$ in $S^d$ is any collection of hyperplanes \[(H_{ij})\subset S^d,\ {i\in [1,n],\ j\in[1,c-1]},\] such that $\cap_{j\in [1,c-1]} H_{ij}=h_i$ for each $i$. A \Defn{hyperplane extension} $\EH$ of $\HA$ in $S^d$ is a sign extension of $\HA$ together with an arbitrary (but finite) collection of hyperplanes in $S^d$. The subdivision of $S^d$ into convex sets induced by $\EH$ is the \Defn{stratification} of $S^d$ given by $\EH$. A stratification is \Defn{combinatorial} if it is a regular CW complex.

Let $\s$ be a combinatorial stratification of the sphere $S^d$, given by some fine extension of a $c$-arrangement~$\HA$. Let $\s^\ast$ be the dual block complex of $\s$. Define $\K(\HA,\s):=\RS^\ast(\s^\ast,S^d{\setminus} \HA)$. The regular CW complex $\K(\HA,\s)$ is the \Defn{complement complex} of $\HA$ induced by $\EH$. Complement complexes of $c$-arrangements in $\R^d$ are defined by restricting the complement complex of a spherical arrangement to a general position hemisphere, cf.\ Definition~\ref{def:affine}.

\begin{theorem}\label{thm:BZperfectc}
Any complement complex of any $c$-arrange\-ment $\HA$ in $S^d$ or~$\R^d$ admits a perfect Morse matching.
\end{theorem}

\begin{proof}
The proof of this theorem is identical to the proof of Theorem~\ref{thm:BZperfect} up to replacing $2$ by $c$.
\end{proof}

\part{Projectively unique polytopes}\label{pt:pup}

\vspace*{\fill}
\thispagestyle{empty}
\begingroup
{\em
{\bf The space of realizations of a convex polytope} is one of the oldest subjects in polytope theory, most likely initiated by Legendre \cite{Legendre} and his contemporary Cauchy. Comparatively young results of Mn\"ev \cite{Mnev}, Richter-Gebert \cite{RG} and others showed that the subject is very hard in general, but its fascination remains unbroken. One of the problems concerning realization spaces is the study of polytopes which have essentially a unique realization: These are the projectively unique polytopes, for which the group of projective transformations acts transitively on the realization space.

Projectively unique polytopes were first studied explicitly by Perles, Shephard and McMullen (cf.~\cite{Grunbaum, PerlesShephard, McMullen}), who explored many of their properties. Still, we are far from a complete understanding of projectively unique polytopes. The purpose of this part of the thesis is to give new constructions for projectively unique polytopes, and use these constructions to solve old problems in the theory of polytopes. In the introductory Chapter \ref{ch:substacked}, we prove a universality property for projectively unique polytopes. As a corollary, we disprove a conjecture of Shephard. Furthermore, Chapter \ref{ch:substacked} is used to anticipate a technique needed in Chapter \ref{ch:lowdim}, where we subsequently provide an infinite family of projectively unique polytopes in bounded dimension, thereby answering a classical problem of Perles and Shephard and a related problem going back to Legendre and Steinitz.
}
\endgroup
\vspace*{\fill}
\setcounter{thmmain}{0}
\setcounter{figure}{0}

\newcommand{\Q}{\mathrm{Q}}
\newcommand{\PROJ}[1]{\mathrm{PROJ}\left[{#1}\right]}
\newcommand{\cub}[2]{C^{#1}({#2})}
\newcommand{\dH}{\mathrm{d}_{\mathrm{H}}}
\newcommand{\di}{\mathrm{d}_i}
\newcommand{\FUNC}[1]{\mathrm{FUNC}\left[{#1}\right]}
\newcommand{\COOR}[1]{\mathrm{COOR}\left[{#1}\right]}
\newcommand{\kF}[1]{{#1}}
\newcommand{\vecone}{{1}}
\newcommand{\veczero}{{0}}
\newcommand{\vect}[1]{{#1}}
\newcommand{\ve}{{e}}

\chapter{Subpolytopes of stacked polytopes}\label{ch:substacked}

\section{Introduction}

In this chapter, we provide the following universality theorem for projectively unique polytopes.

\begin{thmmain}\label{thm:ratprj}
For any algebraic polytope $P$, there exists a polytope $\widehat{P}$ that is projectively unique and that contains a face projectively equivalent to $P$.
\end{thmmain}

Here, a polytope is \Defn{algebraic} if the coordinates of all of its vertices are real algebraic numbers, and a polytope $P$ in $\R^d$ is \Defn{projectively unique} if any polytope ${P}'$ in $\R^d$ combinatorially equivalent to $P$ is \Defn{projectively equivalent} to~$P$. In other words, $P$ is projectively unique if for every polytope ${P}'$ combinatorially equivalent to $P$, there exists a projective transformation of $\R^d$ that realizes the given combinatorial isomorphism from $P$ to~${P}'$.

\begin{rem} Theorem~\ref{thm:ratprj} is sharp: We cannot hope that every polytope is the face of a projectively unique polytope. Indeed, it is a consequence of the Tarski-Seidenberg Theorem \cite{BierstoneMilman, Lindstrom} that every combinatorial type of polytope has an algebraic realization. In particular, every projectively unique polytope, and every single one of its faces, must be projectively equivalent to an algebraic polytope. Hence, a $d$-dimensional polytope with $n\ge d+3$ vertices whose set of vertex coordinates consists of algebraically independent transcendental numbers is not a face of any projectively unique polytope.
\end{rem}

\begin{rem}
A consequence of Theorem~\ref{thm:ratprj} is that for every finite field extension $F$ over $\mathbb{Q}$, there exists combinatorial type of polytope $\mathrm{NR}(F)$ that is projectively unique, but not realizable in any vector space over $F$. This extends on a famous result of Perles, who constructed a projectively unique polytope that is not realizable in any rational vector space, cf.\ \cite[Sec.\ 5.5, Thm.\ 4]{Grunbaum}.
\end{rem}

In the second part of this paper, we consider a conjecture of Shephard, who asked whether every polytope is a \Defn{subpolytope} of some stacked polytope, i.e.\ whether it can be obtained as the convex hull of some subset of the vertices of some stacked polytope. While he proved this wrong in~\cite{Shephard74}, he conjectured it to be true in a combinatorial sense.

\begin{conjecture}[Shephard~\cite{Shephard74}, Kalai {\cite[p.\ 468]{Kalai}}, \cite{KalaiKyoto}]\label{con:shka}
For every $d\ge 0$, every combinatorial type of $d$-dimensional polytope can be realized using subpolytopes of $d$-dimensional stacked polytopes.
\end{conjecture}

The conjecture is true for $3$-dimensional polytopes, as seen by K\"omhoff in~\cite{Komhoff80}, but remained open for dimensions $d> 3$. On the other hand, Theorem~\ref{thm:ratprj} encourages us to attempt a disproof of Shephard's conjecture. The idea is to use the universality theorem above to provide a projectively unique polytope that is not a subpolytope of any stacked polytope. Since any admissible projective transformation of a stacked polytope is a stacked polytope, no realization of the polytope provided this way is a subpolytope of any stacked polytope. 

Unfortunately, the method of Theorem~\ref{thm:ratprj} is highly ineffective: The counterexample to Shephard's conjecture it yields is of a very high dimension. We use a refined method, building on the same idea, to present the following result.

\begin{thmmain}\label{thm:proj}
There exists a combinatorial type of $5$-dimensional polytope that cannot be realized as a subpolytope of any stacked polytope.
\end{thmmain}

It remains open to decide whether every combinatorial type of $4$-dimensional polytope can be realized as the subpolytope of some stacked polytope. 

\section{Point configurations, PP configurations and weak projective triples}\label{sec:ppc}

We follow Gr\"unbaum~\cite{Grunbaum} for the notions of projectively unique polytopes and point configurations, and Richter-Gebert~\cite{RG} for Lawrence equivalence and Lawrence extensions.

\begin{definition}[PP configurations, Lawrence equivalence, projective uniqueness]

A \Defn{point configuration} in $\R^d$ is a finite collection $R$ of points in~$\R^d$. If $H$ is an oriented hyperplane in $\R^d$, then we use $H_+$ resp.\ $H_-$ to denote the open halfspaces 
bounded by $H$. If $P$ is a polytope in $\R^d$ such that $P\cap R=\emptyset$ then the pair $(P,R)$ is a 
\Defn{polytope--point configuration}, or \Defn{PP configuration}. A hyperplane $H$ is \Defn{external to $P$} 
if $H\cap P$ is a face of~$P$.

Two PP configurations $(P,R)$, $(P',R')$ in $\R^d$ are \Defn{Lawrence equivalent} if there is a 
bijection $\varphi$ between the vertex sets of $P$ and $P'$ and the points of $R$ and $R'$, respectively, 
such that, if $H$ is any hyperplane with the property that the closure of $H_-$ contains $P$, there exists an oriented hyperplane $H'$ with the property that the closure of $H'_-$ contains $P'$ and  
\[
    \varphi(\F_0(P)\cap H_-)=\F_0(P')\cap H'_-, \qquad 
    \varphi(R\cap H_+)={R'}\cap H'_+, \qquad 
    \varphi(R\cap H_-)={R'}\cap H'_-.
\]
A PP configuration $(P,R)$ in~$\R^d$ is \Defn{projectively unique} if for any PP configuration $(P',R')$ in~$\R^d$ 
Lawrence equivalent to it, and every bijection $\varphi$ that induces the Lawrence equivalence, there is a projective transformation $T$ that realizes $\varphi$. A point configuration $R$ in $\R^d$ is \Defn{projectively unique} if the PP configuration $(\emptyset, R)$ is projectively unique, and it is an easy exercise to see that a polytope $P$ in $\R^d$ is projectively unique if and and only if the PP configuration $(P,\emptyset)$ is projectively unique.
\end{definition}

\begin{prp}[Lawrence extensions, cf.\ {\cite[Lem.\ 3.3.3 and 3.3.5]{RG}}, {\cite[Thm.\ 5]{ZNonr}}]\label{prp:mlwextn}
Let $(P,R)$ be a projectively unique PP configuration in $\R^d$. Then there exists a $(\dim P+f_0(R))$-dimensional polytope on $f_0 (P) + 2 f_0(R)$ vertices that is projectively unique and that contains $P$ as a face.
\end{prp}

\begin{proof}[Sketch of Proof]
Set $k=f_0(R)$, and denote the elements of $R$ by $r_i,\, 1\leq i\leq k$. To obtain the Lawrence extension of $(P,R)$, consider the canonical embedding of $\R^{d}$ (and with it $(P,R)$) into $\R^{d+k}$, i.e.\ $\R^d$ is the subspace of $\R^{d+k}$ spanned by the first $d$ base vectors, where $\R^{d+k}$ is considered with the basis 

\[\big(e_1,\, \cdots, e_d, e_{r_1}, \, \cdots, e_{r_{k}}\big).\]

\begin{figure}[htbf]
\centering 
  \includegraphics[width=0.53\linewidth]{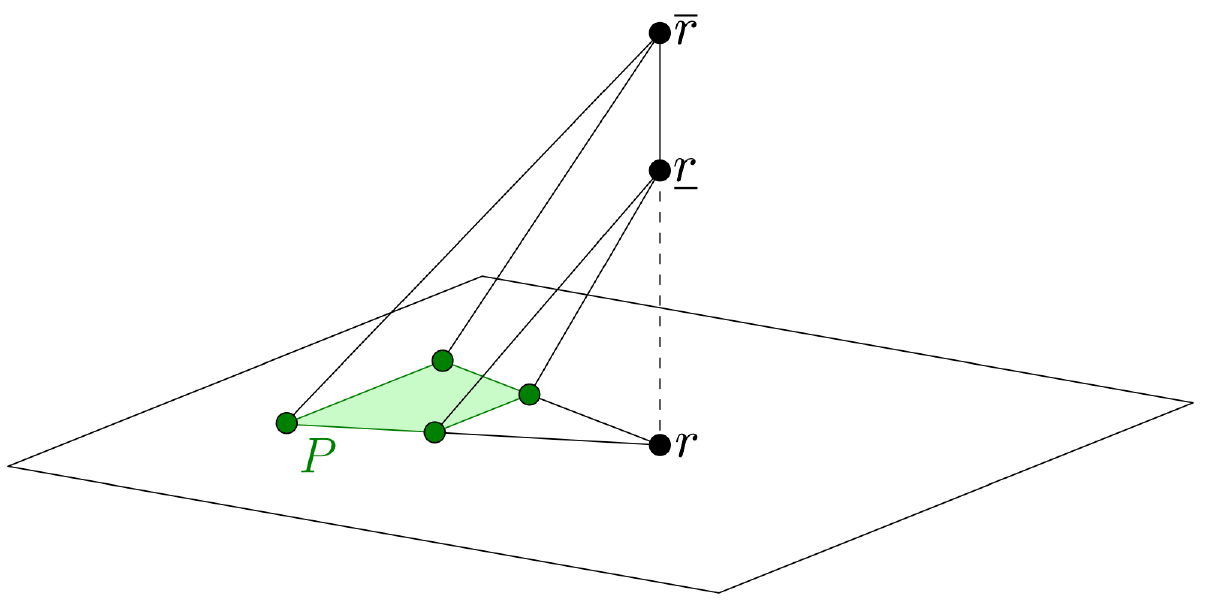} 
  \caption{\small  The Lawrence polytope associated to a polytope $P$ and a point $r\notin P$ is obtained by lifting $r$ to two new points $\underline{r}$ and $\overline{r}$.} 
  \label{fig:lawrence}
\end{figure}

\noindent The \Defn{Lawrence polytope} of a PP-configuration $(P,R)$ is defined as \[\operatorname{L}(P,R):=\conv \Big(P \cup \bigcup_{1\le i \le k} \{ \underline{r}_i,\, \overline{r}_i\} \Big),\ \ \underline{r}_i:={r}_i+e_{r_i},\ \overline{r}_i:={r}_i+2e_{r_i}.\] This polytope has the desired number of vertices, dimension, and contains $P$ as a face. One can show that $\operatorname{L}(P,R)$ is projectively unique. Here, we discuss only the special case where $R$ is a single point $r$; the full case is analogous.

Consider any polytope $\operatorname{L}'$ in $\R^{d+k}$ combinatorially equivalent to $\operatorname{L}(P,R)$. Intersecting the ray from $\underline{r}'$ through $\overline{r}'$ with the affine span of the facet $P'$ gives a point $r'$, and  $(P',r')$ is Lawrence equivalent to $(P,r)$, cf.~\cite[Lem.~3.3.5]{RG}. We construct a projective transformation from $\operatorname{L}$ to $\operatorname{L}'$ in three steps.
\begin{compactitem}[$\circ$]
\item Since $(P,r)$ is projectively unique, there exists a projective transformation $T$ of $S^{d}$ that maps~$(P,r)$ to~$(P',r')$.
\item The lines $\aff\{T(\underline{r}),\, T(\overline{r})\}$ and $\aff\{\underline{r}',\, \overline{r}'\}$ intersect in $r'$. Thus, we can use an affine transformation $U$ restricting to the identity on $\aff P'$ to map  $\aff\{T(\underline{r}),\, T(\overline{r})\}$ to $\aff\{\underline{r}',\, \overline{r}'\}$, and $T(\underline{r})$ to $\underline{r}'$.
\item A projective transformation $V$ fixing $\underline{r}'$ and $\aff P'$ can then be used to map $(UT)(\overline{r})$ to $\overline{r}'$. 
\end{compactitem}
The configuration of maps $V\circ U\circ T$ gives the desired projective transformation from $\operatorname{L}(P,R)$ to $\operatorname{L}'$.
\end{proof}

\section{Weak projective triples}\label{sec:wpt}

\begin{definition}[Framed PP configurations]\label{def:framep} Let $(P, R)$ denote any PP configuration in $\R^d$, and let $Q$ be any subset of $\F_0(P) \cup R$.
Let $(P',R')$ be any PP configuration in $\R^d$ Lawrence equivalent to $(P, R)$, and let $\varphi$ denote the labeled isomorphism inducing the Lawrence equivalence. 

The PP configuration $(P, R)$ is \Defn{framed} by the set $Q$ if $\varphi_{|Q}=\mathrm{id}_{|Q}$ implies $\varphi=\mathrm{id}$. Similarly, a polytope~$P$ (resp.\ a point configuration~$R$) is \Defn{framed} by a set $Q$ if $(P,\emptyset)$ (resp.\ $(\emptyset,R)$) is framed by $Q$.
\end{definition}

\begin{examples}[Some instances of framed PP configurations]\label{ex:stdet} $\mbox{}$
\begin{compactenum}[(i)]
\item If $(P,R)$ is any PP configuration, then $\F_0(P)\cup R$ frames $(P,R)$.
\item If $(P,R)$ is any PP configuration framed by a set $Q$, then every superset of $Q$ frames $(P,R)$ as well.
\item If $(P,R)$ is any projectively unique PP configuration, and $Q\subset \F_0(P) \cup R$ is a projective basis, then $Q$ frames $(P,R)$.
\item Any $d$-cube, $d\geq 3$, is framed by $2^d-1$ of its vertices, cf.\ \cite[Lem.~3.4]{AZ12}.
\item In the next chapter we will construct an infinite family of (combinatorially distinct) $4$-dimensional polytopes that are all framed by subsets of their vertex-sets of cardinality $24$ (Theorem~\ref{mthm:Lowdim}).
\end{compactenum}
\end{examples}

\begin{definition}[Weak projective triple in $\R^d$]\label{def:wpt}
A triple $(P,Q,R)$ of a polytope $P$ in $\R^d$, a subset $Q$ of $\F_0(P)$ and a point configuration $R$ in $\R^d$ is a \Defn{weak projective triple} in $\R^d$ if and only if

\begin{compactenum}[\rm(a)]
\item $(\emptyset, Q \cup R)$ is a projectively unique point configuration,
\item $Q$ frames the polytope $P$, and
\item some subset of $R$ spans a hyperplane $H$, called the \Defn{wedge hyperplane}, which does not intersect $P$.
\end{compactenum}
\end{definition}

\begin{definition}[Subdirect Cone]\label{def:subd}
Let $(P,Q,R)$ be a weak projective triple in $\R^d$, seen as canonical subspace of $\R^{d+1}$. Let $H$ denote the wedge hyperplane in $\R^d$ spanned by vertices of $R$ with $H\cap P=\emptyset$. Let $v$ denote any point not in $\R^d$, and let $\widehat{H}$ denote any hyperplane in $\R^{d+1}$ such that $\widehat{H}\cap \R^d=H$ and $\widehat{H}$ separates $v$ from $P$. Consider, for every vertex $p$ of $P$, the point $p^v=\conv\{v,\,p\}\cap \widehat{H}$. Denote by $P^v$ the pyramid
\[P^v:=\conv \Big( v\cup \bigcup_{p\in \F_0(P)} p^v \Big).\]
 The PP configuration $(P^v,Q \cup R)$ in $\R^{d+1}$ is a \Defn{subdirect cone} of $(P,Q,R)$.
\end{definition}

\begin{figure}[htbf]
\centering 
  \includegraphics[width=0.90\linewidth]{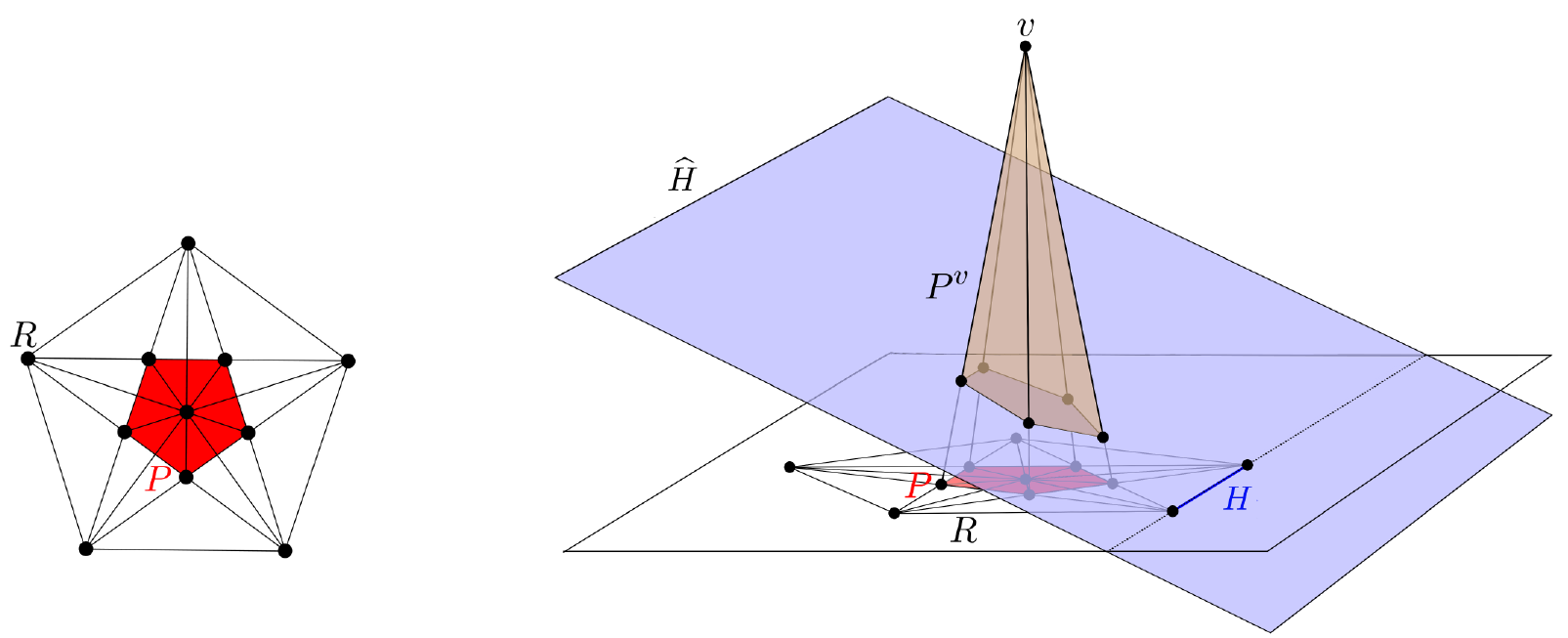} 
  \caption{\small The subdirect cone of the weak projective triple $(P,Q,R)$ in the special case where $Q$ coincides with the vertex-set of $P$.} 
  \label{fig:subdirect}
\end{figure}

\begin{lemma}\label{lem:subdc}
For any weak projective triple $(P,Q,R)$ the subdirect cone $(P^{\vect {v}},Q \cup R)$ is a projectively unique PP configuration, and the base of the pyramid $P^{\vect {v}}$ is projectively equivalent to $P$.\end{lemma}	

\begin{proof}
Define $Q^v$ as the set of all points $q^v=\conv\{v,\,q\}\cap \widehat{H}$, where $q$ ranges over all elements of $Q$. Then $Q^v$ frames the polytope $P^v\cap \widehat{H}$.

Consider any realization $(A^w,B\cup C)$ of $(P^v,Q \cup R)$. Let $H$ denote the wedge hyperplane of the weak projective triple $(P,Q,R)$, and let $I$ denote the hyperplane spanned by the points of $C$ corresponding to $H\cap R$ under the Lawrence equivalence. Let $\widehat{I}$ denote the hyperplane containing $I$ and the facet of $A^w$ not containing $w$. 

The point configuration $B\cup C$ is a realization of $Q \cup R$ since both sets lie in hyperplanes not intersecting the polytopes $A^w$ resp.\ $P^v$. 
Thus we may assume, after a projective transformation, that $B=Q$, $C=R$, $w=v$ and $\widehat{I}=\widehat{H}$. The vertices of $B^w\subset\F_0(A^w)=\F_0(A^v)$ are determined as the intersections of $\SSp\{v,\,b\},\, b\in B$
with $\widehat{I}=\widehat{H}$. In particular $B^w=B^v=Q^v$. As $Q^v$ frames the facet $P^v\cap \widehat{H}$, we have $A^w\cap \widehat{I}=P^v\cap \widehat{H}$, and consequently $A^w=P^v$.\end{proof}

\begin{cor}\label{cor:wpt}
If $(P,Q,R)$ is a weak projective triple in $\R^d$, there exists a projectively unique polytope of dimension $\dim P + f_0(Q) + f_0(R) +1$ with $f_0(P) + 2f_0(Q) + 2f_0(R) +1$ vertices that has a face projectively equivalent to $P$.
\end{cor}

\begin{proof}
By Lemma~\ref{lem:subdc}, the subdirect cone $(P^v,Q \cup R)$ of the triple $(P,Q,R)$ is a projectively unique PP configuration, and by construction, the polytope $P^v$ (of dimension $\dim P +1$) has a facet projectively equivalent to $P$. Consequently, by Proposition~\ref{prp:mlwextn}, there exists a projectively unique polytope of the desired dimension, with the desired number of vertices and with a face projectively equivalent to $P$. 
\end{proof}

\section{Universality of projectively unique polytopes}

Let $P$ be any algebraic polytope. Our goal for this section is to find a weak projective triple that contains~$P$. Applying Corollary~\ref{cor:wpt} then finishes the proof of Theorem~\ref{thm:ratprj}. The main step towards that goal is to embed $\F_0(P)$ into a projectively unique point configuration. In the construction, we will use the following straightforward observation repeatedly.

\begin{lemma}\label{lem:union}
Let $R$ be a projectively unique point configuration, let $Q\subseteq R$ and let $R'\supseteq Q$ be a point configuration framed by $Q$. Then $R\cup R'$ is a projectively unique point configuration.
\end{lemma}

\begin{prp}\label{prp:Qdm}
The point configuration $\Q^d:=\{v\in \Z^d\subset \R^d : ||v||_\infty\le  1\}$
is projectively unique for every $d\ge3$.
\end{prp}
\begin{proof}
The proof is by induction on $d$. We start proving that $\Q^3$ is projectively unique. This implies that $\Q^d$ is projectively unique for any $d\geq 3$.

\medskip
\noindent \textbf{$\Q^3$ is projectively unique:} To see that $\Q^3$ is projectively unique, we start with the folklore observation that the points $(\pm 1, \pm 1, \pm 1)$, together with the origin $(0,0,0)$, form a projectively unique configuration $W\subsetneq \Q^3$ (cf.\ Figure~\ref{fig:Q31_1}). Furthermore, we claim $W$ frames $\Q^3$, thereby proving that $\Q^3$ is projectively unique since $W$ is projectively unique.

\begin{figure}[htpb]
\centering
\begin{subfigure}[t]{.3\linewidth}
\centering
\quad\includegraphics[width=.7\linewidth]{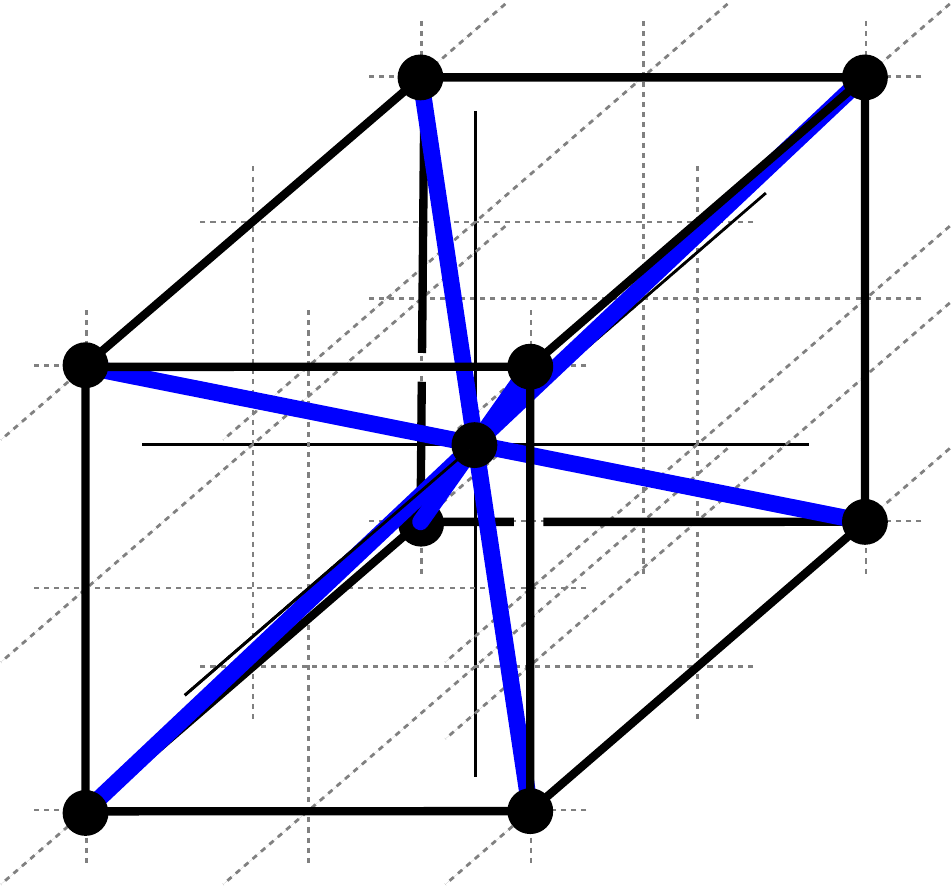}
\caption{$W$}\label{fig:Q31_1}
\end{subfigure}\qquad
\begin{subfigure}[t]{.3\linewidth}
\centering
\includegraphics[width=.7\linewidth]{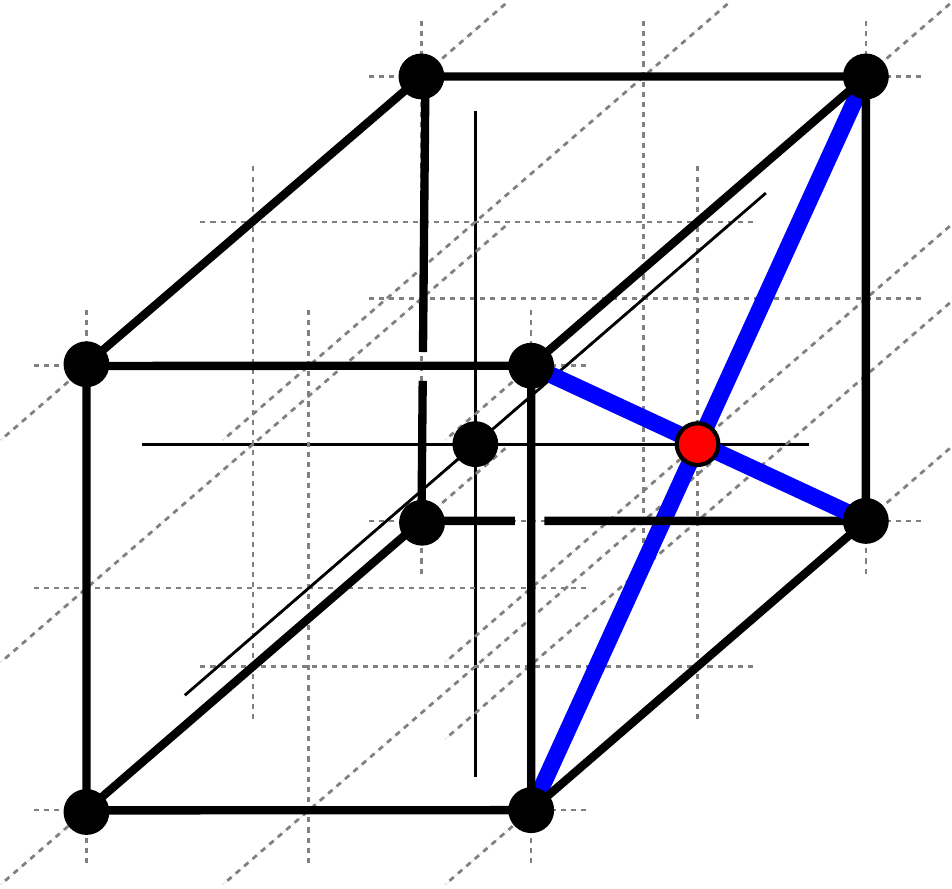}
\caption{Center points of facets.}\label{fig:Q31_2}
\end{subfigure}\qquad
\begin{subfigure}[t]{.3\linewidth}
\centering
\includegraphics[width=.7\linewidth]{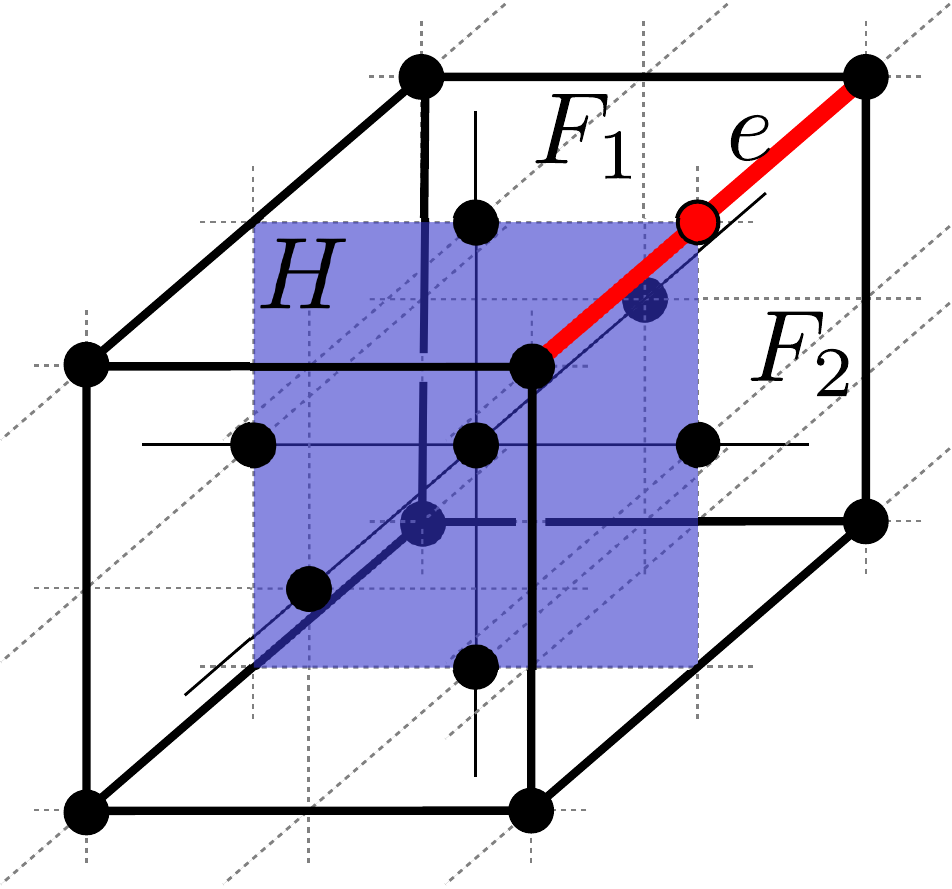}
\caption{Center points of edges.}\label{fig:Q31_3}
\end{subfigure}
\caption{Showing that $\Q^3$ is projectively unique.}\label{fig:Q31}
\end{figure}

To see this, notice that the point $(1,0,0)$ of $\Q^3\setminus W$ is determined as the intersection of the lines $\aff\{(+1,+1,+1),\, (+1,-1,-1)\}$ and $\aff\{(+1,+1,-1),\, (+1,-1,+1)\}$, which are spanned by points of~$W$. Similarly, all points that arise as coordinate permutations and/or sign changes from $(+1,0,0)$ are determined this way. Geometrically, these are the center points of the facets of the cube $[-1,1]^3=\conv W=\conv \Q^3$ (cf.~Figure~\ref{fig:Q31_2}). 

The remaining lattice points of $\Q^3$ coincide with the midpoints of the edges of said cube. To determine them, let $e$ be any edge of $\Q^3$ and let $F_1$ and $F_2$ be the facets of $\Q^3$ incident to that edge. Finally, let $H$ be the hyperplane spanned by the center point of $\Q^3$ and the center points of $F_1$ and $F_2$. The midpoint of~$e$ is the unique point of intersection of $e$ and $H$ (cf.\ Figure~\ref{fig:Q31_3}).

\medskip
\noindent \textbf{$\Q^d$ is projectively unique:} For $d\geq 4$, consider the projective basis $B$ of $\R^d$ consisting of the vertex $v_0:=(+1,+1,\dots,+1)$ of $[-1,1]^d$, together with the neighboring vertices $v_1:=(-1,+1,\dots,+1)$, $v_2:=(+1,-1,\dots,+1)$, $\dots$, $v_d:=(+1,+1,\dots,-1)$ and the origin $o:=(0,\dots,0)$ (cf.\ Figure~\ref{fig:Qd1_1}). We will see that once the coordinates of the elements in $B$ are fixed, then the coordinates of all the remaining lattice points of $\Q^d$ can be determined uniquely.

\begin{figure}[htpb]
\centering
\begin{subfigure}[t]{.3\linewidth}
\centering
\includegraphics[width=.7\linewidth]{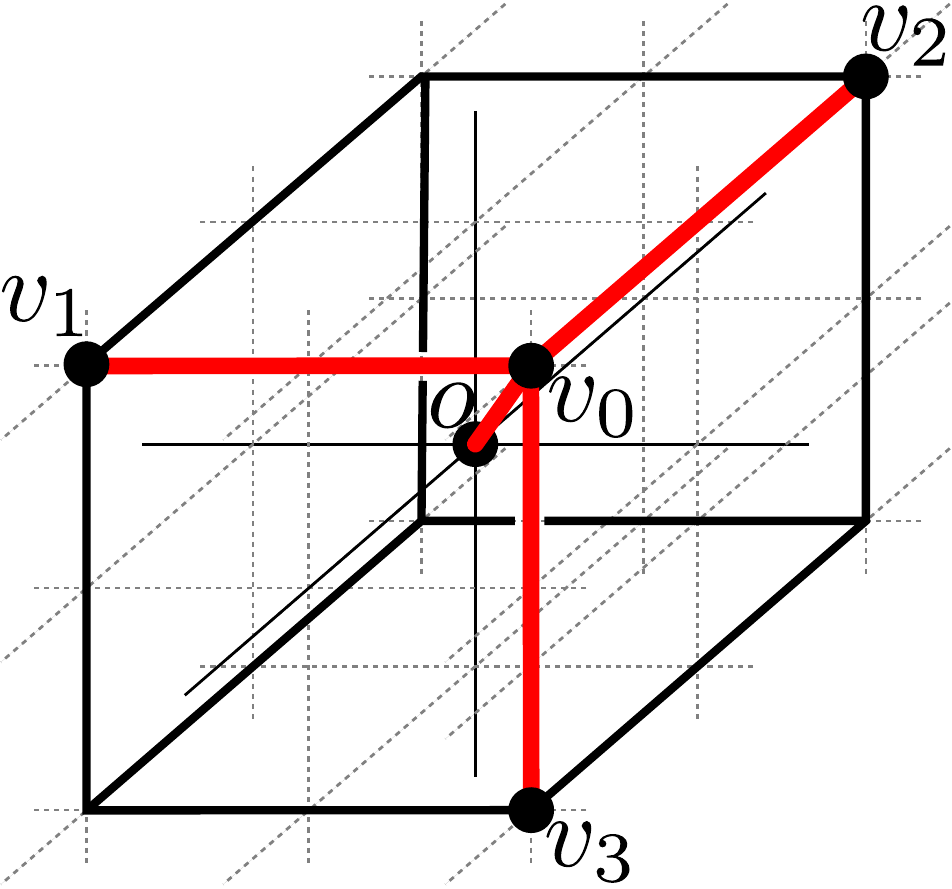}
\caption{$B$}\label{fig:Qd1_1}
\end{subfigure}\qquad
\begin{subfigure}[t]{.3\linewidth}
\centering
\includegraphics[width=.7\linewidth]{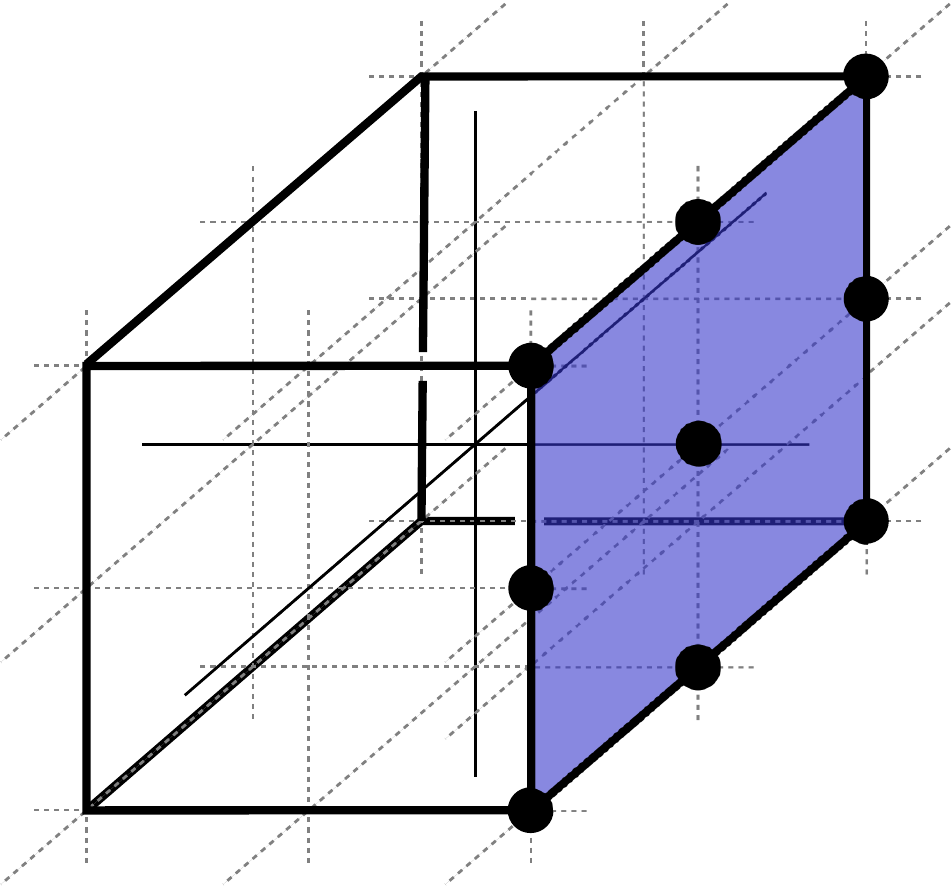}
\caption{$R_1$}\label{fig:Qd1_2}
\end{subfigure}\qquad
\begin{subfigure}[t]{.3\linewidth}
\centering
\includegraphics[width=.7\linewidth]{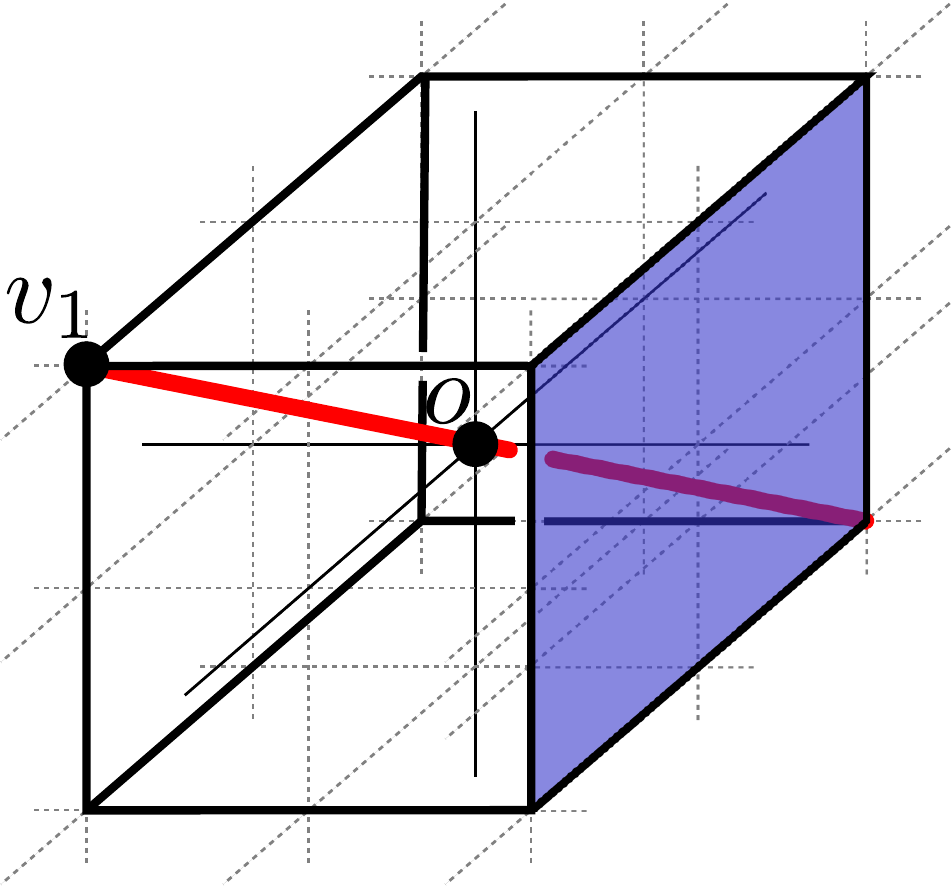}
\caption{$\aff\{o,v_1\}\cap\aff\{R_1\}$}\label{fig:Qd1_3}
\end{subfigure}\qquad
\caption{Scheme for showing that $\Q^d$ is projectively unique. (The picture displays $\Q^3$ for the sake of clarity, but the proof starts with $d\geq4$.)}\label{fig:Qd1}
\end{figure}

Consider the set of points of $\Q^d$ lying in a common facet of $[-1,1]^d$ that is incident to $v_0$; for example, $R_1:=\Q^d\cap\aff\{v_0,v_2,\dots,v_d\}$ (cf.\ Figure~\ref{fig:Qd1_2}). Observe that $R_1$ is just an affine embedding of $\Q^{d-1}$ into~$\R^d$. As such, $R_1$ is projectively unique, and thus it is determined uniquely if a projective basis for its span is fixed.

Clearly, the points $v_0,v_2,\dots,v_d$ of $B$ form an affine basis for the affine span of $R_1$. Furthermore, the coordinates of the point $w=(+1,-1,\dots,-1)$, are fixed by $B$. Indeed, $w$ is the the point of intersection of the line $\aff\{o,v_1\}$ with the hyperplane $\aff\{R_1\}$ (cf.\ Figure~\ref{fig:Qd1_3}). To sum up, we have that
\begin{compactitem}[$\circ$]
\item the points $v_0,v_2,\dots,v_d, w$ are determined uniquely from the points of $B$,
\item the points $v_0,v_2,\dots,v_d, w$ are elements of $R_1$, and
\item the points $v_0,v_2,\dots,v_d, w$ form a projective basis for the span of $R_1$. 
\end{compactitem}
Consequently, $R_1\cup B$ is a projectively unique point configuration, since $B$ is projectively unique. We can repeat this argumentation for all point configurations \[R_i:=\Q^d\cap\aff\left(\{v_0,v_1, v_2,\dots,v_d\} \setminus \{v_i\}\right),\ i\in \{1,\dots, d\}.\] In particular, the point configuration \[\widetilde{\Q}^d=B\cup\bigcup_{i\in \{1,\dots, d\}} R_i\]
is projectively unique. Moreover, since the last vertex of a cube of dimension $d\geq 3$ is determined by the remaining ones by (cf.\ \cite[Lem~3.4]{AZ12}, compare also Example~\ref{ex:stdet}(iv)), the configuration $\widetilde{\Q}^d\cup \{-v_0\}$ is projectively unique as well. By symmetry,

\[-\widetilde{\Q}^d\cup \{v_0\}=-B\cup\bigcup_{i\in \{1,\dots, d\}} -R_i \cup \{v_0\}\]

\noindent is also projectively unique.
Clearly, $\widetilde{\Q}^d\cup \{-v_0\}$ and $-\widetilde{\Q}^d\cup \{v_0\}$ intersect along a projective basis: for instance, the set $B$ lies in both $-\widetilde{\Q}^d\cup \{v_0\}$ and $\widetilde{\Q}^d\cup \{-v_0\}$ and forms a projective basis as desired. Thus, the point configuration $\Q^d=\widetilde{\Q}^d\cup \{-v_0\} \cup -\widetilde{\Q}^d\cup \{v_0\}$ is projectively unique.
\end{proof}

\paragraph*{Embedding vertex sets of algebraic polytopes}

We start with a point configuration $\PROJ{\vect p}$ that shows that it is enough to fix each coordinate of a point to frame it. 

\begin{lemma}\label{lem:cube}
For each point $\vect p$ in the positive orthant $\R^d_+$ of $\R^d,\ d\ge 3$, there is a point configuration $\PROJ{\vect p}\in \R^d$ that contains $\vect p=(p_1,\dots,p_d)$ and is framed by the points \[L(\vect p):=\left\{\veczero,p_1 \ve_1,\dots,p_d \ve_d, \frac{p_1\ve_1}{2} ,\dots,\frac{p_d\ve_d}{2} \right\}.\]
\end{lemma}

\begin{proof}
We denote by $\Q^d+\mathbf{1}$ the translation of $\Q^d$ by the all-ones vector. Moreover, let $\mathrm{D}=\mathrm{D}[p_1,\dots,p_d]$ denote the diagonal matrix with diagonal entries $p_1,\dots,p_d$. Notice that $\PROJ{\vect p}:=\frac{\mathrm{D}}{2}(\Q^d+\mathbf{1})$ is projectively unique (by Proposition~\ref{prp:Qdm}), contains $\vect p$ and the set $L(\vect p)$. Since every projective transformation fixing $L(\vect p)$ is the identity, the subset $L(\vect p)$ frames $\PROJ{\vect p}$.
\end{proof}

Finally, we only need to frame each coordinate of the point. The idea is to realize the defining polynomial of any real algebraic number in a \Defn{functional arrangement} (cf.\ \cite[Def.\ 9.6]{Kapovich}), which conversely determines the real algebraic number. 

\begin{definition}\label{def:func}
For a function $f:\R^k\mapsto \R$, a \Defn{functional arrangement} $\FUNC{f}=\FUNC{f}(\vect x)$ for $f$ is a $k$-parameter family of point configurations in $\R^2$ such that the following conditions hold:
\begin{compactenum}[\rm (i)]
\item For all $\vect x=(x_1,\dots, x_k)$ in $\R^k$, the functional arrangement $\FUNC{f}(\vect x)$ contains the \Defn{output point} $f(\vect x) e_1$, the \Defn{input points} $x_i e_1,\ i\in\{1,\dots,k\},$ and the set $\Q^2+\mathbf{1}$. 
\item For all $\vect x\in \R^k$, the set $\{x_i e_1: i\in\{1,\dots,k\}\} \cup (\Q^2+\mathbf{1})$ frames the point configuration $\FUNC{f}(\vect x)$.
\end{compactenum}
For the last condition, let $\varphi(\FUNC{f})(x)$ denote any point configuration Lawrence equivalent the functional arrangement $\FUNC{f}(x)$, where $\varphi$ is the bijection of points that induces the Lawrence equivalence.
\begin{compactenum}[\rm (i)]
\setcounter{enumi}{2}
\item For all $\vect x\in \R^k$, if $\varphi$ is the identity on $\Q^2+\mathbf{1}$ and $\varphi(f(x)e_1)=f(x)e_1$, we have $\varphi(xe_1)\in f^{-1}(f(x))e_1 \subset \R^{2\times k}$.
\end{compactenum}
\end{definition}

Hence, a functional arrangement essentially computes a function and its inverse by means of its point-line incidences alone. An just as like functions, they can be combined:

\begin{lemma}\label{lem:composition}
Let $\mathrm{F}(x,z)$ and $\mathrm{G}(y,z)$, $x\in \R^k, y\in \R^\ell, z\in \R^m$, denote functional arrangements for functions $f:\R^{k+m}\mapsto \R$ and $g:\R^{\ell+m}\mapsto \R$, respectively. Then \[\mathrm{F}(g(y,z),x',z)\cup \mathrm{G}(y,z),\quad x':=(x_2,\dots,x_k)\]
is a functional arrangement for the function $f(g(y,z),x',z)$ from  $\R^{k+m+\ell-1}$ to $\R$.
\end{lemma}

\begin{prp}[cf.\ \cite{Staudt}, {\cite[Thm.~D]{Kapovich}}]\label{prp:VonStaudt}
Every integer coefficient polynomial $\psi$ is realized by a functional arrangement $\FUNC{\psi}$.  
\end{prp}

\begin{proof}
The proof of this fact is based on the classical von Staudt constructions (\cite{Staudt}, compare also \cite[Ch.~5]{Richter-Gebert2011}, \cite[Sec.~11.7]{RG} or \cite[Sec.~5]{Kapovich}), which are a standard tool to encode algebraic operations in point-line incidences. 

To construct the desired functional arrangements, notice that every integer coefficient polynomial in variable $x$ can be written using $0$, $1$ and $x$, combined by addition and multiplication. Hence, thanks to Lemma~\ref{lem:composition}, it suffices to provide:
\begin{compactitem}[$\circ$]
\item A functional arrangement $\mathrm{ADD}(\alpha,\beta)$ for the function $\mathrm{a}(\alpha,\beta)=\alpha+\beta$ computing the addition of two real numbers.
\item A functional arrangement $\mathrm{MLT}(\alpha,\beta)$ for the function $\mathrm{m}(\alpha,\beta)=\alpha\cdot\beta$ computing the product of two real numbers.  
\end{compactitem}
Both functional configurations are shown below. We invite the reader to check that they indeed form functional arrangements for addition and multiplication.

\begin{figure}[htpb]
\centering
\begin{subfigure}[t]{.37\linewidth}
\centering
\includegraphics[width=\linewidth]{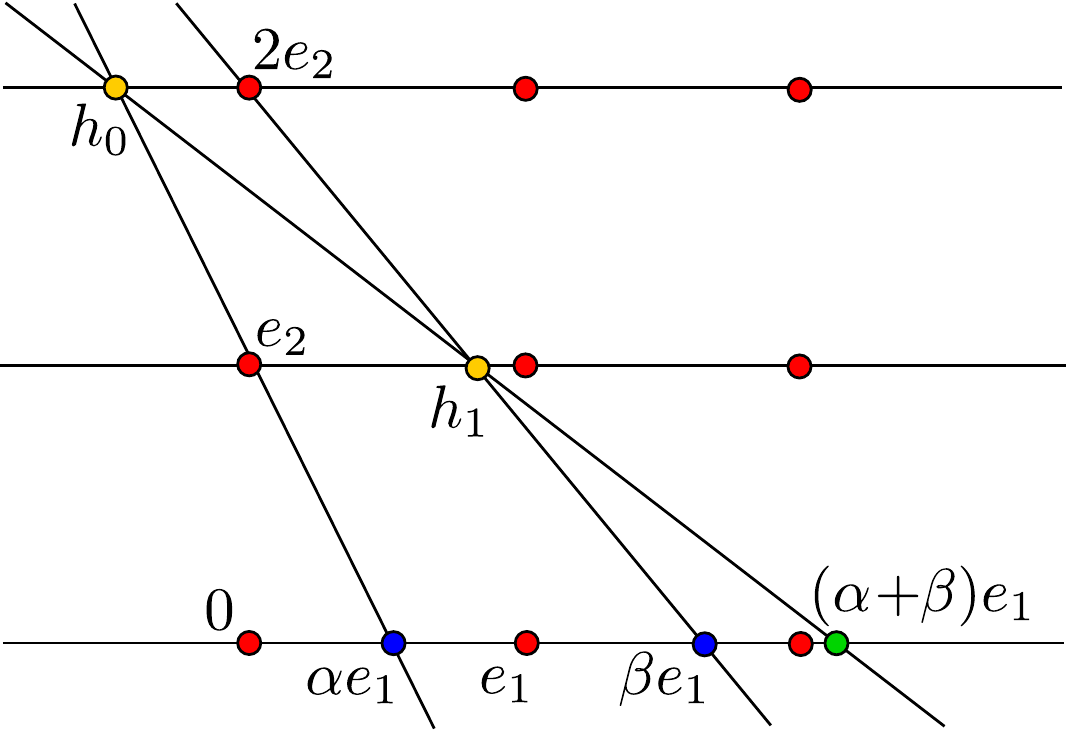}
\caption{$\mathrm{ADD}(\alpha,\beta)$}
\end{subfigure}\qquad\quad\qquad
\begin{subfigure}[t]{.37\linewidth}
\centering
\includegraphics[width=\linewidth]{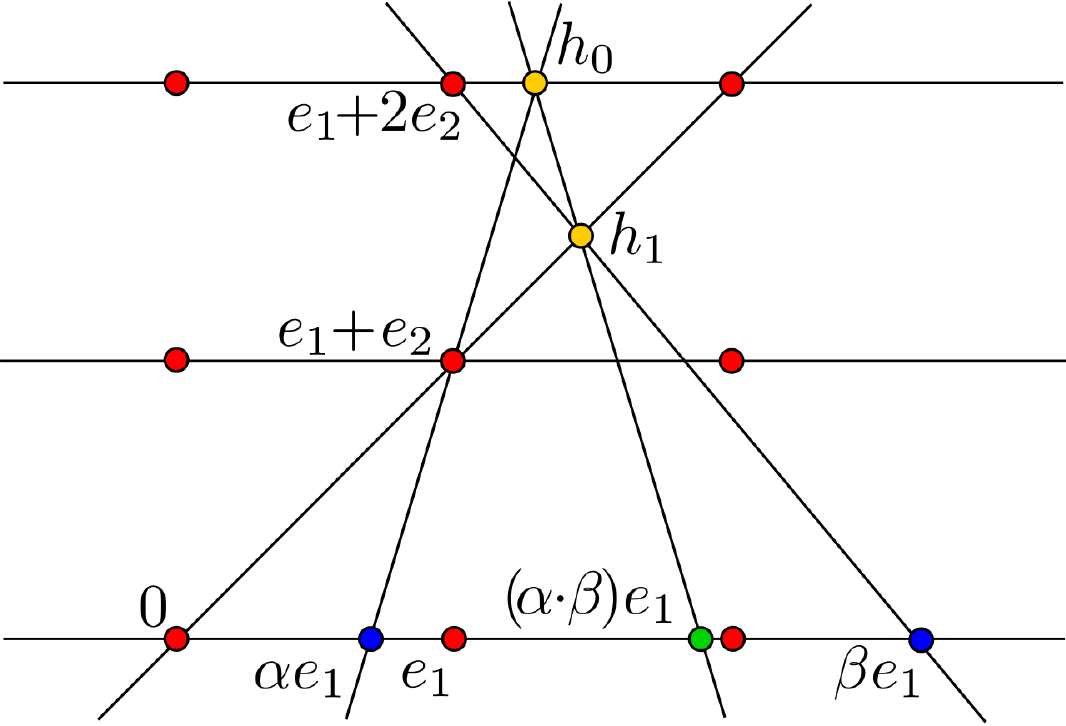}
\caption{$\mathrm{MLT}(\alpha,\beta)$ }
\end{subfigure}
\caption{Von Staudt constructions for addition, $\mathrm{ADD}(\alpha,\beta)$, and multiplication, $\mathrm{MLT}(\alpha,\beta)$. The blue points of the configuration form the input, the red points show $\Q^2+\mathbf{1}$, the yellow points are auxiliary to the construction and the green points give the output.}  \label{fig:VonStaudt}
\end{figure}

By switching output and input points of these functional arrangements, we also obtain functional arrangements $\mathrm{SUB}(\alpha,\beta)$ and $\mathrm{DIV}(\alpha,\beta)$ for $\mathrm{s}(\alpha,\beta)=\alpha-\beta$ and  $\mathrm{d}(\alpha,\beta)=\frac{\alpha}{\beta}$.
\end{proof}

\begin{example}
Let us construct a functional arrangement for $x\mapsto x^2-2=\mathrm{s}(\mathrm{m}(x,x),\mathrm{a}(1,1))$. Using Lemma~\ref{lem:composition}, this arrangement can be written as combination of the functional arrangements for addition, subtraction and multiplication:
\[
\FUNC{x^2-2}(x)=\FUNC{\mathrm{s}(\mathrm{m}(x_1,x_2),\mathrm{a}(x_3,x_4))}(x,x,1,1)= \mathrm{SUB}(x^2,2) \cup \mathrm{MLT}(x,x) \cup \mathrm{ADD}(1,1).
\]
Figure~\ref{fig:VonStaudtExample} shows the evaluations of this functional arrangement at $\sqrt{2}$ and $\sqrt{3}$.

\begin{figure}[h!tpf]
\centering
\begin{subfigure}[t]{.47\linewidth}
\centering
\includegraphics[width=\linewidth]{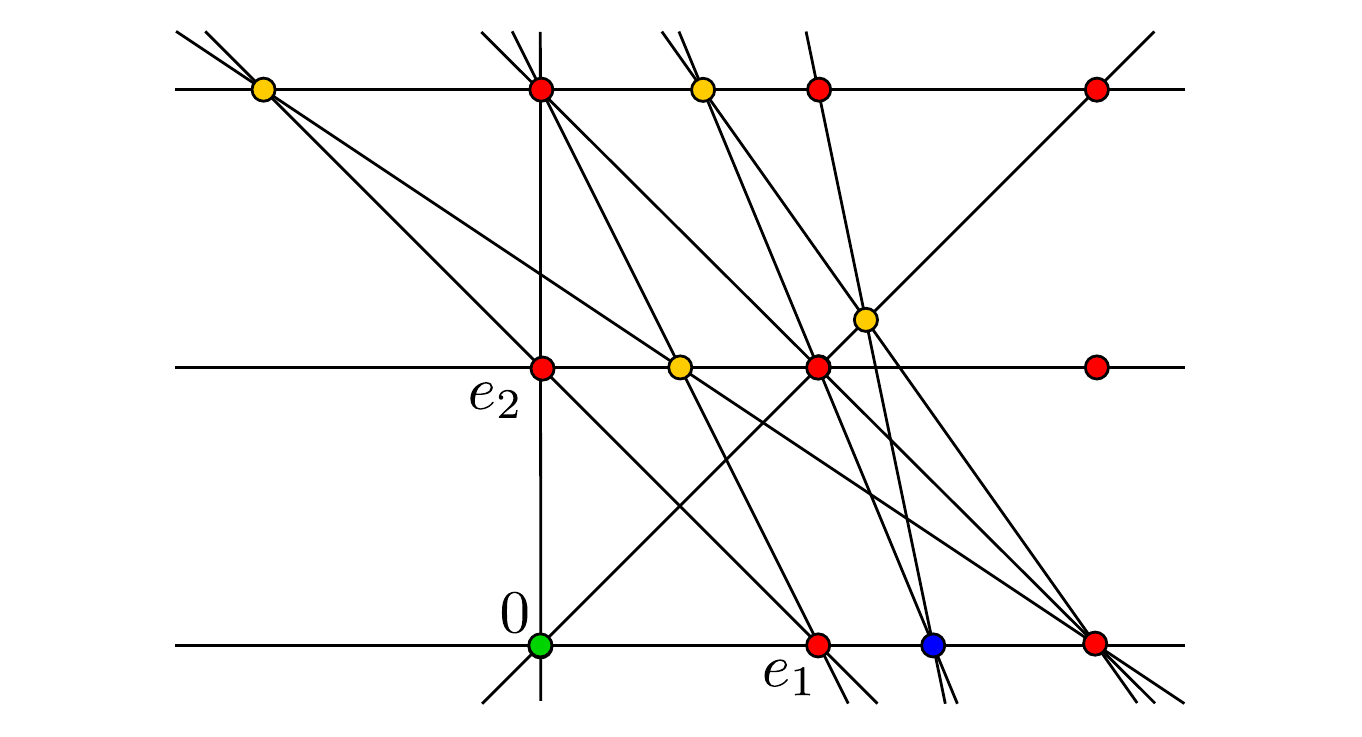}
\caption{$\FUNC{x^2-2}(\sqrt{2})$}\label{sfig:sqrt2}
\end{subfigure}\qquad
\begin{subfigure}[t]{.47\linewidth}
\centering
\includegraphics[width=\linewidth]{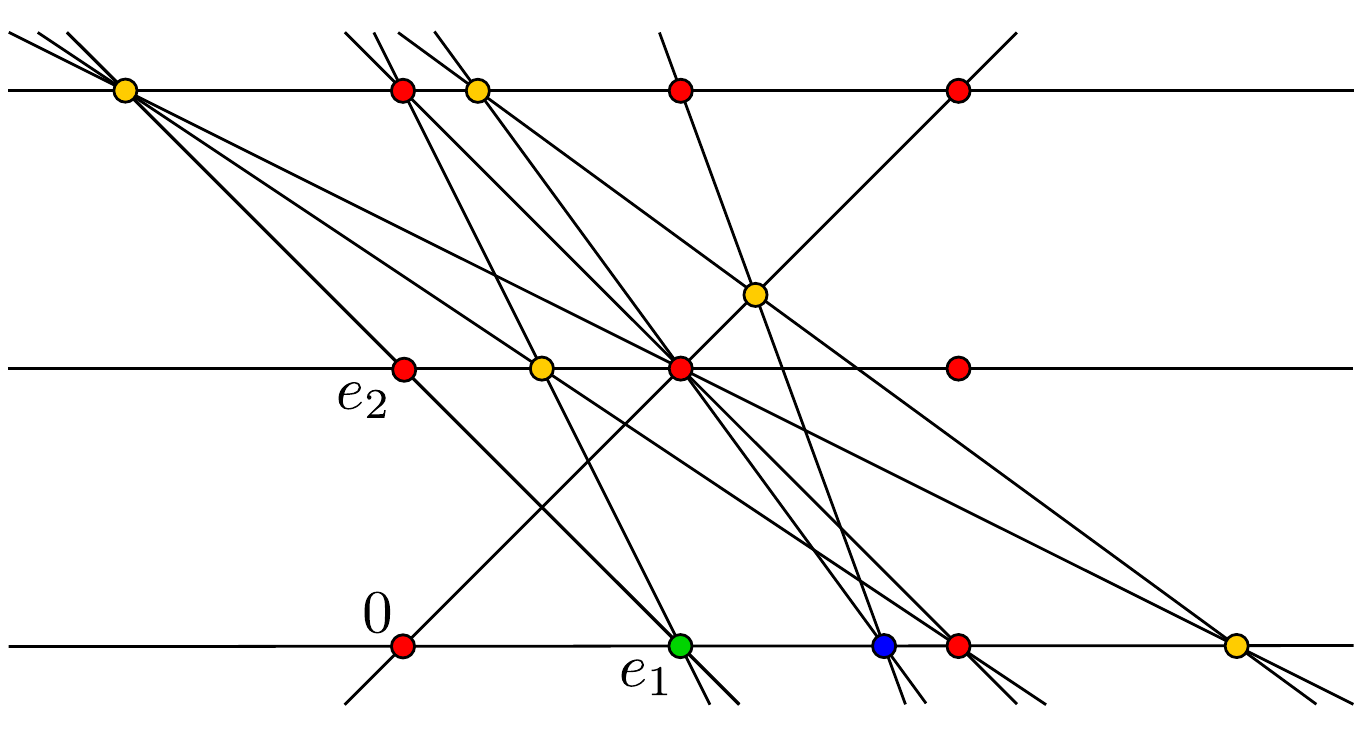}
\caption{$\FUNC{x^2-2}(\sqrt{3})$}
\end{subfigure}
\caption{Two evaluations of the functional arrangement $\FUNC{x^2-2}$.}\label{fig:VonStaudtExample}
\end{figure}

The point configuration $\FUNC{x^2-2}(\sqrt{2})$, as given above, is framed by $\Q^2+\mathbf{1}$. Hence, it enables us to compute $\sqrt{2}$ from $\Q^2+\mathbf{1}$. Similar point configurations for any algebraic number are given in the following corollary.

\end{example}

\begin{cor}\label{cor:vonS}
For each real algebraic number $\zeta$, there is a point configuration $\COOR{\zeta}\subset \R^2$ framed by $\Q^2+\mathbf{1}$ such that $\zeta \ve_1\in \COOR{\zeta}$.
\end{cor}

\begin{proof}
If $\zeta$ is algebraic of degree $\le 1$ (i.e.\ $\zeta$ is rational), let $p,q\in\Z$ be such that $\zeta=\frac{p}{q}$ and set $\psi(x)=qx-p$. We then define $\COOR{\zeta}:=\FUNC{\psi}(\zeta)$.

If $\zeta$ is instead of degree $\ge 2$ (i.e.\ $\zeta$ is irrational, but algebraic), and $\psi$ is an integer coefficient polynomial with root $\zeta$, then let $\zeta^+$ and $\zeta^-$ denote rational numbers with the property that $\zeta$ is the only root of $\psi$ in the interval $[\zeta^-, \zeta^+]$. 
Then the desired point configuration is given by $\COOR{\zeta}:=\COOR{\zeta^-} \cup \FUNC{\psi}(\zeta)\cup \COOR{\zeta^+}, $
which contains $\zeta$ by construction, and it is framed by $\Q^2+\mathbf{1}$.

 Indeed, let $\varphi(\COOR{\zeta})$ be a configuration Lawrence equivalent to $\COOR{\zeta}$, where the equivalence is induced by the bijection $\varphi$. If $\varphi$ is the identity on $\Q^2+\mathbf{1}$, then by Definition~\ref{def:func}(iii) and since ${\psi}(\zeta)e_1=0\in\Q^2+\mathbf{1}$, we obtain $\varphi(\zeta^-e_1)=\zeta^-e_1$, $\varphi(\zeta^+e_1)=\zeta^+e_1$ and $\varphi(\zeta e_1)=\zeta' e_1$. Here $\zeta'$ must be a root of $\psi$, which lies in the interval $[\zeta^-, \zeta^+]$ by Lawrence equivalence. This root is unique, so $\varphi(\zeta e_1)=\zeta e_1$. 
To sum up, $\Q^2+\mathbf{1}$ determines $\zeta e_1$, which together with $\Q^2+\mathbf{1}$ frames $\COOR{\zeta}$ by Definition~\ref{def:func}(ii). 
\end{proof}

\begin{cor}\label{cor:algpoint}
Let $\zeta$ be any point in $\R^d_+$, $d\ge 3$, with algebraic coordinates. Then there is a projectively unique point configuration $\COOR{\zeta}$ containing $\zeta$ and $\Q^d+\mathbf{1}$.
\end{cor}

\begin{proof}
Let  $\zeta=(\zeta_1, \dots, \zeta_d)$. For any $i$, let $\COOR{\zeta_i}_{i,i+1}$ resp.\ $\COOR{\nicefrac{\zeta_i}{2}}_{i,i+1}$ denote the configurations provided by Corollary~\ref{cor:vonS}, naturally embedded into the plane spanned by $e_i$ and $e_{i+1}$ (using a cyclic labelling for the base vectors $e_{\ast}$). We obtain in \[\COOR{\zeta}:=(\Q^d+\mathbf{1})\cup \PROJ{\zeta}\cup\bigcup_{i\in \{1,\dots,d\}}\COOR{\zeta_i}_{i,i+1}\cup\bigcup_{i\in \{1,\dots,d\}}\COOR{\nicefrac{\zeta_i}{2}}_{i,i+1}\]
the desired point configuration: $\COOR{\zeta}$ contains $\zeta$ and $\Q^d+\mathbf{1}$ by construction and it is projectively unique by Lemma~\ref{lem:union}.
\end{proof}

\paragraph*{Conclusion of proof}

\begin{proof}[{\bf Proof of Theorem~\ref{thm:ratprj}}]
Let $P$ denote a algebraic polytope in $\R^d$. We assume that $d\ge 3$, since if $d\leq 2$ we can realize $P$ as a face of some $3$-dimensional pyramid. By dilation and translation, we may assume that $P$ lies in the interior of the cube $\mathrm{C}:=\conv (\Q^d+\mathbf{1})$. Consider now the point configuration
\[\COOR{P}:=\bigcup_{v\in \F_0(P)} \COOR{v},\] \vskip -1.5mm
\noindent where $\COOR{v},\ v\in \F_0(P),$ is the point configuration provided by Corollary~\ref{cor:algpoint}. Set $Q:=\F_0(P)$ and $R:=\COOR{P}\setminus Q$. With this, we have that $(P,Q,R)$ is a weak projective triple. Indeed,

\begin{compactenum}[(1)]
\item $Q\cup R=\COOR{P}$ is a projectively unique point configuration by Corollary~\ref{cor:algpoint} and Lemma~\ref{lem:union},
\item $Q$ obviously frames $P$ (cf.\ Example~\ref{ex:stdet}(i)), and
\item since $P\subset \mathrm{int}\, \mathrm{C}$, we have $\F_0(\mathrm{C})\subset R$, and hence any of the facet hyperplanes of $\mathrm{C}$ can be chosen as wedge hyperplane for the triple.\end{compactenum}
Thus, by Corollary~\ref{cor:wpt}, there exists a projectively unique polytope that contains a face projectively equivalent to $P$.
\end{proof}

\section{Subpolytopes of stacked polytopes}

In this section, we disprove Shephard's conjecture. While this alone could be done using Theorem~\ref{thm:ratprj} (with arguments slightly differing from those below), we here use a refined argumentation to construct combinatorial types of $5$-dimensional polytopes that are not realizable as subpolytopes of $5$-dimensional stacked polytopes (Theorem~\ref{thm:proj}). 
Instead of Theorem~\ref{thm:ratprj}, we will use the following result of Below.

\begin{theorem}[{\cite[Ch.~5]{Below2002}}, see also {\cite[Thm.~4.1]{Dobbins2011}}]\label{thm:Below}
Let $P$ be any algebraic $d$-dimensional polytope. Then there is a polytope $\widehat{P}$ of dimension $d+2$ that contains a face $F$ that is projectively equivalent to $P$ in every realization of $\widehat{P}$.
\end{theorem}

The remainder of this section is concerned with the proof of Theorem~\ref{thm:proj}. For the convenience of the reader, we retrace, in a higher generality and with improved quantitative bounds, Shephard's ideas that lead him to discover $3$-dimensional polytopes that are not subpolytopes of stacked polytopes \cite{Shephard74}. We recall some notions.

\begin{definition}
A polytope is \Defn{$\kF{k}$-stacked} if it is the connected sum (cf.\ \cite[Sec.\ 3.2]{RG}) of polytopes with at most $k$ facets each. With this, a $\kF{(d+1)}$-stacked $d$-dimensional polytope is simply a \Defn{stacked} polytope.
\end{definition}
\vskip -2.4mm
\begin{figure}[htpb]
\centering
\begin{subfigure}[t]{.45\linewidth}
\centering
\includegraphics[width=.45\linewidth]{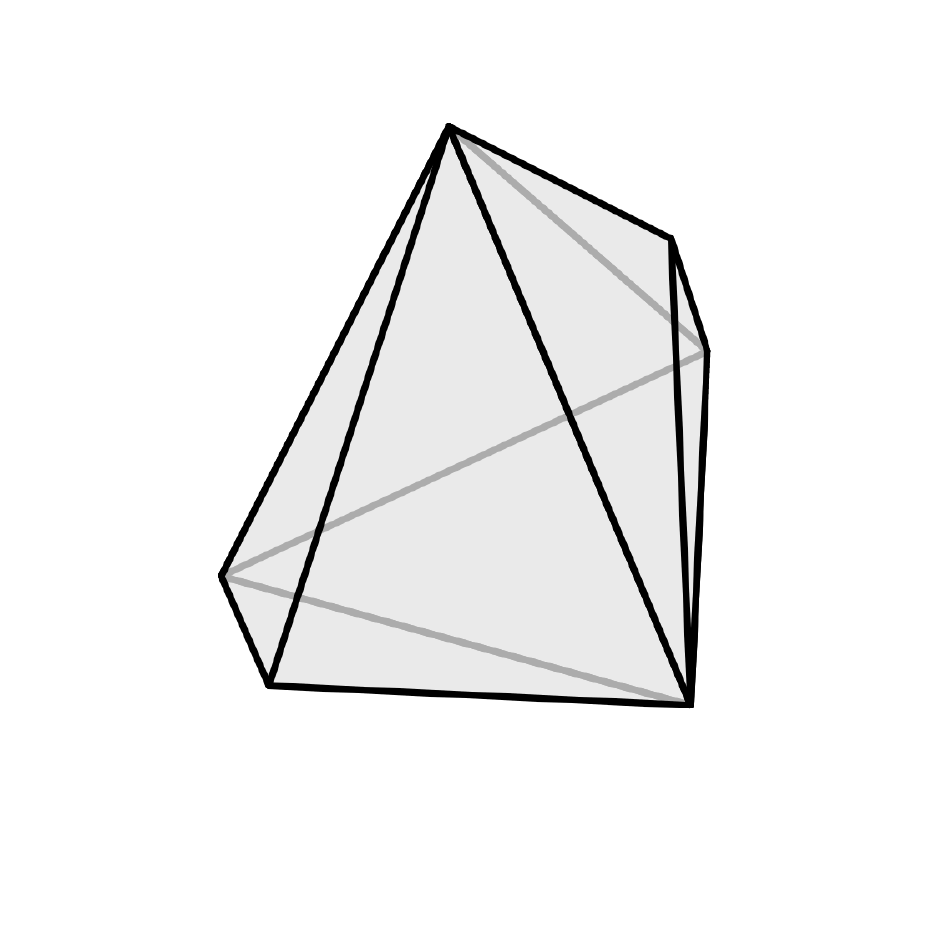}\qquad
\includegraphics[width=.45\linewidth]{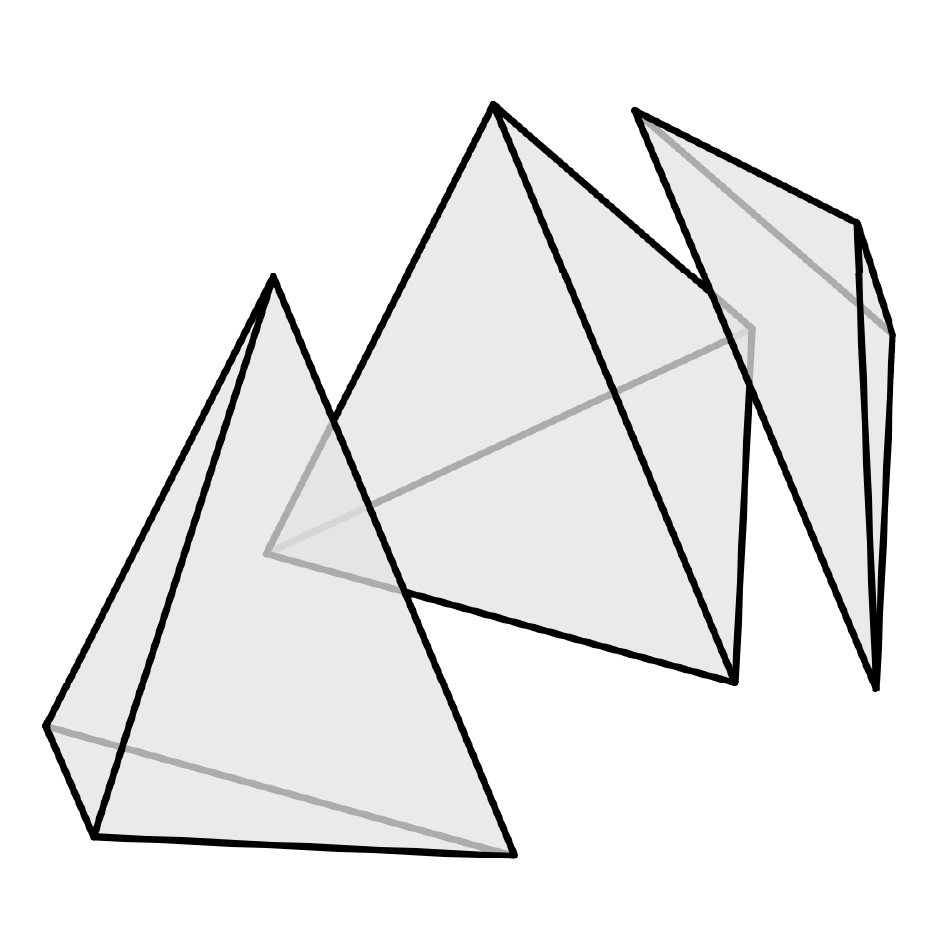}
\caption{$\kF{4}$-stacked}\label{fig:4stacked}
\end{subfigure}\qquad\qquad
\begin{subfigure}[t]{.45\linewidth}
\centering
\includegraphics[width=.45\linewidth]{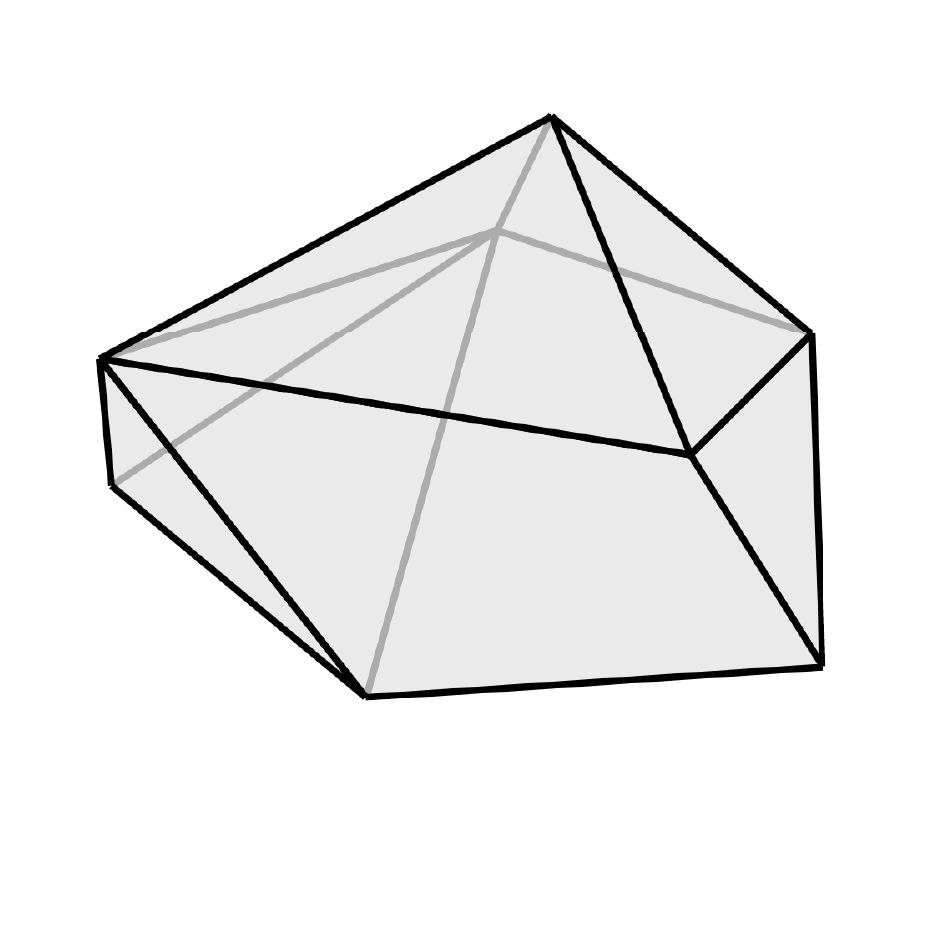}\qquad
\includegraphics[width=.45\linewidth]{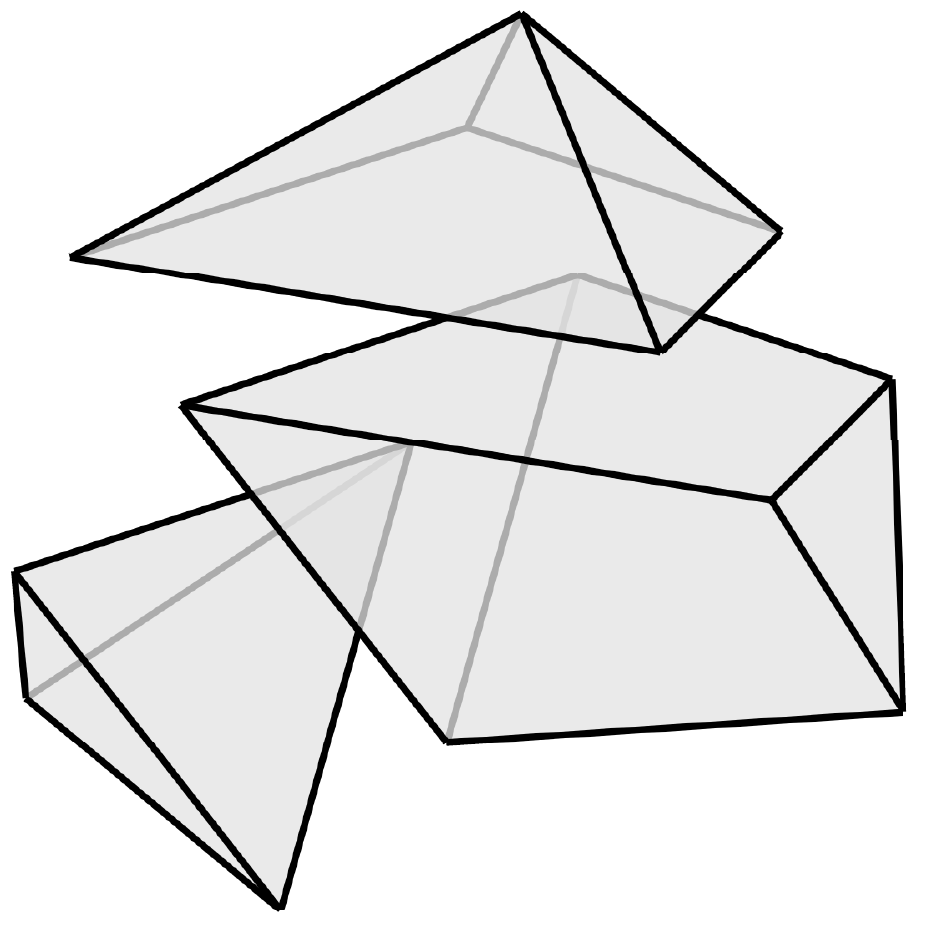}
\caption{$\kF{5}$-stacked}\label{fig:5stacked}
\end{subfigure}
\caption{A $\kF{4}$-stacked $3$-polytope (i.e.\ a stacked $3$-polytope) and a $\kF{5}$-stacked $3$-polytope.}\label{fig:stacked}
\end{figure}

Let $\dH(\cdot,\cdot)$ denote the Hausdorff distance between compact convex subsets of $\R^d$, cf.\ \cite[Sec.\ 1.8]{Schneider}. Let $B_r(x)$ denote the metric ball in $\R^d$ with center $x$ and radius $r$. Shephard's main observations are:

\begin{lemma}[cf.\ {\cite[$\mathbf{2.}$(i) {\&} (ii)]{Shephard74}}]\label{lem:dist}
Let $P$ be any polytope with $\dH(P, B_1(0))\le \e\le\frac{1}{36}$. Then, for any polytope $P'$ with $\F_0(P)\subset \F_0(P')$, we have $\dH(P', B_1(0))\le 6\sqrt{\e}$. 
\end{lemma}
Equivalently, if $Q'$ is any polytope with $\dH(Q', B_1(0))> \delta$, $0<\delta<1$, then $\dH(Q, B_1(0))>\frac{\delta^2}{36}$ for any subpolytope $Q$ of $Q'$.
\begin{proof}[Sketch of Proof]
By assumption, we have $B_{1-\e}(0)\subset P \subset  B_{1+\e}(0)$. Hence, every face of $P$ can be enclosed in some ball of radius $\sqrt{(1+\e)^2-(1-\e)^2}=2\sqrt{\e}$ and so every point in $\partial P$ is at distance at most $2\sqrt{\e}$ from some vertex of~$P$. 
Let now $v$ be any vertex of $P'$ not in $P$, and let $v_P$ denote the point of intersection of the line segment $\conv \{0,v\}$ with $\partial P$. Then $ \conv (\{v\}\cup  B_{1-\e}(0))\subset \conv (\{v\}\cup P)$, and since $\conv (\{v\}\cup P)$ contains no vertex of $P$ in the interior, we have \[\conv B_{2\sqrt{\e}}(v_P) \not\subset \conv (\{v\}\cup  B_{1-\e}(0)).\]
If $2\sqrt{\e}<1-\e$, this can be used to estimate the euclidean norm of $v$ as $||v||_2\le \frac{(1+\e)(1-\e)}{1-\e-2\sqrt{\e}}$, and hence if $\sqrt{\e}\le\nicefrac{1}{6}$, then $||v||_2\le1+6\sqrt{\e}$. This gives $B_{1-\e}(0)\subset P' \subset B_{1+6\sqrt{\e}}(0)$, or $\dH(P', B_1(0))\le6\sqrt{\e}$.
\end{proof}

\begin{lemma}[cf.\ {\cite[$\mathbf{2.}$(iii) {\&} (iv)]{Shephard74}}]\label{lem:dist2} 
For any $\kF{k}$-stacked $d$-dimensional polytope $S$, $d\ge 3$, we have \[\dH(S,B_1(0))\ge 2^{-2k-4}.\]
\end{lemma}

\begin{proof}[Sketch of Proof]
Assume $\dH(S,B_1(0))\le (1-\nicefrac{1}{\sqrt 2})$. As observed in the proof Lemma~\ref{lem:dist}, the edges of $S$ have length at most $4\sqrt{\dH(S,B_1(0))}$. 

Now, the polytope $S$ can be written as the connected sum of polytopes $S_i$, each of which has at most $k$ facets. Let $\varSigma$ be any one of the $S_i$ that contains the origin. We claim that $\varSigma$ has an edge of length at least $2^{-k}$. Indeed, since $\varSigma$ contains the origin, is has two vertices $v,w$ that enclose an angle at least $\nicefrac{\pi}{2}$ with respect to the origin. Furthermore, $\F_0(\varSigma)\cap B_{\nicefrac{1}{\sqrt 2}}(0)\subset \F_0(S)\cap B_{\nicefrac{1}{\sqrt 2}}(0)=\emptyset$, so these two vertices are at least at distance $1$ from each other. Since the graph of each polytope is connected, there must be a path of edges in $\varSigma$ from $v$ to $w$, and so one of these edges must be of length $(f_1(\varSigma))^{-1}$ or more. Finally, Sperner's Theorem shows that a polytope with $k$ facets has at most $\binom{k}{[\nicefrac{k}{2}]}\le2^k$ edges, so that $(f_1(\varSigma))^{-1}\ge 2^{-k}$, which gives the desired bound.
To combine the two observations, notice that since $d\ge 3$, all edges of $\varSigma$ are edges of $S$, so 
\[4\sqrt{\dH(S,B_1(0))}\ge2^{-k}\Longrightarrow \dH(S,B_1(0))\ge2^{-2k-4}.\qedhere\]
\end{proof}
 Combining the two lemmas above, we recover Shephard's main result.

\begin{cor}[cf.\ {\cite{Shephard74}}]\label{cor:sh} 
For any subpolytope $P$ of a $\kF{k}$-stacked $d$-dimensional polytope, $d\ge 3$,  we have \[\dH(P,B_1(0))\ge 2^{-4k-10}\cdot3^{-2}.\]
\end{cor}

In particular, every $d$-polytope that approximates $B_1(0)$ closely is not the subpolytope of any stacked polytope. We now only need to add a simple observation to Shephard's ideas:

\begin{prp}\label{prp:ratsub}
For $P$ is a subpolytope of a $\kF{k}$-stacked polytope $S$, then any face $\sigma$ of $P$ is a subpolytope of a $\kF{k}$-stacked polytope as well.
\end{prp}

\begin{proof}
It suffices to prove this in the case where $\sigma$ is a facet of $P$. Let $H$ denote the hyperplane spanned by $\sigma$. Recall that $S$ is obtained as the connected sum of polytopes $S_1,\dots, S_n$, and so $H\cap S$ is the connected sum of the polytopes $H\cap S_1,\dots, H\cap S_n$. Now every single one of the polytopes $S_i$ has most $k$ facets, and every facet of $H\cap S_i$ is obtained as the intersection of a facet of $S_i$ with $H$, so $H\cap S$ is $\kF{k}$-stacked. Observing that $\sigma=H\cap P$ is a subpolytope of $H\cap S$ finishes the proof.
\end{proof}

\paragraph*{Conclusion of proof}
\begin{proof}[{\bf Proof of Theorem~\ref{thm:proj}}]
Let $P$ be any $3$-dimensional polytope with $ \dH(P,B_1(0))<2^{-4\cdot6-10}\cdot 3^{-2}$. By Corollary~\ref{cor:sh}, $P$ is not a subpolytope of any $\kF{6}$-stacked polytope, and the same holds for any polytope projectively equivalent to $P$. 

Theorem~\ref{thm:proj} now provides a polytope $\widehat{P}$ of dimension $5$ that contains a face $F$ that is projectively equivalent to $P$ in every realization of $\widehat{P}$. Assume now that some polytope $O$ combinatorially equivalent to $\widehat{P}$ is a subpolytope of some stacked polytope. By Proposition~\ref{prp:ratsub}, any face of $O$ is a subpolytope of some $\kF{6}$-stacked polytope. But the face of $O$ corresponding to $F$ is projectively equivalent to $P$, and hence not obtained by deleting vertices of a $\kF{6}$-stacked polytope. A contradiction.
\end{proof}

\setstretch{1.09}

\setcounter{thmmain}{0}
\setcounter{figure}{0}

\chapter{Many polytopes with small realization space}\label{ch:lowdim}
\section{Introduction}

Legendre initiated the study of the spaces of geometric realizations of polytopes, motivated by problems in mechanics. One of the questions studied in his 1794 monograph on geometry~\cite{Legendre} is:
\begin{quote}
\emph{How many variables are needed to determine a geometric realization of a given (combinatorial type~of) polytope?}
\end{quote}
In other words, Legendre asks for the dimension of the \Defn{realization space} $\cR (P)$ of a given polytope $P$, that is, the space of coordinatizations for the particular combinatorial type of polytope. For $2$-polytopes, it is not hard to see that the number of variables needed is given by~$f_0(P)+f_1(P)$. Using Euler's formula, Legendre concluded that for $3$-polytopes the number of variables needed to determine the polytope up to congruence equals the number of its edges 
\cite[Note~VIII]{Legendre}. While his argument made use of some tacit assumptions, Legendre's reasoning was later confirmed by Steinitz who supplied the first full proof 
(cf.~\cite[Sec.\ 34]{steinenc}~\cite[Sec.\ 69]{stein}) of what we here call the Legendre--Steinitz formula:
    
\begin{theorem}[Legendre--Steinitz Formula 
	{\cite[Note VIII]{Legendre}} {\cite[Sec.\ 34]{steinenc}}]\label{thm:legst} For any $3$-polytope $P$, the realization space has dimension $f_0(P)+f_2(P)+4=f_1(P)+6$.
\end{theorem}

\medskip

\noindent {\large \textbf{Problems}}

\smallskip

In this chapter we treat two questions concerning the spaces of geometric realizations of 
higher-dimensional polytopes. The first problem originates with Perles and Shephard~\cite{PerlesShephard}:

\begin{problem}[Perles \& Shephard~\cite{PerlesShephard}, Kalai~\cite{Kalai}]\label{prb:projun}
Is it true that, for each fixed $d\geq 2$, the number of distinct combinatorial types of projectively unique $d$-polytopes is finite?
\end{problem}

For the case of $d=4$, McMullen and Shephard made a bolder conjecture, referring to a list of $11$ projectively unique $4$-polytopes constructed by Shephard in the sixties (see Appendix~\ref{ssc:Shphrdlist}).

\begin{conjecture}[McMullen \& Shephard~\cite{McMullen}]\label{conj:mcmsh}
Every projectively unique $4$-polytope is accounted for in Shephard's list of $11$ combinatorial types of
projectively unique $4$-polytopes.
\end{conjecture}

The connection of Problem~\ref{prb:projun} to the study of the dimension of $\cR(P)$ is this: A $d$-polytope $P$ can be projectively unique only if $\dim \cR (P)$ is smaller or equal to the dimension of the projective linear group $\mathrm{PGL}(\R^{d+1})$ of projective transformations on $\R^d$. In particular, we obtain:

\begin{compactitem}[$\circ$]
\item A $2$-polytope can be projectively unique only if $f_0(P)+f_1(P)\leq 8= \dim \mathrm{PGL}(\R^3)$. 
\item A $3$-polytope can be projectively unique only if $f_1(P)+6\leq 15= \dim \mathrm{PGL}(\R^4)$. 
\end{compactitem}

\noindent A more careful analysis reveals that this is a complete characterization of projectively unique polytopes in dimensions up to $3$: A $2$-polytope is projectively unique if and only it has $3$ or $4$ vertices;
a $3$-polytope is projectively unique if and only if it has at most $9$ edges~\cite[Sec.\ 4.8, Pr.\ 30]{Grunbaum}.

Projectively unique polytopes in dimensions higher than $3$ are far from understood. There has been substantial progress in the understanding of realization spaces of polytopes up to “stable equivalence” (thus, in particular, up to homotopy equivalence), due to the work of Mn\"ev~\cite{Mnev} and Richter-Gebert~\cite{RG}.
Nevertheless, no substantial progress was made on the problem of Perles and Shephard since it was asked in the sixties (see~\cite{PerlesShephard}). Related results on projectively unique polytopes include:

\begin{compactitem}[$\circ$]
\item Any $d$-polytope with at most $d+2$ vertices is projectively unique.
\item There are projectively unique $d$-polytopes with exponentially many vertices (w.r.t.\ $d$)
\cite{McMullen},~\cite{PerlesShephard}.
\item Perles proved that there are projectively unique polytopes that, while realizable in $\R^8$, are not realizable such that all vertices have rational coordinates~\cite[Sec.\ 5.5, Thm.\ 4]{Grunbaum}. 
\end{compactitem}

\smallskip

\noindent 
The above discussion for the low-dimensional cases of Problem~\ref{prb:projun} motivates one to look for bounds on the parameter $\dim \cR (P)$ for $d$-dimensional polytopes $P$ as a step towards the problem of projectively unique polytopes. 
Does the Legendre--Steinitz formula have a 
high-dimensional analogue? If the \Defn{size} of a polytope is defined as the
dimension times the total number of its vertices and facets, 
\[
\size(P)\ := \ d\big(f_0(P)+f_{d-1}(P)\big),
\]
this problem can be made more concrete as follows.

\begin{problem}[Legendre--Steinitz in general dimensions~\cite{ZA}]\label{prb:steinitz}
How does, for $d$-dimensional polytopes, the dimension of the realization space grow with the size of the polytope?
\end{problem}

We have $\dim\cR(P)=\frac12\size(P)$ for $d=2$ and $\dim\cR(P)=\frac13\size(P)+4$ for $d=3$, so in both cases
the dimension of the realization space grows linearly with the size of the polytope.
In contrast, it is known that for $d\ge 4$ the $f$-vector of a $d$-polytope $P$ does not suffice to 
determine $\dim \cR(P)$. One cannot even determine from the $f$-vector whether a polytope is projectively unique 
(cf.~\cite[Sec.\ 4.8, Pr.\ 30]{Grunbaum}). Thus Problem~\ref{prb:steinitz} does not ask for a formula 
for $\dim \cR(\cdot)$ in terms of the size 
of~$P$, but rather for good upper and lower bounds. 
Upper bounds are given by
\[
\dim \cR(P)\leq d\,f_0(P)\qquad \text{and}\qquad \dim \cR(P)\leq d\,f_{d-1}(P);
\]
these bounds are sharp for simplicial resp.\ simple polytopes. In particular,
we always have
\[
\dim \cR(P)\le \size(P).
\]
The quest for lower bounds, however, relates Problem~\ref{prb:steinitz} to the McMullen--Shephard Conjecture and Problem~\ref{prb:projun}, and is apparently a much harder problem about which little is known. Still,   
$4$-polytopes for which the dimension of the realization space grows sublinearly with the size are known:  
In~\cite{ZA}, we argued that for the \Defn{neighborly cubical polytopes} $\operatorname{NCP}_4[n]$ 
constructed by Joswig \& Ziegler~\cite{JZ} the dimension of the realization spaces are low relative to their size: 
\[
	\dim \cR (\operatorname{NCP}_4[n]) \sim 
	(\log\size \operatorname{NCP}_4[n])^2.\]
	
\smallskip

\noindent {\large \textbf{Main Results}}

\smallskip

In this chapter, we answer Problem~\ref{prb:steinitz} by showing
that, in all dimensions $d\ge 4$, the trivial lower bound 
\[
	\dim \cR(P)\geq d(d+1)
\]
is asymptotically optimal: There exists, in every dimension $d\geq 4$, an infinite family of 
combinatorially distinct $d$-polytopes for which the dimension of the realization space is a 
constant that depends on $d$ --- indeed, we can bound this by $d(d+1)$ plus an absolute constant.

Furthermore, we solve Problem~\ref{prb:projun} in the negative for high-dimensional polytopes.

\begin{thmmain}[Cross-bedding cubical torus polytopes]\label{mthm:Lowdim}
There exists an infinite family  of combinatorially distinct $4$-dimensional polytopes $\operatorname{CCTP}_4[n]$ with $12(n+1)$ vertices such that $\dim \cR (\operatorname{CCTP}_4[n])\leq96$ for all $n\geq 1$.
\end{thmmain}

By considering iterated pyramids over the polytopes $\operatorname{CCTP}_4[n]$, we obtain as an immediate corollary:

\begin{mcor}\label{cor:lowdim}
For each $d\geq 4$, there exists an infinite family  of combinatorially distinct $d$-dimensional polytopes $\operatorname{CCTP}_d[n]$ with $12(n+1)+d-4$ vertices such that $\dim \cR (\operatorname{CCTP}_d[n])\leq76+d(d+1)$ for all~$n\geq 1$.
\end{mcor}

Not only does this settle Problem~\ref{prb:steinitz}, it also provides strong evidence that Conjecture~\ref{conj:mcmsh} is wrong. Building on the proof of Theorem~\ref{mthm:Lowdim}, we then obtain:

\begin{thmmain}[Projectively unique cross-bedding cubical torus polytopes]\label{mthm:projun}
There exists an infinite family of combinatorially distinct $69$-dimensional polytopes $\operatorname{PCCTP}_{69} [n],\, n\in \mathbb{N}$ with $12(n+1)+129$ vertices, all of which are projectively unique.
\end{thmmain}

Again, by passing to pyramids, we conclude:

\begin{mcor}\label{cor:projun}
For each $d\geq 69$, there is an infinite family of combinatorially distinct $d$-dimensional polytopes $\operatorname{PCCTP}_{d} [n],\, n\in \mathbb{N}$ with $12(n+1)+60+d$ vertices, all of which are projectively unique.
\end{mcor}

This resolves the Problem of Perles and Shephard (Problem~\ref{prb:projun}) for all dimensions high enough, and may be seen as further evidence against Conjecture~\ref{conj:mcmsh}. 

\medskip

\noindent {\large \textbf{Ansatz}}

\smallskip

Our work starts with the simple observation that every realization of the $3$-cube is determined by any seven of its vertices.
(A sharpened version of this will be provided in Lemma~\ref{lem:cubecmpl}.)
Thus there are cubical complexes, and indeed cubical $4$-polytopes,
for which rather few vertices of a realization successively determine all the others.
For example, for the neighborly-cubical $4$-polytopes $\operatorname{NCP}_4[n]$, which have $f_3=(n-2)2^{n-2}$ facets,
any realization is determined by any vertex, its $n$ neighbors, and the $\binom n2$ vertices at distance~$2$, that is,
by $1+n+\binom n2$ of the $f_0=2^n$ vertices.

In order to obtain \emph{arbitrarily large} cubical complexes determined by a \emph{constant} number of vertices, 
we consider the standard unit cube tiling $\RC$ of $\R^3$; let  the $i$th “layer” of this tiling for $i\in\Z$ be given by all the 
$3$-cubes whose vertices have sum-of-coordinates between $i$ and $i+3$, and thus centers with sum-of-coordinates equal to $i+\frac32$.
Successively realizing all the $3$-cubes of the abstract cubical $3$-complex given by $\RC$, starting from those at level~$0$, amounts to a solution of a
Cauchy problem for Q-nets, as studied by Adler \& Veselov~\cite{AdlVes} in a “discrete differential geometry” setting.
Here the number of initial values, namely the vertex coordinates for the cubes in the layer~$0$ of~$\RC$, is still infinite.
However, if we divide the standard unit cube tiling $\RC$ by a suitable $2$-dimensional integer lattice $\Lambda_2$ spanned by two vectors
of sum-of-coordinates~$0$, we obtain an infinite (abstract) cubical $3$-complex $\RC/\Lambda_2$ that is homeomorphic to $(S^1)^2\times\R$, for
which coordinates for a finite number $K$ of cubes determines all the others, where $K$ is given by the determinant of the lattice $\Lambda_2$.
Similarly, if we consider only the subcomplex $\RC[N+3]$ of the unit cube tiling formed by all cubes of layers $0$ to~$N$,
then the quotient $\RC[N+3]/\Lambda_2$ by the $2$-dimensional lattice is a finite $3$-complex homeomorphic to $(S^1)^2\times[0,N]$
which are arbitrarily large (consisting of $K(N+1)$ $3$-cubes), where any realization is determined by the coordinates
for the vertices first layer of $K$ $3$-cubes.
The cubical $3$-complexes $\CT[n]:=\RC[N+3]/\Lambda_2$ for $n\ge1$, will be called \Defn{cross-bedding cubical tori}\footnote{In geology, “cross-bedding” refers to horizontal geological structures that are internally composed from inclined layers --- says \emph{Wikipedia}.}, short \Defn{CCTs},
in the following. 
The major part of this work will be to construct realizations for the CCTs in convex position, that is, in the boundaries of $4$-dimensional polytopes.
For this our construction will be inspired by the fibration of $S^3$ into Clifford tori, as used in Santos' work~\cite{Santostriang, Santos}. We note that in~\cite{Santostriang}, the idea to construct polytopal complexes along Clifford tori is used to obtain a result related to ours in spirit: Santos provides simplicial complexes that admit only few geometric bistellar flips, whereas we construct polytopes with few degrees of freedom with respect to possible realizations.

\medskip

\noindent {\large \textbf{Outline of the chapter}}

\medskip

We now sketch the main steps of the proofs.

In \textbf{Section~\ref{sec:nota}}, we recall basic facts and definitions about spherical geometry, realization spaces of polytopes and projectively unique polytopes.  

In \textbf{Section~\ref{sec:bblocks}}, we define the family of \Defn{cross-bedding cubical tori}, short \Defn{CCTs}: 
Throughout the paper, $\CT[n]$ will denote a CCT of \Defn{width} $n$, which is a cubical complex on $12(n+1)$ vertices. 
It is of dimension~$3$ for $n\geq 3$. These complexes allow for a natural application of Lemma~\ref{lem:cubecmpl} and 
form the basic building blocks for our constructions. For our approach to the main theorems we use a class of very 
symmetric geometric realizations of the CCTs, the \Defn{ideal CCTs}.

In \textbf{Section~\ref{sec:Lowdim}} we construct in four steps the family of 
convex $4$-polytopes $\operatorname{CCTP}_4[n]$ of Theorem~\ref{mthm:Lowdim}. 
In order to avoid potential problems with unboundedness, we choose to perform this construction in~$S^4$, 
while measuring “progress” with respect to the Clifford torus fibration of the equator $3$-sphere:

\begin{compactenum}[(1)]
\item We start with a CCT $\PS[1]$ in~$S^4$. The first two extensions of~$\PS[1]$ will be treated manually. Thus we obtain $\PS[3]$, an ideal CCT in convex position in~$S^4$. 
\item We prove that our extension techniques apply to $\PS[n-1],\ n\geq 4$, providing the existence of a family of 
polytopal complexes $\PS[n]$ in~$S^4$. The proof works in the following way: We project $\PS[n-1]$ to the equator $3$-sphere, 
prove the existence of the extension, and lift the construction back to $\PS[n]$ in~$S^4$. The existence of the extension is 
the most demanding part of the construction, even though it uses only elementary spherical geometry, since we have to ensure 
that new facets of $\PS[n]$ intersect only in ways predicted by the combinatorics of the complex. 
\item Now that we have constructed the complexes $\PS[n]$, we need to verify that they are in \Defn{convex position}, 
i.e.\ that every $\PS[n]$ is the subcomplex of the boundary complex of a convex polytope. A natural corollary of 
the construction is that the $\PS[n]$ are in \Defn{locally convex position} (i.e.\ the star of each vertex is 
in convex position). 

A theorem of Alexandrov and van Heijenoort~\cite{Heij} states that, for $d\geq 3$, 
a locally convex $(d-1)$-manifold without boundary in $\R^d$ is in fact the boundary of a convex body. 
As the complexes $\PS[n]$ are manifolds with boundary, we need a version of the Alexandrov--van Heijenoort Theorem 
for polytopal manifolds with boundary. We provide such a result (Theorem~\ref{thm:locglowib} in Section~\ref{ssc:convex}), 
and use it to prove that the complexes constructed are in convex position. 
\item Next, we introduce the family of polytopes $\operatorname{CCTP}_4[n]:=\conv \PS[n],\ n\geq 1$. 
This is the family announced in Theorem~\ref{mthm:Lowdim}: The realization space of $\operatorname{CCTP}_4[n]$ 
is naturally embedded into the realization space of $\PS[n]$, which in turn is embedded in the realization space of $\PS[1]$, 
in particular, $\dim \cR (\operatorname{CCTP}_4[n])$ is bounded from above by $ \dim \cR (\PS[1])$. 
A straightforward analysis then yields the desired bound on $\dim \cR (\operatorname{CCTP}_4[n])$.
\end{compactenum}

\smallskip 

 In \textbf{Section~\ref{sec:projun}}, we turn to the proof of Theorem~\ref{mthm:projun}. 
The idea is to use Lawrence extensions (cf.~\cite[Sec.\ 3.3.]{RG}) which produce 
projectively unique polytopes from projectively unique polytope--point configurations. In order to circumvent difficulties that arise from the fact that we have only recursive descriptions of the $\operatorname{CCTP}$ available, 
we use the notion of \Defn{weak projective triples} (Definition~\ref{def:wpts}) developed in Chapter~\ref{ch:substacked}. Using a technique from the same chapter, we then obtain the desired family of projectively unique polytopes $\operatorname{PCCTP}_{69}[n]$.

\smallskip

Finally, in the \textbf{Appendix} we provide the following:

\begin{compactitem}[$\circ$]
\item In \textbf{Section~\ref{sec:convps}}, we discuss the notion of a polytopal complex in (locally) convex position, and establish an extension of the Alexandrov--van Heijenoort Theorem for polytopal manifolds with boundary.
\item In \textbf{Section~\ref{ssc:Shphrdlist}}, we record Shephard's (conjecturally complete) list of $4$-dimensional projectively unique polytopes.
\item An explicit recursion formula for vertex coordinates of the polytopes $\operatorname{CCTP}_4[n]$ is given in \textbf{Section~\ref{ssc:expformula}}. Using this formula, we give a list of for the vertex coordinates of the polytopes $\operatorname{CCTP}_4[n]$ for all $n$ from $1$ to $10$.
\item In \textbf{Section~\ref{sec:Lemmas}}, we provide proofs for two lemmas that were deferred in order to get a more transparent presentation for the proof of Theorem~\ref{mthm:Lowdim}.
\end{compactitem}

\setstretch{1.15}

\section{Set-up}\label{sec:nota}
For notational reasons, we work almost entirely in the spherical setting, that is, we consider spherical polytopes, projective transformations of the sphere etc. Furthermore, with the exception of Sections~\ref{ssc:trc} and~\ref{sec:convps}, we work with geometric polytopal complexes only.
We consider $\R^d$ with the standard orthonormal basis $(e_1,\, \cdots,\,e_d)$, and the (unit) sphere $S^d\subseteq \R^{d+1}$ with the canonical intrinsic space form metric induced by the euclidean metric on $\R^{d+1}$. For a point $x$ in~$S^d$ or in $\R^{d+1}$, we denote the coordinates of $x$ with respect to the canonical basis $(e_1,\, \cdots,\,e_{d+1})$ by~$x_i$ for~$1\leq i\leq d+1$. 
 
\begin{definition}[The realization space of a polytope, cf.~\cite{RG}]
 Let $P$ be a convex $d$-polytope, and consider the vertices of $P$ labeled with the integers from $1$ to $n=f_0 (P)$. A $d$-polytope $P'\subset S^d$ with a labeled vertex set is said to \Defn{realize} $P$ if there exists an isomorphism between the face lattices of $P$ and $P'$ respecting the labeling of their vertex sets. We define the \Defn{realization space of $P$} as 
\[ \cR(P):=\big\{V\in(S^{d})^{n}:\conv(V)\text{ realizes }P \big\},\]
that is, as the set of vertex descriptions of realizations of $P$. 
\end{definition}

This realization space is a primary semialgebraic subset of $(S^{d})^{n}$ defined over $\Z$~\cite{Grunbaum},~\cite{ZA}; in particular, its \Defn{dimension} is well-defined. Using the notion of realization spaces of polytopes allows us to give a new characterization of projectively unique polytopes.

\begin{definition}[Projectively unique polytopes in $S^d$]
A polytope $P$ in $S^d$ is \Defn{projectively unique} if the group $\mathrm{PGL}(\R^{d+1})$ of projective transformations on~$S^d$ acts transitively on~$\cR(P)$. In particular, for projectively unique polytopes we have $\dim \cR(P) \le \dim \mathrm{PGL}(\R^{d+1}) = d(d+2)$, with equality if the vertex set of the polytope contains a projective basis. 
\end{definition}

Recall that a \Defn{subspace} of $S^d$ denotes the intersection of a linear subspace in $\R^{d+1}$ with $S^d$. The \Defn{upper hemisphere} $S^d_+$ of $S^d$ is the open hemisphere with center $e_{d+1}$. The \Defn{equator} $S^{d-1}_\eq:=\partial S^d_+$ is the $(d-1)$-sphere, considered as a subspace of $S^d$. For points in $S^d_+\subset S^d$ we use homogeneous coordinates, that is, we normalize the last coordinate of a point in $S^d_+$ to $1$. The \Defn{orthogonal projection} from $S^d{\setminus} \{\pm e_{d+1}\}$ to the equator $S^{d-1}_\eq$ associates to $x\in S^d{\setminus} \{\pm e_{d+1}\}$ the unique point $\pp(x)\in S^{d-1}_\eq$ that minimizes the distance to $x$ among all the elements of $S^{d-1}_\eq$. 

Recall that if $x,y$ are two points in $\R^d$ or two non-antipodal points in $S^d$, then $[x,y]$ is the \Defn{segment} from $x$ to $y$. Following a convention common in the literature, we will not strictly differentiate between a segment and its image. The \Defn{midpoint} of a segment $[x,y]$ is the unique point $m$ in $[x,y]$ whose distance to $x$ equals its distance to $y$. Recall that the \Defn{angle} between segments sharing a common starting point $x$ is the angle between the tangent vectors of the segments at $x$; it takes a value in the interval $[0,\pi]$. If $C$ is any polytopal complex, and $v$, $u$ are vertices of $C$ connected by an edge of $C$, then we denote this edge by $[v,u]$, in an intuitive overlap in notion with the segment from $v$ to $u$.

We denote the $d \times d$ identity matrix by $\operatorname{I}_{d}$, and the matrix representing the reflection at the orthogonal complement of a vector $\nu$ in $\R^d$ by $\spi^\nu_{d}$. For example, $\spi^{e_4}_{5}$ is the linear transformation of $\R^5$ that is the identity on $\Sp\{e_i:i\in \{1,2,3,5\}\}$ and takes $e_4$ to $-e_4$. 

Conclusively, we set 
\[
R(\beta)\,:=\,
\left(\begin{array}{cc} 
\cos{\beta} & -\sin{\beta} \\
\sin{\beta} & \cos{\beta} \end{array} \right)
\]
\noindent With this, we define the following rotations in~$S^4$ resp.\ $\R^5$:
\[\rot_{1,2}\,:=\,
\left(\begin{array}{ccc} 
R(\nicefrac{\pi}{2}) & 0		&0    \\
0 & \operatorname{I}_2		&0    \\
0                 & 0 & \operatorname{I}_1 \end{array} \right),
 \qquad
\rot_{3,4}\,:=\, 
\left(\begin{array}{ccc} 
\operatorname{I}_2 & 0		&0    \\
0 & R(\nicefrac{\pi}{3})	&0    \\
0                 & 0 & \operatorname{I}_1 \end{array} \right).
\]

\section{Cross-bedding cubical tori}\label{sec:bblocks}

In this section, we define our basic building blocks for the proofs of Theorem~\ref{mthm:Lowdim} and, ultimately, Theorem~\ref{mthm:projun}. 
These building blocks, called \Defn{cross-bedding cubical tori}, short \Defn{CCTs}, are cubical complexes homeomorphic to $(S^1)^2 \times I$. 
They are obtained as quotients of periodic subcomplexes of the regular unit cube tiling in~$\R^3$. 
The section is divided into three parts: We start by defining CCTs abstractly (Section~\ref{ssc:trc}), 
then define the particular geometric realizations of CCTs used in our construction (Section~\ref{ssc:trg}) 
and close by observing some properties of these geometric realizations (Section~\ref{ssc:trp}). 
Before we start, let us remark that it is not a priori clear that CCTs exist as geometric polytopal complexes 
satisfying the constraints that we define in Section~\ref{ssc:trg}. Explicit examples will be obtained in Section~\ref{ssc:example}.

\subsection{Cross-bedding cubical tori, the combinatorial picture}\label{ssc:trc}

We will now define our building blocks, first as abstract cubical complexes, obtained as quotients of infinite complexes by a lattice. 
We then provide Lemma~\ref{lem:cubecmpl}, a sharpened version of the observation that any seven vertices of a $3$-cube determine
the eighth one, and observe that CCTs allow for a direct application of this lemma.

\begin{definition}[Cross-bedding cubical torus, CCTs, $k$-CCTs]\label{def:CCT}
Let $\RC$ be the unit cube tiling of $\R^3$ on the vertex set $\Z^3\subset \R^3$. For $k\geq 0$, let $\RC[k]$ denote the subcomplex on the vertices $v$ of $\RC$ with~$0\leq\langle v, \mathbbm{1} \rangle\leq k$. The complexes $\RC$ and $\RC[k]$ are invariant under translations by vectors $(1,-1,0)$, $(1,0,-1)$ and $(0,1,-1)$, and in particular under translation by $(3,-3,0)$ and $(-2,-2,4)$. A cubical complex $\CT$ in some $\R^d$ or $S^d$ is called a \Defn{$k$-CCT} (\Defn{cross-bedding cubical torus of width $k$}) if it is combinatorially equivalent to the abstract polytopal complex \[\T[k]:=\bigslant{\RC[k]}{(3,-3,0)\Z \times (-2,-2,4)\Z.}\]
A \Defn{CCT}, or \Defn{cross-bedding cubical torus}, is a finite polytopal complex that is a $k$-CCT for some $k\geq 0$. The vertices of $\T[k]$ are divided into layers $\T_\ell[k],\ 0\le \ell \le k$, defined as the sets of vertices $\widetilde{v}$ of $\T$ such that for any representative $v$ of $\widetilde{v}$ in $\RC$, we have $\langle v, \mathbbm{1} \rangle = \ell$. For a $k$-CCT $\CT$, let $\varphi_{\CT}$ denote the isomorphism 
\[\varphi_{\CT}: \T[k] \longmapsto \CT\]
The $\ell$-th \Defn{layer} of $\CT$, $0\leq \ell\leq k$ of $\CT$ is defined to be the vertex set $\RR(\CT,\ell):=\varphi_{\CT} (\T_\ell)$, the \Defn{restriction} of~$\CT$ to the $\ell$-th layer. More generally, if $I$ is a subset of $\Z$, then we denote by $\RR(\CT,I)$ the \Defn{restriction} of~$\CT$ to the subcomplex induced by the vertices $\bigcup_{i\in I}\, \RR(\CT,i)$.
\end{definition}
 
\begin{rem}[On cross-bedding cubical tori]\label{rem:prp} $\quad$
\begin{compactitem}[$\circ$]

\item The $f$-vector $f(\CT)=(f_0(\CT),\,f_1(\CT),\,f_2(\CT),\,f_3(\CT))$ of any CCT is given by $f(\T[0])=\big(12,\, 0,\, 0,\, 0)$, $f(\T[1])=\big(24,\, 36,\, 0,\, 0)$ and, for $k\geq 2$, \[f(\T[k])=\big(12(k+1),\,36k,\, 36(k-1),\, 12(k-2)  \big).\]
\item A $0$-CCT consists of $12$ vertices, and no faces of higher dimension. A $1$-CCT is a bipartite $3$-regular graph on $24$ vertices. For $k\geq 2$, any $k$-CCT is homotopy equivalent to the $2$-torus. If $k=2$, it is even homeomorphic to the $2$-torus, and if $k\geq 5$, it is homeomorphic to the product of the $2$-torus with an interval. 
\end{compactitem}
\end{rem}

\begin{definition}[Extensions]
If $\CT$ is a $k$-CCT in some euclidean or spherical space, a CCT $\CT'$ of width $\ell>k$ is an \Defn{extension} of $\CT$ if $\RR(\CT',[0,k])=\CT$. If $\ell=k+1$, then $\CT'$ is an \Defn{elementary extension} of $\CT$.
\end{definition}

The following two lemmas will show that extensions are unique.

\begin{lemma}\label{lem:cubecmpl}
Let $Q_1$, $Q_2$, $Q_3$ be three quadrilaterals in some euclidean space (or in some sphere) on vertices $\{a_1,\, a_2,\, a_3,\, a_4\}$, $\{a_1,\, a_4,\, a_5,\, a_6\}$ and $\{a_1,\, a_2,\, a_7,\, a_6\}$ respectively, such that the three quadrilaterals do not lie in a common plane. If $Q'_1$, $Q'_2$, $Q'_3$ is any second such triple with the property that $a_i=a'_i$ for all $i\in \{2,\, \, \cdots,\, 7\}$, then we have $a_1=a'_1$ and $Q_j=Q'_j$  for all $j\in\{1,\,2,\,3\}$. In other words, the coordinates of the vertex $a_1$ can be reconstructed from the coordinates of the vertices $a_i,\, i\in \{2,\, \, \cdots,\, 7\}$. 

\begin{figure}[htbf]  
\centering 
  \includegraphics[width=0.5\linewidth]{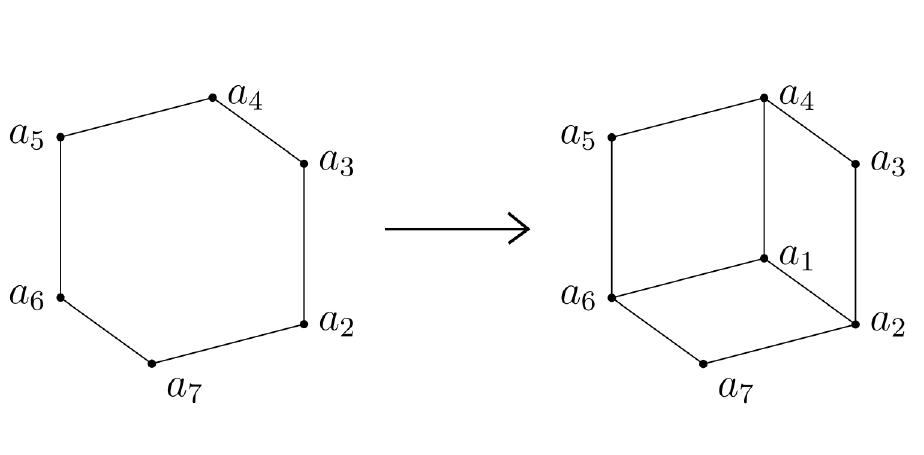} 
  \caption{\small  The vertex $a_1$ of can be reconstructed from the remaining vertices.} 
  \label{fig:cube}
\end{figure}
\end{lemma}

\begin{lemma}[Unique Extension, I]\label{lem:uniext}
Let $\CT$ be a CCT of width $k\geq 2$ in $\R^d$ or $S^d$. If $\CT'$ and $\CT''$ are two elementary extensions of $\CT$, then $\CT''$ coincides with $\CT'$.
\end{lemma}

\begin{proof}
Let $W$ denote any facet of $\CT'$ that is not in $\CT$. Then the geometric realization of $W$ is determined by $\CT$: $W\cap \CT$ consists of three $2$-faces of $W$, in particular, by Lemma~\ref{lem:cubecmpl}, $W$ is determined by $\CT$. The same applies to the facets of $\CT''$ not in $\CT$. Thus $\CT'$ and $\CT''$ coincide.
\end{proof}

\begin{rem}
The assumption $k\geq2$ in Lemma~\ref{lem:uniext} can be weakened to $k\geq1$ if the vertices of $\CT$ are in sufficiently general position: If for all vertices $v$ of $\RR(\CT',2)$ the faces of $\St(v,\CT')$ are \emph{not} coplanar, then $\CT'$ is uniquely determined by $\CT$.
\end{rem}

\subsection{Cross-bedding cubical tori, the geometric picture}\label{ssc:trg}

In this section we give a geometric framework for CCTs. The control structures for our construction are provided by the weighted Clifford tori in~$S^3_\eq$.

\begin{definition}[Weighted Clifford tori $\mathcal{C}_\lambda$ and Clifford projections $\pi_\lambda$]\label{def:clifpauto}
The \Defn{weighted Clifford torus} $\mathcal{C}_\lambda$, where $\lambda\in [0,2]$, is the algebraic subset of~$S^3_\eq$ given by the equations $x_1^2+x_2^2=1-\nicefrac{\lambda}{2}$ and $x_3^2+x_4^2=\nicefrac{\lambda}{2}$. If $\lambda$ is in the open interval $(0,2)$, then $\mathcal{C}_\lambda$ is a flat $2$-torus. In the extreme cases $\lambda=0$ and $\lambda=2$, $\mathcal{C}_\lambda$ degenerates to (isometric copies of) $S^1$. 

For $\lambda\in (0,2)$, we define 
\begin{eqnarray*}
\pi_\lambda :\quad	S^3_\eq{\setminus} (\mathcal{C}_0\cup \mathcal{C}_2) &\longmapsto& \mathcal{C}_\lambda,\\ 
	y & \longmapsto & \big(\mu y_1, \mu y_2, \nu y_3, \nu y_4\big),\ \ \mu =\sqrt{ \frac{1-\tfrac{\lambda}{2}}{y_1^2+y_2^2}},\ \ \nu=\sqrt{\frac{\tfrac{\lambda}{2}}{y_3^2+y_4^2}}.
\end{eqnarray*}
Similarly, we define 
\begin{eqnarray*}
\pi_0 :\quad S^3_\eq{\setminus} \mathcal{C}_2 &\longmapsto& \mathcal{C}_0,\qquad
	y \ \longmapsto \ \frac{1}{\sqrt{y_1^2+y_1^2}}\big(y_1, y_2, 0, 0\big),\\ 
\noalign{\noindent and}
\pi_2 :\quad S^3_\eq{\setminus} \mathcal{C}_0 &\longmapsto& \mathcal{C}_2,\qquad
	y \ \longmapsto \ \frac{1}{\sqrt{y_3^2+y_4^2}}\big(0, 0, y_3, y_4\big). 
\end{eqnarray*}
The maps $\pi_\lambda$ are the \Defn{Clifford projections}.
\end{definition}

We now define restrictions on the geometry of CCTs that will be crucial for the most difficult steps in our proof of Theorem~\ref{mthm:Lowdim}, namely 
to establish that extensions of the CCTs \emph{exist} (Section~\ref{ssc:extension}) and that they are \emph{in} \Defn{convex position} (Section~\ref{ssc:convex}). 
Here is a preview on how we will use the restrictions:
\begin{compactitem}[$\circ$]
\item The property of \Defn{symmetry} (Definition~\ref{def:symrig}) reduces the construction and analysis of the symmetric CCTs to local problems in a fundamental domain of the symmetry. 
\item The property of \Defn{transversality} (Definition~\ref{def:trnsm}) reduces several steps of our study to ``planar'' problems in the torus $\mathcal{C}_1$. It is critically used in Lemma~\ref{lem:localem} (an important step in the proof that extensions of CCTs exist) and in Proposition~\ref{prp:loccrt1lay} (which provides a local-to-global convexity result for CCTs).
\item A CCT is \Defn{ideal} (Definition~\ref{def:slor}) if it satisfies, in addition to symmetry and transversality, two technical requirements: The first requirement is an inequality on a quantity that we call the \Defn{slope} of a symmetric transversal CCT. This inequality is crucially used in the proof that extensions exist (Proposition~\ref{prp:locatt}). The second condition, \Defn{orientation}, is merely a convention that simplifies the notation, and ensures that we extend a CCT into a fixed direction. 
\end{compactitem} 

\begin{figure}[htbf]
\centering 
 \includegraphics[width=0.95\linewidth]{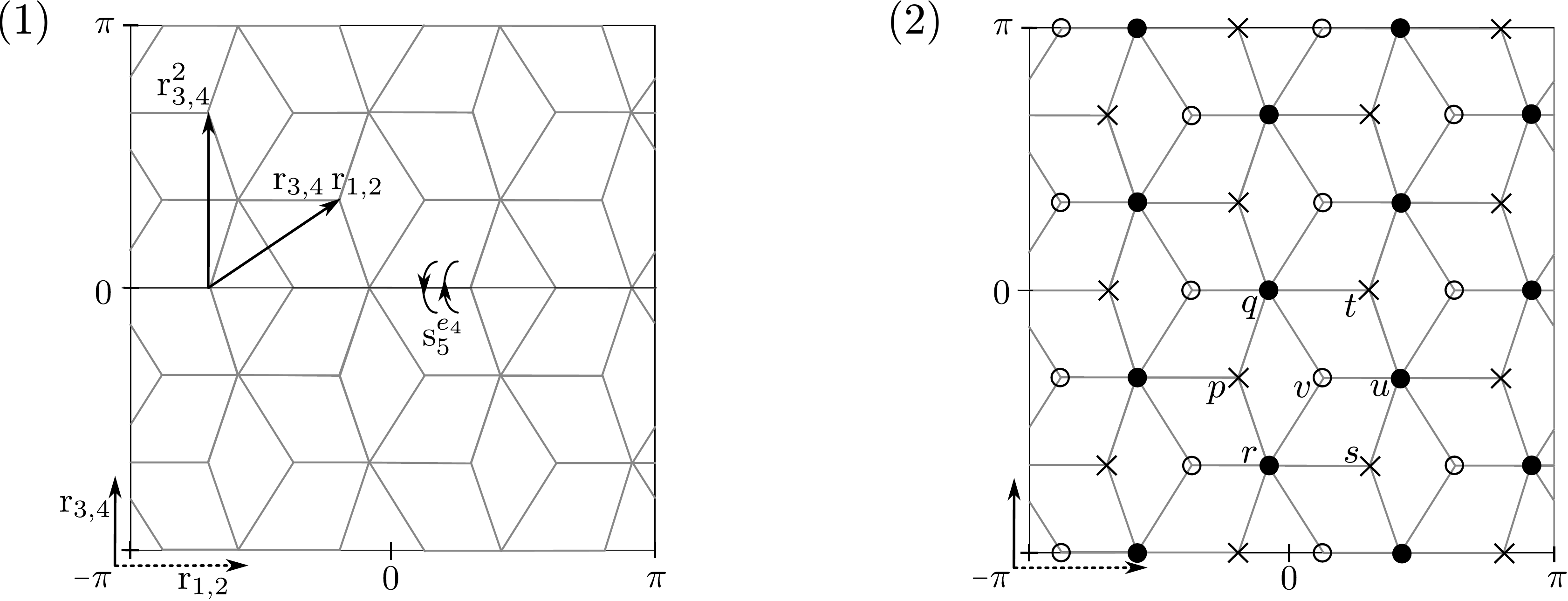} 
 \caption{\small A picture of $\pi_1(\CT)\subset \mathcal{C}_1$ for a transversal symmetric CCT $\CT$ of width $2$ in~$S^3_\eq$. \newline (1) We indicate the action of the symmetries $\mathfrak{S}$ on $\CT$. The compass rose in the lower left corner indicates the action of the rotations $\rot_{1,2} $ of the $\Sp\{e_1,\,e_2\}$-plane and the  action of the rotations $\rot_{3,4}$ of the $\Sp\{e_3,\,e_4\}$-plane. \newline (2) We labeled the images of vertices of layer $0$ by $\circ$, vertices of layer $1$ by $\bullet$ and of layer $2$ by $\times$.}
  \label{fig:protorus}
\end{figure}

\begin{definition}[Symmetric CCTs]\label{def:symrig}
A CCT $\CT$ in~$S^3_\eq$ or~$S^4$ is \Defn{symmetric} if it satisfies the following three conditions:
\begin{compactenum}[\rm(a)]
\item The automorphism $\spi^{(1,-1,0)}_{3}$ of $\T[k]$ corresponds to the reflection $\spi^{e_4}_{5}:\CT\mapsto \CT$. In terms of the restriction 
$\varphi_{\CT}:\T[k]\mapsto \CT$ of Definition~\ref{def:CCT}, 
\[\spi^{e_4}_{5}(\varphi_{\CT}(\T[k]))=\varphi_{\CT} (\spi^{(1,-1,0)}_{3}\T[k]).\] 
\item The automorphism of $\T[k]$ induced by the translation by vector $(-1,1,0)$ corresponds to the rotation $\rot_{3,4}^2:\CT\mapsto \CT$, i.e.\  
\[\rot_{3,4}^2(\varphi_{\CT}(\T[k]))=\varphi_{\CT} (\T[k]+(-1,1,0)).\]
\item The automorphism of $\T[k]$ induced by the translation by vector $(0,-1,1)$ corresponds to the rotation $\rot_{3,4}\rot_{1,2}:\CT\mapsto \CT$, i.e.\  
\[\rot_{3,4}\rot_{1,2}(\varphi_{\CT}(\T[k]))=\varphi_{\CT} (\T[k]+(0,-1,1)).\] 
\end{compactenum}
\end{definition}

\begin{definition}[Groups of symmetries $\mathfrak{S}$ and $\mathfrak{R}$]
The subgroup $\mathfrak{S}$ of $O(\R^5)$ is generated by the elements $\spi^{e_4}_{5}$, $\rot_{3,4}^2$ and $\rot_{3,4}\rot_{1,2}$. 
Its subgroup $\mathfrak{R}$ is generated by the rotations $\rot_{3,4}^2$ and $\rot_{3,4}\rot_{1,2}$.
\end{definition}

\begin{obs} The group $\mathfrak{R}$ acts simply transitively on the vertices of each layer of a symmetric~CCT.
	In particular, $|\mathfrak{R}|=12$ and $|\mathfrak{S}|=2|\mathfrak{R}|=24$.
\end{obs}

\begin{obs}\label{obs:fix}
The translations by vectors $(-1,1,0)$ and $(-1,-1,2)$ induce fixed-point free actions on $\T[k]$. Thus if $\CT$ is a symmetric CCT in~$S^4$ or~$S^3_\eq$, then the actions of $\rot_{3,4}^2$ and $\rot_{1,2}^2$ on $\operatorname{T}$ are fixed point free. In particular, $\CT$ does not intersect $\Sp\{e_1,\, e_2,\, e_5\}\cup \Sp\{e_3,\, e_4,\, e_5\}$.
\end{obs}

\begin{definition}[Control CCT]\label{def:ctrcct}
Let $\CT$ denote a symmetric CCT in~$S^4$. The orthogonal projection of $S^4{\setminus} \{\pm e_5\}$ to~$S^3_\eq$ is well-defined on $\CT$ by the previous observation. If the projection is furthermore injective on $\CT$, then the CCT arising as projection of $\CT$ to $S^3_\eq$ is a symmetric CCT, the \Defn{control CCT} of~$\CT$ in $S^3_\eq$.
\end{definition}

\begin{definition}[Transversal symmetric CCTs]\label{def:trnsm}
A symmetric $2$-CCT $\CT$ in~$S^3_\eq$ is \Defn{transversal} if the Clifford projection $\pi_1: S^3_\eq{\setminus}(\mathcal{C}_0\cup \mathcal{C}_2)\mapsto \mathcal{C}_1 $ is injective on $\CT$. A symmetric $k$-CCT $\CT$ in~$S^3_\eq$  is \Defn{transversal} if  $\RR(\CT,[i-1,i+1])$ is transversal for all $i\in [1,k-1]$. A symmetric CCT $\CT$ in~$S^4$ is \Defn{transversal} if its control CCT in~$S^3_\eq$ is well-defined and transversal.
\end{definition}

\begin{figure}[htbf]
\centering 
  \includegraphics[width=0.72\linewidth]{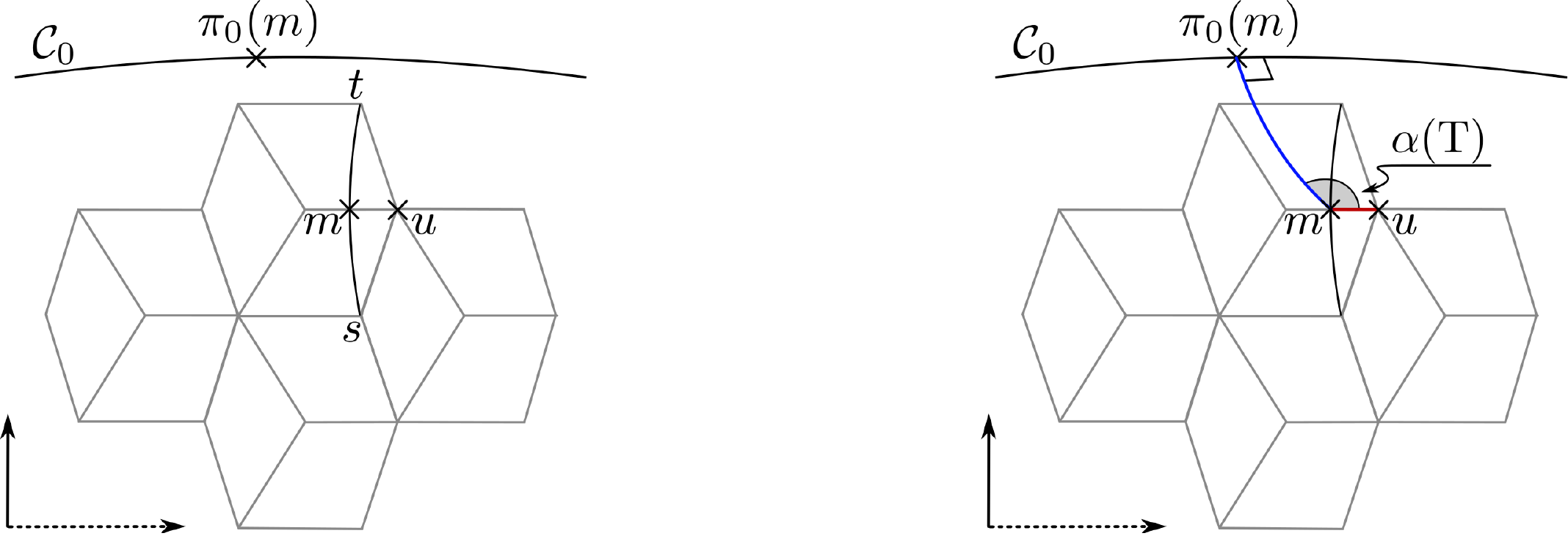}
  \caption{\small  The definition of the slope $\alpha(\CT)$ of a symmetric CCT: $\alpha(\CT)$ is defined as the angle between the segments $[m,\pi_0 (m)]$ and $[m,u]$ at $m$.}
  \label{fig:defslope}
\end{figure}

\begin{definition}[Slope, orientation and ideal CCTs]\label{def:slor}

Let $s$ be any vertex of layer $2$ of a transversal symmetric $2$-CCT $\CT$ in $S^3_\eq$. Then $t:=\rot_{3,4}^2 s$ is in $\RR(\CT,2)$, and $s$ and $t$ are connected by a unique length $2$ edge-path in $\CT$ whose middle vertex $u$ lies in $\RR(\CT,1)$, cf.\ Figure~\ref{fig:defslope}. If $m=m(s,\,t)$ is the midpoint of the geodesic segment $[s,t]$ in $S^3_\eq$, then the angle $\alpha$ between the segments $[m,\pi_0 (m)]$ and $[m,u]$ at $m$ is called the \Defn{slope} $\alpha(\CT)$ of $\CT$.

A transversal symmetric $k$-CCT $\CT$, $k\geq 2$ in $S^3_\eq$ is \Defn{ideal} if $\alpha(\RR(\CT,[k-2,k]))>\nicefrac{\pi}{2}$ and it is \Defn{oriented} towards $\mathcal{C}_0$, i.e.\ the component of $S^3_\eq{\setminus} \CT$ containing $\mathcal{C}_0$ whose closure intersects $\RR(\CT,k)$. A transversal symmetric CCT in $\R^4$ or~$S^4$ is \Defn{ideal} if the associated control CCT in $S^3_\eq$ is ideal.
\end{definition}

\subsection{Some properties of ideal cross-bedding cubical tori}\label{ssc:trp}

Now we record some properties of ideal CCTs, which follow in particular from the conditions of transversality and symmetry (Proposition~\ref{prp:alignsymm}). Furthermore, we give a tool to check transversality (Proposition~\ref{prp:inj3}). The verification of Propositions~\ref{prp:alignsymm} and~\ref{prp:inj3} is straightforward, and left to the reader. We close the section with a Proposition~\ref{prp:slmono}, which says that the slope is monotone under extensions.
\enlargethispage{5mm}
\begin{prp}\label{prp:alignsymm}
Let $\CT$ be a symmetric $2$-CCT in~$S^3_\eq$, let $v$ be a vertex of layer $0$ of $\CT$, and let the remaining vertices of $\CT$ be labeled as in Figure~\ref{fig:protorus}. Then
\begin{compactenum}[\rm(a)]
\item $\CT\cap (\mathcal{C}_0\cup \mathcal{C}_2)$ is empty,
\item $\pi_2(s)=\pi_2(r)$, $\pi_2(p)=\pi_2(v)=\pi_2(u)$ and $\pi_2(t)=\pi_2(q)$,
\item $\pi_0(t)=\pi_0(s)$ and $\pi_0(q)=\pi_0(r)$, and
\item $\pi_2(p)$ is the midpoint of the (nontrivial) segment $[\pi_2(s),\pi_2(t)]$.
\end{compactenum}
\noindent If $\CT$ is additionally transversal, then
\begin{compactenum}[\rm(a)]
\setcounter{enumi}{+4}
\item $\pi_0(v)$, $\pi_0(s)$, $\pi_0(r)$ lie in the interior of segment $[\pi_0(u),\pi_0(p)]$,
\item $\pi_0(s)$ and $\pi_0(v)$ lie in the interior of segment $[\pi_0(u),\pi_0(r)]$, and
\item $\pi_0(r)$ lies in the interior of segment $[\pi_0(v),\pi_0(p)]$ and segment $[\pi_0(s),\pi_0(p)]$.
\end{compactenum}
\end{prp}

\begin{proof}
(a) was already proven in Observation~\ref{obs:fix}. The statements (b), (c) and (d) can also be seen using the symmetry assumption on $\CT$. For example, $\CT$ is symmetric with respect to reflection at the subspace $\SSp(\pi_2^{-1}(\pi_2(s)))$. Thus 
$\pi_2(r)\in\SSp(\pi_2^{-1}(\pi_2(s)))$. Now, $s$ and $r$ are connected by an edge, and this edge does not intersect $\mathcal{C}_0$ (by (a)), consequently $s$ and $r$ lie in the same component of \[\SSp(\pi_2^{-1}(\pi_2(s))){\setminus} \mathcal{C}_0=\pi_2^{-1}(\pi_2(s))\cup \pi_2^{-1}(\pi_2(-s)).\] Thus $\pi_2(r)\in\pi_2^{-1}(\pi_2(s))$, and $\pi_2(s)=\pi_2(r)$. The other statements follow analogously.

Statements (e) to (g) use additionally transversality of $\CT$. Observe that 
\[\pi_1=\tfrac{1}{\sqrt{2}}\big(\pi_0+\pi_2\big).\] Consider now, for example, the statement $\pi_0(v)\in\rint[\pi_0(u),\pi_0(p)]$ of (e). Since $\ell=[p,v]\cup [v,u]\cup [u,(-p)]$ is a line embedded in $\CT$, and $\pi_2(p)=\pi_2(v)=\pi_2(u)=\pi_2(\ell)$ by (b), the map $\pi_0$ must be injective on~$\ell$. The points $\pi_0(u)$ and $\pi_0(-u)$ are antipodal in the $1$-sphere $\mathcal{C}_0$, thus $\ell$ must be mapped injectively into one of the components of $\mathcal{C}_0{\setminus} \{\pi_0(u),\,\pi_0(-u)\}$. The map is in particular injective on $[p,v]\cup[v,u]$, and the length of the image in $\mathcal{C}_0$ is less than $\pi$. It follows that $\pi_0(v)$ lies in the interior of the segment $[\pi_0(u),\pi_0(p)]$. The other statements are obtained analogously.
\end{proof}

For $x\in S^3_\eq{\setminus} \mathcal{C}_0$, we write $\pi^{\operatorname{f}}_2 (x)$ to denote the $\pi_2$-fiber $\pi_2^{-1}(\pi_2(x))$ in~$S^3_\eq$, 
and $\pi^{\operatorname{f}}_0(x),\ x\in S^3_\eq{\setminus} \mathcal{C}_2$ to denote the $\pi_0$-fiber $\pi_0^{-1}(\pi_0(x))\subset S^3_\eq$. 
Furthermore, we denote by $\pi^{\SSp}_2 (x),\ x\in  S^3_\eq{\setminus} \mathcal{C}_0$  the hyperplane in~$S^3_\eq$ spanned by $\pi^{\operatorname{f}}_2 (x)$, i.e.\  
\[
\pi^{\SSp}_2 (x):=\SSp(\pi^{\operatorname{f}}_2 (x))=\pi^{\operatorname{f}}_2 (x)\cup \pi^{\operatorname{f}}_2 (-x)\cup \mathcal{C}_0\subset S^3_\eq,
\] 
and similarly, for $x\in S^3_\eq{\setminus} \mathcal{C}_2$, 
\[
\pi^{\SSp}_0 (x):=\SSp(\pi^{\operatorname{f}}_0 (x))=\pi^{\operatorname{f}}_0 (x)\cup \pi^{\operatorname{f}}_0 (-x)\cup \mathcal{C}_2\subset S^3_\eq.
\] 
Then the statements of Proposition~\ref{prp:alignsymm} can be reformulated using the following dictionary:

\begin{prp}\label{prp:dict}
Let $a$,$b$ and $x$ be three points in $S^3_\eq{\setminus} \mathcal{C}_{2-i}\,,\ i\in\{0,\,2\}$. Then $\pi_i(x)$ lies in the interior of the segment $[\pi_i(a),\pi_i(b)]$ if any two of the three following statements hold:
\begin{compactenum}[\rm(a)]
\item $a$ and $b$ lie in different components of $S^3_\eq{\setminus}\pi^{\SSp}_i (x)$, 
\item $a$ and $x$ lie in the same component of 
$S^3_\eq{\setminus}\pi^{\SSp}_i (b)$, 
\item $b$ and $x$ lie in the same component of 
$S^3_\eq{\setminus}\pi^{\SSp}_i (a)$.
\end{compactenum}
\noindent Conversely, if $\pi_i(x)$ lies in the interior of the segment $[\pi_i(a),\pi_i(b)]$, then all three of the above statements hold.
\end{prp}

Finally, we present our tool for checking transversality. 

\begin{prp}\label{prp:inj3}
Let $C$ denote a cubical complex that arises as the union of three quadrilaterals on vertices $\{u,\,t,\,v,\,q\}$, $\{u,\,s,\,v,\,r\}$ and $\{p,\,q,\,v,\,r\}$ respectively, such that
\begin{compactenum}[ \rm(a) ]
\item $C\cap (\mathcal{C}_0\cup \mathcal{C}_2)$ is empty,
\item $\pi_2(s)=\pi_2(r)$, $\pi_2(p)=\pi_2(v)=\pi_2(u)$ and $\pi_2(t)=\pi_2(q)$,
\item $\pi_0(t)=\pi_0(s)$ and $\pi_0(q)=\pi_0(r)$,
\item $\pi_2(p)$ lies in the interior of the segment $[\pi_2(s),\pi_2(t)]$,
\item $\pi_0(v)$, $\pi_0(r)$ lie in the interior of segment $[\pi_0(u),\pi_0(p)]$, and
\item $\pi_0(s)$ lies in the interior of segment $[\pi_0(u),\pi_0(r)]$.
\end{compactenum}
\noindent Then $\pi_1$ is injective on $C$.
\end{prp} 

For the proof of Proposition~\ref{prp:inj3}, we start with two simple observations:

\begin{lemma}\label{lem:inj1}
Let $Q$ be a convex spherical quadrilateral on vertices $a_1$, $a_2$, $b_1$ and $b_2$ in $S^3_\eq{\setminus} (\mathcal{C}_0\cup \mathcal{C}_2)$, with diagonals $[a_1,a_2]$ and $[b_1,b_2]$, such that $\pi_0(b_1)=\pi_0(b_2)$ and $\pi_2(a_1)=\pi_2(a_2)$. Then $\pi_1$ is injective on $Q$ if and only if $\pi_0(a_1)\neq\pi_0(a_2)$
and $\pi_2(b_1)\neq\pi_2(b_2)$.
\end{lemma}

\begin{lemma}\label{lem:inj2}
Let $Q$ be a convex spherical quadrilateral on vertices $c_1$, $c_2$, $d_1$ and $d_2$ in $S^3_\eq{\setminus} (\mathcal{C}_0\cup \mathcal{C}_2)$ (with diagonals $[c_1,d_2]$ and $[d_1,c_2]$ such that $\pi_2(c_1)=\pi_2(c_2)$ and $\pi_2(d_1)=\pi_2(d_2)$. Then $\pi_1$ is injective on $Q$ if and only if $\pi_2(c_1)\neq\pi_2(d_1)$, $\pi_0(c_1)\neq\pi_0(c_2)$, $\pi_0(d_1)\neq\pi_0(d_2)$ and the orientation of $\mathcal{C}_0$ induced by the segment $[\pi_0(c_1),\pi_0(c_2)]$ (from $\pi_0(c_1)$ to $\pi_0(c_2)$) coincides with the orientation induced by $[\pi_0(d_1),\pi_0(d_2)].$
\end{lemma}

\begin{proof}[\textbf{Proof of Proposition~\ref{prp:inj3}}]
Let us label the quadrilateral on vertices $\{v,\,r,\,q,\,p\}$ by $Q_p$, the quadrilateral on vertices $\{v,\,r,\,s,\,u\}$ by $Q_s$, 
and the quadrilateral on vertices $\{v,\,q,\,t,\,u\}$ by $Q_t$. The assumptions, together with Lemmas~\ref{lem:inj1} and~\ref{lem:inj2}, 
imply directly that $\pi_1$ is injective on each of the three quadrilaterals. Furthermore, by assumptions (b) and (d), 
$Q_s{\setminus} [u,v]$ and $Q_t{\setminus} [u,v]$ lie in different components of $S^3_\eq{\setminus}\pi_2^{\SSp}(v)$. 
Thus by the identity $\pi_1=\nicefrac{1}{\sqrt{2}}(\pi_0+\pi_2)$, $\pi_1$ is injective on $Q_s\cup Q_t$, and in order to prove that 
$\pi_1$ is injective on $C$, it suffices to prove injectivity for $Q_s\cup Q_p$ and $Q_t\cup Q_p$. We prove only the latter, the former case is analogous.

Let $H$ denote the closed hemisphere delimited by $\pi_2^{\SSp}(v)$ and containing $Q_t$. If for two distinct points $x,\, y$ in $Q_t$ and~$Q_p$, respectively, we have that $\pi_1(x)=\pi_1(y)$, then in particular $\pi_2(x)=\pi_2(y)$. As $Q_t\subset H$, we have $x,y \in H$. Thus if we assume that $\pi_1$ is not injective on the domain $Q_t\cup Q_p$, then $\pi_1$ is not injective on \[D:=H\cap( Q_p\cup Q_t) =(H\cap Q_p)\cup Q_t.\] We will show that the latter statement leads to a contradiction.

\begin{figure}[htbf]
\centering 
  \includegraphics[width=0.55\linewidth]{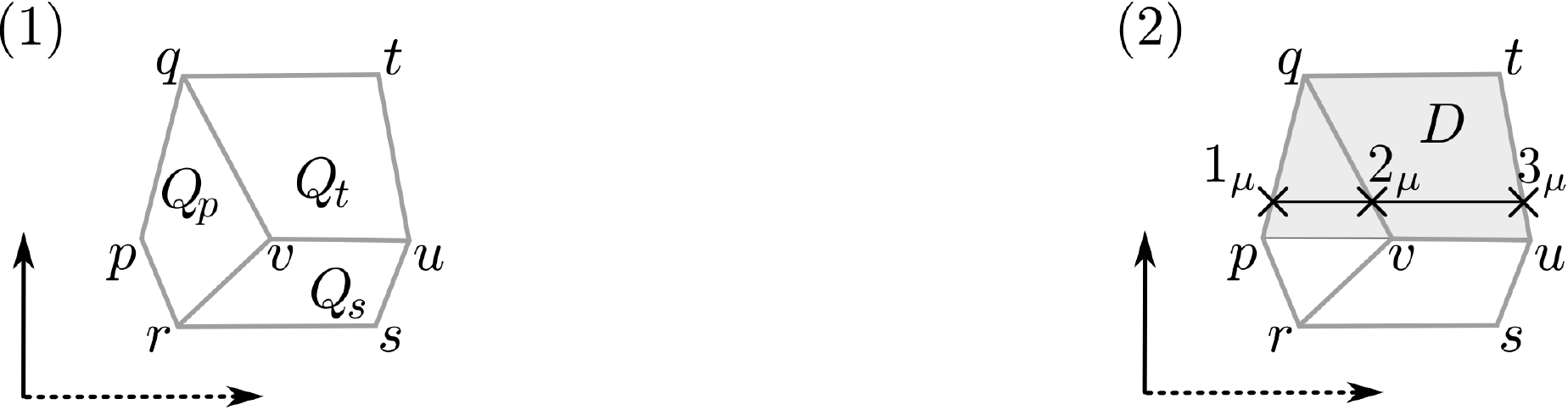} 
  \caption{\small Illustration for the proof of Proposition~\ref{prp:inj3}.}
  \label{fig:checkalign}
\end{figure}

Let us now consider the family of curves $\ell_{\mu}=\pi_2^{\operatorname{f}}({\mu})\cap D,\ {\mu}\in [p,q]$ which cover $D$ by assumption (b). Except when ${\mu}=q$, $\ell_{\mu}$ has three distinct vertices, i.e.\ the intersection of $\pi_2^{\operatorname{f}}({\mu})$ with the edges $[p,q]$, $[v,q]$ and $[u,t]$, which we will denote by $1_{\mu}$, $2_{\mu}$ and $3_{\mu}$. If $\pi_1$ is not injective on $D$, then $\pi_0$ must be non-injective on one of the $\ell_{{\mu}_0}$, where ${\mu}_0\in[p,q]$, and since $\pi_0$ is injective on $\ell_q=[t,q]\subset Q_t$, we obtain  ${\mu}_0\neq q$. 

Since $\pi_0$ is injective on $[1_{{\mu}_0},2_{{\mu}_0}]\subset Q_p$ and on $[2_{{\mu}_0},3_{{\mu}_0}]\subset Q_t$, we have
 \[\pi_0(2_{{\mu}_0})\notin\rint[\pi_0(1_{{\mu}_0}),\pi_0(3_{{\mu}_0})]
\subset[\pi_0(u),\pi_0(p)].\] 
Furthermore, by assumption (e), we have 
\[\pi_0(2_p)=\pi_0(v)\in[\pi_0(1_p),\pi_0(3_p)]=[\pi_0(p),\pi_0(u)].\]
As $\pi_0$ is continuous on $S^3_\eq{\setminus} \mathcal{C}_2$, there must be a point ${\mu}_1\in [p,{\mu}_0]\subsetneq [p,q]{\setminus}\{q\}$ such that $\pi_0(2_{{\mu}_1})=\pi_0(1_{{\mu}_1})$ or $\pi_0(2_{{\mu}_1})=\pi_0(3_{{\mu}_1})$, which contradicts the injectivity of $\pi_0$ on the segments $[2_{{\mu}_1},1_{{\mu}_1}]$ or $[2_{{\mu}_1},3_{{\mu}_1}]$, respectively.
\end{proof}

We close with the crucial monotonicity result for the slope.

\begin{prp}[The slope is monotone]\label{prp:slmono}
Let $\CT$ denote an ideal CCT in~$S^3_\eq$ and let $\CT'$ denote its elementary extension. Then the slope of $\CT'$ is larger or equal to the slope of $\CT$, i.e.\ $\alpha(\CT')\geq\alpha(\CT)$.
\end{prp}

\begin{proof}
Let $D'$ denote the restriction of $\CT'$ to the top three layers, i.e.\ if $\CT$ is a $k$-CCT, then $\CT'$ is a symmetric and transversal $(k+1)$-CCT and $D':=\RR(\CT,[k-1,k+1])$. Let $s$ denote any vertex of layer $k+1$ of $D'$. The vertices $t:=\rot_{3,4}^{2}s$ and $s$ of $D'$ are connected by a unique length $2$ edge-path in $\CT'$ whose middle vertex is $u\in \RR(\CT,k)$. Let $m:=m(s,t)$ denote the midpoint of the segment $[s,t]$. Let $\beta$ denote the angle between the segment $[u,m]$ and the segment $[u,\pi_0(u)]$. We claim the following two inequalities: 
\begin{compactenum}[\rm(a)]
\item $\alpha(\CT')\geq\pi-\beta$.
\item $\beta\leq\pi-\alpha(\CT)$.
\end{compactenum}
This immediately implies that $\alpha(\CT')\geq\alpha(\CT)$, and thus finishes the proof. For inequality (a), consider the convex quadrilateral on the vertices $u$, $m$, $\pi_0(u)$ and $\pi_0(m)$. The slope $\alpha(\CT')$ of $\CT'$ is the angle at $m$, the angle $\beta$ is the angle at $u$. The remaining two angles of the quadrilateral measure to $\nicefrac{\pi}{2}$. The angle-sum in a convex spherical quadrilateral is at least $2\pi$, so \[\beta + \alpha(\CT') + \nicefrac{\pi}{2} + \nicefrac{\pi}{2} \geq 2\pi\Rightarrow\alpha(\CT')\geq \pi-\beta.\]

To see inequality (b),  we work in the $2$-sphere $\RN^1_u S^3_\eq$. Consider the vertices \[a^+:=[u\rot_{3,4}^{2},u],\ a^-:=[u\rot_{3,4}^{-2},u],\ b^+:=[u,t] \text{ and } b^-:=[u,s]\] in $\RN^1_u S^3_\eq$, cf.\ Figure~\ref{fig:sphcosine}. Symmetry and transversality of $D'$ imply that these four vertices form the vertices of a convex quadrilateral $Q$ that is symmetric under reflection at the axis $\Sp\{[u,\pi_0(u)],\, [u,m]\} \subset \RN_u^1 S^3_\eq$. In particular,  the segment $[u,m]$, seen as point in $\RN^1_u S^3_\eq$, coincides with the midpoint of the edge $[b^+,b^-]\subset \RN^1_u S^3_\eq$ of $Q$, and similarly, $[u,\pi_0(u)]$ is the midpoint of the edge $[a^+,a^-]\subset \RN^1_u S^3_\eq$ of $Q$. 

\begin{figure}[htbf] 
\centering 
 \includegraphics[width=0.96\linewidth]{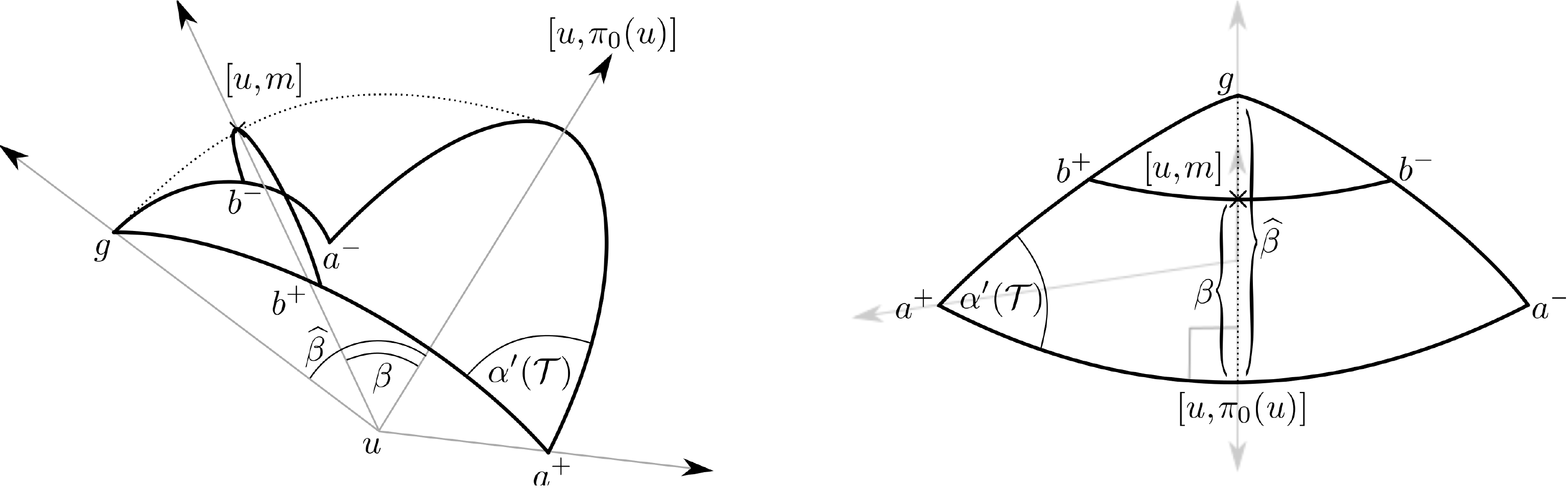} 
\caption{\small The figures show part of the sphere $\RN^1_u S^3_\eq$ from two perspectives.}
  \label{fig:sphcosine}
\end{figure}

To estimate the distance $\beta$ between $[u,m]$ in $[u,\pi_0(u)]$ in $\RN^1_u S^3_\eq$, consider the intersection point $g$ in \[\Sp\{a^+,\, b^+\} \cap \Sp \{a^-,\, b^-\}\subset \RN^1_u S^3_\eq\] for which $b^+\in [a^+,g]$, cf.\ Figure~\ref{fig:sphcosine}. Then, $[u,m]$ lies in the segment from $g$ to $[u,\pi_0(u)]$, and thus the angle $\beta$ is bounded above by the distance $\widehat{\beta}$ of $g$ to $[u,\pi_0(u)]$ in $\RN^1_u S^3_\eq$. 

To bound $\widehat{\beta}$, consider the triangle in $\RN^1_u S^3_\eq$ on $[u, \pi_0(u)]$, $g$ and $a^+=[u,\rot_{3,4}^{2}u]$. Let $\eta$ denote the angle of this triangle at $g$, and notice that the angle at $[u,\pi_0(u)]$ is $\nicefrac{\pi}{2}$ by reflective symmetry at $\SSp\{[u,\pi_0(u)],\, [u,m]\}$, and that the angle at $[u,\rot_{3,4}^{-2}u]$ measures to $\pi-\alpha(\CT)$. Then the second spherical law of cosines gives \[\cos\big(\pi-\alpha(\CT)\big)=\sin\eta\cos\widehat{\beta},\]   
and since $\pi-\alpha(\CT)<\nicefrac{\pi}{2}$, we have \[\cos\big(\pi-\alpha(\CT)\big)\leq\cos\widehat{\beta}\]
and consequently $\pi-\alpha(\CT)\geq\widehat{\beta}$, which gives the desired bound $\pi-\alpha(\CT)\geq\beta$.
\end{proof}

\section{Many 4-polytopes with low-dimensional realization space}\label{sec:Lowdim}

This section is dedicated to the proof of Theorem~\ref{mthm:Lowdim}. 
As outlined above, we will proceed in four steps:
\begin{compactitem}[$\circ$]
\item \emph{Section~\ref{ssc:example}:}
In this section, we provide the initial CCT for the construction, onto which we will build larger and larger CCTs by iterative extension. Our extension techniques for ideal CCTs develop their full power only if $k\geq 3$ (cf.\ Theorem~\ref{thm:convp}). Thus the initial example is $\PS[3]$, an ideal $3$-CCT in convex position in~$S^4$ which is constructed manually from a CCT $\PS[1]$ of width $1$.
\item \emph{Section~\ref{ssc:extension}:} Any ideal CCT can be extended:
\begin{theorem}\label{thm:ext}
Let $\CT\subset S^4$ be an ideal CCT of width $k\geq 3$. Then there exists an ideal $(k+1)$-CCT $\CT'$ extending $\CT$.
\end{theorem}

An ideal $k$-CCT $\CT$ in~$S^4$, $k\geq 2$, has an ideal elementary extension if and only if the associated ideal control CCT in~$S^3_\eq$ has an ideal elementary extension. Thus Theorem~\ref{thm:ext} is equivalent to the following theorem, which we will prove in Section~\ref{ssc:extension}.
\begin{theorem}\label{thm:exts}
Let $\CT\subset S^3_\eq$ be an ideal CCT of width $k\geq 3$. Then there exists an ideal $(k+1)$-CCT $\CT'$ extending $\CT$.
\end{theorem}
\item \emph{Section~\ref{ssc:convex}:} Extensions are in convex position:
\begin{theorem}\label{thm:convp}
Let $\CT\subset S^4$ denote an ideal CCT of width $k\geq 3$ in convex position, and let $\CT'$ be an elementary extension of $\CT$. Then $\CT'$ is in convex position as well.
\end{theorem}
The idea for the proof is to derive from the convex position of the $3$-complex $\RR(\CT, [k-3,k])\subset \CT$ the convex position of $\RR(\CT',[k-2,k+1])$, and then to apply the Alexandrov--van Heijenoort Theorem for polytopal manifolds with boundary, as given in Section~\ref{sec:convps}. 
\item \emph{Section~\ref{ssc:pfmthm1}:} The two previous steps provide an infinite family $\PS[n]$ of ideal CCTs in convex position in $S^4$ extending $\PS[1]$. We define the realization space of a complex, set $\operatorname{CCTP}_4[n]:=\conv (\PS[n])$, and get
\[
\dim \cR(\operatorname{CCTP}_4[n])\leq \dim \cR(\PS[n])\leq \dim \cR(\PS[1]).
\] 
A closer inspection gives the desired bound.
\end{compactitem}

\subsection{An explicit ideal cross-bedding cubical torus in convex position}\label{ssc:example}

The purpose of this section is to provide an ideal $3$-CCT $\PS[3]$ in convex position in the upper hemisphere $S^4_+$ of~$S^4$. We use homogeneous coordinates, 
so we describe points in $S^4_+$ by coordinates in $\R^4\times\{1\}$. 
The coordinates of the specific example we present are chosen in such a way that the complex can be reused later for the proof of Theorem~\ref{mthm:projun}.

\smallskip
The complex $\PS[3]$ is constructed by first giving a $1$-CCT $\PS[1]$ in~$S^4$, and then extending it twice. Instead of developing machinery that could provide the first two extensions, we will describe them directly in terms of vertex coordinates, and indicate how to verify that $\PS[3]$ is ideal and in convex position.

\medskip 

\enlargethispage{-5mm}

We start off with the ideal $1$-CCT $\PS[1]$ in~$S^4\subset\R^5$. Set
\[\vartheta_0=\big(\sqrt{2}-1,\,1-\sqrt{2},\,2,\,0,\, 1\big)\ \ \text{and}\ \ \vartheta_1:=\big(1,\,0,\,1,\,0,\, 1\big).\]
Let $L_0$ denote the orbit of $\vartheta_0$, and $L_1$ the orbit of $\vartheta_1$, under the group $\mathfrak{R}\subset O(\R^5)$ generated by $\rot_{3,4}^2$ and $\rot_{3,4} \rot_{1,2}$. The point configuration $L_0\cup L_1$ forms the vertex set of $\PS[1]$. 

The edges of $\PS[1]$ are given by
\begin{compactitem}[$\circ$]
\item the edge connecting $\vartheta_0$ with the vertex $\vartheta_1$, and its orbit under $\mathfrak{R}$,  
\item the edge connecting $\vartheta_0$ with the vertex $\vartheta_1^+=\rot_{1,2}^{-1}\rot_{3,4}^{-1}\vartheta_1$, and its orbit under $\mathfrak{R}$, and
\item the edge connecting $\vartheta_0$ with the vertex $\vartheta_1^-=\rot_{1,2}^{-1}\rot_{3,4}\vartheta_1$, and its orbit under $\mathfrak{R}$.
\end{compactitem}

\noindent $L_0$ is layer $0$ of $\PS[1]$, and consequently, the orbit $L_1$ corresponds to layer $1$ of $\PS[1]$. Let $\PS[2]$ denote the elementary extension of $\PS[1]$ (cf.\ Figure~\ref{fig:0transmit}), and let $\PS[3]$ denote the elementary extension of $\PS[2]$.

\begin{figure}[htbf]
\centering 
  \includegraphics[width=0.38\linewidth]{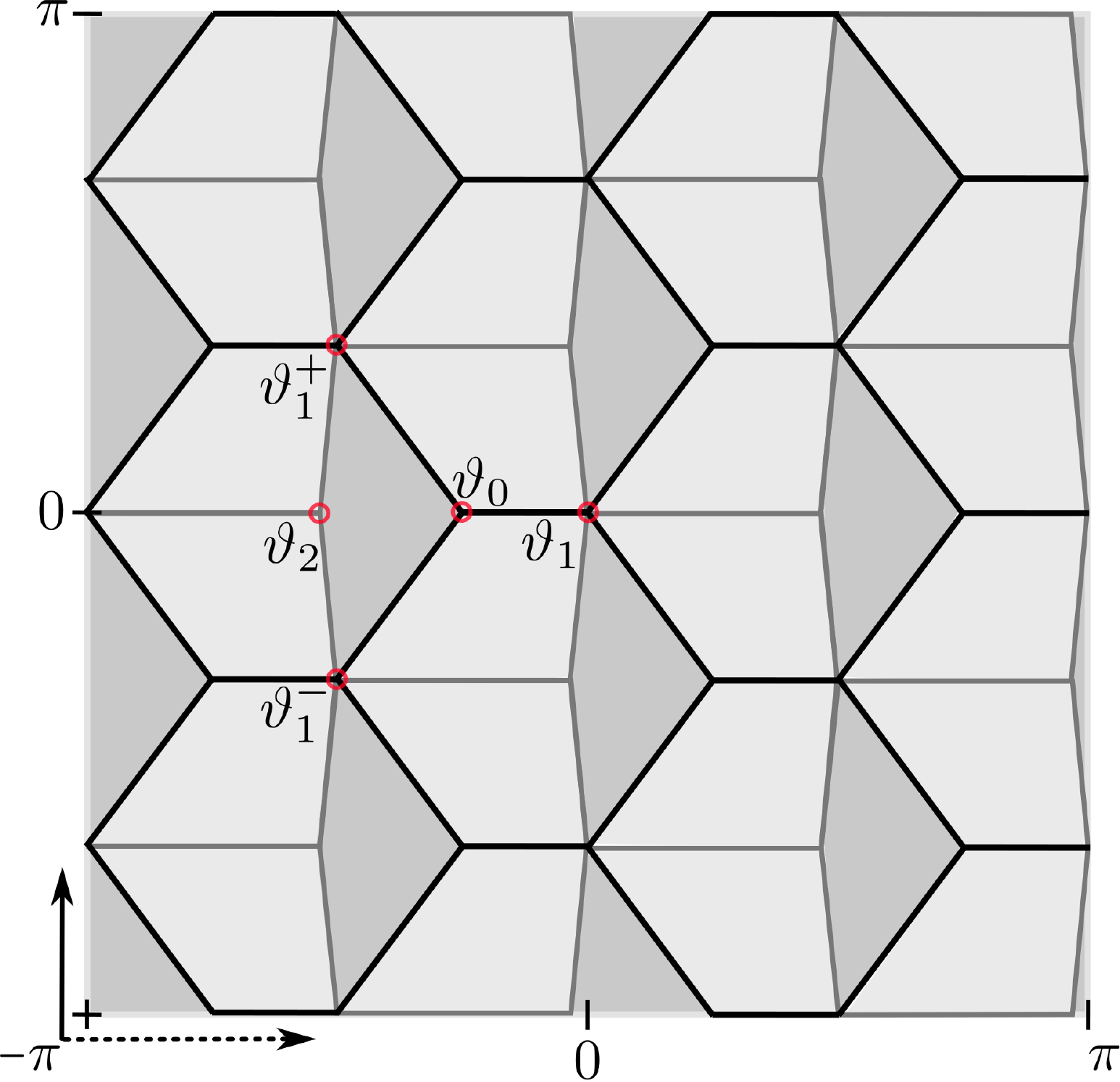} 
  \caption{\small The $2$-CCT $\PS[2]$, obtained as the elementary extension of $\PS[1]$.
     This image in $\mathcal{C}_1$ is produced by orthogonal projection to the equator $S^3_\eq$, followed by the Clifford projection~$\pi_1$.} 
  \label{fig:0transmit}
\end{figure}

We record some properties of the complexes $\PS[2]$ and $\PS[3]$ thus obtained:
\begin{compactitem}[$\circ$]
\item $\PS[2]$ is a CCT of width $2$. The coordinates of layer $2$ are obtained as \[ \vartheta_2= \big(\tfrac{1}{23}(-11+7\sqrt{2}),\, \tfrac{1}{23}(-9-11\sqrt{2}),\, \tfrac{1}{23}(16-6\sqrt{2}),\, 0,\, 1 \big)\]
and its orbit under the group of rotational symmetries $\mathfrak{R}$; see also Figure~\ref{fig:0transmit}.
\item $\PS[3]$ is a CCT of width $3$. The coordinates of layer $3$ are obtained as \[\vartheta_3=\big(\tfrac{1}{49}(37+11\sqrt{2}),\,\tfrac{1}{49}(-11+6\sqrt{2}),\tfrac{1}{49}(22-12\sqrt{2}),0,\,1\big)\]
and its orbit under $\mathfrak{R}$. Here $\vartheta_3$ is the unique vertex of $\PS[3]$ that lies in a joint facet with $\vartheta_0$.

\item $\PS[3]$ is ideal. The symmetry of $\PS[3]$ is obvious from the construction.
Transversality can be checked by straightforward computation, or by using the injectivity criterion Proposition~\ref{prp:inj3}: Using it gives directly that $\pi_1$ is injective on the star of every vertex $v$ of degree $3$ in $\RR(\PS[3],[0,2])$, and similarly~$\RR(\PS[3],[1,3])$.

In particular, $\pi_1$ is locally injective on the tori $\RR(\PS[3],[0,2])$ and $\RR(\PS[3],[1,3])$. It now follows from an examination of the action of the rotations that $\pi_1$ must be a trivial covering map, in particular injective. The computation of the slope and checking the orientation is again a simple calculation. 

\item $\PS[3]$ is in convex position. Outer normals of its facets are given by
\[\vv{n}=\big(7+5\sqrt{2},\,-8-5\sqrt{2},2,0, -9-5\sqrt{2}\big)\]
and its orbit under $\mathfrak{R}$. Here, $\vv{n}$ is an outer normal to the facet containing the vertex $\vartheta_0$. While verifying that $\PS[3]$ is in convex position this way is again easy, one can simplify the calculation drastically by using Proposition~\ref{prp:loccrt1lay} below.
\end{compactitem}

\subsection{Existence of the extension}\label{ssc:extension}
This section is devoted to the proof of Theorem~\ref{thm:exts}. 

\smallskip

We divide it into two parts, first proving that the elementary extension exists ``locally'', i.e.\ proving that, for every vertex $v$ of $\RR(\CT,k-2)$, there exists a $3$-cube containing $\St(v, \RR(\CT,[k-2,k]))$ as a subcomplex (Proposition~\ref{prp:locatt}), then concluding that the extension exists and is in fact an ideal CCT. For this section, we work in the equator $3$-sphere~$S^3_\eq$, and stay in the notation of Theorem~\ref{thm:exts}.

\subsubsection*{Local extension}

The goal of this section is to prove Proposition~\ref{prp:locatt}:

\begin{prp}\label{prp:locatt}
Let $\CT^\circ:=\RR(\CT,[k-2,k])$ denote the subcomplex of the ideal CCT $\CT$ induced by the vertices of the last three layers, and let $v$ be any vertex of layer $k-2$ in $\CT$. Then there exists a $3$-cube $X(v)\subset S^3_\eq$ such that $\St(v,\CT^\circ)$ is a subcomplex of $X(v)$. 
\end{prp}

We start with a lemma, the proof of which is postponed to the appendix, Section~\ref{ssc:lemdihang}:

\begin{lemma}\label{lem:dihang}
Let $v$, $\CT^\circ$ be chosen as in Proposition~\ref{prp:locatt}. Then $\St(v,\CT^\circ)$ is in convex position, and the tangent vector of $[v,\pi_0(v)]$ at $v$, seen as an element of $\RN_v^1 S^3_\eq$, lies in $\conv \Lk(v,\CT^\circ)$.
\end{lemma}

Intuitively speaking, the convex position of $\St(v,\CT^\circ)$ is a necessary condition for the existence of the cube $X(v)$, and the
conclusion that the tangent direction of $[v,\pi_0(v)]$ lies in $\conv \Lk(v,\CT^\circ)$ ensures that the new cube $X(v)$ is attached in direction of $\mathcal{C}_0$, and away from the complex $\CT$ already present. This allows us to provide the following technical statement towards the proof of Proposition~\ref{prp:locatt}. We consider, for $v$ and $\CT^\circ$ as above, the complex $\St(v,\CT^\circ)$ with vertices labeled as in Figure~\ref{fig:cubeatt}(1). Let $m$ denote the midpoint of the segment $[s,t]$.

\begin{cor}\label{cor:diffcomp}
$\pi_0(m)=\pi_0(s)$ and $v$ lie in different components of $S^3_\eq{\setminus} \SSp\{p,\, s,\, t\}$.
\end{cor}

\begin{proof}
Let $u'$ be the unique point in the weighted Clifford torus $\mathcal{C}_0$ so that $u$ lies in  $[v,u']$, and analogously let $p'$ denote the point of $\mathcal{C}_0$ for which $p\in [v,p']$; cf.\ Figure~\ref{fig:cubeatt}(2). Let $\Delta$ be the triangle $\conv\{v,u',p'\}\subset \pi_2^{\SSp}(v)$. Since the tangent direction of $[v,\pi_0(v)]$ at $v$ lies in $\conv \Lk(v,\CT^\circ)$ by Lemma~\ref{lem:dihang}, we have that $\pi_0(v)$ lies in~$[u',p']$. Thus for all points $y$ in $\Delta$, the point $\pi_0(y)$ lies in~$[u',p']$, and in particular $[\pi_0(u),\pi_0(p)]$ lies in~$[u',p']$. Moreover, since $m$ lies in the interior of $\Delta$, $\pi_0(m)$ lies in the interior of~$[u',p']$. Let us consider the polygon $\Theta:=\conv\{m,\, \pi_0(m),\,p',\,p\}\subset  \Delta$. We have the following:

\begin{compactitem}[$\circ$]
\item By Proposition~\ref{prp:alignsymm}(e), $\pi_0(m)$ is a point in the relative interior of $[\pi_0(u),\pi_0(p)]\subset [u',p']$. In particular, $p$ and $p'$ lie in the same component of $\pi_2^{\SSp}(v){\setminus}\SSp\{m,\, \pi_0(m)\}$. Thus the segment $[m,\pi_0(m)]$ is an edge of $\Theta$, since $m$ does not coincide with $\pi_0(m)$ by Proposition~\ref{prp:alignsymm}(a).
\item Since $m$ and $p$ lie in $\pi_2^{\operatorname{f}}(v)$ by Proposition~\ref{prp:alignsymm}(b) and (d), the segment $[p',\pi_0(m)]\in \mathcal{C}_0$ is exposed by the subspace $\mathcal{C}_0$ in $\pi_2^{\SSp}(v)$. Consequently, since $p'$ does not coincide with $\pi_0(m)$, we obtain that $[p',\pi_0(m)]$ is an edge of $\Theta$.
\item By an analogous argument, $[p,p']$ is an edge of $\Theta$: The point $m$ lies in the interior of $\Delta$, and $\pi_0(m)$ lies in the interior of $[u',p']$. Consequently, $m$ and $\pi_0(m)$ lie in the same component of $\pi_2^{\SSp}(v){\setminus}\SSp\{p,\, p'\}$. Thus the segment $[p,p']$ is an edge of $\Theta$, since we have $p\neq p'$ by Proposition~\ref{prp:alignsymm}(a).
\end{compactitem}

\noindent Thus $\Theta$ is a convex quadrilateral, the remaining edge of which is given by $[p,m]$, cf.\ Figure~\ref{fig:cubeatt}. In particular $\pi_0(m)$ and $p'$ lie in the same component of $\pi_2^{\SSp}(v){\setminus} \SSp\{p,\, m\}=\pi_2^{\SSp}(v){\setminus} \SSp\{p,\, s,\, t\}$, and since $p'$ and $v$ lie in different components of $\pi_2^{\SSp}(v){\setminus} \SSp\{p,\, m\}$, we obtain that $\pi_0(m)$ and $v$ lie in different components of $\pi_2^{\SSp}(v){\setminus} \SSp\{p,\, m\}$, as desired.
\end{proof}

As announced, we conclude this section with a proof of Proposition~\ref{prp:locatt}.

\begin{proof}[\textbf{Proof of Proposition~\ref{prp:locatt}}]

We continue to use the labeling of Figure~\ref{fig:cubeatt}. We already proved that the vertices of $\St(v,\CT^\circ)$ are not all coplanar (Lemma~\ref{lem:dihang}), so in order to prove $X(v)$ exists, it suffices to show that the vertices $r,\, q,\, u$ lie in the same component of $S^3_\eq{\setminus} \SSp\{p,\, s,\, t\}$ as $v$. For $r$ and $q$, this follows directly from reflective symmetry of $\St(v,\CT^\circ)$ at the hyperplane $\pi^{\SSp}_2(v)$, so we only have to verify the property for $u$. Let $m$ denote the midpoint of the segment $[s,t]$, as before. Denote the component of $S^3_\eq{\setminus} \SSp\{p,\, s,\, t\}$ containing $v$ by $H^v$. Consider now the following segments:  
\begin{compactitem}[$\circ$]
\item the segment $[m,\pi_0(m)]$ from $m$ to $\pi_0(m)$,
\item the segment $[m,u]$ from $m$ to the vertex $u$, and
\item the segment $[m,p]$ from $m$ to the vertex $p$.  
\end{compactitem}
The slope $\alpha(\CT)$ of $\CT$ coincides with the angle between $[m,u]$ and $[m,\pi_0(m)]$.
The angle between $[m,p]$ and $[m,\pi_0(m)]$ is denoted by $\gamma$.
  
\begin{figure}[htbf]
\centering 
  \includegraphics[width=0.78\linewidth]{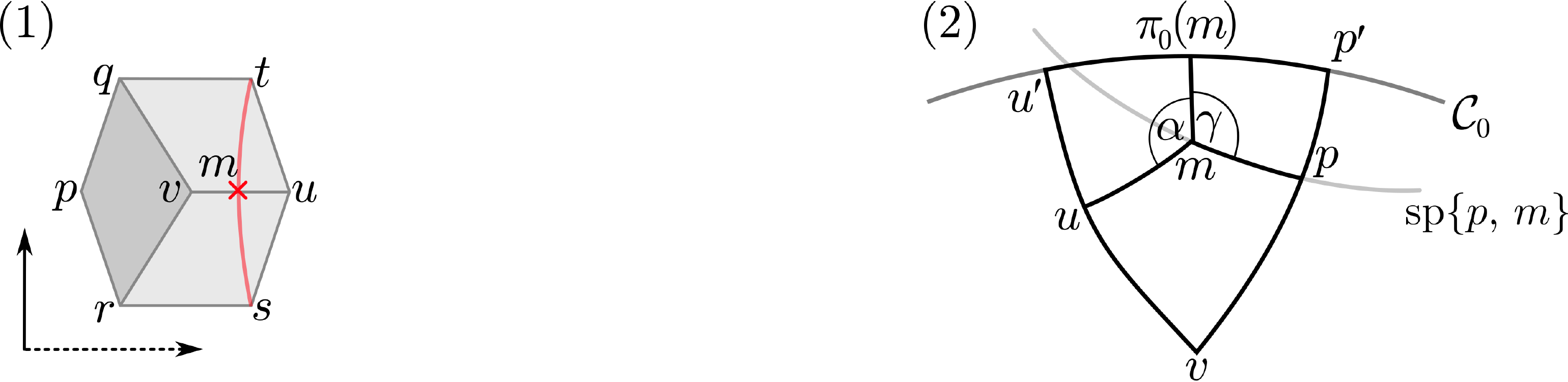} 
  \caption{\small  Illustrations for the proof of Proposition~\ref{prp:locatt}. \newline (1) The labeling of the vertices of $\St(v,\CT^\circ)$. The complex $\St(v,\CT^\circ)$ is the subcomplex of the boundary of a $3$-cube $X(v)$ if and only if the vertices $r,\, q,\, u$ lie on the same side of the hyperplane spanned by $p,\, s,\, t$ as $v$. \newline (2) The triangle $\Delta\in \pi_2^{\SSp}(v)$ on vertices $u'$, $v$ and $p'$. By Corollary~\ref{cor:diffcomp}, $\pi_0(m)$ and $v$ are in different components of $S^3_\eq{\setminus} \SSp\{p,\, s,\, t\}$, and by Proposition~\ref{prp:alignsymm}, $u$ and $p$ lie in different components of $S^3_\eq{\setminus} \pi_2^{\SSp}(m)$. Thus $v$ and $u$ lie in the same components of $S^3_\eq{\setminus} \SSp\{p,\, s,\, t\}$ iff $\alpha(\CT)> \pi-\gamma$.} 
  \label{fig:cubeatt}
\end{figure}

As $\CT^\circ$ is transversal, $u$ and $p$ lie in different components of $S^3_\eq{\setminus} \pi_2^{\SSp}(m)=S^3_\eq{\setminus} \pi_2^{\SSp}(s)$ by Proposition~\ref{prp:alignsymm}(e). Furthermore, $\pi_0(m)$ and $v$ lie in different components of $S^3_\eq{\setminus} \SSp\{p,\, s,\, t\}$ by the previous corollary, so $u$ lies in $H^v$ if and only if $\alpha(\CT)>\pi-\gamma$. Since $\CT^\circ$ is ideal we have $\alpha(\CT)>\nicefrac{\pi}{2}$. Thus let us determine $\gamma$. After possibly applying a rotation of $\Sp\{e_1,\,e_2\}$-plane and $\Sp\{e_3,\,e_4\}$-plane, we may assume that the coordinates of $p$ are given as 
\[
	\big(p_1,\,0,\,p_3,\,0,\, 0\big),\ p_1,\, p_3 > 0,\ p_1^2+p_3^2=1.
\] 
Then 
\[
	\sqrt{(1-\tfrac{3}{4}p_3^2 )}m=\big(0,\,p_1,\,\tfrac{1}{2} p_3,\,0,\, 0\big),\ \ \pi_0(m)=\big(0,\,1,\,0,\,0,\, 0\big)
\] 
and  
\[
\cos(\gamma)=\frac{-p_3}{\sqrt{1+p_3^2}}
\]   
and, consequently, $\cos(\gamma)<0$ and $\gamma>\nicefrac{\pi}{2}$. Thus 
\[
\alpha(\CT)>\tfrac{\pi}{2}> \pi-\gamma,
\] 
which finishes the proof.
\end{proof}

\subsubsection*{The global extension exists and is ideal}

In this section, we prove Theorem~\ref{thm:exts}.

\smallskip

We need to prove that if we attach all cubes of the extension, we obtain a polytopal complex, and that the resulting complex is an ideal CCT. For this, we first prove that the attachment of a cube $X(v)$, where $v$ denotes a vertex of $\RR(\CT,k-2)$ as before, does not change the image of $\St(v,\CT^\circ),\, \CT^\circ=\RR(\CT,[k-2,k])$, under $\pi_1$ (Lemma~\ref{lem:localem}), which allows us to prove both the existence and the transversality of the extension. Theorem~\ref{thm:exts} then follows easily.

\begin{lemma}\label{lem:localem}
Let $X(v)$ denote the facet attached in Proposition~\ref{prp:locatt}, let $x(v)$ be the vertex of $X(v)$ not in $\CT$, and let $\widehat{X}(v)$ denote
the complex $\St(x(v),\partial X(v))$. Then $\pi_1$ is injective on $\widehat{X}(v)$.
\end{lemma}

The set-up for Lemma~\ref{lem:localem} is illustrated in Figure~\ref{fig:inj}.

\begin{figure}[htbf]
\centering 
  \includegraphics[width=0.76\linewidth]{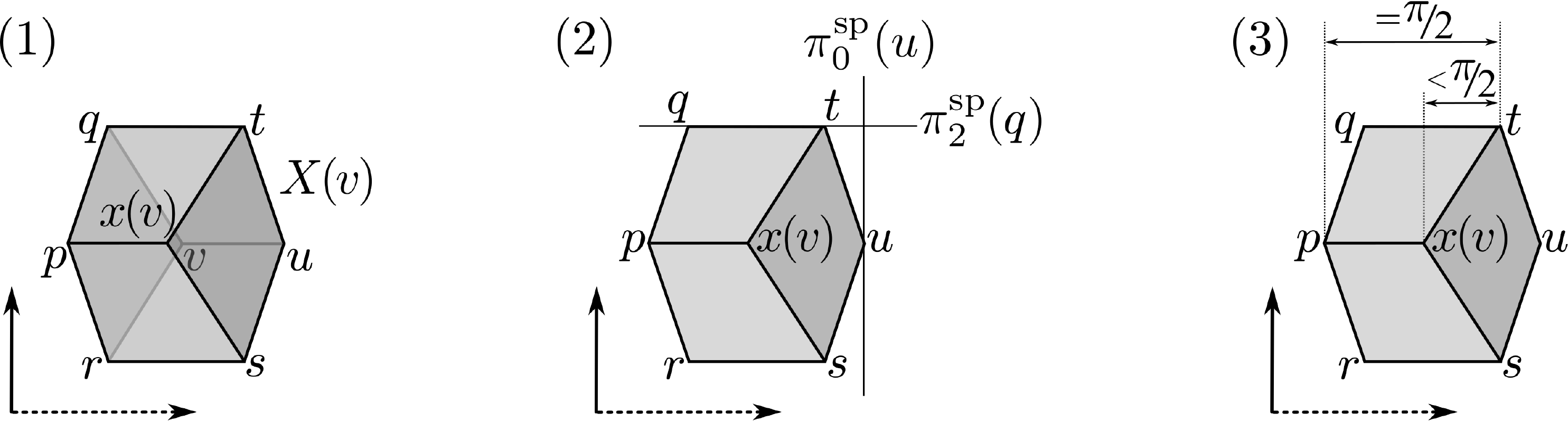} 
  \caption{\small  Illustration for Lemma~\ref{lem:localem}:  \newline 
(1) The set-up: We proved that there exists a $3$-cube $X(v)$ with $\St(v,\CT^\circ)$ as a subcomplex, and wish to prove that $\pi_1$ is injective on $\widehat{X}(v)=\St(x(v),\partial X(v))$. \newline 
(2) and (3) We collect the information needed to apply Proposition~\ref{prp:inj3}.} 
  \label{fig:inj}
\end{figure}

\begin{proof}
We wish to apply Proposition~\ref{prp:inj3} to prove that $\pi_1$ is injective on $\widehat{X}(v)$. All information on vertices not involving $x(v)$ needed for Proposition~\ref{prp:inj3} can be inferred from applying Proposition~\ref{prp:alignsymm} to $\partial\widehat{X}(v)=\partial \St(v,\CT^\circ)$. It remains to prove that 
\[\widehat{X}(v)\cap (\mathcal{C}_0\cup \mathcal{C}_2)=\emptyset,\ \pi_2(x(v))=\pi_2(u),\ \text{and}\ \pi_0(x(v))\in \rint[\pi_0(u),\pi_0(p)]\] to satisfy the requirements of Proposition~\ref{prp:inj3}. For the proof of these statements, we use the convention that for a $2$-dimensional subspace $H$ in~$S^3_\eq$ and a point $y$ not contained in $H$, $H^y$ denotes the component of the complement of $H$ in~$S^3_\eq$ that contains $y$.

Since $t\in\pi_2^{\operatorname{f}}(q)$ by Proposition~\ref{prp:alignsymm}(b) and $p\in \pi_2^{\SSp}(q)^v$ by Proposition~\ref{prp:alignsymm}(d), the remaining vertex $x(v)$ of the quadrilateral $\conv \{x(v),\,r,\,s,\,p\}$ must lie in $\pi_2^{\SSp}(r)^v$. Analogously, $x(v)$ lies in $\pi_2^{\SSp}(q)^v$. Furthermore, by Proposition~\ref{prp:alignsymm}(e), $t,s \in\pi_0^{\SSp}(u)^v$, so the remaining vertex $x(v)$ of the quadrilateral $\conv \{x(v),\,u,\,s,\,t\}$ is contained in $\pi_0^{\SSp}(u)^v$ as well. This suffices to complete the proof of the first two statements.

\begin{compactitem}[$\circ$]
\item Since all vertices of $X(v)$ except $u$, $s$ and $q$ lie in $\pi_2^{\SSp}(q)^v\cap\pi_0^{\SSp}(u)^v$, the $3$-cube $X(v)$ can only intersect $\mathcal{C}_0\cup \mathcal{C}_2$ if either $u$ lies in $\mathcal{C}_0\cup \mathcal{C}_2$, or $[s,q]$ intersects $\mathcal{C}_0\cup \mathcal{C}_2$. This is excluded by Proposition~\ref{prp:alignsymm}(a) since $u,\, [s,q]\in \St(v,\CT^\circ)\subset \CT^\circ$. Thus $X(v)\cap (\mathcal{C}_0\cup \mathcal{C}_2)=\emptyset$, in particular, $\widehat{X}(v)\cap (\mathcal{C}_0\cup \mathcal{C}_2)=\emptyset$.
\item By reflective symmetry at $\pi_2^{\SSp}(u)$, $x(v)$ lies in $\pi_2^{\SSp}(u)$, and we can conclude that \[x(v)\in \pi^{\operatorname{f}}_2(u)=\pi_2^{\SSp}(u)\cap \pi_2^{\SSp}(q)^v.\]
\end{compactitem}
To prove the last statement, i.e.\ that $\pi_0(x(v))\in \rint[\pi_0(u),\pi_0(p)]$, consider the midpoint $m$ of the segment $[s,t]$, and the triangle $\conv\{\pi_0(x(v)),\,\pi_0(m),\,m\}$ in $\pi^{\SSp}_2(v)$. The angle at the vertex $\pi_0(m)$ is a right angle, the angle $\delta$ at vertex $m$ is bounded above by $\pi-\alpha(\CT)<\nicefrac{\pi}{2}$ and the angle at $\pi_0(x(v))$ shall be labeled~$\zeta$. The second spherical law of cosines implies that the distance $d$ between $\pi_0(m)$ and $\pi_0(x(v))$ satisfies
\[\cos d\sin\zeta=\cos\delta\geq\cos\big(\pi-\alpha(\CT)\big).\] 
Since $\delta\leq\pi-\alpha(\CT)<\nicefrac{\pi}{2}$, we thus obtain \[d\leq\delta\leq\pi-\alpha(\CT)<\nicefrac{\pi}{2}.\]
By Proposition~\ref{prp:dict}, we have that $\pi_0(x(v))$ lies in $\rint[\pi_0(u),\pi_0(p)]$ if $\pi_0(x(v))$ lies in $\pi_0^{\SSp}(u)^v$
and in~$\pi_0^{\SSp}(p)^v$.
\begin{compactitem}[$\circ$]
\item $\pi_0(x(v))\in\pi_0^{\SSp}(u)^v$ was already proven explicitly.
\item The distance between $\pi_0(m)$ and $\pi_0(p)$ is $\nicefrac{\pi}{2}$, and the distance between $\pi_0(x(v))$ and $\pi_0(m)$ is smaller than $\nicefrac{\pi}{2}$, and we conclude that $x(v)\in\pi_0^{\SSp}(p)^m$. By Proposition~\ref{prp:alignsymm}(e), $\pi_0^{\SSp}(p)^m$ coincides with~$\pi_0^{\SSp}(p)^v$.
\end{compactitem}
\noindent Thus all conditions of Proposition~\ref{prp:inj3} are met by $\widehat{X}(v)$, and its application gives the injectivity of $\pi_1$ on $\widehat{X}(v)$.
\end{proof}

Since $\widehat{X}(v)\cap(\mathcal{C}_0\cup \mathcal{C}_2)=\emptyset$, the map $\pi_1$ is well-defined and continuous on $\widehat{X}(v)$. In particular, since $\pi_1(x(v))\in\pi_1(\St(v,\CT^\circ))$ and $\partial\widehat{X}(v)=\partial \St(v,\CT^\circ)$, we have:

\begin{cor}\label{cor:localem}
$\pi_1$ embeds $\widehat{X}(v)$ into $\pi_1(\St(v,\CT^\circ))\subset \mathcal{C}_1$. In particular, $\pi_1(\widehat{X}(v))= \pi_1(\St(v,\CT^\circ))$.
\end{cor}

We can now prove that attaching the facets $X(v)$ provides a polytopal complex, i.e.\ that the elementary extension exists, and that it is ideal as well.

\begin{proof}[\textbf{Proof of Theorem~\ref{thm:exts}}]
By combining Corollary~\ref{cor:localem} and Lemma~\ref{lem:dihang}, we see that $X(v){\setminus} \St(v,\CT^\circ)$ lies in the set \[K(v):=\bigcup_{x\in \operatorname{int}(\St(v,\CT^\circ))} [x,\pi_0(x)]\]
for all vertices $v\in \RR(\CT,k-2)$. The sets $K(\cdot)$ are disjoint subsets of the component of $S^3_\eq{\setminus} \CT$ containing $\mathcal{C}_0$ for different elements of $\RR(\CT,k-2)$ since $\pi_1$ is injective on $\CT^\circ$. Consequently, $X(v)\cap \CT$ coincides with~$\St(v,\CT^\circ)$, and the complex $\CT'$ obtained as the union
\[\CT':=\CT\cup\ \bigcup_{v\in \RR(\CT,k-2)} X(v)\]
is a polytopal complex, and consequently a CCT that extends $\CT$. To finish the proof of Theorem~\ref{thm:exts}, it remains to prove that $\CT'$ is ideal.
Symmetry is an immediate consequence of Lemma~\ref{lem:uniext} and the symmetry of $\CT$. Corollary~\ref{cor:localem}  shows that $\RR(\CT',[k-1,k+1])$ is embedded by $\pi_1$ into $\mathcal{C}_1$, thus $\CT'$ is transversal.  The fact that layer $k+1$ of $\CT'$ intersects the component containing $\mathcal{C}_0$ is immediate as well, since the segment $[x(v),\pi_0(x(v))]$ (where $x(v)$ is any layer $(k+1)$-vertex, cf.\ Lemma~\ref{lem:localem}) does not intersect $\CT'$ in any other point than $x(v)$, by transversality. Thus, to show that $\CT'$ is ideal it only remains to prove that 
\[
\alpha(\CT')=\alpha\big(\RR(\CT',[k-1,k+1])\big)>\nicefrac{\pi}{2}.
\] This, however, is immediate from Proposition~\ref{prp:slmono}. 
\end{proof}

\subsection{Convex position of the extension}\label{ssc:convex}

Now we are concerned with the proof of Theorem~\ref{thm:convp}. We will get it from the following proposition.

\begin{prp}\label{prp:convpext}
Let $\CT$ denote an ideal CCT in $S^4$ of width $k\geq 3$ in locally convex position, and let $\CT'$ denote the elementary extension of $\CT$. Then 
\begin{compactenum}[\bf(I)] 
\item \emph{[Convex position for the fattened boundary]} the subcomplex $\RR(\CT',[k-2,k+1])$ of $\CT'$ is in convex position, and
\item \emph{[Locally convex position]} for every vertex $v$ in $\CT'$, $\St(v,\CT')$ is in convex position.
\end{compactenum}
\end{prp}

\begin{proof}[Proposition~\ref{prp:convpext} implies Theorem~\ref{thm:convp}]
Let $\CT$ denote an ideal $k$-CCT in convex position in $S^4$, $k\geq 3$, and choose $\ell:= \max\{k+1, 5\}$. Let $\CT''$ denote the ideal $\ell$-CCT that is an extension of $\CT$. Then
\begin{compactitem}[$\circ$]
 \item $\CT''$ is a manifold with boundary,
 \item the elementary extension $\CT'$ of $\CT$ is a subcomplex of $\CT''$,
\item  $\RR(\CT'',[0,3])=\RR(\CT,[0,3])$ is in convex position by assumption,
\item $\RR(\CT'',[\ell-2,\ell+1])$ is in convex position by Proposition~\ref{prp:convpext}{\bf (I)}, and 
\item $\CT''$ is in locally convex position by Proposition~\ref{prp:convpext}{\bf(II)}.
\end{compactitem}
Thus, we can conclude with the Alexandrov--van Heijenoort Theorem for polytopal manifolds with boundary (Theorem~\ref{thm:locglowib}),
$\CT''$ is in convex position. Since $\CT'$ is a subcomplex of $\CT''$, the elementary extension $\CT'$ of $\CT$ is in convex position as well.
\end{proof}

The proof of Proposition~\ref{prp:convpext} will occupy us for the rest of this section. We work in $S^4\subset \R^5$, and start off with a tool to check the convex position of ideal CCTs of width $3$:

\begin{prp}[Convex position of ideal $3$-CCT]\label{prp:loccrt1lay}
Let $\CT$ be an ideal $3$-CCT in~$S^4$, such that for every facet $\sigma$ of $\CT$, there exists a closed hemisphere $H(\sigma)$ containing $\sigma$ in the boundary, and such that $H(\sigma)$ contains all remaining vertices of $\RR(\CT,1)$ connected to $\sigma$ via an edge of $\CT$ in the interior. Then each $H(\sigma)$ contains all vertices of $\F_0(\CT){\setminus} \F_0(\sigma)$ in the interior.
\end{prp}

This is a strong local-to-global statement for the convex position of ideal CCTs, since it reduces the decision whether an ideal CCT is or is not in convex position to the evaluation of a local criterion. The proof requires a detailed and somewhat lengthy discussion of several cases, and is thus given in the appendix, 
Section~\ref{ssc:localtoglobal}. Here is an immediate corollary. 

\begin{cor}\label{Cor:localglobal2}
Let $\CT$ be an ideal $3$-CCT in~$S^4$. Then $\CT$ is in locally convex position if and only if it is in convex position.
\end{cor}

Finally, a simple observation:

\begin{lemma}[Halfspace selection]\label{lem:righthalfspace}
Let $C$ denote a polytopal complex in a closed hemisphere of $S^d$ with center $x\in S^d$ that consists of only two $(d-1)$-dimensional facets $\sigma$, ${\tau}$ such that $\sigma$ and ${\tau}$ intersect in a $(d-2)$-face, and there exists a closed hemisphere $H(\sigma)$ exposing $\sigma$ in $C$. Then $C$ is in convex position. 

If, additionally, $H(\sigma)$ contains $x\in S^d$, and the hyperplane $\SSp (x\cup(\sigma\cap {\tau}))$ separates $\sigma$ and ${\tau}$ into different components, then the hemisphere $H({\tau})$ exposing ${\tau}$ in $C$ contains $x$ as well.
\end{lemma}

\begin{proof}[\textbf{Proof of Proposition~\ref{prp:convpext}}]
Consider first the case of (II) where $v$ is a vertex of layer $k-2$. The complex $\St(v,\CT)$ is in convex position by assumption. Let us prove that $\St(v,\CT')$ is in convex position as well. For this, denote by $X=X(v)$ the facet added in the extension from $\CT$ to $\CT'$ at the vertex $v$ of $\CT$. Set $\CT^+:=\RR(\CT',[k-3,k+1])$, and $\CT^-:=\RR(\CT',[k-3,k])$.
\smallskip

\noindent\textbf{(1)} \emph{$\St(v,\CT'),\ v\in\RR(\CT',k-2)$ is in convex position if $\St(v,\CT^+)$ is in convex position.}
\smallskip

$\St(v,\CT')$ is in convex position iff $\Lk(v,\CT')$ is in convex position, and by Theorem~\ref{thm:locglowib}, $\Lk(v,\CT')$ is in convex position if and only if it is in locally convex position, i.e.\ if for every edge $e$ of $\CT'$ containing $v$, $\Lk(e,\CT')$ is in convex position. This is already known for all edges that are not edges of ${X}$ since $\CT$ is in convex position. Thus $\St(v,\CT')$ is in convex position if $\Lk(e,\CT')$ is in convex position for edges in ${X}\cap \CT$ containing $v$.

We wish to reduce this further by showing the following:  By Lemma~\ref{lem:convglue}, $\Lk(e,\CT')$ is in convex position (for any edge $e\in X\cap \CT$ containing $v$) if $\Lk(e,\CT^+)$ is in convex position. There are two cases to consider:
\begin{compactitem}[$\circ$]
 \item If $\CT$ is a $3$-CCT, $\Lk(e,\CT^+)=\Lk(e,\CT')$, so the convex position of $\Lk(e,\CT^+)$ is clearly equivalent to the convex position of $\Lk(e,\CT')$.
 \item If $\CT$ is a $k$-CCT, $k\geq 4$, $\Lk(e,\CT)$ is a $1$-ball (and so is $\Lk(e,{X})$) in convex position in the $2$-sphere $\RN^1_e S^4$. Consequently, it follows from
 Lemma~\ref{lem:convglue}, applied to the pair of complexes $\Lk(e,\CT)$ and $\Lk(e,{X})$, that if $\Lk(e,\CT^+)$ is in convex position, then so is  $\Lk(e,\CT')$.
\end{compactitem}

\noindent Thus $\St(v,\CT')$ is in convex position if $\Lk(e,\CT^+)$ is in convex position for all edges of ${X}\cap \CT$ containing $v$. Since the convex position of $\St(v,\CT^+)$ clearly implies the convex position of $\Lk(e,\CT^+)$ for every edge of ${X}\cap \CT$ containing $v$, we have the desired statement.
\smallskip

\noindent\textbf{(2)} \emph{$\St(v,\CT^+),\ v\in\RR(\CT',k-2)$ is in convex position.}
\smallskip

Let $\sigma$ be any facet of $\St(v,\CT^-)$, and let $H(\sigma)$ denote the hemisphere exposing $\sigma$ in $\St(v,\CT)$. Let us first prove that
\[\F_0({X}){\setminus}\F_0(\sigma)\subset \intx H(\sigma),\]
i.e.\ $H(\sigma)$ exposes $\sigma$ in $\sigma\cup {X}$. Since $\St(v,\CT)$ is in convex position, it suffices to prove 
\[
	\F_0({X}){\setminus}\F_0(\St(v,\CT))\subset \intx H(\sigma).
\]
Every vertex $\F_0({X}){\setminus}\F_0(\St(v,\CT))$ (there is only one) lies in the interior of the convex cone $\Gamma$ with apex $v\in H(\sigma)$ and spanned by the vectors $w-v$,
\[
	w\in \F_0({X}\cap \RR(\CT,[k-1,k]))\in H(\sigma).
\]
Thus to show $\F_0({X}){\setminus}\F_0(\St(v,\CT))\subset \intx H(\sigma)$, it suffices to prove that the interior of $\Gamma$ does lie in $\intx H(\sigma)$, or equivalently that there exists one $w\in \F_0({X}\cap \RR(\CT,[k-1,k]))$ such that $w\in \intx H(\sigma)$. This however follows from the locally convex position of $\CT$.

\begin{figure}[htbf]
\centering 
 \includegraphics[width=0.35\linewidth]{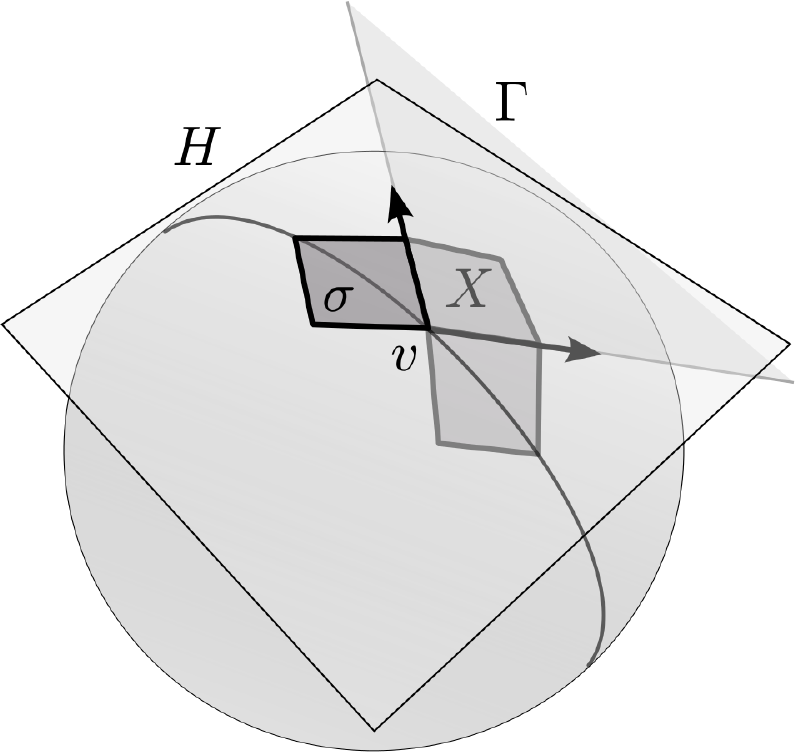} 
\caption{\small The facet ${X}$ added to the star of $v$ in the extension from $\CT$ to $\CT'$ lies $H(\sigma)$, since it is confined to a convex cone in $H(\sigma)$. The figure illustrates a $2$-dimensional analogue.} 
  \label{fig:localconv}
\end{figure}

It remains to find a hemisphere $H({X})$ exposing ${X}$ in $\St(v,\CT^+)$.  Every facet of $\St(v,\CT^-)$ intersects ${X}$ in a $2$-face, thus, by the preceding argument, for every facet $\sigma$ of $\St(v,\CT^-)$ there exists a hemisphere $H_\sigma({X})$ exposing ${X}$ in $\sigma\cup {X}$. It remains to prove that $H_\sigma({X})$ does not depend on $\sigma$, i.e.\ it remains to prove that for every choice of facets $\sigma'$, $\sigma''$ in $\St(v,\CT^-)$ we have \[H_{\sigma'}({X})=H_{\sigma''}({X}).\]

For this, we can use Lemma~\ref{lem:righthalfspace}: $\CT^-$ is in convex position, since it is an ideal $3$-CCT in locally convex position (Corollary~\ref{Cor:localglobal2}). We may assume, after possibly a reflection at $\Sp\{e_1,\, e_2,\, e_3,\, e_4\}\subset \R^5$, that $\CT$ is chosen in such a way $\CT^-$ lies in the hemisphere of~$S^4$ with center $e_5$, so that every hemisphere exposing a facet of it contains the point $e_5$, by symmetry of $\CT$. Furthermore, the orthogonal projection $\pp$ of $S^4{\setminus} \{\pm e_5\}$ to~$S^3_\eq$ is injective on $\CT'$ since $\CT'$ is ideal, so if ${X}$, $\sigma$ are adjacent facets of $\CT'$, then $\SSp (e_5\cup(\sigma\cap {X}))$ separates them into different components.
By Lemma~\ref{lem:righthalfspace}, the hemisphere $H_\sigma({X})$ that exposes ${X}$ in $\sigma\cup {X}$ is uniquely determined as the one containing $e_5$, independent of $\sigma$. In particular, $H({X})=H_\sigma({X})$ exposes ${X}$ in $\St(v,\CT^-)$, and consequently, $\St(v,\CT^+)$ is in convex position. 

As already observed, this proves that $\St(v,\CT')$ is in convex position. We can now complete the proofs of the claims \textbf{(I)} and \textbf{(II)} of Proposition~\ref{prp:convpext}.
\smallskip

\noindent\textbf{(I)}
If $H({X})$ is the hyperplane exposing ${X}$ in $\St(v,\CT')$ (where $v$ is the layer $(k-2)$-vertex of $\CT'$ in ${X}$, as above), then all vertices of $\St(v,\CT')$ that are not in ${X}$ lie in the interior of $H$. In particular, all vertices of layers $k-2$ of $\CT'$ being connected to ${X}$ via an edge lie in the interior of $H({X})$, since they lie in $\St(v,\CT')$. Thus $\RR(\CT',[k-2,k+1])$ lies in convex position by Proposition~\ref{prp:loccrt1lay}.
\smallskip

\noindent\textbf{(II)}
We have to prove that for every vertex $v$ of $\CT'$ the complex $\St(v,\CT')$ is in convex position: For vertices $\RR(\CT',k-2)$ we proved that $\St(v,\CT')$ is in convex position already. For vertices in $\RR(\CT',[0,k-3])$, this follows from the assumptions on $\CT$. Thus it remains to consider the case in which $v$ is a vertex of layer $k-1$, $k$ or $k+1$. If $v$ is in layer $k+1$, this is trivial, since in this case the star consists of a single facet. If $v$ is in layer $k-1$ or $k$, this follows from Theorem~\ref{thm:locglowib}, applied to the complex $\Lk(v,\CT')$: 
Since $\RR(\CT',[k-2,k+1])$ is in convex position, so is in particular $\Lk(v,\RR(\CT',[k-2,k+1]))$. Thus, by Theorem~\ref{thm:locglowib}, it remains to prove that $\Lk(v,\CT')$ is in locally convex position, i.e., to show that $\Lk(e,\CT')$ is in convex position for every $e$ edge of $\CT'$ containing $v$. There are, again, three cases to consider:
\begin{compactitem}[$\circ$]
\item Edges $e$ between vertices of layer $k-2$ and $k-1$ (for which the convex position of $\Lk(e,\CT')$ follows from the convex position of the stars of layer $k-2$ vertices),
\item edges between vertices of layer $k-1$ and $k$ (for which the convex position of $\Lk(e,\CT')$ follows from statement~\textbf{(I)}), and
\item edges between vertices of layers $k$ and $k+1$ (which, again, is a trivial case). \qedhere
\end{compactitem}
\end{proof}

\subsection{Proof of Theorem~\ref{mthm:Lowdim}}\label{ssc:pfmthm1}

Let $C$ denote a polytopal complex in $\R^d$. A polytope $P$ in $\R^d$ \Defn{induces a convex position realization} of~$C$ if the boundary complex of $P$ contains a subcomplex combinatorially equivalent to $C$. We generalize the notion of the realization space of a polytope to polytopal complexes. The \Defn{(convex) realization space} of $C$ is defined as
\[ \cR(C):=
\big\{V\in\R^{d \times f_0(C)}\ :\ \conv(V)\, \text{induces a convex position realization of }\, C \big\},
\]
and just like to the realization space of a polytope, the realization space of $C$ is a semi-algebraic set in~$\R^{d \times f_0(P)}$ and its \Defn{dimension} is well-defined. 
We will use the following elementary lemma.

\begin{lemma}\label{lem:incl}
Let $C$ denote a polytopal complex in convex position. Then \[ \cR(\conv C)\subseteq \cR(C).\]
\end{lemma}

Furthermore, we reformulate Lemma~\ref{lem:uniext} in a slightly stronger form:

\begin{lemma}[Unique Extension]\label{lem:uniextem}
Let $\CT$ denote a CCT of width at least $1$, and let $\CT'$ denote its extension. Then $\cR(\CT')$ embeds into $\cR(\CT)$.
\end{lemma}

\begin{proof}
Any pair of distinct $2$-faces in any convex position realization of $\CT$ are not coplanar. Consequently, by Lemma~\ref{lem:cubecmpl}, the inclusion of vertex sets induces an embedding of $\cR(\CT')$ into $\cR(\CT)$.
\end{proof}

\enlargethispage{-5mm}

We can finish the proof of the first main theorem.

\begin{proof}[\textbf{Proof of Theorem~\ref{mthm:Lowdim}}]
Let $\PS[n]$, $n\ge4$, denote the CCT of width $n$ extending the $3$-CCT $\PS[3]$, which exists by Theorem~\ref{thm:ext} and is in convex position by Theorem~\ref{thm:convp}. 
We define
\[\operatorname{CCTP}_4[n]:=\conv \PS[n].	\]
This is the family announced in Theorem~\ref{mthm:Lowdim}:
\begin{compactitem}[$\circ$]
\item As $\PS[n]$ is in convex position, it lies in some open hemisphere of $S^4$; thus, the sets $\operatorname{CCTP}_4[n]$ are polytopes.
\item The polytopes $\operatorname{CCTP}_4[n]$ are of dimension $4$ by construction.
\item The polytopes $\operatorname{CCTP}_4[n]$ are combinatorially distinct: \[f_0(\operatorname{CCTP}_4[n])=f_0(\PS[n])=12(n+1).\]
\item By Lemma~\ref{lem:incl}, $\cR\operatorname{CCTP}_4[n])$ embeds into $\cR(\PS[n])$, which in turn embeds into the realization space $\cR(\PS[1])$ by Lemma~\ref{lem:uniextem}. It remains to estimate the dimension of the last space:
\[\dim \cR\left(\operatorname{CCTP}_4[n])\right)\leq \dim\cR(\PS[n]) \le\dim\cR(\PS[1]) \leq 4f_0(\PS[1])=96. \qedhere\]
\end{compactitem}
\end{proof}

\section{Many projectively unique polytopes in fixed dimension}\label{sec:projun}

We will now construct infinitely many projectively unique polytopes in $S^{69}$, thus proving Theorem~\ref{mthm:projun}.

We do so by by using the notion of \Defn{weak projective triples} developed in the last chapter, together with a lifting technique for weak projective triples (Corollary~\ref{cor:wpt}). In Section~\ref{ssc:prun} we provide an analogue of Corollary~\ref{cor:wpt} for the spherical setting (Lemma~\ref{lem:subdcs}), and use it to prove Theorem~\ref{mthm:projun}. A crucial part of this proof, Lemma~\ref{lem:affine}, is stated in Section~\ref{ssc:prun}, but the details of its verification are delayed to Section~\ref{ssc:constr}.

\subsection{Proof of Theorem~\ref{mthm:projun}}\label{ssc:prun}

We recall the basics for projectively unique point configurations and weak projective triples for the spherical setting, compare this also to Sections~\ref{sec:ppc} and~\ref{sec:wpt} for the analogous euclidean setting.

\begin{definition}[Point configurations, Lawrence equivalence and projective uniqueness in $S^d$]
A \Defn{point configuration} in $S^d$ is a finite collection $R$ of points in some open hemisphere of~$S^d$. We use $H_-$ and $H_+$ to denote the hemispheres delimited by an oriented hyperplane $H$ in $S^d$. Two point configurations $R$, $R'$ in $S^d$ are \Defn{Lawrence equivalent} if there is a 
bijection $\varphi$ between $R$ and $R'$, 
such that, if $H$ is any oriented hyperplane in~$S^d$, then there exists an oriented hyperplane $H'$ in $S^d$ satisfying
\[
\varphi(R\cap H_+)={R'}\cap H'_+, \qquad
\varphi(R\cap H_-)={R'}\cap H'_-.
\]
A point configuration $R$ in~$S^d$ is \Defn{projectively unique} if for any point configuration $R'$ in~$S^d$ Lawrence equivalent to it, and every bijection $\varphi$ that induces the Lawrence equivalence, there is a projective transformation $T$ that realizes $\varphi$.
\end{definition}

\begin{definition}[Framed polytopes] A polytope $P$ in $S^d$ is \Defn{framed} by a subset $Q$ of its vertex-set if the geometric realization of $Q$ determines $P$. More accurately, let $P'$ be any polytope in $S^d$ combinatorially equivalent of $P$, and let $\varphi$ denote the labeled isomorphism from the faces of $P$ to the faces of $P'$. If $\varphi(q)=q$ for all elements of $Q\subset \F_0(P)$ implies $P=P'$, then $P$ is \Defn{framed} by $Q$.\end{definition}

\begin{definition}[Weak projective triple in $S^d$]\label{def:wpts}
A triple $(P,Q,R)$ of a polytope $P$ in $S^d$, a subset $Q$ of $\F_0(P)$ and a point configuration $R$ in $S^d$ is a \Defn{weak projective triple} in $S^d$ if and only if
\begin{compactenum}[\rm(a)]
\item the three sets are contained in some open hemisphere of $S^d$,
\item $(\emptyset, Q \cup R)$ is a projectively unique point configuration,
\item $Q$ frames the polytope $P$, and
\item some subset of $R$ spans a hyperplane $H$, called the \Defn{wedge hyperplane}, which does not intersect $P$.
\end{compactenum}
\end{definition}

With these notions, we have the following Lemma.

\begin{lemma}\label{lem:subdcs}
If $(P,Q,R)$ is a weak projective triple in $S^d$, then there exists a projectively unique polytope of dimension $\dim P + f_0(Q)+ f_0(R) +1$ on $f_0(P)+2f_0(Q)+2f_0(R)+1$ vertices that contains a face projectively equivalent to $P$.
\end{lemma}

\begin{proof}
 Use a central projection to transfer $(P,Q,R)$ to a weak projective triple in $\R^d$, and apply Corollary~\ref{cor:wpt}.
\end{proof}

It remains to show that a weak projective triple based on the pair $(\operatorname{CCTP}_4[n],\F_0(\PS[1]))$ exists.

\begin{lemma}\label{lem:affine}
There is a point configuration $R$ of $40$ points in $S^4$ such that for all $n\geq 1$, the triple $(\operatorname{CCTP}_4[n],\F_0(\PS[1]),R)$ is a weak projective triple.
\end{lemma}

We defer the construction of $R$ to the next section. Using Lemma~\ref{lem:affine}, we can now prove Theorem~\ref{mthm:projun}.

\begin{proof}[\textbf{Proof of Theorem~\ref{mthm:projun}}]
Consider, for any $n\geq 1$, the polytope $\operatorname{PCCTP}_{69}[n]$, given by applying Lemma~\ref{lem:subdcs} to the weak projective triple $(\operatorname{CCTP}_4[n],\F_0(\PS[1]),R)$.

\begin{compactitem}[$\circ$]
\item The polytopes $\operatorname{PCCTP}_{69}[n]$ are projectively unique by Lemma~\ref{lem:subdcs}.
\item The polytopes $\operatorname{PCCTP}_{69}[n]$ are combinatorially distinct: We have
\[f_0(\operatorname{PCCTP}_{69}[n])=f_0(\operatorname{CCTP}_{4}[n])+2 f_0(R) + 2 f_0(\PS[1]) +1=12(n+1)+129.\] 
\item The dimension of the polytopes $\operatorname{PCCTP}_{69}[n]$ is given by 
\[\dim \operatorname{PCCTP}_{69}[n]=\dim \operatorname{CCTP}_{4}[n] +  f_0(R) +  f_0(\PS[1]) +1=69. \qedhere \]
\end{compactitem} 
\end{proof}

\subsection{Construction of a weak projective triple}\label{ssc:constr}

To prove Lemma~\ref{lem:affine}, we construct a projectively unique point configuration $K:=R\cup \F_0(\PS[1])$ in the closure of $S^4_{+}$. For this we start with a projectively unique PP configuration (the vertex set of a product of two triangles $\Delta_2 \times \Delta_2$) and extend it to the desired point configuration $K$ step by step. We do so in such a way that the ultimate point configuration constructed is determined uniquely (up to Lawrence equivalence) from the vertices of $\Delta_2 \times \Delta_2$.

\smallskip

The following three notational remarks should help the reader navigate through the construction:
\begin{compactitem}[$\circ$]
\item Coordinates for vertices in $S^4_+$ are given in homogeneous coordinates.
\item We start the construction by giving coordinates to the vertex set of $P_8=\Delta_2 \times \Delta_2$. Then we construct another copy $\widetilde{P}_8$ of $\Delta_2 \times \Delta_2$ antipodal to the first. This gives us the $6$-fold symmetry that is inherent in the vertex set $\F_0(\PS[1])$. After this, we will construct the remaining points of $K$ from the vertices of $P_8$ and of $\widetilde{P}_8$ with little interdependence between the two constructions. Thus, we mark points that correspond to $\widetilde{P}_8$ with a tilde. 
\item The symbols $+$ and $-$ will be used to signify parity in the first and second coordinate. A pre-superscript ${\circ}$ will indicate that a point has zero first coordinate function, but nonzero second coordinate function, where there exists a corresponding point with nonzero first coordinate function, but zero second coordinate function.
\item The first three construction steps are direct, in the sense that we obtain new points as intersections of subspaces spanned by points already obtained. From step four on, our construction steps are more involved; the point configurations depend on a parameter $\lambda$ whose value will be determined at the end.
\end{compactitem}

\medskip

\noindent \textbf{I. \emph{The framework, (a).}} The polytope $\Delta_2 \times \Delta_2$ in $S^4$ is projectively unique (cf.\ Shephard's List polytope $P_8$ in Section~\ref{ssc:Shphrdlist}), and so is consequently its vertex set. We choose as vertices of $\Delta_2 \times \Delta_2$
    \vskip -26pt
\begin{align*}
      \parbox{5.5cm}{\begin{align*}
        a_1^\pm&:= \big(\pm 1,\, 0 ,\,-2 ,\,0,\, 1\big),\\
        {^\circ}a_1^+&:= \big( 0,\, 1 ,\,-2 ,\,0,\, 1\big),
      \end{align*}} 
      \parbox{5.5cm}{\begin{align*}
        a_2^\pm&:=  \big(\pm 1,\, 0 ,\, 1 ,\, \sqrt{3},\, 1\big),\\
        {^\circ}a_2^+&:=  \big( 0,\, 1 ,\, 1 ,\, \sqrt{3},\, 1\big),
      \end{align*}}
        \parbox{5.5cm}{\begin{align*}
        a_3^\pm&:=  \big(\pm 1,\, 0 ,\, 1 ,\,-\sqrt{3},\, 1\big),\\
        {^\circ}a_3^+&:=  \big( 0,\, 1 ,\, 1 ,\,-\sqrt{3},\, 1\big).
      \end{align*}}
\end{align*}
    \vskip -13pt
\noindent For $1\leq i\neq j\leq 3$, let us denote by $b_{ij}=b_{ji}$ the intersection of $\SSp\{a_i^+,\, a_j^-\}$ and $\SSp\{a_i^-,\, a_j^+\}$ in $S^d_+$, and let ${^\circ}a_i^-$ denote the intersection point of the affine $2$-plane $\SSp\{ a_i^+,\, a_i^-,\,{^\circ}a_i^{+}\}$  with $\SSp\{b_{ij},\, {^\circ}a_j^+\}$. This yields
\vskip -26pt
 \begin{align*}
      \parbox{5.5cm}{\begin{align*}
        b_{23}&:= \big(0,\,0 ,\,1 ,\,0,\, 1\big),\\
        {^\circ}a_1^-&:= \big(0,\,- 1  ,\,-2 ,\,0,\, 1\big),
      \end{align*}} 
      \parbox{5.5cm}{\begin{align*}
        b_{13}&:=  \big(0,\,0,\, -\tfrac{1}{2} ,\, -\tfrac{\sqrt{3}}{2},\, 1\big),\\
        {^\circ}a_2^-&:=  \big(0,\,- 1 ,\, 1 ,\, \sqrt{3},\, 1\big),
      \end{align*}}
        \parbox{5.5cm}{\begin{align*}
        b_{12}&:=  \big(0,\,0,\, -\tfrac{1}{2} ,\, \tfrac{\sqrt{3}}{2},\, 1\big),\\
        {^\circ}a_3^-&:=  \big(0,\,- 1  ,\, 1 ,\,-\sqrt{3},\, 1\big).
      \end{align*}}
    \end{align*}
        \vskip -13pt
\noindent Let us denote by $b_0$ the intersection of the $2$-planes  $\SSp\{a_i^+,a_j^-,a_k^-\}$, $i,\, j,\, k \in \{1,2,3\},\, i\neq j\neq k\neq i$ in~$S^d_+$. It is not hard to check that $b_0=(0,\,0,\,0,\,0,\,1)$.
\medskip

\noindent \textbf{II. \emph{The framework, (b).}} For $1\leq i\neq j\leq 3$, let us denote by $k$ the last remaining index $\{1,2,3\}{\setminus} \{i,j\}$, and let $\widetilde{a}_{k}^+$ denote the intersection of the $2$-plane $\SSp\{a_1^+,\,a_2^+,\,a_3^+\}$ with the line $\SSp\{b_0,\, a_{ij}^-\}$ in $S^d_+$.

 We obtain, 
     \vskip -26pt
    \begin{align*}
      \parbox{5.5cm}{\begin{align*}
       \widetilde{a}_1^+&:= \big( 1,\, 0 ,\,2 ,\,0,\, 1\big),
      \end{align*}} 
      \parbox{5.5cm}{\begin{align*}
       \widetilde{a}_2^+:=  \big(1,\, 0 ,\, -1 ,\, -\sqrt{3},\, 1\big),
      \end{align*}}
        \parbox{5.5cm}{\begin{align*}
        \widetilde{a}_3^+:=  \big(1,\, 0 ,\, -1 ,\,\sqrt{3},\, 1\big).
      \end{align*}}
    \end{align*}
     \vskip -13pt
\noindent Analogously, we obtain
     \vskip -26pt
    \begin{align*}
      \parbox{5.5cm}{\begin{align*}
       \widetilde{a}_1^-&:= \big( -1,\, 0 ,\,2 ,\,0,\, 1\big),\\
        {^\circ} \widetilde{a}_1^\pm&:= \big( 0,\, \pm 1  ,\,2 ,\,0,\, 1\big),
      \end{align*}} 
      \parbox{5.5cm}{\begin{align*}
        \widetilde{a}_2^-&:=  \big( -1,\, 0 ,\, -1 ,\,-\sqrt{3},\, 1\big),\\
        {^\circ} \widetilde{a}_2^\pm&:=  \big( 0,\,\pm 1  ,\, -1 ,\,-\sqrt{3},\, 1\big),
      \end{align*}}
        \parbox{5.5cm}{\begin{align*}
       \widetilde{a}_3^-&:=  \big(-1,\, 0 ,\, -1 ,\,\sqrt{3},\, 1\big),\\
        {^\circ} \widetilde{a}_3^\pm&:=  \big(0,\,\pm 1 ,\, -1 ,\,\sqrt{3},\, 1\big)
      \end{align*}}
    \end{align*}
         \vskip -13pt
\noindent uniquely from the points already constructed. At this point, we have constructed the vertices of both polytopes ${P}_8$ and $\widetilde{P}_8$ (and more) from the initial point configuration.
\medskip

\noindent \textbf{III. \emph{The points of layer 1.}} Again for $1\leq i\neq j\leq 3$, we denote by $\psi_{k}^+$ the intersection point of the lines $\SSp\{a_{k}^+,\, \widetilde{a}_{k}^+\}$ with $\SSp\{a_i^+,\, a_j^+\}$ in $S^d_+$. We obtain
     \vskip -26pt
    \begin{align*}
      \parbox{5.5cm}{\begin{align*}
       \psi_{1}^+&:= \big( 1,\, 0 ,\,1 ,\,0,\, 1\big),
      \end{align*}} 
      \parbox{5.5cm}{\begin{align*}
       \psi_{2}^+&:=  \big(1,\, 0 ,\, -\tfrac{1}{2} ,\, -\tfrac{\sqrt{3}}{2},\, 1\big),
      \end{align*}}
        \parbox{5.5cm}{\begin{align*}
        \psi_{3}^+&:=  \big( 1,\, 0 ,\, -\tfrac{1}{2} ,\,\tfrac{\sqrt{3}}{2},\, 1\big).
      \end{align*}}
    \end{align*}
         \vskip -13pt
\noindent Analogously, we obtain 
     \vskip -26pt
    \begin{align*}
      \parbox{5.5cm}{\begin{align*}
       \psi_{1}^-&:= \big( -1,\, 0 ,\,1 ,\,0,\, 1\big),\\
        {^\circ}\widetilde{\psi}_{1}^\pm&:= \big( 0,\, \pm 1 ,\, -1 ,\,0,\, 1\big),
      \end{align*}} 
      \parbox{5.5cm}{\begin{align*}
        \psi_{2}^-&:=  \big(-1,\, 0 ,\, -\tfrac{1}{2} ,\, -\tfrac{\sqrt{3}}{2},\, 1\big),\\
        {^\circ}\widetilde{\psi}_{2}^\pm&:=  \big(0,\, \pm 1 ,\, \tfrac{1}{2} ,\, \tfrac{\sqrt{3}}{2},\, 1\big),
      \end{align*}}
        \parbox{5.5cm}{\begin{align*}
       \psi_{3}^-&:=  \big(-1,\, 0 ,\, -\tfrac{1}{2} ,\,\tfrac{\sqrt{3}}{2},\, 1\big),\\
        {^\circ}\widetilde{\psi}_{3}^\pm &:=  \big(0,\, \pm 1 ,\, \tfrac{1}{2} ,\,-\tfrac{\sqrt{3}}{2},\, 1\big)
      \end{align*}}
    \end{align*}
     \vskip -13pt
\noindent from the points already constructed. The points ${^\circ}\widetilde{\psi}_{i}^\pm$ and $\psi_{i}^\pm$, $i\in \{1,2,3\}$ together form $\RR(\PS[1],1)$, with  $\psi_{1}^+=\vartheta_1$, cf.\ Section~\ref{ssc:example}.
\medskip

\noindent \textbf{IV. \emph{Transition to layer 0.}} Define \[b_1:=\SSp\{a_1^+,\, a_1^-\}\cap\SSp\{{^\circ}a_1^+,\,{^\circ} a_1^-\}\cap S^d_+\]
Similarly, define \[b_{12}^{++}:=\SSp\{a_1^+,\,{^\circ} a_2^+\}\cap\SSp\{a_2^+,\,{^\circ} a_1^+\}\cap S^d_+\quad \text{and}\quad b_{12}^{+-}:=\SSp\{a_1^+,\,{^\circ} a_2^-\}\cap \SSp\{a_2^+,\,{^\circ} a_1^-\}\cap S^d_+.                                                                                                                                                                                                                 \]
We obtain
     \vskip -26pt
    \begin{align*}
      \parbox{5.5cm}{\begin{align*}
       b_1&:=\big(0,\,0,\,-2,\,0,\, 1\big),
      \end{align*}} 
      \parbox{5.5cm}{\begin{align*}
       b_{12}^{++}&:=\big(\tfrac{1}{2},\,\tfrac{1}{2},\,-\tfrac{1}{2},\,\tfrac{\sqrt{3}}{2},\, 1\big),
      \end{align*}}
        \parbox{5.5cm}{\begin{align*}
        b_{12}^{+-}&:=\big(\tfrac{1}{2},\,-\tfrac{1}{2},\,-\tfrac{1}{2},\,\tfrac{\sqrt{3}}{2},\, 1\big).
      \end{align*}}
    \end{align*}
         \vskip -13pt

\noindent Let us now consider any set of distinct points $\{\omega_a,\, \omega_b,\,\omega_c,\,\omega_d\}$ in $S^d_+$, cf.\ Figure~\ref{fig:quad}, such that
\begin{compactitem}[$\circ$]
\item all the points lie in the affine $2$-plane $\SSp\{a_1^+,\, a_1^-,\,{^\circ}a_1^+,\,{^\circ} a_1^-\}$,
\item $\omega_a$, $\omega_c$ lie in the intersection $X$ of $\SSp\{b_1,\, b_{12},\, b_{12}^{++}\}$ with $\SSp\{a_1^+,\, a_1^-,\,{^\circ}a_1^+,\,{^\circ} a_1^-\}$,
\item $\omega_b$, $\omega_d$ lie in the intersection $Y$ of $\SSp\{b_1,\, b_{12},\, b_{12}^{+-}\}$ with $\SSp\{a_1^+,\, a_1^-,\,{^\circ}a_1^+,\,{^\circ} a_1^-\}$,
\item the intersection point $\omega_{ad}$ of $\SSp\{\omega_a, a_1^-\}$ with 
$\SSp\{\omega_d, a_1^+\}$ and the intersection point $\omega_{bc}$ of $\SSp\{\omega_b, a_1^-\}$ with $\SSp\{\omega_c, a_1^+\}$) 
 lie in $\SSp\{{^\circ}a_1^+,\,{^\circ} a_1^-\}$,
\item the intersection point $\omega_{ab}$ of $\SSp\{\omega_a,{^\circ}a_1^-\}$ with $\SSp\{\omega_b, {^\circ}a_1^+\}$ lie in $\SSp\{a_1^+,\, a_1^-\}$, and
\item $\omega_a$ and ${^\circ}a_1^+$ lie in the same component of $\SSp\{a_1^+,\, a_1^-,\,{^\circ}a_1^+, \,{^\circ} a_1^-\}{\setminus} \SSp\{a_1^+,\, a_1^-\}$.
\end{compactitem}
\noindent The quadrilateral $W_1$ on points $\{\omega_a,\, \omega_b,\,\omega_c,\,\omega_d\}$ is a square; its barycenter coincides with $b_1$, but the dilation factor $\lambda$ of $W_1$ is not determined yet. Since the vertices of $W_1$ are of the form $\big(s_3 \lambda,\, s_4 \lambda,\,-2,\,0,\, 1\big)$, where $\lambda>0$ and $s_3,s_4\in\{+,-\}$,
we relabel \[\omega_1^{+,+}:=\omega_a,\ \ \omega_1^{+,-}:=\omega_b,\ \ \omega_1^{-,-}:=\omega_c\ \ \text{and}\ \omega_1^{-,+}:=\omega_d\]
with
$\omega_1^{s_3,s_4}:=\big(s_3 \lambda,\, s_4 \lambda,\,-2,\,0,\, 1\big),\  \lambda>0,\, s_3,s_4\in\{+,-\}.$
\noindent Call this PP configuration $K'_\lambda$. Before we turn to the issue of fixing $\lambda$, let us construct analogous quadrilaterals $W_2$ and $W_3$ and $\widetilde{W}_1$, $\widetilde{W}_2$ and $\widetilde{W}_3$.
\medskip

\begin{figure}[htbf]
\centering 
 \includegraphics[width=0.34\linewidth]{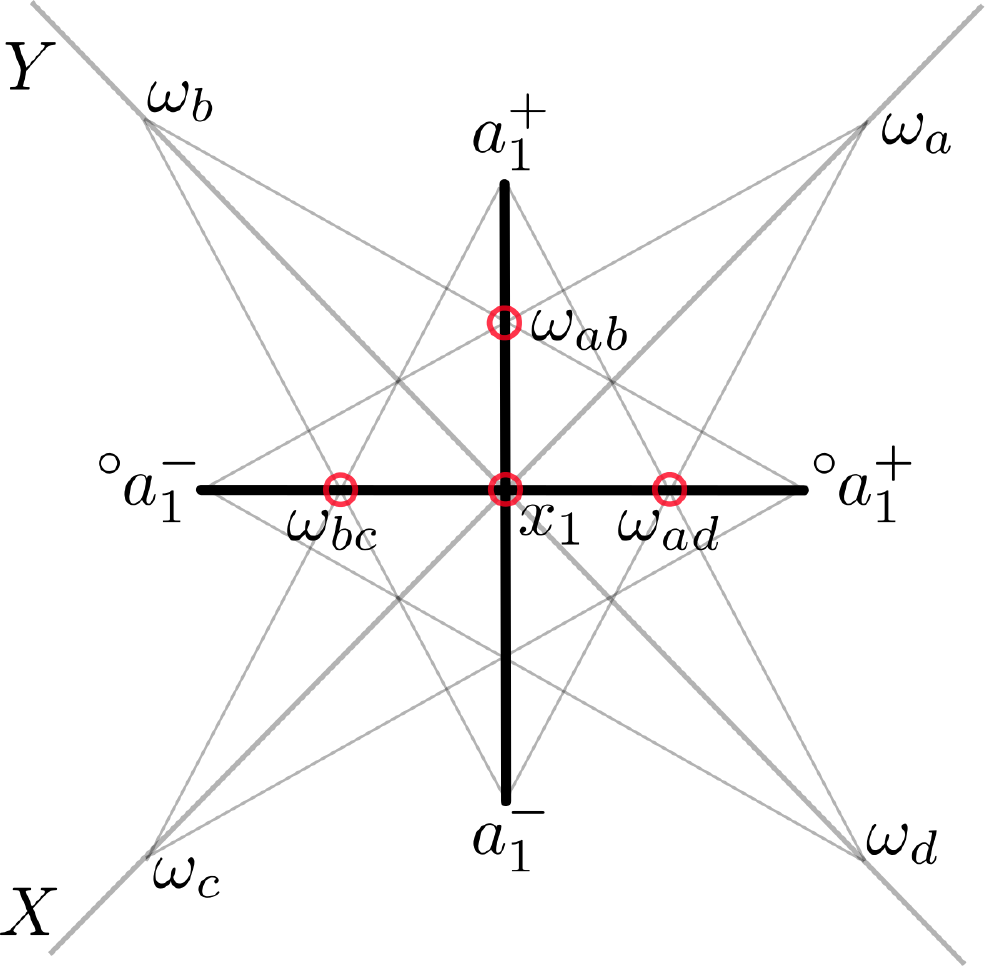} 
\caption{\small  The quadrilateral $W_1$ on vertices $\{\omega_a,\, \omega_b,\,\omega_c,\,\omega_d\}$ is, up to dilation, determined by the quadrilateral on vertices $\{a_1^+,\, a_1^-,\,{^\circ}a_1^+,\,{^\circ} a_1^-\}$ and the lines $X$ and $Y$ in $\SSp\{a_1^+,\, a_1^-,\,{^\circ}a_1^+,\,{^\circ} a_1^-\}$.} 
  \label{fig:quad}
\end{figure}

\noindent \textbf{V. \emph{The points of layer 0, (a).}} Let $i\in \{2,3\}$. Define, for any choice of signs $s_3$ and $s_4$, $s_3=s_4$, the points $\omega_i^{s_3,s_4}$ in $S^d_+$ by
\[\omega_i^{s_3,s_4}=\SSp\big\{a_i^+,\, a_i^-,\,{^\circ}a_i^+,\,{^\circ} a_i^-\big\}\cap \SSp\{a_1^+,\, a_i^+,\,\omega_1^{s_3,s_4}\}  \cap \SSp\{{^\circ}a_1^+,\, {^\circ}a_i^+,\,\omega_1^{s_3,s_4} \}\cap  S^d_+.\]
We obtain:
\[\omega_2^{s_3,s_4}:=\big(s_3 \lambda,\, s_4 \lambda,\,1,\,\sqrt{3},\, 1\big),\ \quad \omega_3^{s_3,s_4}:=\big(s_3 \lambda,\, s_4 \lambda,\,1,\,-\sqrt{3},\, 1\big),\  \lambda>0,\, s_3,s_4\in\{+,-\},\, s_3=s_4.\]

\begin{figure}[htbf]
\centering 
 \includegraphics[width=0.4\linewidth]{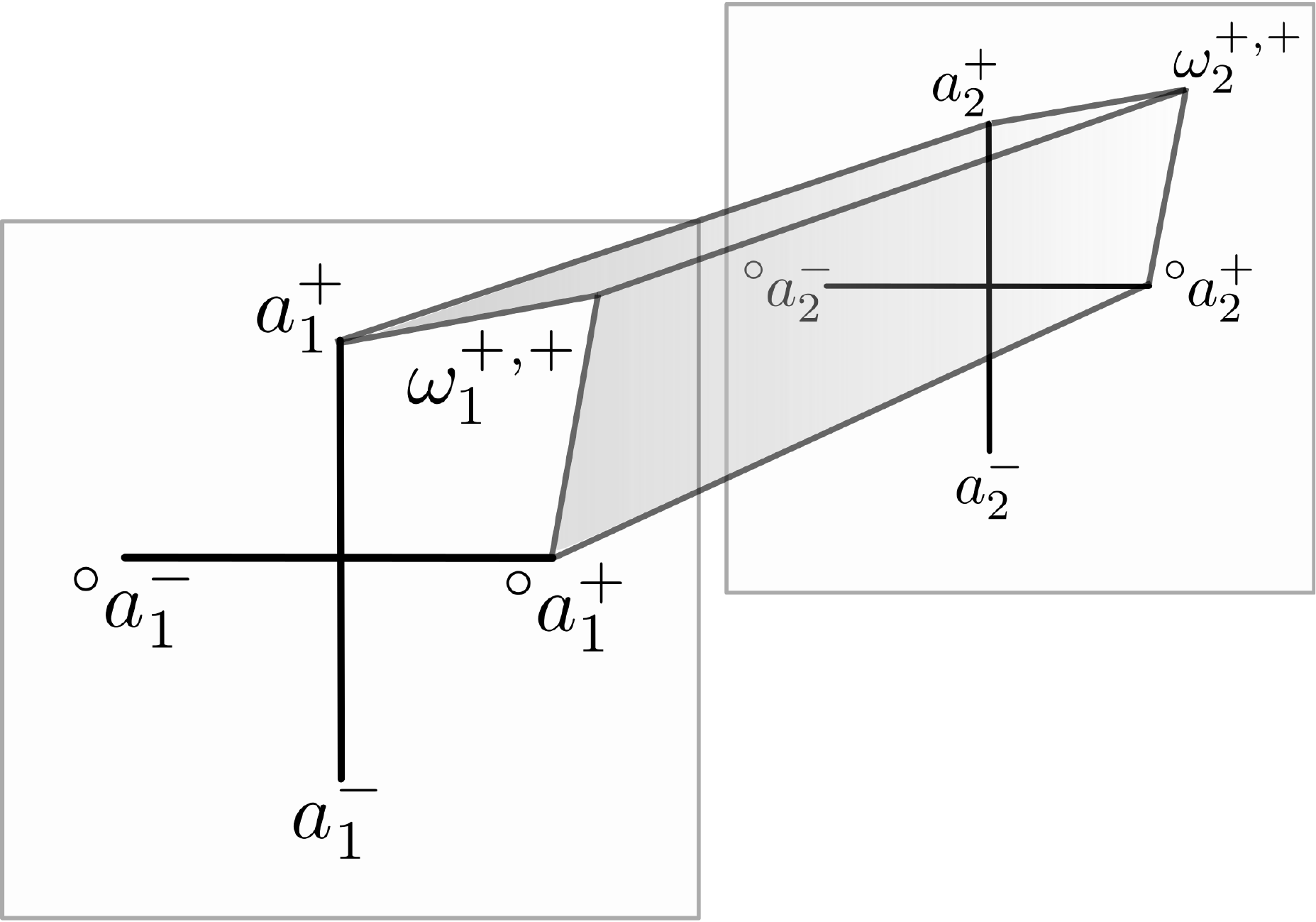} 
\caption{\small Transferring points between planes: We determine $\omega_2^{+,+}\in \SSp\{a_2^+,\, a_2^-,\,{^\circ}a_2^+,\,{^\circ} a_2^-\}$ from the points $a_1^+,\, a_1^-,\,{^\circ}a_1^+,\,{^\circ} a_1^-$, $\,a_2^+,\, a_2^-,\,{^\circ}a_2^+,\,{^\circ} a_2^-$ and $\omega_1^{+,+}\in \SSp\{a_1^+,\, a_1^-,\,{^\circ}a_1^+,\,{^\circ} a_1^-\}$.} 
  \label{fig:quadtrans}
\end{figure}

\noindent Similarly, for $i\in \{1,2,3\}$ and for any choice of signs $s_3$ and $s_4$ with $s_3\neq s_4$, define the points $\widetilde{\omega}_i^{s_3,s_4}$  in $S^d_+$ by
\[\widetilde{\omega}_i^{s_3,s_4}=\SSp\big\{\widetilde{a}_i^+,\, \widetilde{a}_i^-,\,{^\circ}\widetilde{a}_i^+,\,{^\circ} \widetilde{a}_i^-\big\}\cap \SSp\{a_1^+,\, \widetilde{a}_i^+,\,\omega_1^{s_3,s_4}\} \cap \SSp\{{^\circ}a_1^+,\, {^\circ}\widetilde{a}_i^+,\,\omega_1^{s_3,s_4} \}\cap  S^d_+ .\]
We have, for $\lambda>0$ and $s_3,s_4\in\{+,-\}$, $s_3\neq s_4$,
     \vskip -26pt
    \begin{align*}
      \parbox{5.5cm}{\begin{align*}
      \widetilde{\omega}_1^{s_3,s_4}&:=\big(s_3 \lambda,\, s_4 \lambda,\, 2,\,0 ,\, 1\big),
      \end{align*}} 
      \parbox{5.5cm}{\begin{align*}
       \widetilde{\omega}_2^{s_3,s_4}&:=\big(s_3 \lambda,\, s_4 \lambda,\,-1,\,-\sqrt{3},\, 1\big),
      \end{align*}}
       \parbox{5.5cm}{\begin{align*}
       \widetilde{\omega}_3^{s_3,s_4}&:=\big(s_3 \lambda,\, s_4 \lambda,\,-1,\,\sqrt{3},\, 1\big).
      \end{align*}}
    \end{align*}
         \vskip -13pt
\noindent The point configuration $K_{\lambda}$ constructed depends on a parameter $\lambda>0$.
\medskip

\noindent \textbf{VI. \emph{The points of layer 0, (b).}} The point set \[\Omega(\lambda):=\big\{\widetilde{\omega}^{\pm,\mp}_i:\ i \in \{1,2,3\}\} \cup\, \{\omega^{\pm,\pm}_i:\ i \in \{1,2,3\}\big\}\]
is obtained from the point 
\[\widetilde{\omega}^{+,-}_1=\big(\lambda,\,-\lambda,\,2,\,0,\, 1\big),\	 \lambda>0\]
as its orbit under $\mathfrak{R}$. Recall that the points of $\PS[1]$ are given as the orbits of $\vartheta_0$ and $\vartheta_1$ under the group of rotational symmetries $\mathfrak{R}\subset O(\R^5)$, respectively (cf.\ Section~\ref{ssc:example}). We need to determine the values of~$\lambda$ for which the point set 
\[\Omega'(\lambda):=\big\{\psi_{i}^\pm:\ i\in \{1,2,3\}\big\}\cup \big\{{^\circ}\widetilde{\psi}_{i}^\pm:\ i\in \{1,2,3\}\big\}\cup \Omega(\lambda)\]
is Lawrence equivalent to the point configuration $\F_0(\PS[1])$, where we label the points such that
\begin{itemize}[$\circ$]
 \item $(\rot_{1,2}\rot_{3,4})^l (\rot_{3,4})^{2m} \widetilde{\omega}^{+,-}_1$ corresponds to $(\rot_{1,2}\rot_{3,4})^l (\rot_{3,4})^{2m} \vartheta_0$ for all $m,l \in \N$. In particular, the point configuration $\Omega(\lambda)\subset\Omega'(\lambda)$ corresponds to the point configuration $\RR(\PS[1], 0)\subset \F_0(\PS[1]).$
 \item $(\rot_{1,2}\rot_{3,4})^l (\rot_{3,4})^{2m} \psi_{1}^+$ corresponds to $(\rot_{1,2}\rot_{3,4})^l (\rot_{3,4})^{2m} \vartheta_1$ for all $m,l \in \N$. In particular, the point configuration \[\{\psi_{i}^\pm:\ i\in \{1,2,3\}\}\cup \{{^\circ}\widetilde{\psi}_{i}^\pm:\ i\in \{1,2,3\}\}\subset\Omega'(\lambda)\] corresponds to the point configuration $\RR(\PS[1], 1)\subset \F_0(\PS[1])$.
\end{itemize}

\smallskip

\noindent This requirement on $\Omega'(\lambda)$ imposes strong conditions on $\lambda$:
The points  \[\vartheta_1,\ \rot_{1,2}\rot_{3,4}\vartheta_1,\ \rot_{1,2}\rot^{-1}_{3,4}\vartheta_1,\ \rot_{1,2}^2\vartheta_0,\ \rot_{1,2}\rot_{3,4}^{-1}\vartheta_0\ \text{and}\ \rot_{1,2}\rot_{3,4}\vartheta_0\]
of $\F_0(\PS[1])$ lie in a common hyperplane. In particular, the corresponding points \[\psi_{1}^+=\vartheta_1,\ \rot_{1,2}\rot_{3,4}\psi_{1}^+,\ \rot_{1,2}\rot^{-1}_{3,4}\psi_{1}^+,\ \rot_{1,2}^2\widetilde{\omega}^{+,-}_1=\widetilde{\omega}^{-,+}_1,\ \rot_{1,2}^{-1}\rot_{3,4}^{-1}\widetilde{\omega}^{-,+}_1\ \text{and}\ \rot_{1,2}^{-1}\rot_{3,4}\widetilde{\omega}^{-,+}_1\] 
of $\Omega'(\lambda)$ with coordinates
\[
\big(1,\,0,\,1,\,0,\, 1\big), \big(0,\,1,\,\tfrac{1}{2},\,\tfrac{\sqrt{3}}{2},\, 1\big), \big(0,\,1,\,\tfrac{1}{2},\,-\tfrac{\sqrt{3}}{2},\, 1\big),
\big(-\lambda,\,\lambda,\,2,\,0,\, 1\big), \big(\lambda,\,\lambda,\,1,\,\sqrt{3},\, 1\big)\ \text{and}\ \big(\lambda,\,\lambda,\,1,\,-\sqrt{3},\, 1\big)\]
 must lie in a common hyperplane. In particular, $\lambda$ must satisfy
\[\lambda^2+2\lambda-1=0\Longleftrightarrow \lambda=\pm \sqrt{2}-1.\]
Since $\lambda>0$, the only choice for $\lambda$, assuming $\F_0(\PS[1])'(\lambda)$ is Lawrence equivalent to $\F_0(\PS[1])$, is \[\lambda= \sqrt{2}-1.\]
In particular, $K_{\sqrt{2}-1}$ is projectively unique, and it contains $\F_0(\PS[1])$. 
\medskip

\noindent \textbf{VII. \emph{Four points at infinity}.}
No subset of $K_{\sqrt{2}-1}$ spans a hyperplane that does not intersect the polytopes $\operatorname{CCTP}_{4}[n]$, but there is an easy way to obtain such a point set from $K_{\sqrt{2}-1}$:
Denote the intersection of $\SSp\{a_1^+,\,a_2^+\}$ with $\SSp\{a_1^-,\,a_2^-\}$ (in our case, this intersection point lies in the equator $S^3_\eq$) by $\infty_1$. Similarly, denote the intersection of $\SSp\{a_3^+,\,a_2^+\}$ with $\SSp\{a_3^-,\,a_2^-\}$ by $\infty_2$, the intersection of $\SSp\{a_1^+,\,a_1^-\}$ with the ray $\SSp\{a_2^+,\,a_2^-\}$ by $\infty_3$, and the intersection of $\SSp\{{^\circ}a_1^+,\,{^\circ}a_1^-\}$ with  $\SSp\{{^\circ}a_2^+,\,{^\circ}a_2^-\}$ by $\infty_4$. The points $\infty_1,\, \cdots,\ \infty_4$ span $S^3_\eq$. Denote the union of these four points with $K_{\sqrt{2}-1}$ by $K$, and set $R=K{\setminus} \F_0(\PS[1])$. For the proof of Lemma~\ref{lem:subdcs}, we now only have to examine~$K$.

\begin{proof}[\textbf{Proof of Lemma~\ref{lem:affine}}]
First, we note that $f_0(K)=64$ and $f_0(R)=40$. It remains to verify that $(\operatorname{CCTP}_4[n],\F_0(\PS[1]),R)$ is a weak projective triple:
\begin{compactenum}[\rm(a)]
\item All points of $K_{\sqrt{2}-1}\subset K$ lie in $S^4_+$, as well as the polytopes $\operatorname{CCTP}_4[n]$. The $4$ points $\infty_1,\,\infty_2,\, \infty_3,\ \infty_4$ form the vertices of a tetrahedron in $S^3_\eq$. By perturbing the closure of $S^4_{+}$ a little, we can find an open hemisphere that contains the points of $K$ and the polytopes $\operatorname{CCTP}_4[n]$ in the interior.
\item $K=R\cup \F_0(\PS[1])$ is projectively unique, since it is uniquely determined from the projectively unique PP configuration $\F_0(\Delta_2\times \Delta_2)$ up to Lawrence equivalence.
\item $\F_0(\PS[1])$ frames $\operatorname{CCTP}_{4}[n]$ for all $n$. See also Lemma~\ref{lem:uniextem}.
\item Since the polytopes $\operatorname{CCTP}_4[n]$ lie in the interior of the upper hemisphere, the hyperplane given as $\SSp\{\infty_1,\, \cdots,\ \infty_4\}$ does not intersect any of the cross-bedding cubical torus polytopes $\operatorname{CCTP}_4[n]$. \qedhere
\end{compactenum}
\end{proof}

\section{Appendix}\label{sec:app}

\subsection{Convex position of polytopal complexes}\label{sec:convps}

In this section, we establish a polyhedral analogues of the Alexandrov--van Heijenoort Theorem, a classical result in the theory of convex hypersurfaces.

\begin{theorem}[{\cite[Main Theorem]{Heij}}]\label{thm:heijor}
Let $M$ be an immersed topological $(d-1)$-manifold without boundary in $\R^d,\, d\geq 3$, such that
\begin{compactenum}[\rm(a)]
\item $M$ is complete with respect to the metric induced on $M$ by the immersion,
\item $M$ is connected,
\item $M$ is locally convex at each point (that is, there exists a neighborhood, with respect to the topology induced by the immersion, of every point $x$ of $M$ in which $M$ coincides with the boundary of some convex body~$K_x$), and
\item $M$ is strictly convex in at least one point (that is, for at least one $x\in M$, there exists a hyperplane $H$ intersecting $K_x$ only in $x$).
\end{compactenum}
\noindent Then $M$ is embedded, and it is the boundary of a convex body. 
\end{theorem}

This remarkable theorem is due to Alexandrov~\cite{Aleks} in the case of surfaces. Alexandrov did not state it explicitly (his motivation was to prove far stronger results on intrinsic metrics of surfaces), and his proof does not extend to higher dimensions. Van Heijenoort also proved Theorem~\ref{thm:heijor} only for surfaces, but his proof extends to higher dimensions. Expositions of the general case of this theorem and further generalizations are available in~\cite{TrudingerWang} and~\cite{Rybnikov}. In this section, we adapt Theorem~\ref{thm:heijor} to polytopal manifolds with boundary. We start off by introducing the notion of immersed polytopal complexes.

\begin{definition}[Precomplexes] Let $C$ denote an abstract polytopal complex, and let $f$ denote an immersion of $C$ into $\R^d$ (resp.\ $S^d$) with the property that $f$ is an isometry on every face $\sigma$ of $C$. Then $f(C)$ is called a \emph{precomplex} in $\R^d$ (resp.\ $S^d$). While not necessarily a polytopal complex, it is not hard to see that for every face $\sigma$ of $C$, $f(\St(\sigma,C))$ is a polytopal complex combinatorially equivalent to $\St(\sigma,C)$. 

A polytopal complex $C$, abstract or geometric, is a \Defn{(PL) $d$-manifold} if for every vertex $v$ of $C$, $\St(v,C)$ is PL-homeomorphic to a $d$-simplex~\cite{RourkeSanders}. If $C$ is a manifold, then $f(C)$ is called a \Defn{premanifold}. If $\sigma$ is a face of $C$, then $f(\sigma)$ is a \Defn{face} of the polytopal precomplex $f(C)$. The complexes $\St(f(\sigma),f(C))$ resp.\ $\Lk(f(\sigma),f(C))$ are defined to be the polytopal complexes $f(\St(\sigma,C))$ and $\Lk(f(\sigma), f(\St(\sigma,C)))$ respectively. 
\end{definition}

\begin{definition}[The gluing of two precomplexes along a common subcomplex]
Consider two precomplexes $f(C)$, $g(C')$, and assume there are subcomplexes $D$, $D'$ of $C$, $C'$ respectively with $f(D)=g(D')$ such that for every vertex $v$ of $f(D)=g(D')$, $\St(v,f(C))\cup\St(v,g(C'))$ is a polytopal complex. Let $\Sigma$ be the abstract polytopal complex given by identifying $C$ and $C'$ along the map $g^{-1}\circ f:D\mapsto D'$, and let $s$ denote the immersion of $\Sigma$ defined as 
\[s(x):=\left\{\begin{array} {cl} 
f(x) &\ \text{for }x\in C,    \\
g(x)&\ \text{for }x\in C'. \end{array}\right.\]
Then the \Defn{gluing} of $f(C)$ and $g(C')$ at $f(D)=g(D')$, denoted by $f(C)\sqcup_{f(D)} g(C')$, is the polytopal precomplex obtained as the image of $\Sigma$ under $s$. 
\end{definition}

For the rest of the section, we will use the term \Defn{halfspace} in $S^d$ synonymously with hemisphere in $S^d$. 

\begin{definition}[Locally convex position and convex position for polytopal (pre-)complexes]
A pure polytopal (pre-)complex $C$ in $S^d$ or in $\R^{d}$ is \Defn{in convex position} if one of the following three equivalent conditions is satisfied:
\begin{compactitem}[$\circ$]
\item For every facet $\sigma$ of $C$, there exists a closed halfspace $H(\sigma)$ containing $C$, and such that $\partial H(\sigma)$ contains the vertices of $\sigma$ but no other vertices of $C$. We say that such a halfspace $H(\sigma)$ \Defn{exposes} $\sigma$ in $C$.
\item Every facet is exposed by some linear functional, i.e.\ there exists, for every facet $\sigma$ of $C$, a vector $\n(\sigma)$ such that the points of $\sigma$ maximize the linear functional $\langle\n,x\rangle$ among all points $x\in C$. For $C\subset S^d$, we additionally demand $\langle\n(\sigma),x\rangle =0$ for all $x$ in $\sigma$.
\item $C$ is a subcomplex of the boundary complex of a convex polytope. 
\end{compactitem}
\noindent Likewise, a polytopal (pre-)complex $C$ in $S^d$ or $\R^{d}$ is \Defn{in locally convex position} if for every vertex $v$ of $C$, 
the link $\Lk(v,C)$, seen as a subcomplex in the $(d-1)$-sphere $\RN^1_v \R^d$ resp.\ $\RN^1_v S^d$, is in convex position. As $\Lk(v,C)$ is in convex position if and only if $\St(v,C)$ is in convex position, $C$ is in locally convex position iff $\St(v,C)$ is in convex position for every vertex $v$ of $C$.
\end{definition}

It is obvious that ``convex position'' implies ``locally convex position.'' The Alexandrov--van Heijenoort Theorem provides instances where this observation can be reversed. We start with a direct analogue of the Theorem~\ref{thm:heijor} for precomplexes. Notice that a precomplex without boundary in locally convex position is locally convex at every point and strictly convex at every vertex in the sense of Theorem~\ref{thm:heijor}.

\begin{theorem}[AvH for closed precomplexes]\label{thm:heij}
Let $C$ be a $(d-1)$-dimensional connected closed polytopal premanifold in $\R^d$ or $S^d$, $d\geq 3$, in locally convex position. Then $C$ is in convex position.
\end{theorem}

\begin{proof}
Let us first prove the euclidean case: The metric induced on $C$ is complete because $C$ is finite, and $C$ is locally convex at each point since $C$ is in locally convex position. Furthermore, $C$ is strictly convex at every vertex of $C$. Thus by Theorem~\ref{thm:heijor}, $C$ is the boundary of a convex polytope. Since every facet of $C$ is exposed by a linear functional, the boundary complex of this polytope coincides with $C$.

For the spherical case, let $v$ be a vertex of $C$, and let $P$ be the polyhedron in $S^d$ that is obtained by intersecting the halfspaces exposing the facets of $\St(v,C)$. Let $H$ denote a closed halfspace containing $P$, chosen so that $\partial H\cap  \St(v,C)=v$. As $C$ is polytopal, the complex $C$ contains at least one vertex $w$ in the interior of $H$. Consider a central projection $\zeta$ mapping $\intx H$ to $\R^d$. The complex $\zeta(C)$ is in the boundary of a convex polyhedron $K$ in $\R^d$ by Theorem~\ref{thm:heijor}. In particular $K\subseteq \zeta(P)$. Since $K$ is pointed at the vertex $\zeta(w)$, $K$ contains no line, and consequently, $\cl\zeta^{-1}(K)$ contains no antipodal points. Thus \[\cl\zeta^{-1}(K)=\zeta^{-1}(K)\cup v\subsetneq P\] is a polytope in $S^d$. Since every facet of $C$ is exposed by a halfspace, the boundary complex of $\zeta^{-1}(K)$ is~$C$, as desired.
\end{proof}

\begin{example}
A polytopal premanifold $C$ is called \Defn{simple} if for every vertex $v$ of $C$, $\Lk(v,C)$ is a simplex. Consider now any simple, closed and connected $k$-premanifold $C$ in $\R^d$, where $d> k\geq 2$. Since $C$ is simple and connected, it is contained in some affine $(k+1)$-dimensional subspace of $\R^d$. Since $C$ is simple, it is either in locally convex position, or locally flat (i.e.\ locally isometric to $\R^k$). Since $C$ is furthermore compact, only the former is possible. To sum up, $C$ is a $k$-dimensional premanifold that is closed, connected and in locally convex position in some $(k+1)$-dimensional affine subspace. Hence $C$ is in convex position by Theorem~\ref{thm:heij} (cf.~\cite[Sec.\ 11.1, Pr.\ 7]{Grunbaum}).
\end{example}

\begin{definition}[Fattened boundary] Let $C$ be a polytopal $d$-manifold, and let $B$ be a connected component of its boundary. The \Defn{fattened boundary} $\fat(B,C)$ of $C$ at $B$ is the minimal subcomplex of $C$ containing all facets of $C$ that intersect $B$ in a $(d-1)$-face.
\end{definition}

\begin{lemma}[Gluing lemma]\label{lem:convglue}
Let $C$, $C'$ denote two connected polytopal $(d-1)$-manifolds with boundary in $S^d$ or $\R^d,\, d\geq 2 $ with $B:=\partial C= \partial C'$. Assume that $C$ and $C'\cup \fat(B,C)$ are in convex position.
Then $C\cup C'$ is the boundary complex of a convex polytope.
\end{lemma}

\begin{proof}
We proceed by induction on the dimension. First, consider the case $d=2$, which is somewhat different from the case $d> 2$ since Theorem~\ref{thm:heijor} is not applicable. We use the language of curvature of polygonal curves, cf.~\cite{Sullivan}. If $C$, $C'$ are in $S^2$, use a central projection to transfer $C$ and $C'$ to complexes in convex position in $\R^2$. If there are two curves $C$ and $C'$ in convex position in $\R^2$ such that $C'\cup \fat(B,C)$ is in convex position, then $C\sqcup_{\fat(B,C)} C'$ is a 1-dimensional premanifold whose curvature never changes sign, and which is of total curvature less than $4\pi$ since the total curvature of $C'\cup \fat(B,C)$ is smaller or equal to $2\pi$, and $C$ has total curvature less than $2\pi$. Since the turning number of a closed planar curve is a positive integer multiple of $2\pi$, the total curvature of $C\sqcup_{\fat(B,C)} C'$ is $2\pi$. By Fenchel's Theorem~\cite{Fenchel}, $C\sqcup_{\fat(B,C)} C'$ is the boundary of a planar convex body. Since every facet is exposed, the boundary complex of this convex body must coincide with $C\sqcup_{\fat(B,C)} C'=C\cup C'$.

We proceed to prove the lemma for dimension $d> 2$. If $v$ is a vertex of $B$, then $\Lk(v,C\sqcup_{\fat(B,C)} C')$ is obtained by gluing the two complexes $\Lk(v,C)$ and $\Lk(v,C')\cup \Lk(v,\fat(B,C))$ along $\fat(B,C)$.  Each of these is of codimension $1$ and in convex position, so the resulting complex is a polytopal sphere in convex position by induction on the dimension. In particular, $C\sqcup_{\fat(B,C)} C'$ is a premanifold in locally convex position. Thus by Theorem~\ref{thm:heij}, $C\sqcup_{\fat(B,C)} C'=C\cup C'$ is in convex position.
\end{proof}

We will apply Theorem~\ref{thm:heij} in the following version for manifolds with boundary:
\begin{theorem}\label{thm:locglowib}
Let $C$ be a polytopal connected $(d-1)$-dimensional (pre-)manifold in locally convex position in $\R^d$ or in $S^d$ with $d\geq 3$, and assume that for all boundary components $B_i$ of $\partial C$, their fattenings $\fat(B_i,C)$ are (each on its own) in convex position. Then $C$ is in convex position. 
\end{theorem}

\begin{proof}
Consider any boundary component $B$ of $C$ and the boundary complex $\partial \conv \fat(B,C)$ of the convex hull of $\fat(B,C)$, the fattened boundary of $C$ at $B$. The subcomplex $B$ decomposes $\partial \conv \fat(B,C)$ into two components, by the (polyhedral) Jordan--Brouwer Theorem. Consider the component $A$ that does not contain the fattened boundary $\fat(B,C)$ of $C$. The $(d-1)$-complex $A\cup \fat(B,C)\subseteq \partial \conv \fat(B,C)$ is in convex position, and by Lemma~\ref{lem:convglue}, the result $A\sqcup_{\fat(B,C)} C$ of gluing $C$ and $A$ at $B$ is a premanifold in locally convex position.

Repeating this with all boundary components yields a polytopal premanifold without boundary in locally convex position. Thus it is the boundary of a convex polytope, by Theorem~\ref{thm:heij}. Since $C$ is still a subcomplex of the boundary of the constructed convex polytope, $C$ is in convex position.
\end{proof}

\subsection{Iterative construction of CCTs}\label{ssc:expformula}

The main results of this paper were based on an iterative construction of ideal CCTs (Section~\ref{sec:bblocks}). It is natural to ask whether one can provide explicit formulas for this iteration, and indeed, a first attempt to prove Theorem~\ref{mthm:Lowdim} and Theorem~\ref{mthm:projun} would try to understand these iterations in terms of explicit formulas. Since the building block of our construction is Lemma~\ref{lem:cubecmpl}, this amounts to understanding the following problem.

\begin{problemm}\label{prb:iteration}
 Let $Q_1$, $Q_2$, $Q_3$ be three quadrilaterals in some euclidean space (or in some sphere) on vertices $\{a_1,\, a_2,\, a_3,\, a_4\}$, $\{a_1,\, a_4,\, a_5,\, a_6\}$ and $\{a_1,\, a_2,\, a_7,\, a_6\}$, respectively, such that the quadrilaterals do not lie in a common $2$-plane. Give a formula for $a_1$ in terms of the coordinates of the vertices $a_i$,~$i\in \{2,\, \, \cdots,\, 7\}$.
\end{problemm}

It is known and not hard to see that this formula is rational~\cite[Sec.\ 2.1]{BobSur}. The formula is, however, rather complicated, so that it is much easier to follow an implicit approach for the iterative construction of ideal CCTs. 

\medskip

In this section we nevertheless give, without proof, an explicit formula (Formula~\ref{fml:explicit}) to compute, given an ideal $1$-CCT $\CT$ in $S^4_+$, its elementary extension $\CT'$ by solving Problem~\ref{prb:iteration} in $S^4_+$ for cases with a certain inherent symmetry coming from the symmetry of ideal CCTs. More accurately, we provide a rational formula for a map $\operatorname{i}$ that, given two special vertices $a,\, b$ of $\CT$ in layers $0$ and $1$ respectively, obtains a vertex $c:=\operatorname{i}(a,b)$ of layer $2$ of $\CT'$. The map $\operatorname{i}$ is chosen in such a way that we can easily iterate it, i.e.\ in order to obtain a vertex $d$ of the elementary extension $\CT''$ of $\CT'$, we simply compute $\operatorname{i}(b,\rot^{2}_{1,2} c)$ (cf.\ Proposition~\ref{prp:it}).

\begin{rem}
A word of caution: The formula for $\operatorname{i}(a,b)$ is also well-defined for some values of $a,\, b$ for which the extension $\CT'$ of $\CT$ does not exist. In particular, one should be careful not to interpret the well-definedness of $\operatorname{i}(a,b)$ as a direct proof of Theorem~\ref{thm:ext}, rather the opposite: Theorem~\ref{thm:ext} proves that the extensions of ideal CCTs exist, which allows us, if we are so inclined, to use the explicit formula for $\operatorname{i}$ to compute them. For the rest of this section we will simply ignore this problem; we shall assume the extension exists whenever we speak of an extension of a CCT.
\end{rem}

\smallskip
\noindent {\bf Explicit formula for the iteration}
\smallskip

To define $\operatorname{i}$, chose vertices $a\in \RR(\CT,0)$ and $b\in \RR(\CT,1)$ of the ideal $1$-CCT $\CT$ as in Figure~\ref{fig:setupit}, and, to simplify the formula, such that $\langle a,e_4 \rangle$ and $\langle b,e_4 \rangle$ vanish.
\begin{figure}[htbf]
\centering 
  \includegraphics[width=0.64\linewidth]{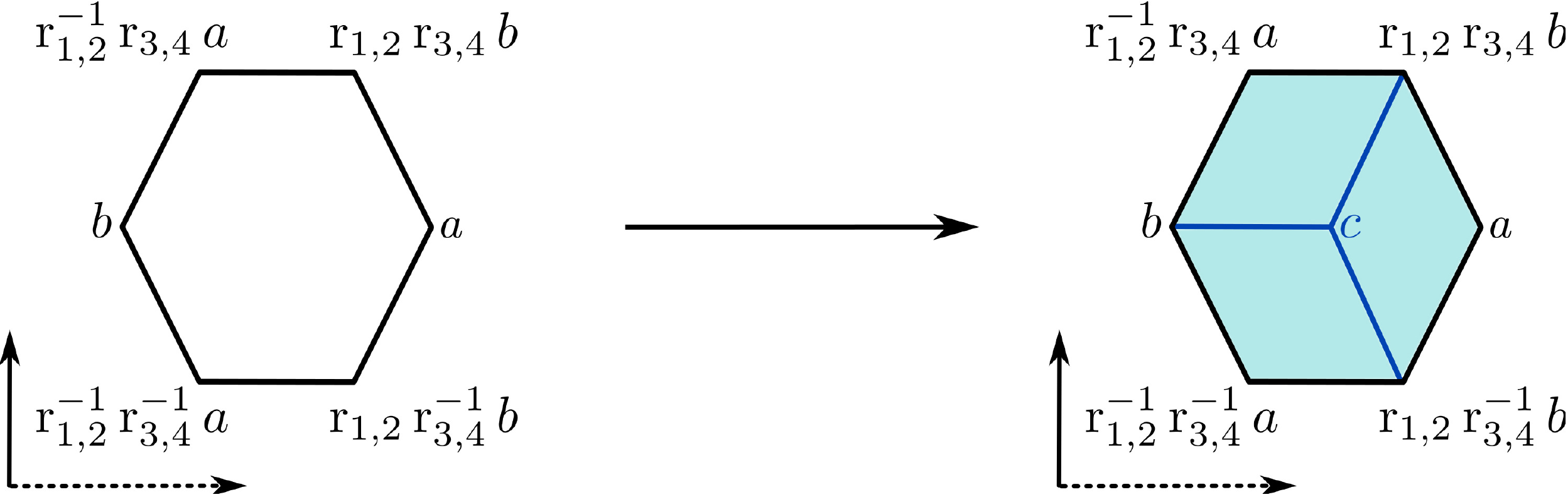} 
  \caption{\small Set-up for an explicit iterative formula; the black vertices and edges are in $\CT$, the blue faces are added in the extension to $\CT'$.} 
  \label{fig:setupit}
\end{figure}
We are going to give the formula for the vertex $c:=\operatorname{i}(a,b) \in \RR(\CT',2)$ as indicated in Figure~\ref{fig:setupit}. Then, it is easy to compute all extensions of an ideal CCTs in~$S^4$ explicitly using iterations of $\operatorname{i}$.
\begin{prp}\label{prp:it}
Let $\CT$ be an ideal $1$-CCT in $S^4_+$, and assume that $a\in \RR(\CT,0)$ and $b\in \RR(\CT,1)$ are chosen as before. Let us denote by $\CT^{[k]}, k\in\N$ the $k$-th layer of the $k$-CCT extending $\CT$. Set $\kappa_0:=a$, $\kappa_1:=b$ and define
\[\kappa_{k+1}:=\rot^{2}_{1,2}\operatorname{i}(\kappa_{k-1},\kappa_{k}).\]
Then $\kappa_{k}\in \CT^{[k]}$ for all $k$.
\end{prp}

\begin{fml}\label{fml:explicit} Consider an ideal $1$-CCT $\CT$ in $S^4_+$. We use homogeneous coordinates for the vertices of~$\CT$. Then, we have the desired formula for $\operatorname{i}$:
\[\operatorname{i}(a,b)=\mu(a,b)a+(1-\mu(a,b)) \frac{\rot_{1,2}\rot_{3,4} b + \rot_{1,2}\rot_{3,4}^{-1} b}{2},\] 
where $\mu(a,b)$ is defined as:
\[\mu(a,b)=\frac{(a_3b_3 - 2b_3^2)b_2 +  (a_3b_3 - b_3^2)b_1-a_2b_3^2}{2(a_3b_3 - b_3^2)a_1 - (2a_3b_3 - b_3^2)a_2 - (2a_3^2 - 3a_3b_3 + b_3^2)b_1 - (2a_3^2 - 3a_3b_3 + 2b_3^2)b_2}\]
\end{fml}

\begin{example}
If $\PS[1]$, as given in Section~\ref{ssc:example}, is our starting CCT for the proof of Theorem~\ref{mthm:Lowdim}, then we can choose \[a=(\sqrt{2}-1,\,1-\sqrt{2},\, 2,\, 0,\, 1)= \vartheta_0\in \RR(\PS[1],0)\] and \[b=(-1,\,0,\, 1,\, 0,\, 1)=\rot^2_{1,2}\vartheta_1\in \RR(\PS[1],1).\]  Then, \[\mu(a,b)=\tfrac{1}{23}(3-4\sqrt{2})\] and \[\operatorname{i}(a,b)= \big(\tfrac{1}{23}(-11+7\sqrt{2}),\,\tfrac{1}{23}(-9-11\sqrt{2}),\, \tfrac{1}{23}(16-6\sqrt{2}),\, 0 ,\, 1\big)=\vartheta_2\in \RR(\PS[2],2).\]
\end{example}

\begin{example}
More generally, by setting $\kappa_0:=\vartheta_0$, $\kappa_1:=\rot^2_{1,2}\vartheta_1$, we can use the iteration procedure of Proposition~\ref{prp:it} to inductively construct all complexes $\PS[n]$. We compute $\kappa_i$ explicitly for $i\leq 10$; the fourth coordinate is always $0$, the fifth is always $1$ so that we omit them in the list. Furthermore we compute the value $\lambda(\kappa_i)$ such that $\pp(\kappa_i)$ lies in $\mathcal{C}_{\lambda(\kappa_i)}$. We constructed the complexes $\PS[n]$ towards $\mathcal{C}_0$, so these values should decrease. In fact, it is not hard to see that $\lambda(\kappa_i)$ converges to $0$.
\renewcommand{\arraystretch}{1.3}
\begin{table}[h!tbf]
\centering
$\begin{array}{|c||c|c|c|c|}
\hline
\text{Vertex}&\text{First coordinate} & \text{Second coordinate} &\text{Third coordinate} & \lambda(\kappa_i)\ \text{(float)} \\
\hline
\hline
\kappa_0 & \sqrt{2}-1&-\sqrt{2}+1 & 2&1.8419\\
\hline
\kappa_1 & -1 & 0& 1&1\\
\hline
\kappa_2 & \frac{-7\sqrt{2}+11}{23} &\frac{11\sqrt{2}+9}{23}& \frac{16-6\sqrt{2}}{23}&0.1709\\ 
\hline
\kappa_3 &\frac{11\sqrt{2}+37}{49} &\frac{6\sqrt{2}-11}{49}& \frac{-12\sqrt{2} + 22}{49}&0.0181\\
\hline
\kappa_4 &\frac{145\sqrt{2} - 241}{697} &\frac{-241\sqrt{2} - 407}{697} &\frac{-168\sqrt{2} + 260}{697}&  1.7906{\cdot}10^{-2} \\ 
\hline
\kappa_5 &\frac{ -192\sqrt{2} - 457 }{679}& \frac{ -111\sqrt{2} + 192 }{679} &\frac{ -138\sqrt{2} + 202}{679}  & 1.7580{\cdot}10^{-3}\\
\hline
\kappa_6 & \frac{-341\sqrt{2} + 577}{1837}& \frac{577\sqrt{2} + 1155}{1837} &  \frac{-324\sqrt{2} + 464}{1837}&1.7247{\cdot}10^{-5}\\
\hline
\kappa_7 &\frac{11471\sqrt{2} + 25057}{38473}& \frac{6708\sqrt{2} - 11471}{38473} & \frac{ -5712\sqrt{2} + 8116}{38473}& 1.6920{\cdot}10^{-6}\\ 
\hline
\kappa_8 & \frac{137\sqrt{2} - 233}{761}&\frac{ -233\sqrt{2} - 487}{761} & \frac{ -96\sqrt{2} + 136}{761}& 1.6598{\cdot}10^{-7}\\
\hline
\kappa_9 & \frac{-165588\sqrt{2} - 353893}{548089}&\frac{-97098\sqrt{2} + 165588}{548089} & \frac{-58344\sqrt{2} + 82564}{548089}&1.6283{\cdot}10^{-8}\\ \hline
\kappa_{10} & \frac{-2955751\sqrt{2} + 5033675}{16549127}&\frac{5033675\sqrt{2} + 10637625}{16549127}&   \frac{-1490520\sqrt{2} + 2108416}{16549127}&1.5974{\cdot}10^{-9}\\
\hline \end{array}$
\end{table}
\end{example}

\subsection{Shephard's list}\label{ssc:Shphrdlist}

Construction methods for projectively unique $d$-polytopes
were developed by Peter McMullen in his doctoral thesis (Birmingham 1968) 
directed by G.\ C.\ Shephard; see~\cite{McMullen}, where
McMullen writes: 
\begin{quote}
 ``Shephard (private communication) has independently made a list, believed to be complete,
	  of the projectively unique $4$-polytopes. All of these polytopes can be constructed by the methods
	  described here.''
\end{quote}	  
If the conjecture is correct, then the following list of eleven projectively unique $4$-polytopes
(all of them generated by McMullen's techniques, duplicates removed) should be complete:

\begin{table}[h!tbf]
\centering
{\small
\begin{tabular}{llllcll}
	  \hline
  & Construction	& dual&	type		& $(f_0,f_1,f_2,f_3)$  & facets\\
	  \hline
	  \hline
$P_1$ &  $\Delta_4$     & $P_1$     & simplicial & (5,10,10,5) & 5\,tetrahedra					\\
$P_2$ &  $\Box * \Delta_1$ & $P_2$ & & (6,11,11,6) & 4\,tetrahedra, 2\,square\,pyramids			\\
$P_3$ &  $(\Delta_2\oplus\Delta_1)*\Delta_0$  & $P_4$  &    &  (6,14,15,7) & 6\,tetrahedra, 1\,bipyramid 	\\
$P_4$ &  $(\Delta_2\times\Delta_1) * \Delta_0$ & $P_3$  &  & (7,15,14,6) & 2\,tetrahedra, 3\,square\,pyramids, 1\,prism  \\
$P_5$  &  $\Delta_3\oplus\Delta_1$ & $P_6$ & {simplicial} & (6,14,16,8)  & 8\,tetrahedra			\\
$P_6$  &  $\Delta_3\times\Delta_1$ & $P_5$ & {simple} & (8,16,14,6)  & 2\,tetrahedra, 4\,prisms		\\
$P_7$  &  $\Delta_2\oplus\Delta_2$ & $P_8$ & {simplicial} & (6,15,18,9)  & 9\,tetrahedra			\\
 $P_8$ &  $\Delta_2\times\Delta_2$ & $P_7$ & {simple} & (9,18,15,6) & 6\,prisms				\\
$P_9$  &  $(\Box,v)\oplus(\Box,v)$ & $P_{10}$ & &(7,17,18,8)   & 4\,square\,pyramids, 4\,tetrahedra\\
$P_{10}$  &  & $P_9$     &  & (8,18,17,7)  & 2\,prisms, 4\,square\,pyramids, 1\,tetrahedron\\
$P_{11}$ & ${\rm v.split}(\Delta_2\times\Delta_1)$  & $P_{11}$  & & (7,17,17,7)  &  3\,tetrahedra, 2\,square\,pyramids, 2\,bipyramids\\
  \hline
  \end{tabular}
}
\end{table}

\subsection{Some technical Lemmas}\label{sec:Lemmas}

\subsubsection{Proof of Lemma~\ref{lem:dihang}}\label{ssc:lemdihang}

We now prove Lemma~\ref{lem:dihang}, which was used to prove that ``locally'' ideal CCTs admit an extension. We do so by first translating it into the language of \Defn{dihedral angles}, and then applying a local-to-global theorem for convexity. We stay in the notation of Lemma~\ref{lem:dihang}: $\CT$ is an ideal CCT of width at least $3$ and $\CT^\circ:=\RR(\CT,[k-2,k])$ denotes the subcomplex of $\CT$ induced by the vertices of the last three layers. $\CT^\circ$ is homeomorphic to $(S^1)^2$.

\begin{definition}[Dihedral angle]
Let $M$ be a $d$-manifold with polytopal boundary in $\R^d$ or $S^d$. Let $\sigma,\, \tau$ be two facets in $\partial M$ that intersect in a $(d-2)$-face. The \Defn{(interior) dihedral angle} at $\sigma\cap \tau$ with respect to $M$ is the angle between the hyperplanes spanned by $\sigma$ and $\tau$, respectively, measured in the interior of $M$.
\end{definition}

With this notion, we formulate a lemma that generalizes Lemma~\ref{lem:dihang}.

\begin{lemma}\label{lem:dihang2}
Let $M$ denote the closure of the component of $S^3_\eq{\setminus} \CT$ that contains $\mathcal{C}_0$. Then the dihedral angles at edges in $\RR(\CT,[k-2,k-1])\subset \CT^\circ$ with respect to $M$ are strictly smaller than $\pi$.
\end{lemma}

\begin{proof}[Lemma~\ref{lem:dihang2} implies Lemma~\ref{lem:dihang}]
Let $v$ denote any vertex of $\RR(\CT,k-2)$, cf.\ Figure~\ref{fig:torusdihan}. 
It follows with Lemma~\ref{lem:dihang2} that all dihedral angles of the $\Lk(v,\CT^\circ)$ are smaller than $\pi$, so $\Lk(v,\CT^\circ)$ is the boundary of the convex triangle $\RN_v^1 M\subset \RN_v^1 S^3_\eq$, proving the first statement of Lemma~\ref{lem:dihang}. 
Furthermore, since $\CT$ is transversal, the tangent direction of $[v,\pi_0(v)]$ at $v$ lies in $\RN_v^1 M=\conv\Lk(v,\CT^\circ)$, proving the second statement.
\end{proof}
 
\smallskip

The rest of this section is consequently dedicated to the proof of Lemma~\ref{lem:dihang2}. We need the following elementary observation.

\begin{obs}\label{obs:geodesictr}
Consider the union $s$ of three segments $[a,b]$, $[b,c]$ and $[a,c]$ on vertices $a$, $b$ and $c$, and any component $B$ of the complement of $s$ in $S^2$. Then the angles between the three segments with respect to $B$ are either all smaller or equal to $\pi$ or all greater or equal to $\pi$. If $a,\,b,\,c$ do not all lie on some common great circle, these inequalities are strict.
\end{obs}

\begin{proof}[\textbf{Proof of Lemma~\ref{lem:dihang2}}]

\begin{figure}[htbf]
\centering 
  \includegraphics[width=0.25\linewidth]{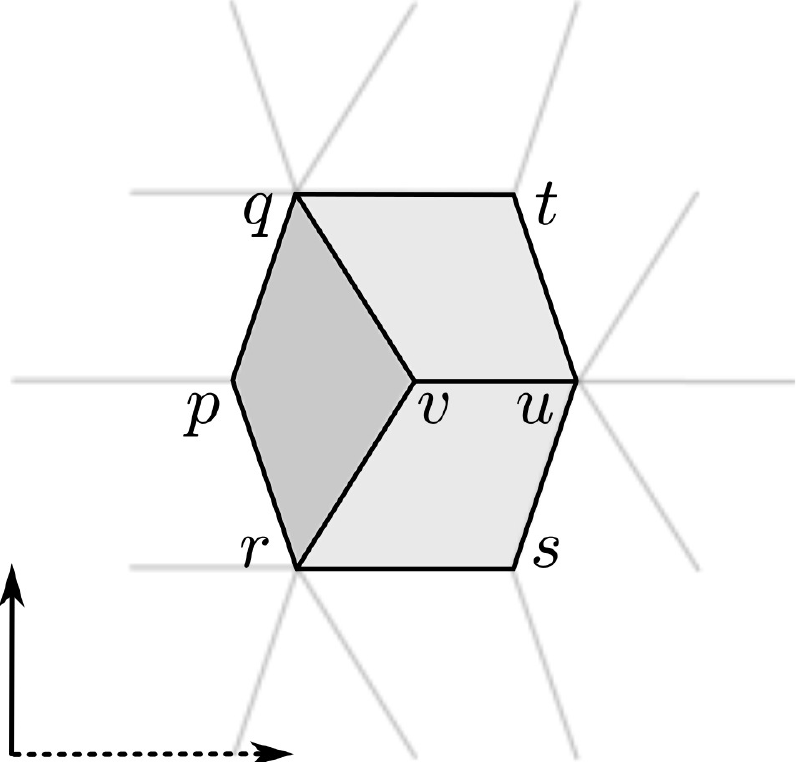} 
  \caption{\small We show a part of the CCT $\CT^\circ\subset \CT$. $v$ is a vertex of layer $k-2$, $q,r,u$ are vertices of layer $k-1$, and the vertices $p,s,t$ lie in layer $k$ of $\CT$.} 
  \label{fig:torusdihan}
\end{figure}
                                                                                                                                                                                                                                                                                                                
Since the facets of $\CT$ are convex, all dihedral angles at vertices of layer $k$ must be larger than $\pi$. Thus the dihedral angles at vertices of layer $k-2$ must be smaller or equal to $\pi$, and at least one of the dihedral angles at each edge of layer $k-2$ must be strictly smaller than $\pi$, since the contrary assumption would imply that $\CT^\circ$ is the boundary of a convex body in~$S^3_\eq$ by Theorem~\ref{thm:heij}, or the more elementary Theorem of Tietze~\cite[Satz 1]{Tietze}, and in particular homeomorphic to a $2$-sphere, contradicting the assumption that $\CT^\circ$ is a torus.

Consider the star $\St(v,\CT^\circ)$ of a vertex $v$ of layer $k-2$ of $\CT$ (Figure~\ref{fig:torusdihan}) for any vertex $v$ of $\RR(\CT,k-2)$. As observed, one of the dihedral angles at edges $[v,u]$, $[v,r]$ and $[v,q]$ must be smaller than $\pi$. In particular, not all vertices of $\Lk(v,\CT^\circ)$ lie on a common great circle, and all dihedral angles at edges $[v,u]$, $[v,r]$ and $[v,q]$ with respect to $M$ are smaller than $\pi$ by Observation~\ref{obs:geodesictr}. 
\end{proof}

\subsubsection{Proof of Proposition~\ref{prp:loccrt1lay}}\label{ssc:localtoglobal}
The goal of this section is to prove Proposition~\ref{prp:loccrt1lay}, which provides a local-to-global criterion for the convex position of ideal CCTs of width $3$. We work in the $4$-sphere $S^4\subset \R^5$, and the equator sphere~$S^3_\eq$. Furthermore, $\pp$ shall denote the projection from $S^4{\setminus} \{\pm e_5\}$ to~$S^3_\eq$.

\begin{lemma}\label{lem:localcrit}
Let $\CT$ be an ideal $3$-CCT in~$S^4$ such that for every facet $\sigma$ of $\CT$, there exists a closed hemisphere $H(\sigma)$ containing $\sigma$ in the boundary and such that $H(\sigma)$ contains all remaining vertices of layers $1,\, 2$ connected to $\sigma$ via an edge of $\CT$ in the interior. Then $H(\sigma)$ contains all vertices of $\F_0(\CT){\setminus} \F_0(\sigma)$ in the interior.
\end{lemma}

\begin{figure}[htbf]
\centering 
\includegraphics[width=0.3\linewidth]{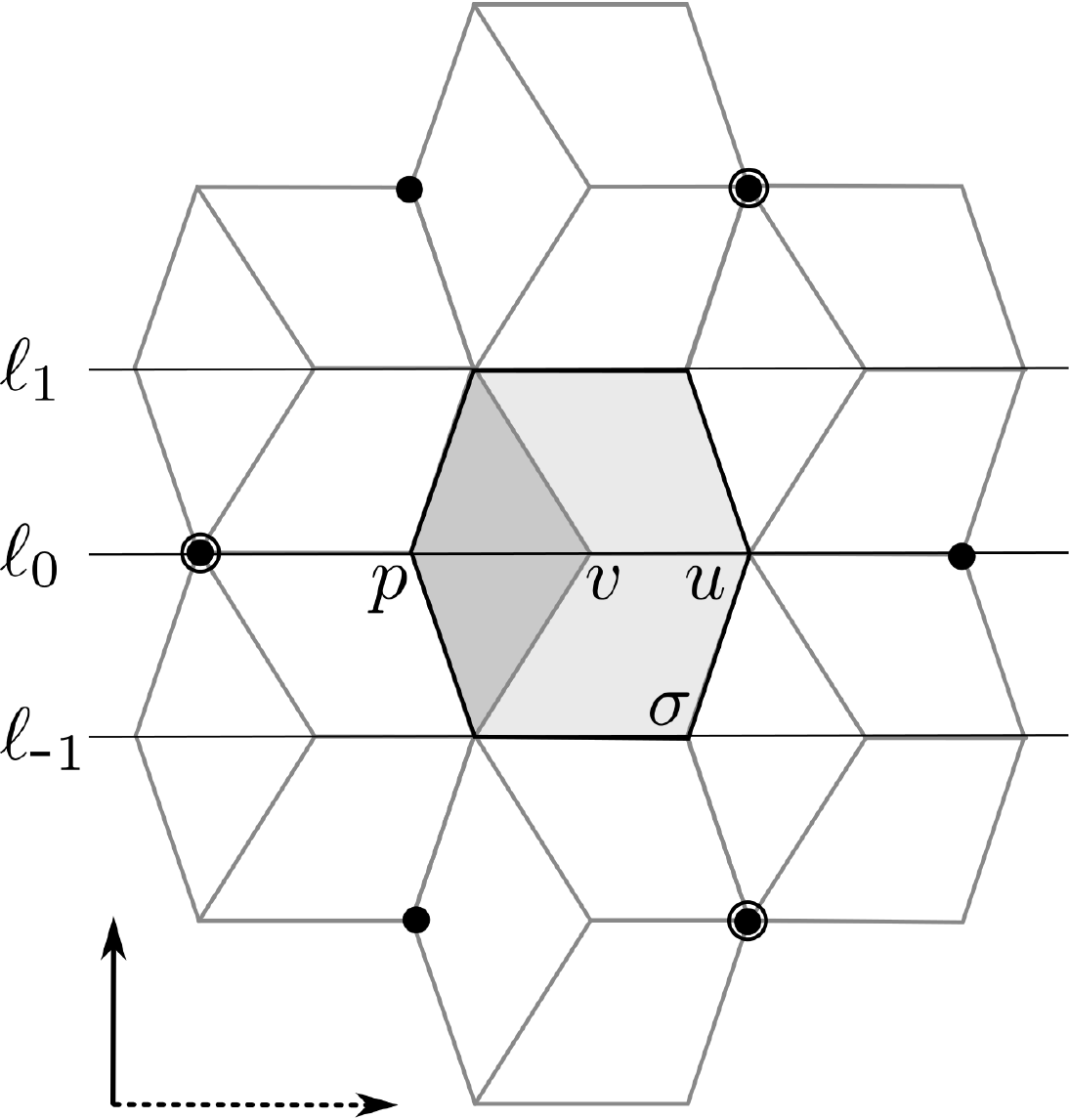} 
\caption{\small The picture shows part of the complex $\RR(\CT,[0,2])$. Lemma~\ref{lem:localcrit} is concerned with the marked by $\bullet$ (connected to $\sigma$ via an edge, in layers $1$ and $2$); Proposition~\ref{prp:loccrt1lay} considers those vertices marked with an additional $\circ$ (connected to $\sigma$ via an edge and in layer $1$).}
\label{fig:localglobalsetup}
\end{figure}

\begin{proof}
We will prove the claim by contraposition. Let $H=H(\sigma)$ denote a hemisphere in~$S^4$ containing $\sigma$ in the boundary, with outer normal $\n=\n(H)$. To prove the claim of the lemma, assume that $\widetilde{H}:=S^4{\setminus} \mathrm{int}H$ contains a vertex $w$ of $\CT$ that is not a vertex of $\sigma$. We have to prove that there is another vertex that
\begin{compactenum}[(i)]
\item lies in layers $1$ or $2$ of $\CT$,
\item lies in $\widetilde{H}$,
\item is not a vertex of $\sigma$, but that
\item is connected to $\sigma$ via some edge of $\CT$.
\end{compactenum}
\noindent We consider this problem in four cases. Let $\sigma$ be any facet of $\CT$, and let us denote the  orthogonal projection of $x\in \R^5$ to the $\Sp\{e_i,\,e_j\}$-plane in $\R^5$ by $x_{i,j}$. By symmetry, $\vv{n}=(\ast,\,\ast,\,\e v_3,\, \e v_4,\,\ast)$, where $v=v(\sigma)=(v_1(\sigma),\, \cdots,\,v_5(\sigma))$ denotes the only layer $0$ vertex of $\sigma$, and $\e=\e(H)$ is some real number. We can take care of the easiest case right away:

\medskip \textbf{Case (0)} If $w$ lies in facet ${\tau}$ of $\CT$ that is obtained from $\sigma$ by a rotation of the $\Sp\{e_3,\,e_4\}$-plane, then $\e=\e(H)\leq 0$ by Proposition~\ref{prp:alignsymm}(b) and (d). Thus $\widetilde{H}$ must contain ${\tau}$, and all other facets of $\CT$ obtained from $\sigma$ by rotation of the $\Sp\{e_3,\,e_4\}$-plane. In particular, it contains all vertices of adjacent facets that are obtained from $\sigma$ by a rotation of the $\Sp\{e_3,\,e_4\}$-plane, among which we find the desired vertex, even a vertex satisfying (i) to (iv) among the vertices of $\RR(\CT,1)$. We may assume from now on that $\e$ is positive.

\medskip

It remains to consider the case in which $w$ satisfies (ii) and (iii) and is obtained from a vertex of $\sigma$ only from a nontrivial rotation of the $\Sp\{e_1,\,e_2\}$-plane followed by a (possibly trivial) rotation of the $\Sp\{e_3,\,e_4\}$-plane. Since $\e$ is positive, there exists a vertex $w'$ satisfying (ii) and (iii) in the same layer of $\CT$ as $w$, but which lies in \[\ell_0=\pp^{-1}\pi_2^{\operatorname{f}}(\pp(v)),\ \ell_1=\pp^{-1}\pi_2^{\operatorname{f}}(\pp(q))=\rot_{3,4} \ell_0\ \text{or}\  \ell_{-1}=\pp^{-1}\pi_2^{\operatorname{f}}(\pp(r))=\rot_{3,4}^{-1}\ell_0.\]
The existence of a vertex satisfying (i)-(iv) now follows from the following observation:

\smallskip
\noindent \emph{Let $x,y$ be any two points in $S^1$ not antipodal to each other, and let $m$ be any point in the segment $[x,y]$. Assume $n$ is any further point in $S^1$ such that $\langle n,m\rangle\leq \langle n,-m\rangle$. Then $\langle n,y\rangle\leq \langle n,-y\rangle$ or $\langle n,x\rangle\leq \langle n,-x\rangle$.}
\smallskip

\medskip \textbf{Case (1)} Assume $w'\in\ell_0$. This case is only nontrivial if $w'$ is not in $\RR(\CT,[1,2])$. Thus assume (w.l.o.g.) that $w'$ lies in layer $0$ (i.e.\ it is the vertex circled in Figure~\ref{fig:localcglobal}(1)), the other case is fully analogous. Then $\rot_{1,2}^2 w'=v$. To construct the desired vertex $x$, note that $\pi_2(\pp(v))$ lies in the segment $[\pi_2(\pp(u)),\pi_2(\pp(p))]$ by Proposition~\ref{prp:alignsymm}(e). Now, \[w'_{1,2}=-v_{1,2},\, w'_{3,4}=v_{3,4}\ \text{and}\ w'_{5}=v_{5},\] and since $w'\in \widetilde{H}$, $\langle \n,v_{1,2}\rangle\leq \langle \n,w'_{1,2}\rangle$. Consequently, for $x={u}$ or $x={p}$, $\langle \n,x_{1,2}\rangle\leq \langle \n,-x_{1,2}\rangle$ and thus \[\langle \n,x\rangle\leq \langle \n,\rot_{1,2}^2 x\rangle\ \Longleftrightarrow \ \rot_{1,2}^2 x\in \widetilde{H}.\] 
Since both $\rot_{1,2}^2{u}$ and $\rot_{1,2}^2{p}$ satisfy (i), (iii) and (iv), this vertex satisfies the properties (i) to (iv).

\medskip \textbf{Case (2)} Assume $w'\in \RR(\CT,[1,2])\cap (\ell_1\cup\ell_{-1})$. Then $w'$ is obtained from a vertex $y$ of $\sigma$ by a rotation of the $\Sp\{e_1,\,e_2\}$-plane, i.e.\ $w'=\rot_{1,2}^2y$. It then follows, as in Case (1), that $\rot_{1,2}^2{u}$ or $\rot_{1,2}^2{p}$ lie in $\widetilde{H}$, since $\pi_2(\pp(y))\in[\pi_2(\pp(u)),\pi_2(\pp(p))]$ by Proposition~\ref{prp:alignsymm}(e). This vertex satisfies the properties (i) to (iv), because both $\rot_{1,2}^2{u}$ and $\rot_{1,2}^2{p}$ satisfy (i), (iii) and (iv).

\begin{figure}[htbf]
\centering 
 \includegraphics[width=1\linewidth]{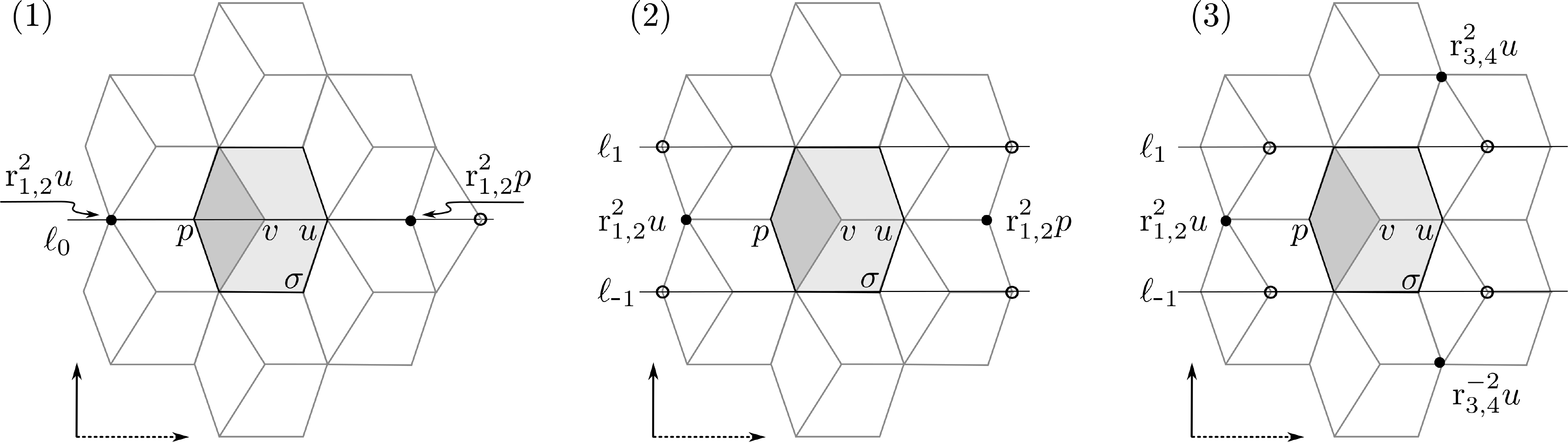} 
\caption{\small The three cases show how to, given a vertex satisfying (ii) and (iii) (circled black), obtain a vertex satisfying (i)-(iv) (found among the vertices marked with a black disk).}
  \label{fig:localcglobal}
\end{figure}

\medskip \textbf{Case (3)} 
If $w'\in \RR(\CT,\{0\}\cup\{3\})\cap (\ell_1\cup\ell_{-1})$, then it must lie in a $2$-face $F$ which intersects $\sigma$ in an $1$-face (cf.\ Figure~\ref{fig:localcglobal}(3) for the case of layer $0$ vertices). The remaining vertex of $F$ that does not lie in $\sigma$ is a vertex of layers $1$ or $2$. Since $e\subseteq \partial \widetilde{H}$ and $w'\in \widetilde{H}$, this vertex lies in $\widetilde{H}$ and must be connected to $\sigma$ via an edge, and consequently satisfies properties (i) to (iv), as desired.
\end{proof}

We can now prove Proposition~\ref{prp:loccrt1lay}.

\begin{proof}[\textbf{Proof of Proposition~\ref{prp:loccrt1lay}}] For the labeling of vertices, we refer to Figure~\ref{fig:localglobal2}(1). Staying in the notation of Lemma~\ref{lem:localcrit} and using its result we only have to prove that if $\sigma$ is a facet of $\CT$ and $H(\sigma)$ with $\sigma\subseteq \partial H(\sigma)$ is a hemisphere such that the vertices $\rot^2_{1,2}{u},\, \rot^2_{3,4}{u}$ and $\rot^{\,-2}_{3,4}{u}$ of $\RR(\CT,1)\subset \CT$ lie in the interior of $H(\sigma)$, then the vertices $\rot^2_{1,2}{p},$ $ \rot^2_{3,4}{p}$ and $\rot^{\,-2}_{3,4}{p}$ of $\RR(\CT,2)\subset \CT$ lie in $\intx H(\sigma)$ as well. Let $\n=\n(\sigma)$ denote the outer normal to $H(\sigma)$.

As already observed in the proof of Lemma~\ref{lem:localcrit}, $\n$ is of the form $\n=(\ast,\, \ast,\,\e v_3,\, \e v_4,\, \ast)$, and if $\rot^2_{3,4}{u}\in \intx H(\sigma)$, then $\e>0$ and \[\langle\n,\rot^{\,\pm 2}_{3,4}{p}\rangle<\langle\n,{p}\rangle\ \Longleftrightarrow\ \rot^{\,\pm 2}_{3,4}{p}\in \intx H(\sigma) .\]
Thus it remains to prove that $w=\rot^2_{1,2}{p}$ lies in $\intx H(\sigma)$, which we will do in two steps. If $x$ and $y$ lie in $S^4{\setminus} \{z\in S^4:\ z_{1,2}=0\}$, let $\operatorname{d}(x,y)$ measure the distance between $\pi_0(\pp(x))$ and $\pi_0(\pp(y))$ in~$S^4$.

\begin{compactenum}[\rm(a)]
\item We prove that $\operatorname{d}({p},q)$ is smaller than $\nicefrac{\pi}{4}$.
\item From $\operatorname{d}({p},q)<\nicefrac{\pi}{4}$ we conclude that $w=\rot_{1,2}^2 {p}$ lies in the interior of $H(\sigma)$.
\end{compactenum}
\noindent To see the inequality of (a), let $m$ denote the midpoint of the segment $[r,q]$. Consider the unique point $m'$ in $\mathcal{C}_0$ such that $\pp(m)\in[\pp({u}),m']$, which is guaranteed to exist since $\pi_2(\pp({u}))=\pi_2(\pp(m))$ (Proposition~\ref{prp:alignsymm}(b) and (d)). We locate $\pp(p)$ with respect to $\pp({m})$, $\pi_0(\pp({m}))$ and $m'$: 
\begin{compactitem}[$\circ$]
\item By Proposition~\ref{prp:alignsymm}(f), $\pp({p})$ and $m'$ lie in the same component of $S^3_\eq{\setminus} \pi^{\SSp}_0(\pp(q))$.
\item Furthermore, the complement of $\SSp\{\pp(r),\, \pp(q),\, \pp({u})\}$ in~$S^3_\eq$ contains $\pi_0(\pp(m))$ and $\pp({p})$ in the same component, since $\pp(v)$ and $\pi_0(\pp(m))$ lie in different components.
\item Finally, by (b) and (d) of the same Proposition, $\pi_2(\pp(m))=\pi_2(\pp({p}))$. 
\end{compactitem}

\begin{figure}[htbf]
\centering 
 \includegraphics[width=0.9\linewidth]{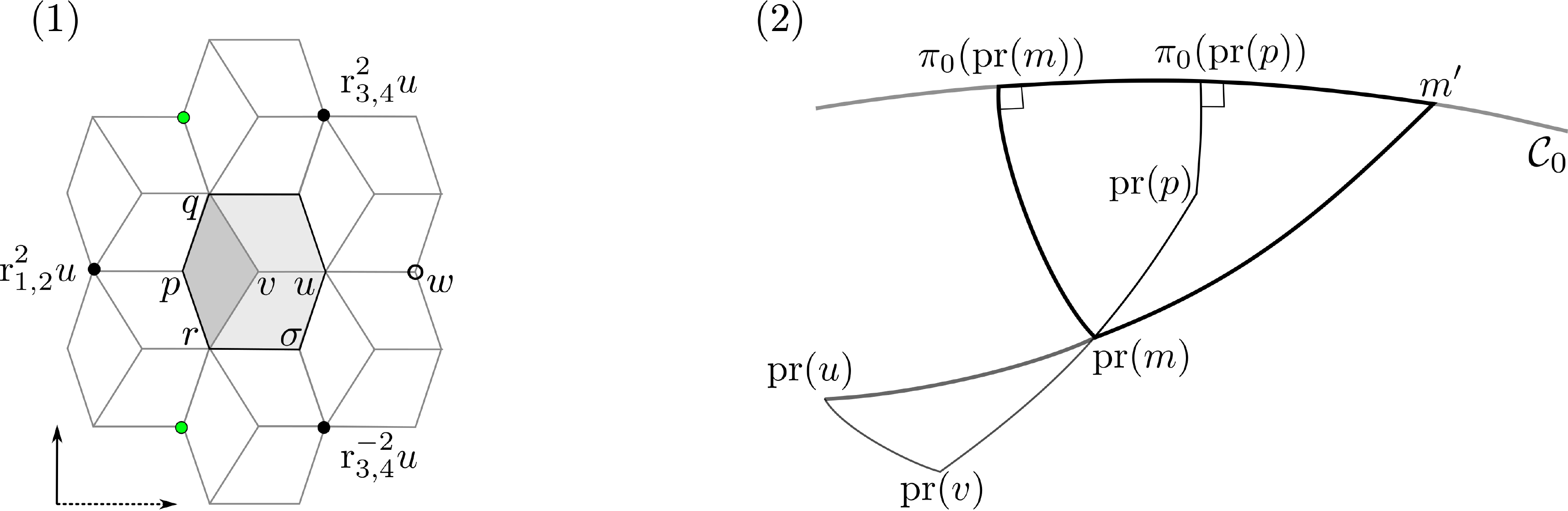} 
\caption{\small  (1) As before, the picture shows part of the complex $\RR(\CT,[0,2])$. We have to show that if all layer $1$ vertices satisfying (iii) and (iv) lie in $H(\sigma)$ (full black disks), then so do all vertices of layer $2$ satisfying (iii) and (iv) (circled black). For the vertices of layer $2$ colored green, this follows as in Lemma~\ref{lem:localcrit}.
\newline (2) Illustration for inequality $\operatorname{d}({p},q)\leq \operatorname{d}(m',q)$.} 
  \label{fig:localglobal2}
\end{figure}

\noindent Thus $\pp({p})$ is contained in the triangle on vertices $m'$, $\pi_0(\pp(m))$ and $\pp(m)$ in $\pi_2^{\SSp}(m)$, see also Figure 
\ref{fig:localglobal2}. As the projection of $\pp(m)$ to $\mathcal{C}_0$ lies in the segment $[\pi_0(\pp(m)),m']$, $\pi_0(\pp({p}))$ lies in the segment from $\pi_0(\pp(q))=\pi_0(\pp(m))$ to $m'$. Thus \[\operatorname{d}({p},q)=\operatorname{d}({p},m)\leq \operatorname{d}(m',m)=\operatorname{d}(m',q).\] To compute the latter, notice that, after applying rotations of planes $\Sp\{e_1,\,e_2\}$ and $\Sp\{e_3,\,e_4\}$, we may assume that \[{u}=\big(u_1,0,u_3,0,1\big),\ \   u_1,\, u_3>0.\]
Then the coordinates of $q$ and $r$ are given as \[\big(0,-u_1,\tfrac{1}{2}u_3,\pm\tfrac{\sqrt{3}}{2}u_3,1\big)\ \text{and} \ m'=\tfrac{1}{\sqrt{5}}\big(1,-2,0,0,0\big).\] 
In particular, \[\operatorname{d}({p},q)\leq \operatorname{d}(m',q)=\arctan \tfrac{1}{2}<\tfrac{\pi}{4}.\]
As for step (b): After applying rotations of planes $\Sp\{e_1,\,e_2\}$ and $\Sp\{e_3,\,e_4\}$ of $\CT$, we may assume that $(u_1,0,u_3,0,1)$, as above. Then, as before, the coordinates of the remaining layer $1$ vertices $q$ and $r$ of $\sigma$ are $(0,-u_1,\nicefrac{1}{2}u_3,\pm\nicefrac{\sqrt{3}}{2}u_3,1)$. A straightforward calculation shows that any normal to $\conv\{{u},\,q,\,r\}$ is a dilate of 
\[\vv{n}=\big(\mu,\,-(\mu+\tfrac{u_3}{2u_1}),\,1,\,0,\,n_5\big),\ \ \mu\in \R,\ n_5=-\mu u_1-u_3\]
and if $\vv{n}$ is the outer normal to a hemisphere $H(\sigma)$ that exposes the triangle $\conv\{{u},\,q,\,r\}$ among all vertices of $\RR(\CT,1)$ connected to $\sigma$, then the dilation is by a positive real and $\mu>0$. 

Since $\CT$ is transversal, Proposition~\ref{prp:alignsymm}(e) gives ${p}_1<0$. If additionally $\operatorname{d}({p},q)<\nicefrac{\pi}{4}$, then ${p}_2<{p}_1<0$. In particular, we have that
\[2\mu({p}_1-{p}_2)-\tfrac{u_3}{u_1}{p}_2>0.\]
Due to the fact that $w_{1,2}=-{p}_{1,2},\, w_{3,4}={p}_{3,4}$ and $w_{5}={p}_{5}=1$, this is equivalent to 
\[
\langle\vv{n},{p}\rangle={p}_1 \mu-\big(\mu+\tfrac{u_3}{2u_1}\big){p}_2+{p}_3+n_5>-{p}_1 \mu + \big(\mu+\tfrac{u_3}{2u_1}\big){p}_2+{p}_3+n_5= \langle\vv{n},w\rangle,
\]
and consequently $w$ is in the interior of $H(\sigma)$.
\end{proof}

\bibliographystyle{myamsalpha}
\bibliography{Ref}
\addcontentsline{toc}{part}{Bibliography}

\chapter*{Appendix}
\addcontentsline{toc}{part}{Appendix}

\section*{Curriculum Vitae}
\addcontentsline{toc}{chapter}{Curriculum Vitae}

\begin{center}
{\large \textbf{Personal information}} 
\end{center}

  \begin{table}[h!tbf]
  \centering
  {{\bf
  \begin{tabular}{|rl|}
  	  \hline
  Name &  Karim Alexander Adiprasito \\
  Nationality  &  German \\
  Birthdate &  26.05.1988	\\
  Birthplace & Aachen, Germany \\
  E-mail  & \url{sunya00@gmail.com},  \url{adiprasito@math.fu-berlin.de}\\
    \hline
    \end{tabular}}}
  \end{table}

\begin{center}
{\large \textbf{Education}} 
\end{center}

\begin{table}[h!tbf]
\centering
{
\begin{tabular}{|rl|}
	  \hline
1999--2007 &  \textbf{High School diploma} \emph{Bergstadt Gymnasium L\"udenscheid} \\
2007-2010  &   \textbf{Diploma in Mathematics} \emph{Technische Universit\"at Dortmund}\\
2008-2009  &   \textbf{International Visiting Student} \emph{Indian Institute of Technology, Bombay}\\
2010-2013  &   \textbf{Ph.D. in Mathematics} \emph{Freie Universit\"at Berlin}\\
Nov.\ 2011-Feb.\ 2012 &   \textbf{Guest Researcher} \emph{Hebrew University Jerusalem}\\
  \hline
  \end{tabular}}
\end{table}

\newpage

\section*{Zusammenfassung}

\addcontentsline{toc}{chapter}{Zusammenfassung}

In dieser Arbeit werden Methoden der metrischen Geometrie, der Differentialgeometrie und der kombinatorischen Topologie benutzt um klassische Probleme in der Theorie der Polytope, der Theorie der polytopalen Komplexe und der Theorie der Unterraumarrangements zu l\"osen. 

\smallskip

Wir beginnen mit einer Anwendung von Alexandrov's Konzept von Räumen von nach oben beschränkter Krümmung auf die (verallgemeinerte) Vermutung von Hirsch den Durchmesser von Polyedern betreffend. Wir zeigen insbesondere dass die Vermutung für eine weite Klasse von Simplizialkomplexen gilt. Dies wird erreicht indem wir dem unterliegenden Raum des Komplexes eine Längenmetrik zuweisen, die es uns erlaubt kombinatorische Pfade und kürzeste geometrische Kurven miteinander in Verbindung zu setzen.

\smallskip

In dem zweiten Teil dieser Dissertation widmen wir uns der Anwendungen und Problemen der diskreten Morsetheorie. Ein zentrales Problem dieser Theorie ist es, gute Morsefunktionen zu finden, das hei{\ss}t, Morsefunktionen die möglichst viel über die Topologie des untersuchten Raumes aussagen. Wir beginnen diesen Teil mit der Präsentation von Fortschritten in Fragen und Vermutungen von Lickorish, Goodrick und Hudson aus der PL Topologie die dieses Problem betreffen. Diese Resultate, sowie die Resultate aus dem ersten Teil meiner Arbeit, wurden gemeinsam mit Bruno Benedetti erarbeitet.

Die Hauptresultate dieses Teiles meiner Doktorarbeit jedoch beschäftigen sich mit der Untersuchung von Komplementen von Unterraumarrangements: Wir verallgemeinern Resultate von Hattori, Falk, Dimca--Papadima, Salvetti--Settepanella und anderen indem wir zeigen, dass das Komplement eines $2$-Arrangements ein minimaler topologischer Raum ist. Mit anderen Worten, wir zeigen dass das Komplement jedes $2$-Arrangements einen homotopie\"aquivalenten CW Komplex hat der genauso viele Zellen der Dimension $i$ aufweist wie durch die $i$-te rationale Bettizahl des Komplex vorgegeben. Um dies zu erreichen verwenden wir diskrete Morsetheorie, angewendet auf die von Bj\"orner und Ziegler für Arrangements definierten Modellkomplexe.

\smallskip

Wir schlie{\ss}en die Arbeit ab indem wir uns einer Frage widmen die bis auf die Forschung von Legendre und Steinitz zur\"uckgeht. Diese beschäftigten sich mit dem Problem, eine Formel zu finden die, abhängig von dem $f$-vektor eines Polytopes, die Dimension des dem Polytop zugeordneten Realisierungsraumes bestimmt. Legendre und Steinitz bestimmten eine solche Formel f\"ur Polytope der Dimension $3$. Dies wirft die nat\"urliche Frage auf ob eine \"ahnliche Formel auch f\"ur h\"oherdimensionale Polytope gilt. Basierend auf einer gemeinsamen Arbeit mit G\"unter M. Ziegler beantworten wir diese Frage negativ, indem wir beweisen dass es eine unendliche Familie von $4$-dimensionalen Polytopen gibt deren Realisierungsr\"aume uniform beschr\"ankte Dimension aufweisen. Darauf aufbauend beantworten wir eine Frage von Perles und Shephard: wir konstruieren eine unendliche Familie von $69$-dimensionalen projektiv eindeutigen Polytopen. 

Die hier entwickelten Methoden sind des weiteren von Nutzen um zu einer Vermutung von Shephard ein Gegenbeispiel zu konstruieren; dieses letzte Resultat basiert auf gemeinsamer Arbeit mit Arnau Padrol.

\newpage
\vspace*{\fill}
\begingroup

\section*{Eidesstattliche Erkl\"arung}

\addcontentsline{toc}{chapter}{Eidesstattliche Erkl\"arung}

Gem\"a\ss\ \S7 (4), der Promotionsordnung des Fachbereichs Mathematik und Informatik der
Freien Universit\"at Berlin versichere ich hiermit, dass ich alle Hilfsmittel und Hilfen 
angegeben und auf dieser Grundlage die Arbeit selbst\"andig verfasst habe.
Des Weiteren versichere ich, dass ich diese Arbeit nicht schon einmal zu 
einem fr\"uheren Promotionsverfahren eingereicht habe.

\bigskip
\bigskip
\bigskip
\bigskip

Berlin, den 

\bigskip
\bigskip
\bigskip
\bigskip

\hfill Karim Adiprasito

\endgroup
\vspace*{\fill}

\vfill


\end{document}